\documentclass{amsart}

\usepackage[all]{xy}

\usepackage{hyperref}

\usepackage{cleveref}

\usepackage{amssymb}

\usepackage{todonotes}

\usepackage{chngcntr}
\usepackage{apptools}
\AtAppendix{\counterwithin{theorem}{section}}

\usepackage{relsize}
\usepackage[bbgreekl]{mathbbol}
\usepackage{amsfonts}

\DeclareSymbolFontAlphabet{\mathbb}{AMSb}
\DeclareSymbolFontAlphabet{\mathbbl}{bbold}
\newcommand{\prism}{{\mathlarger{\mathbbl{\Delta}}}}

\theoremstyle{plain}
\newtheorem*{thm}{Theorem}
\newtheorem{theorem}{Theorem}[subsection]
\newtheorem{proposition}[theorem]{Proposition}
\newtheorem{lemma}[theorem]{Lemma}
\newtheorem{corollary}[theorem]{Corollary}

\newtheorem{question}[theorem]{Question}
\theoremstyle{definition}
\newtheorem{definition}[theorem]{Definition}
\newtheorem{example}[theorem]{Example}
\newtheorem{remark}[theorem]{Remark}
\newtheorem{notation}[theorem]{Notation}

\newcommand{\crys}{\mathrm{crys}}

\newcommand{\dR}{\mathrm{dR}}

\newcommand{\F}{\mathbb{F}}

\newcommand{\Fil}{\mathrm{Fil}}

\newcommand{\Gm}{{\mathbb{G}_m}}

\newcommand{\Id}{\textrm{Id}}

\renewcommand{\inf}{\mathrm{inf}}

\newcommand{\N}{\mathbb{N}}

\newcommand{\pris}{\mathrm{pris}}
\newcommand{\Q}{\mathbb{Q}}

\newcommand{\Spa}{\mathrm{Spa}}
\newcommand{\Spec}{\mathrm{Spec}}
\newcommand{\Spf}{\mathrm{Spf}}
\newcommand{\tilxi}{\tilde{\xi}}

\newcommand{\Z}{\mathbb{Z}}

\newcommand{\Win}{\mathrm{Win}}
\newcommand{\Ext}{\mathcal{E}xt}
\newcommand{\BT}{\mathrm{BT}}
\newcommand{\zp}{\mathbb{Z}_p}

\newcommand{\DM}{\mathrm{DM}}
\newcommand{\DF}{\mathrm{DF}}

\newcommand{\purple}[1]{{\color{black} #1}}
\newcommand{\red}[1]{{\color{black} #1}}

\newcommand{\green}[1]{{\color{black}#1}}
\newcommand{\orange}[1]{{\color{black}#1}}
\newcommand{\blue}[1]{{\color{black}#1}}

\begin{document}

\title{Prismatic Dieudonn\'e theory}

\author{Johannes Ansch\"{u}tz}
\address{Mathematisches Institut, Universit\"at Bonn, Endenicher Allee 60, 53115 Bonn, Deutschland}
\email{ja@math.uni-bonn.de}
\thanks{During this project, J.A.\ was partially supported by the ERC 742608, GeoLocLang.}

\author{Arthur-C\'esar Le Bras}
\address{Universit\'e Sorbonne Paris Nord, LAGA, C.N.R.S., UMR 7539, 93430 Villetaneuse,
France}
 \email{lebras@math.univ-paris13.fr}

\begin{abstract}
We define, for each quasi-syntomic ring $R$ (in the sense of Bhatt-Morrow-Scholze), a category \red{$\DM^{\rm adm}(R)$ of \textit{admissible prismatic Dieudonn\'e crystals over $R$} and a functor from $p$-divisible groups over $R$ to $\DM^{\rm adm}(R)$}. We prove that this functor is an antiequivalence. Our main cohomological tool is the prismatic formalism recently developed by Bhatt and Scholze.
\end{abstract}

\maketitle
\tableofcontents

\newpage

\section{Introduction}
\label{sec:introduction}

Let $p$ be a prime number. The goal of the present paper is to establish classification theorems for $p$-divisible groups over \textit{quasi-syntomic rings}. This class of rings is a non-Noetherian generalization of the class of $p$-complete locally complete intersection rings and contains also big rings, such as perfectoid rings. Our main theorem is as follows.

\begin{thm}
Let $R$ be a quasi-syntomic ring. There is a natural functor from the category of $p$-divisible groups over $R$ to the category \red{$\DM^{\rm adm}(R)$ of admissible} prismatic Dieudonn\'e crystals over $R$, \purple{which is an antiequivalence}.
\end{thm}

A more precise version of this statement and a detailed explanation will be given later in this introduction. For now, let us just say that the category \red{$\DM^{\rm adm}(R)$} is formed by objects of semi-linear algebraic nature. 

The problem of classifying $p$-divisible groups and finite locally free group schemes by semi-linear algebraic structures has a long history, going back to the work of Dieudonn\'e on formal groups over characteristic $p$ perfect fields. In characteristic $p$, as envisionned by Grothendieck, and later developed by Messing (\cite{messing_the_crystals_associated_to_barsotti_tate_groups}), Mazur-Messing (\cite{mazur_messing_universal_extensions_and_one_dimensional_crystalline_cohomology}), Berthelot-Breen-Messing (\cite{berthelot_breen_messing_theorie_de_dieudonne_cristalline_II}, \cite{berthelot_messing_theorie_de_dieudonne_cristalline_III}), the formalism of crystalline cohomology provides a natural way to attach such invariants to $p$-divisible groups. This theory goes by the name of \textit{crystalline Dieudonn\'e theory} and leads to classification theorems for $p$-divisible groups over a characteristic $p$ base in a wide variety of situations, which we will not try to survey but for which we refer the reader, for instance, to \cite{lau_divided_dieudonne_crystals}. In mixed characteristic, the existing results have been more limited. Fontaine (\cite{fontaine_groupes_p_divisibles_sur_les_corps_locaux}) obtained complete results when the base is the ring of integers of a finite totally ramified extension $K$ of the ring of Witt vectors $W(k)$ of a perfect field $k$ of characteristic $p$, with ramification index $e<p-1$. This ramification hypothesis was later removed by Breuil (\cite{breuil_groupes_p_divisibles_groupes_finis_et_modules_filtres}) for $p>2$, who also conjectured an alternative reformulation of his classification in \cite{breuil_schemas_en_groupes_et_corps_de_normes}, simpler and likely to hold even for $p=2$, which was proved by Kisin (\cite{kisin_crystalline_representations_and_f_crystals}),  for odd $p$, and extended by Kim (\cite{kim_the_classification_of_p_divisible_groups_over_2_adic_discrete_valuation_rings}), Lau (\cite{lau_relations_between_dieudonne_displays_and_crystalline_dieudonne_theory}) and Liu (\cite{liu_the_correspondence_between_barsotti_tate_groups_and_kisin_modules_when_p_equal_2}) to all $p$. Zink, and then Lau, gave a classification of \textit{formal} $p$-divisible groups over very general bases using his theory of \textit{displays} (\cite{zink_the_display_of_a_formal_p_divisible_group}). More recently, $p$-divisible groups have been classified over perfectoid rings (\cite{lau_dieudonne_theory_over_semiperfect_rings_and_perfectoid_rings}, \cite[Appendix to Lecture XVII]{scholze2020berkeley}). 

The main interest of our approach is that it gives a uniform and geometric construction of the classifying functor on quasi-syntomic rings. This is made possible by the recent spectacular work of Bhatt-Scholze on \textit{prisms} and \textit{prismatic cohomology} (\cite{bhatt_scholze_prisms_and_prismatic_cohomology}, \cite{bhatt_lectures_on_prismatic_cohomology}). So far, such a cohomological construction of the functor had been available only in characteristic $p$, using the crystalline theory. This led in practice to some restrictions, when trying to study $p$-divisible groups in mixed characteristic by reduction to characteristic $p$, of which Breuil-Kisin theory is a prototypical example : there, no direct definition of the functor was available when $p=2$! Replacing the crystalline formalism by the prismatic formalism, we give a definition of the classifying functor very close in spirit to the one used by Berthelot-Breen-Messing (\cite{berthelot_breen_messing_theorie_de_dieudonne_cristalline_II}) and which now makes sense without the limitation to characteristic $p$. Over a quasi-syntomic ring $R$, our functor takes values in the category of \textit{\red{admissible} prismatic Dieudonn\'e crystals} over $R$.
As the name suggests, \textit{prismatic Dieudonn\'e crystals} are prismatic analogues of the classical notion of a Dieudonn\'e crystal on the crystalline site. 

Before stating precisely the main results of this paper and explaining the techniques involved, let us note that several natural questions are not addressed in this paper. 
\begin{enumerate}
\item It would be interesting to go beyond quasi-syntomic rings. By analogy with the characteristic $p$ story, one would expect that the prismatic theory should also shed light on more general rings. In the general case, \red{admissible} prismatic Dieudonn\'e crystals will not be the right objects to work with. One should instead define analogues of the \textit{divided} Dieudonn\'e crystals introduced recently by Lau \cite{lau_divided_dieudonne_crystals} in characteristic $p$.
\item Even for quasi-syntomic rings, our classification is explicit for the so called \textit{quasi-regular semiperfectoid} rings or for complete regular local rings with perfect residue field of characteristic $p$ (cf.\ \Cref{sec:comp-case-mathc}), as will be explained below, but quite abstract in general. Classical Dieudonn\'e crystals can be described as modules over the $p$-completion of the PD-envelope of a smooth presentation, together with a Frobenius and a connection satisfying various conditions. Is there an analogous concrete description of \red{(admissible)} prismatic Dieudonn\'e crystals?
\item Finally, it would also be interesting and useful to study deformation theory (in the spirit of Grothendieck-Messing theory) for the prismatic Dieudonn\'e functor. 
\end{enumerate}

We now discuss in more detail the content of this paper. 

\subsection{Quasi-syntomic rings}
\label{sec:quasi-syntomic-rings}
Let us first define the class of rings over which we study $p$-divisible groups. 

\begin{definition}[cf.\ \Cref{sec:quasi-syn-rings-definition-quasi-syn-rings}]
\label{sec:quasi-syntomic-rings-1-definition-introduction-quasi-syntomic-rings}
A ring $R$ is \textit{quasi-syntomic} if $R$ is $p$-complete with bounded
$p^{\infty}$-torsion and if the cotangent complex $L_{R/\zp}$ has $p$-complete
Tor-amplitude in $[-1,0]$\footnote{This means that the complex $M=L_{R/\zp}
\otimes_R^{\mathbb{L}} R/p \in D(R/p)$ is such that $M \otimes_R^{\mathbb{L}} N \in
D^{[-1,0]}(R/p)$ for any $R/p$-module $N$.}. The category of all quasi-syntomic
rings is denoted by $\mathrm{QSyn}$. 

Similarly, a map $R \to R'$ of $p$-complete rings with bounded $p^{\infty}$-torsion
is a \textit{quasi-syntomic morphism} if
$R'$ is $p$-completely flat over $R$ and
$L_{R'/R} \in D(R')$ has $p$-complete Tor-amplitude in $[-1,0]$. 
\end{definition}

\begin{remark}
This definition is due to Bhatt-Morrow-Scholze \cite{bhatt_morrow_scholze_topological_hochschild_homology} and extends (in the $p$-complete world) the usual notion of l.c.i.\
rings and syntomic morphisms (flat and l.c.i.) to the
non-Noetherian, non finite-type setting. The interest of this
definition, apart from being more general, is that it more clearly shows why this
category of rings is relevant : the key property of (quasi-)syntomic rings is that
they have a well-behaved ($p$-completed) cotangent complex. The work of
Avramov shows that the cotangent complex is very badly behaved for all other rings,
at least in the Noetherian setting: it is left unbounded (cf.\ \cite{avramov_locally_complete_intersection_homomorphisms_and_a_conjecture_of_quillen_on_the_vanishing_of_cotangent_homology}).  
\end{remark}

\begin{example}
Any $p$-complete l.c.i.\ Noetherian ring is in $\mathrm{QSyn}$. But there are also big rings in $\mathrm{QSyn}$ : for example, any (integral)
perfectoid ring  is in $\mathrm{QSyn}$ (cf.\ \Cref{sec:quasi-syntomic-rings-example-quasi-syntomic-rings}). As a consequence of this, the $p$-completion of a smooth algebra over a
perfectoid ring is also quasi-syntomic, as well as any bounded $p^{\infty}$-torsion $p$-complete ring which can be presented as the quotient of an integral
perfectoid ring by a finite regular sequence. For example, the rings
$$ \mathcal{O}_{\mathbb{C}_p}\langle T \rangle \quad ; \quad \mathcal{O}_{\mathbb{C}_p}/p \quad ; \quad \mathbb{F}_p[T^{1/p^{\infty}}]/(T-1)$$ are quasi-syntomic. 
\end{example} 

The category of quasi-syntomic rings is endowed with a natural topology : the Grothendieck topology for which covers are given by \textit{quasi-syntomic covers}, i.e., morphisms $R \to R'$ of $p$-complete rings which are quasi-syntomic and $p$-completely faithfully flat.

An important property of the quasi-syntomic topology is that \textit{quasi-regular semiperfectoid rings} form a basis of the topology (cf.\ \Cref{sec:quasi-syn-rings-proposition-quasi-regular semiperfectoid-basis}).

\begin{definition}[cf.\ \Cref{sec:quasi-syn-rings-definition-quasi-regular semiperfectoid}]
  \label{sec:quasi-syntomic-rings-2-definition-qrsp-introduction}
A ring $R$ is \textit{quasi-regular semiperfectoid} if $R \in \mathrm{QSyn}$ and
there exists a perfectoid ring $S$ mapping surjectively to $R$.
\end{definition}

As an example, any perfectoid ring, or any $p$-complete bounded $p^{\infty}$-torsion quotient of a perfectoid ring by a finite regular
sequence, is quasi-regular semiperfectoid.

\subsection{Prisms and prismatic cohomology (after Bhatt-Scholze)}
Our main tool for studying $p$-divisible groups over quasi-syntomic rings is the recent prismatic theory of Bhatt-Scholze \cite{bhatt_scholze_prisms_and_prismatic_cohomology}, \cite{bhatt_lectures_on_prismatic_cohomology}. This theory relies on the seemingly simple notions of \textit{$\delta$-rings} and \textit{prisms}. In what follows, all the rings considered are assumed to be $\mathbb{Z}_{(p)}$-algebras.

A \textit{$\delta$-ring} is a commutative ring $A$, together with a map of sets $\delta : A\to A$, with $\delta(0)=0$, $\delta(1)=0$, and satisfying the following identities :
\[ \delta(xy)= x^p \delta(y) + y^p \delta(x)+ p\delta(x)\delta(y) \quad ; \quad \delta(x+y)=\delta(x)+\delta(y)+\frac{x^p+y^p-(x+y)^p}{p}, \]
for all $x, y \in A$. For any $\delta$-ring $(A,\delta)$, denote by $\varphi$ the map defined by $$\varphi(x)=x^p+p\delta(x).$$ The identities satisfied by $\delta$ are made to make $\varphi$ a ring endomorphism lifting Frobenius modulo $p$. Conversely, a $p$-torsion free ring equipped with a lift of Frobenius gives rise to a $\delta$-ring. A pair $(A,I)$ formed by a $\delta$-ring $A$ and an ideal $I \subset A$ is a \textit{prism} if $I$ defines a Cartier divisor on $\mathrm{Spec}(A)$, if $A$ is (derived) $(p,I)$-complete and if $I$ is pro-Zariski locally generated\footnote{In practice, the ideal $I$ is always principal.} by a distinguished element, i.e., an element $d$ such that $\delta(d)$ is a unit.

\begin{example}
\begin{enumerate}
\item For any $p$-complete $p$-torsion free $\delta$-ring $A$, the pair $(A,(p)$) is a prism.
\item Say that a prism is \textit{perfect} if the Frobenius $\varphi$ on the \blue{underlying} $\delta$-ring is an isomorphism. Then the category of perfect prisms is equivalent to the category of (integral) perfectoid rings : in one direction, one maps a perfectoid ring $R$ to the pair $(A_{\inf}(R):=W(R^{\flat}), \mathrm{ker}(\theta))$ (here $\theta : A_{\rm inf}(R) \to R$ is Fontaine's theta map) ; in the other direction, one maps $(A,I)$ to $A/I$. Therefore, one sees that, in the words of the authors of \cite{bhatt_scholze_prisms_and_prismatic_cohomology}, prisms are some kind of \textit{"deperfection" of perfectoid rings}.
\end{enumerate}
\end{example}

The crucial definition for us is the following. We stick to the affine case for simplicity, but it admits an immediate extension to $p$-adic formal schemes. 
\begin{definition}
Let $R$ be a $p$-complete ring. The \textit{(absolute) prismatic site $(R)_{\prism}$ of $R$} is the opposite of the category of bounded\footnote{A prism $(A,I)$ is \textit{bounded} if $A/I$ has bounded $p^{\infty}$-torsion.} prisms $(A,I)$ together with a map $R \to A/I$, endowed with the Grothendieck topology for which covers are morphisms of prisms $(A, I) \to (B,J)$, such that the underlying ring map $A\to B$ is $(p,I)$-completely faithfully flat. 
\end{definition}

Bhatt and Scholze prove that the functor $\mathcal{O}_{\prism}$ (resp. $\overline{\mathcal{O}}_{\prism}$) on the prismatic site valued in $(p,I)$-complete $\delta$-rings (resp. in $p$-complete $R$-algebras), sending $(A,I) \in (R)_{\prism}$ to $A$ (resp. $A/I$), is a sheaf. The sheaf $\mathcal{O}_{\prism}$ (resp. $\overline{\mathcal{O}}_{\prism}$) is called the \textit{prismatic structure sheaf} (resp. the \textit{reduced prismatic structure sheaf}).  

From this, one easily deduces that the presheaves $I_{\prism}$ (resp. $\mathcal{N}^{\geq 1} \mathcal{O}_{\prism}$) sending $(A,I)$ to $I$ (resp. $\mathcal{N}^{\geq 1} A:=\varphi^{-1}(I)$) are also sheaves on $(R)_{\prism}$. 
\\

Let $R$ be a $p$-complete ring. One proves the existence of a morphism of topoi:

\[ v : \mathrm{Shv}((R)_{\prism}) \to \mathrm{Shv}((R)_{\rm qsyn}). \]

Set :
\[ \mathcal{O}^{\rm pris} := v_* \mathcal{O}_{\prism} ~~ ; ~~ \mathcal{N}^{\geq 1} \mathcal{O}^{\rm pris} := v_* \mathcal{N}^{\geq 1} \mathcal{O}_{\prism} ~~ ; ~~  \mathcal{I}^{\rm pris} := v_* I_{\prism} . \]
The sheaf $\mathcal{O}^{\rm pris}$ is endowed with a Frobenius lift $\varphi$. Moreover, if $R$ is quasi-syntomic, the quotient sheaf $\mathcal{O}^{\rm pris} / \mathcal{N}^{\geq 1} \mathcal{O}^{\rm pris}$ is naturally isomorphic to the structure sheaf $\mathcal{O}$ of $(R)_{\mathrm{qsyn}}$.

\subsection{\red{Admissible} prismatic Dieudonn\'e crystals and modules}
We are now in position to define the category of objects classifying $p$-divisible groups.

\begin{definition}
Let $R$ be a quasi-syntomic ring. A \textit{prismatic Dieudonn\'e crystal over $R$} is a finite locally free $\mathcal{O}^{\rm pris}$-module $\mathcal{M}$ together with $\varphi$-linear morphism
  $$
  \varphi_{\mathcal{M}} \colon  \mathcal{M}\to \mathcal{M}
  $$
  whose linearization $\varphi^\ast \mathcal{M}\to \mathcal{M}$ has its cokernel is killed by $\mathcal{I}^{\rm pris}$. It is said to be \textit{admissible} if the image of the composition
    \[
      \mathcal{M}\xrightarrow{\varphi_{\mathcal{M}}} \mathcal{M}\to \mathcal{M}/\mathcal{I}^\pris \mathcal{M}
    \]
  \blue{ is a finite, locally free $\mathcal{O}$-module $\mathcal{F}_{\mathcal{M}}$ such that the map $(\mathcal{O}^{\rm pris}/\mathcal{I}^{\rm pris}) \otimes_{\mathcal{O}} \mathcal{F}_{\mathcal{M}} \to \mathcal{M}/\mathcal{I}^{\rm pris}\mathcal{M}$ induced by $\varphi_{\mathcal{M}}$ is a monomorphism. }
\end{definition}

\begin{definition}
  \label{sec:abstr-divid-prism-definition-category-of-divided-dieudonne-crystals-intro}
  Let $R$ be a quasi-syntomic ring. We denote by \red{$\DM(R)$ the category of prismatic Dieudonn\'e crystals over $R$ (with morphisms the $\mathcal{O}^{\rm pris}$-linear morphisms commuting with the Frobenius), and by $\DM^{\rm adm}(R)$ its full subcategory of admissible prismatic Dieudonn\'e crystals}.
\end{definition}

\begin{remark}
\label{change-of-definition-thanks-to-the-referee}
\red{In a former version of the paper, we used the notion of \textit{filtered prismatic Dieudonn\'e crystal}. A filtered prismatic Dieudonn\'e crystal over a quasi-\blue{syntomic} ring $R$ is a collection $(\mathcal{M}, \mathrm{Fil} \mathcal{M}, \varphi_{\mathcal{M}})$ consisting of a finite locally free $\mathcal{O}^{\rm pris}$-module $\mathcal{M}$, a $\mathcal{O}^{\rm pris}$-submodule $\mathrm{Fil} \mathcal{M}$, and a $\varphi$-linear map $\varphi_{\mathcal{M}} : \mathcal{M} \to \mathcal{M}$, satisfying the following conditions :
    \begin{enumerate}
  \item $\varphi_{\mathcal{M}}(\mathrm{Fil} \mathcal{M}) \subset \mathcal{I}^{\rm pris}.\mathcal{M}$.
    \item $\mathcal{N}^{\geq 1}\mathcal{O}^{\rm pris}. \mathcal{M} \subset \mathrm{Fil} \mathcal{M}$ and $\mathcal{M}/\mathrm{Fil} \mathcal{M}$ is a finite locally free $\mathcal{O}$-module. 
  \item $\varphi_{\mathcal{M}}(\mathrm{Fil} \mathcal{M})$ generates $\mathcal{I}^{\rm pris}.\mathcal{M}$ as an $\mathcal{O}^{\rm pris}$-module.
\end{enumerate}
However, as was pointed out to us by the referee, the category of filtered prismatic Dieudonn\'e crystals embeds fully faithfully in the category of prismatic Dieudonn\'e crystals, with essential image given by the admissible objects (this essentially follows from \Cref{sec:abstr-divid-prism-proposition-equivalence-divided-prismatic-dieudonne-modules-windows} below). Since admissible prismatic Dieudonn\'e crystals are easier to work with than filtered prismatic Dieudonn\'e crystals, we decided to work only with the first; hence, the results stayed the same, but their formulation changed slightly.}
\end{remark}

For quasi-regular semiperfectoid rings, these abstract objects have a concrete incarnation. Let $R$ be a quasi-regular semiperfectoid ring. The prismatic site $(R)_{\prism}$ admits a final object $(\prism_R, I)$.

\begin{example}
\begin{enumerate}
\item If $R$ is a perfectoid ring, $(\prism_R,I)=(A_{\rm inf}(R),\ker(\tilde{\theta}))$. 
\item If $R$ is quasi-regular semiperfectoid and $pR=0$, $(\prism_R,I) \cong (A_{\rm crys}(R),(p))$. 
\end{enumerate}
\end{example}

\begin{definition}
A \textit{prismatic Dieudonn\'e module over $R$} is a finite locally free $\prism_R$-module $M$ together with a $\varphi$-linear morphism
  $$
  \varphi_M \colon M\to M
  $$
  whose linearization $\varphi^\ast M \to M$ has its cokernel is killed by $I$. It is said to be \textit{admissible} if the composition
  \[
    M\xrightarrow{\varphi_M}M\to M/I\cdot M
  \]
  \blue{is a finite, locally free $R\cong \prism_R/\mathcal{N}^{\geq 1}\prism_R$-module $F_M$ such that the map $\green{\prism_R/I\prism_R} \otimes_{R} F_M \to M/IM$ induced by $\varphi_{M}$ is a monomorphism.}
\end{definition}

\begin{proposition}[\Cref{sec:abstr-divid-prism-proposition-equivalence-crystals-modules-for-quasi-regular semiperfectoid}]
Let $R$ be a quasi-regular semiperfectoid ring. The functor of global sections induces an equivalence between the category of \red{(admissible)} prismatic Dieudonn\'e crystals over $R$ and the category of \red{(admissible)} prismatic Dieudonn\'e modules over $R$.
\end{proposition}

\subsection{Statements of the main results}
In all this paragraph, $R$ is a quasi-syntomic ring.

\begin{theorem}[\Cref{sec:divid-prism-dieud-modul-proposition-window}]
\label{sec:intro-main-theorem-1}
Let $G$ be a $p$-divisible group over $R$. The \red{pair
\[  \left(\mathcal{M}_{\prism}(G)=\Ext^1(G, \mathcal{O}^{\rm pris}), \varphi_{\mathcal{M}_{\prism}(G)} \right), \]
where the $\Ext$ is an Ext-group of abelian sheaves on $(R)_{\rm qsyn}$ and $\varphi_{\mathcal{M}_{\prism}(G)}$ is the Frobenius induced by the Frobenius of $\mathcal{O}^{\rm pris}$, is an admissible prismatic Dieudonn\'e crystal over $R$, often denoted simply by $\mathcal{M}_{\prism}(G)$. }
\end{theorem}

\begin{remark}
The rank of the finite locally free $\mathcal{O}^{\rm pris}$-module $\mathcal{M}_{\prism}(G)$ is the height of $G$, and the quotient $\mathcal{M}_{\prism}(G)/\red{\varphi_{\mathcal{M}_\prism(G)}^{-1}(\mathcal{I}^{\rm pris}.\mathcal{M}_{\prism}(G))}$ is naturally isomorphic to $\mathrm{Lie}(\check{G})$, where $\check{G}$ is the Cartier dual of $G$.
\end{remark}

\begin{remark}
When $pR=0$, the \textit{crystalline comparison theorem} for prismatic cohomology allows us to prove that this construction coincides with the functor usually considered in crystalline Dieudonn\'e theory, relying on Berthelot-Breen-Messing's constructions (\cite{berthelot_breen_messing_theorie_de_dieudonne_cristalline_II}).
\end{remark}

\begin{theorem}[\Cref{sec:divid-prism-dieud-modul-theorem-main-theorem-of-the-paper}] 
\label{sec:intro-main-theorem-2}
\red{The prismatic Dieudonn\'e functor 
\[ \mathcal{M}_{\prism} : G \mapsto (\mathcal{M}_{\prism}(G),\varphi_{\mathcal{M}_\prism(G)}) \]
induces an antiequivalence between the category $\mathrm{BT}(R)$ of $p$-divisible groups over $R$ and the category $\DM^{\rm adm}(R)$ of admissible prismatic Dieudonn\'e crystals over $R$.}
\end{theorem}

\begin{remark}
\purple{\Cref{sec:intro-main-theorem-1} and \Cref{sec:intro-main-theorem-2} immediately extend to $p$-divisible groups over a quasi-syntomic formal scheme.}
\end{remark}

\begin{remark}
It is easy to write down a formula for a functor attaching to \red{an admissible} prismatic Dieudonn\'e crystal an abelian sheaf on $(R)_{\rm qsyn}$, which will be a quasi-inverse of the prismatic Dieudonn\'e functor : see \Cref{writing-a-quasi-inverse}. But such a formula does not look very useful.
\end{remark}

\begin{remark}
As a corollary of the theorem and the comparison with the crystalline functor, one obtains that the (contravariant) Dieudonn\'e functor from crystalline Dieudonn\'e theory is an antiequivalence for quasi-syntomic rings in characteristic $p$. For excellent l.c.i. rings, fully faithfulness was proved by de Jong-Messing; the antiequivalence was proved by Lau for $F$-finite l.c.i. rings (which are in particular excellent rings). 
\end{remark}

\begin{remark}
\purple{It is not difficult to prove that if $R$ is perfectoid, \red{admissible} prismatic Dieudonn\'e crystals (or modules) over $R$ are equivalent to minuscule Breuil-Kisin-Fargues modules for $R$, in the sense of \cite{bhatt_morrow_scholze_integral_p_adic_hodge_theory}.
Therefore, \Cref{sec:intro-main-theorem-2} contains as a special case the results of Lau and Scholze-Weinstein. But the proof of the theorem actually requires this special case\footnote{In fact, as observed in \cite{scholze2020berkeley}, only the case of perfectoid valuation rings with algebraically closed and spherically complete fraction field is needed.} as an input.}
\end{remark}

\begin{remark}
In general, the prismatic Dieudonn\'e functor (without the \red{admissibility condition}) is not essentially surjective, but we prove it is an antiequivalence for \blue{complete regular (Noetherian) local} rings in \Cref{sec:comp-case-mathc-proposition-forgetful-functor-regular}, \red{i.e., in this case the admissibility condition is automatic}.

Moreover, we explain in \Cref{sec:comp-case-mathc} how to recover Breuil-Kisin's classification (as extended by Kim, Lau and Liu to all $p$) of $p$-divisible groups over $\mathcal{O}_K$, where $K$ is a discretely valued extension of $\mathbb{Q}_p$ with perfect residue field, from \Cref{sec:intro-main-theorem-2}.
\end{remark} 

\begin{remark}
\Cref{sec:comparison-with-zink-displays} shows how to extract from the \red{admissible} prismatic Dieudonn\'e functor a functor from $\mathrm{BT}(R)$ to the category of displays of Zink over $R$. Even though the actual argument is slightly involved for technical reasons, the main result there ultimately comes from the following fact : if $R$ is a quasi-regular semiperfectoid ring, the natural morphism $\theta : \prism_R \to R$ gives rise by adjunction to a morphism of $\delta$-rings $\prism_R \to W(R)$, mapping $\mathcal{N}^{\geq 1} \prism_R$ to the image of Verschiebung on Witt vectors.

Zink's classification by displays works on very general bases but is restricted (by design) to formal $p$-divisible groups or to odd $p$ ; by contrast, our classification is limited to quasi-syntomic rings but do not make these restrictions. 
\end{remark}

\begin{remark}
  \label{remark-finite-flat-group-schemes-introduction}
As in Kisin's article \cite{kisin_crystalline_representations_and_f_crystals}, it should be possible to deduce from \Cref{sec:intro-main-theorem-2} a classification result for finite locally free group schemes. We only write this down over a perfectoid ring, in which case it was already known for $p>2$ by the work of Lau, \cite{lau_dieudonne_theory_over_semiperfect_rings_and_perfectoid_rings}. This result is used in \purple{the recent} work of \u{C}esnavi\u{c}ius and Scholze \cite{cesnavicius_scholze_purity_for_flat_cohomology}. 
\end{remark}

\subsection{Overview of the proof and plan of the paper}
\Cref{sec:generalities-prisms} and \Cref{sec:gener-prism-cohom} contain some useful basic results concerning prisms and prismatic cohomology, with special emphasis on the case of quasi-syntomic rings. Most of them are extracted from \cite{bhatt_morrow_scholze_topological_hochschild_homology} and \cite{bhatt_scholze_prisms_and_prismatic_cohomology}, but some are not contained in loc. cit. \orange{(for instance, the definition of the $q$-logarithm, \Cref{sec:q-logarithm}, or the K\"unneth formula, \Cref{sec:kunn-form-prism})}, or only briefly discussed there (for instance, the description of truncated Hodge-Tate cohomology, \Cref{sec:trunc-prism-cotang-complex}).

\Cref{sec:prism-dieu-theory-for-p-divisible-groups} is the heart of this paper. We first introduce the category \red{$\DM^{\rm adm}(R)$ of admissible} prismatic Dieudonn\'e crystals over a quasi-syntomic ring $R$ and discuss some of its abstract properties (\Cref{sec:abstr-divid-dieud}). We then introduce a candidate functor from $p$-divisible groups over $R$ to \red{$\DM^{\rm adm}(R)$} (\Cref{sec:divid-prism-dieud-definition-divided-prismatic-dieudonne-crystals}). That it indeed takes values in the category \red{$\DM^{\rm adm}(R)$} is the content of \Cref{sec:intro-main-theorem-1}, which we do not prove immediately. We first relate this functor to other existing functors, for characteristic $p$ rings or perfectoid rings (\Cref{sec:comp-with-cryst}). The next three sections are devoted to the proof of \Cref{sec:intro-main-theorem-1}. This proof follows a road similar to the one of \cite[Ch. 2, 3]{berthelot_breen_messing_theorie_de_dieudonne_cristalline_II}. The basic idea is to reduce many statements to the case of $p$-divisible groups attached to abelian schemes, using a theorem of Raynaud ensuring that a finite locally free group scheme on $R$ can always be realized as the kernel of an isogeny between two abelian schemes over $R$, Zariski-locally on $R$. For abelian schemes, via the general device, explained in \cite[Ch. 2]{berthelot_breen_messing_theorie_de_dieudonne_cristalline_II} and recalled in \Cref{sec:calc-ext-groups}, for computing Ext-groups in low degrees in a topos, one needs a good understanding of the prismatic cohomology. It relies on the degeneration of the conjugate spectral sequence abutting to reduced prismatic cohomology, in the same way as the description of the crystalline cohomology of abelian schemes is based on the degeneration of the Hodge-de Rham spectral sequence. We prove it in \Cref{sec:prism-dieud-modul-dieud-mod-abel-schemes} by appealing to the \red{group structure on the abelian scheme. Alternatively, one could use an} identification of some truncation of the reduced prismatic complex with some cotangent complex, in the spirit of Deligne-Illusie (or, more recently, \cite{bhatt_morrow_scholze_integral_p_adic_hodge_theory}), proved in \Cref{sec:trunc-prism-cotang-complex}.   

To prove \Cref{sec:intro-main-theorem-2}, stated as \Cref{sec:divid-prism-dieud-modul-theorem-main-theorem-of-the-paper} below, one first observes that the functors 
\[ R \mapsto \mathrm{BT}(R) \quad ; \quad R \mapsto \red{\mathrm{DM}^{\rm adm}(R)} \]
on $\mathrm{QSyn}$ are both stacks for the quasi-syntomic topology (for $\mathrm{BT}$, this is done in the Appendix). Therefore, to prove that the functor \red{$\mathcal{M}_{\prism}$} is an antiequivalence, it is enough to prove it for $R$ quasi-regular semiperfectoid, since these rings form a basis of the topology, in which case one can simply consider the more concrete functor \red{$M_{\prism}$} taking values in \red{admissible} prismatic Dieudonn\'e modules over $R$, defined by taking global sections of \red{$\mathcal{M}_{\prism}$}. Therefore, one sees that, even if one is ultimately interested only by Noetherian rings, the structure of the argument forces to consider large quasi-syntomic rings\footnote{In characteristic $p$, Lau has recently and independently implemented a similar strategy in \cite{lau_divided_dieudonne_crystals}.}.

Assume from now on that $R$ is quasi-regular semiperfectoid. The proof of fully faithfulness is ultimately reduced to the identification of the syntomic sheaf $\Z_p(1)$ (as defined using prismatic cohomology) to the $p$-adic Tate module of $\mathbb{G}_m$, a result of Bhatt-Morrow-Scholze recently reproved without $K$-theory by Bhatt-Lurie (\cite[Theorem 7.5.6]{bhatt2022absolute}). (A former version of this paper attempted to prove fully faithfulness using the strategy of \cite{scholze_weinstein_moduli_of_p_divisible_groups} (following an idea of de Jong-Messing) : one first proves it for morphisms from $\mathbb{Q}_p/\mathbb{Z}_p$ to $\mu_{p^{\infty}}$ and then reduces to this special case. This reduction step works fine in many cases of interest -- such as characteristic $p$ or $p$-torsion free quasi-regular semiperfectoid rings -- but we encountered several technical difficulties while trying to push it to the general case.)
Once fully faithfulness is acquired, the proof of essential surjectivity is by reduction to the perfectoid case. One can actually even reduce to the case of perfectoid valuation rings with algebraically closed fraction field. In this case, the result is known, and due - depending whether one is in characteristic $p$ or in mixed characteristic - to Berthelot and Scholze-Weinstein.   

Finally, \Cref{complements} gathers several complements to the main theorems, already mentioned above : the classification of finite locally free group schemes of $p$-power order over a perfectoid ring, Breuil-Kisin's classification of $p$-divisible groups over the ring of integers of a finite extension of $\Q_p$, the relation with the theory of displays and the description of the Tate module of the generic fiber of a $p$-divisible group from its prismatic Dieudonn\'e crystal.

\subsection{Notations and conventions}

In all the text, we fix a prime number $p$.

\begin{itemize}
\item All finite locally free group schemes will be assumed to be commutative.

\item If $R$ is a ring, we denote by $\BT(R)$ the category of $p$-divisible groups over $R$. 

\item If $A$ is a ring, $I \subset A$ an ideal, and $K \in D(A)$ an object of the derived category of $A$-modules, $K$ is said to be \textit{derived $I$-complete} if for every $f \in I$, the derived limit of the inverse system
\[ \dots K \overset{f} \to K \overset{f} \to K \]
vanishes. Equivalently, when $I=(f_1,\dots,f_r)$ is finitely generated, $K$ is derived $I$-complete if the natural map
\[ K \to R\lim(K \otimes_A^{\mathbb{L}} K_n^{\bullet}) \]
is an isomorphism in $D(A)$, where for each $n\geq 1$, $K_n^{\bullet}$ denotes the Koszul complex $K_{\bullet}(A;f_1^n,\dots,f_r^n)$ (one has $H^0(K_n^{\bullet})=A/(f_1^n,\dots,f_r^n)$, but beware that in general $K_n^{\bullet}$ may also have cohomology in negative degrees, unless $(f_1,\dots,f_r)$ forms a regular sequence). An $A$-module $M$ is said to be \textit{derived $I$-complete} if $K=M[0] \in D(A)$ is derived $I$-complete. The following properties are useful in practice :
\begin{enumerate}
\item A complex $K \in D(A)$ is derived $I$-complete if and only if for each integer $i$, $H^i(K)$ is derived $I$-complete (this implies in particular that the category of derived $I$-complete $A$-modules form a weak Serre subcategory of the category of $A$-modules).
\item If $I=(f_1,\dots,f_r)$ is finitely generated, the inclusion of the full subcategory of derived $I$-complete complexes in $D(A)$ admits a left adjoint, sending $K \in D(A)$ to its \textit{derived $I$-completion}
\[ \widehat{K} = R\lim(K \otimes_A^{\mathbb{L}} K_n^{\bullet}). \]
\item (Derived Nakayama) If $I$ is finitely generated, a derived $I$-complete complex $K \in D(A)$ (resp. a derived $I$-complete $A$-module $M$) is zero if and only if $K\otimes_A^{\mathbb{L}} A/I=0$ (resp. $M/IM=0$).
\item If $I$ is finitely generated, an $A$-module $M$ is (classically) $I$-adically complete if and only if it is derived $I$-complete and $I$-adically separated.
\item $I=(f)$ is principal and $M$ is an $A$-module with bounded $f^{\infty}$-torsion (i.e. such that $M[f^{\infty}]=M[f^N]$ for some $N$), the derived $I$-completion of $M$ (as a complex) is discrete and coincides with its (classical) $I$-adic completion.
\end{enumerate}
A useful reference for derived completions is \cite[Tag 091N]{stacks_project}. 

\item Let $A$ be a ring, $I$ a finitely generated ideal. A complex $K\in D(A)$ is \textit{$I$-completely flat} (resp. \textit{$I$-completely faithfully flat}) if $K \otimes_A^{\mathbb{L}} A/I$ is concentrated in degree $0$ and flat (resp. faithfully flat), cf. \cite[Definition
4.1]{bhatt_morrow_scholze_topological_hochschild_homology}. If an $A$-module $M$ is flat, its derived completion $\widehat{M}$ is $I$-completely flat.

Assume that $I$ is principal, generated by $f \in A$ (in the sequel, $f$ will often be $p$). Let $A\to B$ be a map of derived $f$-complete rings. If $A$ has bounded $f^{\infty}$-torsion and $A\to B$ is $f$-completely flat, then $B$ has bounded $f^{\infty}$-torsion. Conversely, if $B$ has bounded $f^{\infty}$-torsion and $A\to B$ is $f$-completely faithfully flat, $A$ has bounded $f^{\infty}$-torsion. Moreover, if $A$ and $B$ both have bounded $f^{\infty}$-torsion, then $A\to B$ is $f$-completely (faithfully) flat if and only if
$A/f^n \to B/f^n$ is (faithfully) flat for all $n\geq 1$. See \cite[Corollary 4.8]{bhatt_morrow_scholze_topological_hochschild_homology}).

\item A derived $I$-complete $A$-algebra $R$ is \textit{$I$-completely \'etale} (resp. \textit{$I$-completely smooth}) if $R \otimes_A^{\mathbb{L}} A/I$ is concentrated in degree $0$ and \'etale (resp. smooth). 
\end{itemize}

\subsection{Acknowledgements}
\label{sec:acknowledgments}
Special thanks go to Bhargav Bhatt who patiently answered our many questions about prismatic cohomology and to Peter Scholze who suggested this project and followed our progress with interest. We particularly thank Bhatt for a discussion regarding \Cref{sec:trunc-prism-cotang-complex}. The papers of Eike Lau had a strong influence on this work, and we thank him heartily for very helpful discussions and explanations. We would also like to thank Sebastian Bartling, Dustin Clausen, Laurent Fargues, \red{Yonatan Harpaz}, Fabian Hebestreit, Ben Heuer and Andreas Mihatsch for useful discussions on topics related to the content of this paper, as well as K\c{e}stutis \u{C}esnavi\u{c}ius for his comments on a first draft. In the winter term 2019/20, the ARGOS seminar in Bonn went through the manuscript, and we are very grateful to the participants for their careful reading and suggestions for improvements or corrections. We heartily thank Kazuhiro Ito for pointing out to us that our first proof of \Cref{sec:comp-case-mathc-theorem-comparison-with-bk} was erroneous and the anonymous referees for their tremendous report, which hopefully helped to improve and simplify the paper and for spotting many mistakes in earlier versions of this paper. Last, but not least, we would like to deeply thank Akhil Mathew for his interest in our work, for discussing with us our (failed) attempts to rescue our original argument for fully faithfulness and for coming up with an alternative simpler argument. 

The authors would like to thank the University of Bonn, the University Paris 13 and the Institut de Math\'ematiques de Jussieu for their hospitality while this work was done.
Moreover, the first author wants to thank Jonathan Schneider for his support during the first author's academic year in Paris.

\newpage
\section{Generalities on prisms}
\label{sec:generalities-prisms}

In this section we review the theory of prisms and collect some additional results. In particular, we present the definition of the $q$-logarithm (cf.\ \Cref{sec:q-logarithm}).

\subsection{Prisms and perfectoid rings}
\label{sec:prisms-perf-rings}

We list here some basic definitions and results from \cite{bhatt_scholze_prisms_and_prismatic_cohomology}, of which we will make constant use in the paper. Let us first recall the definition of a $\delta$-ring $A$. In the following all rings will be assumed to be $\Z_{(p)}$-algebras.

\begin{definition}
\label{sec:misc-things-delta-definition-delta-ring}
A \textit{$\delta$-ring} is a pair $(A,\delta)$ with $A$ a commutative ring and $\delta\colon A\to A$ a map (of sets) such that for $x,y\in A$ the following equalities hold:
$$
\begin{matrix}
\delta(0)=\delta(1)=0 \\
\delta(xy)=x^p\delta(y)+y^p\delta(x)+p\delta(x)\delta(y)\\
\delta(x+y)=\delta(x)+\delta(y)+\frac{x^p+y^p-(x+y)^p}{p}.
\end{matrix}
$$
A morphism of $\delta$-rings $f\colon (A,\delta)\to (A^\prime,\delta^\prime)$ is a morphism $f\colon A\to A^\prime$ of rings such that $f\circ \delta=\delta^\prime\circ f$. 
\end{definition}

By design the morphism
$$
\varphi\colon A\to A,\ x\mapsto x^p+p\delta(x)
$$
for a $\delta$-ring $(A,\delta)$ is a ring homomorphism lifting the Frobenius on $A/p$.
Using $\varphi$ the second property of $\delta$ can be rephrased as
$$
\delta(xy)=\varphi(x)\delta(y)+y^p\delta(x)=x^p\delta(y)+\varphi(y)\delta(x)
$$
which looks close to that of a derivation.
 If $A$ is $p$-torsion free, then any Frobenius lift $\psi\colon A\to A$ defines a $\delta$-structure on $A$ by setting 
$$
\delta(x):=\frac{\psi(x)-x^p}{p}.
$$
Thus, in the $p$-torsion free case a $\delta$-ring is the same as a ring with a Frobenius lift. 

\begin{remark}
The category of $\delta$-rings has all limits and colimits and that these are calculated on the underlying rings\footnote{This does not hold for the category of rings with a Frobenius lift in the presence of $p$-torsion.} (cf.\ \cite[Section 1]{bhatt_scholze_prisms_and_prismatic_cohomology}). In particular, there \blue{exist} free $\delta$-rings (by the adjoint functor theorem). Concretely, if $A$ is a $\delta$-ring and $X$ is a set, then the free $\delta$-ring $A\{X\}$ on $X$ is a polynomial ring over $A$ with variables $\delta^n(x)$ for $n\geq 0$ and $x\in X$ (cf.\ \cite[Lemma 2.11]{bhatt_scholze_prisms_and_prismatic_cohomology}).
Moreover, the Frobenius on $\Z_{(p)}\{X\}$ is faithfully flat (cf.\ \cite[Lemma 2.11]{bhatt_scholze_prisms_and_prismatic_cohomology}).
\end{remark}

\begin{definition}
  \label{sec:generalities-prisms-definition-distinguished-element-rank-one-element} Let $(A,\delta)$ be a $\delta$-ring.
  \begin{enumerate}
  \item An element $x\in A$ is called \textit{of rank $1$} if $\delta(x)=0$.
    \item An element $d\in A$ is called \textit{distinguished} if $\delta(d)\in A^\times$ is a unit. 
  \end{enumerate}
\end{definition}

In particular, $\varphi(x)=x^p$ if $x\in A$ is of rank $1$.

Here is a useful lemma showing how to find rank $1$ elements in a $p$-adically separated $\delta$-ring.

\begin{lemma}
  \label{sec:generalities-prisms-lemma-elements-with-p-th-roots-are-of-rank-1}
  Let $A$ be a $\delta$-ring and let $x\in A$. Then $\delta(x^{p^n})\in p^nA$ for all $n$. In particular, if $A$ is $p$-adically separated and $y\in A$ admits a $p^n$-th root for all $n\geq 0$, then $\delta(y)=0$, i.e., $y$ has rank $1$. 
\end{lemma}
\begin{proof}
  Cf.\ \cite[Lemma 2.31]{bhatt_scholze_prisms_and_prismatic_cohomology}.
\end{proof}

We can now state the definition of a prism (cf.\ \cite[Definition 3.2]{bhatt_scholze_prisms_and_prismatic_cohomology}). Recall that a $\delta$-pair $(A,I)$ is simply a $\delta$-ring $A$ together with an ideal $I\subseteq A$.

\begin{definition}
  \label{sec:generalities-prisms-definition-prism}
  A $\delta$-pair $(A,I)$ is a \textit{prism} if $I\subseteq A$ is an invertible ideal such that $A$ is derived $(p,I)$-complete, and $p\in I+\varphi(I)A$. A prism $(A,I)$ is called \textit{bounded} if $A/I$ has bounded $p^\infty$-torsion.
\end{definition}

\begin{remark}
  \label{remarks-after-definition-prism}
  Some comments about these definitions are in order :
\begin{enumerate}
\item By \cite[Lemma 3.1]{bhatt_scholze_prisms_and_prismatic_cohomology} the condition $p\in I+\varphi(I)A$ is equivalent to the fact that $I$ is pro-Zariski locally on $\Spec(A)$ generated by a distinguished element.
Thus it is usually not much harm to assume that $I=(d)$ is actually principal\footnote{For example, if $A$ is perfect, i.e., the Frobenius $\varphi\colon A\to A$ is bijective, then this condition is automatic by \cite[Lemma 3.7]{bhatt_scholze_prisms_and_prismatic_cohomology}.}.
\item If $(A,I)\to (B,J)$ is a morphism of prisms, i.e., $A\to B$ is a morphism of $\delta$-rings carrying $I$ to $J$, then \cite[Lemma 3.5]{bhatt_scholze_prisms_and_prismatic_cohomology} implies that $J=IB$.
\item An important example of a prism is provided by
$$
(A,I)=(\Z_p[[q-1]],([p]_q))
$$
where
$$
[p]_q:=\frac{q^p-1}{q-1}
$$
is the $q$-analog of $p$. Many other interesting examples will appear below.
\item The prism $(A,I)$ being bounded implies that $A$ is classically $(p,I)$-adically complete (cf.\ \cite[Exercise 3.4]{bhatt_lectures_on_prismatic_cohomology}), and thus in particular $p$-adically separated.
\end{enumerate}
\end{remark}

\begin{lemma}
  \label{sec:concr-pres-prism-lemma-regular-sequence-for-transversal-prism}
  Let $(A,I)$ be a prism and let $d\in I$ be distinguished. If $(p,d)$ is a regular sequence in $A$, then for all $r,s\geq 0$, $r\neq s$ the sequences
  $$
  (p,\varphi^r(d)),(\varphi^r(d),\varphi^s(d))
  $$
  are regular.
\end{lemma}
\begin{proof}
  \red{Note that for the second case, one can always assume $\min(r,s)=0$, up to replacing $d$ by $\varphi^{\min(r,s)}(d)$.} Then the statement is proven in \cite[Lemma 3.3]{anschuetz_le_bras_the_cyclotomic_trace_in_degree_2} and \cite[Lemma 3.6]{anschuetz_le_bras_the_cyclotomic_trace_in_degree_2}.
\end{proof}

Previous work in $p$-adic Hodge theory used, in one form or another, the theory of perfectoid spaces. From the prismatic perspective, this is explained as follows. We recall that a $\delta$-ring $A$ (or prism $(A,I)$) is called \textit{perfect} if the Frobenius $\varphi\colon A\to A$ is an isomorphism. If $A$ is perfect, then necessarily $A\cong W(R)$ for some perfect \red{$\mathbb{F}_p$-algebra} $R$ (cf.\ \cite[Corollary 2.30]{bhatt_scholze_prisms_and_prismatic_cohomology}).

\begin{proposition}
  \label{sec:generalities-prisms-theorem-perfect-prisms-equivalent-to-perfectoid-rings} The functor
  $$
  \{\textrm{perfect prisms} ~ (A,I) \}\to \{ \textrm{(integral) perfectoid rings} ~ R \},\ (A,I)\mapsto A/I.
  $$
  is an equivalence of categories with inverse $R\mapsto (A_{\mathrm{inf}}(R),\mathrm{ker}(\tilde{\theta}))$, where $A_\inf(R):=W(R^\flat)$ and $\tilde{\theta}=\theta \circ \varphi^{-1}$, $\theta$ being Fontaine's theta map.
\end{proposition}
\begin{proof}
  Cf.\ \cite[Theorem 3.9]{bhatt_scholze_prisms_and_prismatic_cohomology}.
\end{proof}

\begin{remark}
(1) Of course, one could use $\theta$ instead of $\tilde{\theta}$. We make this (slightly strange) choice for coherence with later choices.

(2) The theorem implies in particular that for every perfect prism $(A,I)$, the ideal $I$ is principal. 
\end{remark}

As a corollary, we get the following easy case of almost purity.

\begin{corollary}
  \label{sec:essent-surj-lemma-existence-of-perfectoid-covers-which-are-henselian-along-ideal} Let $R$ be a perfectoid ring and let $R\to R^\prime$ be $p$-completely \'etale. Then $R^\prime$ is perfectoid. Moreover, if $J\subseteq R$ is an ideal, then the $p$-completion $R^\prime$ of the henselization of $R$ at $J$ is perfectoid. 
\end{corollary}
\begin{proof}
  We can lift $R^\prime$ to a $(p,\ker(\theta))$-completely \'etale $A_\inf(R)$-algebra $B$. By \cite[Lemma 2.18]{bhatt_scholze_prisms_and_prismatic_cohomology}, the $\delta$-structure on $A_\inf(R)$ extends uniquely to $B$. Reducing modulo $p$ we see that $B$ is a perfect $\delta$-ring as it is $(p,\ker(\theta))$-completely \'etale over $A_\inf(R)$. Using \Cref{sec:generalities-prisms-theorem-perfect-prisms-equivalent-to-perfectoid-rings} $R^\prime\cong B/\ker(\theta)B$ is therefore perfectoid.
  The statement on henselizations follows from this as henselizations are colimits along \'etale maps (cf.\ the proof of \cite[Tag 0A02]{stacks_project}). \red{(Note that since $R$ has bounded $p^{\infty}$-torsion, the $p$-completion of an \'etale $R$-algebra is $p$-completely \'etale.)}  
\end{proof}

Moreover, perfectoid rings enjoy the following fundamental property.

\begin{proposition}
  \label{sec:prisms-perf-rings-lemma-initial-object-for-perfectoid-rings}
  Let $(A,I)$ be a perfect prism. Then for every prism $(B,J)$ the map
  $$
  \mathrm{Hom}((A,I),(B,J))\to \mathrm{Hom}(A/I,B/J)
  $$
  is a bijection.
\end{proposition}
\begin{proof}
  Cf.\ \cite[Lemma 4.7]{bhatt_scholze_prisms_and_prismatic_cohomology}.
\end{proof}

\subsection{The $q$-logarithm}
\label{sec:q-logarithm}
Each prism is endowed with its \textit{Nygaard filtration} (cf.\ \cite[Definition 11.2]{bhatt_lectures_on_prismatic_cohomology}).
\begin{definition}
  \label{sec:generalities-prisms-definition-nygaard-filtration}
  Let $(A,I)$ be a prism. Then we set
  $$
  \mathcal{N}^{\geq i}A:=\varphi^{-1}(I^i)
  $$
  for $i\geq 0$. The filtration $\mathcal{N}^{\geq \bullet}A$ is called the \textit{Nygaard filtration of $(A,I)$}.
\end{definition}

This filtration (or rather the first piece of this filtration) will play an important role in the rest of this text. It already shows up when proving the existence of the $q$-logarithm
$$
\log_q\colon \Z_p(1)(B/J)\to B,\ x\mapsto \log_q([x^{1/p}]_{\tilde{\theta}})
$$
for a prism $(A,I)$ over $(\Z_p[[q-1]],([p]_q))$ from \Cref{remarks-after-definition-prism}, as we now explain. 

Here,
$$
\Z_p(1):=T_p(\mu_{p^\infty})
$$
is the functor sending a ring $R$ to $T_p(R^\times)=\varprojlim\limits_n\mu_{p^n}(R)$
and
$$
[-]_{\tilde{\theta}}\colon \varprojlim\limits_{x\mapsto x^p} A/I \to A
$$
is the Teichm\"uller lift sending a $p$-power compatible system
$$
x:=(x_0,x_1,\ldots)\in \varprojlim\limits_{x\mapsto x^p} A/I
$$
to the limit
$$
[x]_{\tilde{\theta}}:=\varinjlim\limits_{n\to \infty} \tilde{x}_n^{p^n}
$$
where $\tilde{x}_n\in A$ is a lift of $x_n\in A/I$. \red{By definition, 
$$
\Z_p(1)(A/I)\subseteq \varprojlim\limits_{x\mapsto x^p} A/I
$$
is the subset of the inverse limit consisting of sequences that start with a $1$.}
Moreover, on $\varprojlim\limits_{x\mapsto x^p} A/I$ one can take $p$-th roots
$$
(-)^{1/p}\colon \varprojlim\limits_{x\mapsto x^p} A/I\to \varprojlim\limits_{x\mapsto x^p} A/I,\ (x_0,x_1,\ldots)\mapsto (x_1,x_2,\ldots).
$$

In \cite[Lemma 4.10]{anschuetz_le_bras_the_cyclotomic_trace_in_degree_2} there is the following lemma on the $q$-logarithm.
For $n\in \Z$ we recall that the $q$-number $[n]_q$ is defined as
$$
[n]_q:=\frac{q^n-1}{q-1}\in \Z_p[[q-1]].
$$

\begin{lemma}
  \label{sec:q-logarithm-convergence-of-q-logarithm}
  Let $(B,J)$ be a prism over $(\Z_p[[q-1]],([p]_q))$. Then for every
  element $x\in 1+\mathcal{N}^{\geq 1} B$ of rank $1$, i.e., $\delta(x)=0$, the series
$$
\mathrm{log}_q(x)=\sum\limits_{n=1}^\infty
(-1)^{n-1}q^{-n(n-1)/2}\frac{(x-1)(x-q)\cdots (x-q^{n-1})}{[n]_q}
$$
is well-defined and converges in $B$. Moreover,
$\log_q(x)\in \mathcal{N}^{\geq 1}B$ and, \red{in $$B[1/p][[x-1]]^{\wedge (q-1)},$$ one has the relation $\log_q(x)=\frac{q-1}{\log(q)}\log(x)$, where $\mathrm{log}(x):=\sum\limits_{n=1}^\infty (-1)^{n-1}\frac{\blue{(x-1)^n}}{n}$}.
\end{lemma}

The defining properties of the $q$-logarithm are that $\log_q(1)=0$ and that its $q$-derivative is $\frac{d_qx}{x}$ (cf.\ \cite[Lemma 4.6]{anschuetz_le_bras_the_cyclotomic_trace_in_degree_2}).

One derives easily the existence of the ``divided $q$-logarithm''.

\begin{lemma}
\label{sec:prism-dieud-module-lemma-divided-logarithm}
  Let $(B,J)$ be a \red{bounded} prism over $(\Z_p[[q-1]],([p]_q))$ and let $x\in \Z_p(1)(B/J)$. Then $[x^{1/p}]_{\tilde{\theta}}\in B$ is of rank $1$ and lies in $1+\mathcal{N}^{\geq 1}B$. Thus
  $$
  \mathrm{log}_q([x^{1/p}]_{\tilde{\theta}})=\sum\limits_{n=1}^\infty (-1)^{n-1}q^{-n(n-1)/2}\frac{([x^{1/p}]_{\tilde{\theta}}-1)\ldots ([x^{1/p}]_{\tilde{\theta}}-q^{n-1})}{[n]_q}
  $$
  exists in $B$.
\end{lemma}
\begin{proof}
  By \Cref{sec:generalities-prisms-lemma-elements-with-p-th-roots-are-of-rank-1} \red{(which applies to $B$ as $B$ is bounded and thus classically $(p,[p]_q)$-complete, by \cite[Lemma 3.7 (1)]{bhatt_scholze_prisms_and_prismatic_cohomology})}, the element $[x^{1/p}]_{\tilde{\theta}}$ is of rank $1$ as it admits arbitrary $p^n$-roots. Moreover, $[x^{1/p}]_{\tilde{\theta}}\in 1+\mathcal{N}^{\geq 1}B$ as $\varphi([x^{1/p}]_{\tilde{\theta}})=[x]_{\tilde{\theta}}\equiv 1$ modulo $J$. 
  By \Cref{sec:q-logarithm-convergence-of-q-logarithm} we can therefore conclude.
\end{proof}

\newpage

\section{Generalities on prismatic cohomology}
\label{sec:gener-prism-cohom}

\subsection{Prismatic site and prismatic cohomology}
\label{sec:prism-site-prism-cohom}

In this paragraph, we shortly recall, mostly for the convenience of the reader and to fix notations, some fundamental definitions and results, without proofs, from \cite{bhatt_scholze_prisms_and_prismatic_cohomology}. 
Fix a bounded prism $(A,I)$. Let $R$ be a $p$-complete $A/I$-algebra. 

\begin{definition}
\label{sec:prismatic-site}
The \textit{prismatic site of $R$ relative to $A$}, denoted $(R/A)_{\prism}$, is the category whose objects are given by bounded prisms $(B,IB)$ over $(A,I)$ together with an $A/I$-algebra map $R \to B/IB$, with the obvious morphisms, endowed with the Grothendieck topology for which covers are given by $(p,I)$-completely faithfully flat morphisms of prisms over $(A,I)$. 
\end{definition}

\begin{remark}
  \label{remark-set-theoretic-issues}
  In this remark we deal with the set-theoretic issues arising from \Cref{sec:prismatic-site}. For example, as it stands there does not exist a sheafification functor for presheaves on $(R/A)_\prism$. We will implicitly fix a cut-off cardinal $\kappa$ like in \cite[Lemma 4.1]{scholze_etale_cohomology_of_diamonds} and assume that all objects appearing in \Cref{sec:prismatic-site} (or \Cref{sec:absolute:prismatic-site}) have cardinality $<\kappa$. The results of this paper will not change under enlarging $\kappa$. For example, the category of prismatic Dieudonn\'e crystals on $(R)_\prism$ will be independent of the choice of $\kappa$. Also the prismatic cohomology does not change (because it can be calculated via a \u{C}ech-Alexander complex), and thus the prismatic Dieudonn\'e crystals will be independent of $\kappa$ (by \Cref{sec:calc-ext-groups}).
\end{remark}

This affine definition admits an immediate extension to $p$-adic formal schemes over $\mathrm{Spf}(A/I)$, cf \cite{bhatt_scholze_prisms_and_prismatic_cohomology}. 

\begin{proposition}[\cite{bhatt_scholze_prisms_and_prismatic_cohomology}, Corollary\ 3.12]
\label{sec:prism-structure-sheaf}
The functor $\mathcal{O}_{\prism}$ (resp.\ $\overline{\mathcal{O}}_{\prism}$) on the prismatic site valued in $(p,I)$-complete $\delta-A$-algebras (resp. in $p$-complete $R$-algebras), sending $(B,IB) \in (R/A)_{\prism}$ to $B$ (resp. $B/IB$), is a sheaf. The sheaf $\mathcal{O}_{\prism}$ (resp. $\overline{\mathcal{O}}_{\prism}$) is called the \textit{prismatic structure sheaf} (resp. the \textit{reduced prismatic structure sheaf}). 
\end{proposition}

These constructions have absolute variants, where one does not fix a base prism. Let $R$ be a $p$-complete ring.

\begin{definition}
\label{sec:absolute:prismatic-site}
The \textit{(absolute) prismatic site of $R$}, denoted $(R)_{\prism}$, is the category whose objects are given by bounded prisms $(B,J)$ together with a ring map $R \to B/J$, with the obvious morphisms, endowed with the Grothendieck topology for which covers are given by morphisms of prism $(B,J) \to (C,JC)$ which are $(p,I)$-completely faithfully flat. 
\end{definition}

Exactly as before, one defines functors $\mathcal{O}_{\prism}$ and $\overline{\mathcal{O}}_{\prism}$, which are sheaves on $(R)_{\prism}$. 
\\

We turn to the definition of (derived) prismatic cohomology. Fix a bounded prism $(A,I)$. The prismatic cohomology of $R$ over $A$ is defined in two steps. One starts with the case where $R$ is $p$-completely smooth over $A/I$. 

\begin{definition}
\label{sec:def-prism-cohom-smooth} Let $R$ be a $p$-complete $p$-completely smooth $A/I$-algebra. The prismatic complex $\prism_{R/A}$ of $R$ over $A$ is defined to be the cohomology of the sheaf $\mathcal{O}_{\prism}$ on the prismatic site :
\[ \prism_{R/A} = R\Gamma((R/A)_{\prism}, \mathcal{O}_{\prism}). \]
This is a $(p,I)$-complete commutative algebra object in $D(A)$ endowed with a semi-linear map $\varphi : \prism_{R/A} \to \prism_{R/A}$, induced by the Frobenius of $\mathcal{O}_{\prism}$. 

Similarly, one defines the reduced prismatic complex or Hodge-Tate complex :
\[ \overline{\prism}_{R/A} = R\Gamma((R/A)_{\prism}, \overline{\mathcal{O}}_{\prism}). \]
This is a $p$-complete commutative algebra object in $D(R)$. 
\end{definition}

A fundamental property of prismatic cohomology is the Hodge-Tate comparison theorem, which relates the Hodge-Tate complex to differential forms. For this, first recall that for any $A/I$-module $M$ and integer $n$, the $n$th-Breuil-Kisin twist of $M$ is defined as
\[ M\{n\} := M \otimes_{A/I} (I/I^2)^{\otimes n}. \]
The Bockstein maps
\[ \beta_I : H^i(\overline{\prism}_{R/A})\{i\} \to H^{i+1}(\overline{\prism}_{R/A})\{i+1\} \]
for each $i \geq 0$ make $(H^*(\overline{\prism}_{R/A}\red{)}\{*\},\beta_I)$ a graded commutative $A/I$-differential graded algebra\footnote{For $p=2$ this assertion is non-trivial and part of the proof of \cite[Theorem 4.10]{bhatt_scholze_prisms_and_prismatic_cohomology}.}, which comes with a map $\eta \colon R \to H^0(\overline{\prism}_{R/A})$. 

\begin{theorem}[\cite{bhatt_scholze_prisms_and_prismatic_cohomology}, Theorem 4.10]
\label{sec:hodge-tate-comp-smooth}
The map $\eta$ extends to a map
\[ \eta_R^* : (\Omega_{R/(A/I)}^{\wedge_p}, d) \to (H^*(\overline{\prism}_{R/A}\red{)}, \beta_I) \]  
which is an isomorphism.
\end{theorem}

While proving \Cref{sec:hodge-tate-comp-smooth}, Bhatt and Scholze also relate prismatic and crystalline cohomology when the ring $R$ is an $\F_p$-algebra. The precise statement is the following. Assume that $I=(p)$, i.e. that $(A,I)$ is a crystalline prism. Let $J \subset A$ be a PD-ideal with $p\in J$. Let $R$ be a smooth $A/J$-algebra and 
\[ R^{(1)} = R \otimes_{A/J} A/p, \]
where the map $A/J \to A/p$ is the map induced by Frobenius and the fact that $J$ is a PD-ideal. 

\begin{theorem}[\cite{bhatt_scholze_prisms_and_prismatic_cohomology}, Theorem 5.2]
\label{sec:crystalline-comp-smooth}
Under the previous assumptions, there is a canonical isomorphism of $E_{\infty}-A$-algebras
\[ \prism_{R^{(1)}/A} \simeq R\Gamma_{\rm crys}(R/A), \]
compatible with Frobenius.
\end{theorem}

\begin{remark}
  \label{sec:prism-site-prism-remarks-after-crystalline-comparison}
  \begin{enumerate}
  \item If $J=(p)$, $R^{(1)}$ is just the Frobenius twist of $R$.
  \item The proof of \Cref{sec:crystalline-comp-smooth} goes through for a syntomic $A/J$-algebra $R$. The important point is that in the proof in \cite[Theorem 5.2]{bhatt_scholze_prisms_and_prismatic_cohomology} in each simplicial degree the kernel of the morphism $B^\bullet\to \tilde{R}$ is the inductive limit of ideals of the form $(p,x_1,\ldots, x_r)$ with $(x_1,\ldots, x_r)$ being $p$-completely regular relative to $A$, which allows to apply \cite[Proposition 3.13]{bhatt_scholze_prisms_and_prismatic_cohomology}. The statement extends by descent from the quasi-regular semiperfect case to all quasi-syntomic rings over $\F_p$ (cf.\ \Cref{sec:prism-cohom-quasi-lemma-for-qr-semiperfect-prism-isomorphic-to-acrys}).
  \end{enumerate}

\end{remark}

\Cref{sec:def-prism-cohom-smooth} of course makes sense without the hypothesis that $R$ is $p$-completely smooth over $A/I$. But it would not give well behaved objects ; for instance, the Hodge-Tate comparison would not hold in general\footnote{Nevertheless, in \Cref{sec:prism-cohom-quasi-regular semiperfectoid} we will check that the site-theoretic defined prismatic cohomology is well-behaved for quasi-regular semiperfectoid rings (as it agrees with the derived prismatic cohomology), and also for quasi-syntomic rings}. The formalism of non-abelian derived functors allows to extend the definition of the prismatic and Hodge-Tate complexes to all $p$-complete $A/I$-algebras in a manner compatible with the Hodge-Tate comparison theorem.

\begin{definition}
\label{sec:def-prism-cohom-general} 
The \textit{derived prismatic cohomology} functor $L\prism_{-/A}$ (resp. the \textit{derived Hodge-Tate cohomology} functor $L\overline{\prism}_{-/A}$) is the left Kan extension (cf.\ \cite[Construction 2.1]{bhatt_morrow_scholze_topological_hochschild_homology}) of the functor $\prism_{-/A}$ (resp. $\overline{\prism}_{-/A}$) from $p$-completely smooth $A/I$-algebras to $(p,I)$-complete commutative algebra objects in (the $\infty$-category) $D(A)$ (resp. $p$-complete commutative algebra objects in $D(R)$), to the category of $p$-complete $A/I$-algebras.  
\end{definition}

For short, we will just write $\prism_{R/A}$ (resp. $\overline{\prism}_{R/A}$) for $L\prism_{R/A}$ (resp. $L\overline{\prism}_{R/A}$) in the following.

Left Kan extension of the Postnikov (or canonical filtration) filtration leads to an extension of Hodge-Tate comparison to derived prismatic cohomology.

\begin{proposition}
\label{sec:hodge-tate-comp-general}
For any $p$-complete $A/I$-algebra $R$, the derived Hodge-Tate complex $\overline{\prism}_{R/A}$ comes equipped with a functorial increasing multiplicative exhaustive filtration $\mathrm{Fil}_*^{\rm conj}$ in the category of $p$-complete objects in $D(R)$ and canonical identifications 
\[ \mathrm{gr}_i^{\rm conj}(\overline{\prism}_{R/A}) \simeq \wedge^i L_{R/(A/I)}\{-i\}[-i]^{\wedge_p}. \]
\end{proposition}

Finally, let us indicate how these affine statements globalize. 

\begin{proposition}
\label{sec:derived-prism-cohom-global}
Let $X$ be a $p$-adic formal scheme over $\mathrm{Spf}(A/I)$, \red{which is locally the formal spectrum of a $p$-complete ring with bounded $p^{\infty}$-torsion}. There exists a functorially defined $(p,I)$-complete commutative algebra object $\prism_{X/A} \in D(X,A)$, equipped with a $\varphi_A$-linear map $\varphi_X\colon \prism_{X/A} \to \prism_{X/A}$, and having the following properties :
\begin{itemize}
\item For any affine open $U=\mathrm{Spf}(R)$ in $X$, there is a natural isomorphism of $(p,I)$-complete commutative algebra objects in $D(A)$ between $R\Gamma(U,\prism_{X/A})$ and $\prism_{R/A}$, compatible with Frobenius. 
\item Set $\overline{\prism}_{X/A}=\prism_{X/A} \otimes_A^{\mathbb{L}} A/I \in D(X,A/I)$. Then $\overline{\prism}_{X/A}$ is naturally an object of $D(X)$, which comes with a functorial increasing multiplicative exhaustive filtration $\mathrm{Fil}_*^{\rm conj}$ in the category of $p$-complete objects in $D(X)$ and canonical identifications 
\[ \mathrm{gr}_i^{\rm conj}(\overline{\prism}_{X/A}) \simeq \wedge^i L_{X/(A/I)}\{-i\}[-i]^{\wedge_p}. \]
\end{itemize}
\end{proposition}

\subsection{Truncated Hodge-Tate cohomology and the cotangent complex}
\label{sec:trunc-prism-cotang-complex}

Let $(A,I)$ be a bounded prism, and let $X$ be a $p$-adic $A/I$-formal scheme. The following result also appears in \cite[Proposition 4.14]{bhatt_scholze_prisms_and_prismatic_cohomology}\footnote{Recently, Illusie has also obtained related results in characteristic $p$ (private communication).}. We give a similar argument (suggested to us by Bhatt), with more details than in loc. cit. \green{Since this result is not strictly necessary for the rest of the paper, the reader can safely skip this subsection.}

\begin{proposition}
\label{sec:isom-trunc-prism-cohom-cotangent-complex}
There is a canonical isomorphism :
$$
\alpha_X \colon L_{X/\mathrm{Spf}(A)}\{-1\}[-1]^{\wedge_p}\cong \mathrm{Fil}_{1}^{\mathrm{conj}}(\overline{\prism}_{X/A}),
$$
where the right-hand side is the first piece of the increasing filtration on $\overline{\prism}_{X/A}$ introduced  in \Cref{sec:derived-prism-cohom-global}.
\end{proposition}
\begin{proof}
We can assume that $X=\mathrm{Spf}(R)$ is affine. \blue{Write $\bar{A}=A/I$.} We want to prove that there is a canonical isomorphism
$$
\alpha_R \colon L_{R/A}\{-1\}[-1]^{\wedge_p}\cong \mathrm{Fil}_{1}^{\mathrm{conj}}(\overline{\prism}_{R/A}).
$$
First, let us note that by the transitivity triangle for $A\to \blue{\bar{A}} \to R$ the cotangent complex $L_{R/A}\{-1\}[-1]^{\wedge_p}$ sits inside a triangle
$$
R\cong R\otimes_{\blue{\bar{A}}}L_{\blue{\bar{A}}/A}\{-1\}[-1]^{\wedge_p}\to L_{R/A}\{-1\}[-1]^{\wedge_p}\to L_{R/\blue{\bar{A}}}\{-1\}[-1]^{\wedge_p}
$$
and the outer terms are isomorphic to $R\cong \mathrm{gr}_{0}^{\mathrm{conj}}\overline{\prism}_{R/A}$ and
$$\mathrm{gr}_{1}^{\mathrm{conj}}\overline{\prism}_{R/A}\cong L_{R/\blue{\bar{A}}}\{-1\}[-1]^{\wedge_p}. $$

To construct the isomorphism $\alpha_R$ it suffices to restrict to $\blue{\bar{A}}\to R$ $p$-completely smooth first, and then Kan extend. Thus assume from now on that $R$ is $p$-completely smooth over $\blue{\bar{A}}$.

Let $(B,J) \in (R/A)_{\prism}$, i.e., $(B,J)$ is a prism over $(A,I)$ with a morphism $\iota\colon R\to B/J$.
Pulling back the extension of $A$-algebras
$$
0\to J/J^2\to B/J^2\to B/J\to 0
$$
along $\iota\colon R\to B/J$ defines an extension of $R$ by $J/J^2\cong B/J\{1\}$ and as such, is thus classified by a morphism
$$
\alpha_R^\prime\colon L_{R/A}^{\wedge_p}\to B/J\{1\}[1].
$$
Passing to the (homotopy) limit over all $(B,J)\in (R/A)_\prism$ then defines (after shifting and twisting) the morphism
$$
\alpha_R\colon L_{R/A}\{-1\}[-1]^{\wedge_p}\to \tau^{\leq 1}\overline{\prism}_{R/A}. 
$$
Concretely, if $R=\blue{\bar{A}}\langle x\rangle$, then
$$
L_{R/A}^{\wedge_p}\cong R\otimes_{\blue{\bar{A}}} I/I^2[1]\oplus Rdx.
$$
On the summand $R\otimes_{\blue{\bar{A}}}I/I^2[1]$, the morphism $\alpha_R^\prime$ is simply the base extension of $I/I^2\to J/J^2$ as follows by considering the case $\blue{\bar{A}}=R$. On the summand $Rdx$ the morphism $\alpha_R^\prime$ is (canonically) represented by the $J/J^2$-torsor of preimages of $\iota(x)$ in $B/J^2$ and factors as $R\xrightarrow{\iota}B/J\to B/J\{1\}[1]$ with the second morphism the connecting morphism for $0\to B/J\{1\}\to B/J^2\to B/J\to 0$.
Thus, after passing to the limit, we get a diagram
$$
\xymatrix{
  R\ar[d]\ar[rd] \\
  \overline{\prism}_{R/A}\ar[r]& \overline{\prism}_{R/A}\{1\}[1]
}
$$
and on $H^0$ the horizontal morphism induces the Bockstein differential
$$
\beta\colon H^0(\overline{\prism}_{R/A})\to H^0(\overline{\prism}_{R/A}\{1\}[1])=H^1(\overline{\prism}_{R/A})\{1\}. 
$$
Thus the image of $dx\in H^0(L_{R/A}^{\wedge_p})$ under $\alpha_R$ is $\beta(\iota(x))$. Therefore we see that on $H^0$ the morphism $\alpha_R$ induces the identity under the identifications
$$
(\Omega^{1}_{R/\blue{\bar{A}}})^{\wedge_p}\cong H^0(L_{R/A}^{\wedge_p})
$$
and
$$
(\Omega^1_{R/\blue{\bar{A}}})^{\wedge_p}\cong H^1(\overline{\prism}_{R/A})\{1\}
$$
(the second is the Hodge-Tate comparison).
Moreover, the morphism $$R\otimes_{\blue{\bar{A}}}I/I^2\cong H^{-1}(L_{R/A}^{\wedge_p})\xrightarrow{H^{-1}(\alpha_R)} H^{-1}(\overline{\prism}_{R/A}\{1\}[1])$$ is the canonical one obtained by tensoring $R\to H^0(\overline{\prism}_{R/A})$ with $I/I^2$.
By functoriality (and as $\Omega^1_{R/A}$ is generated by $dr$ for $r\in R$), we can conclude that for every $p$-completely smooth algebra $R$ over $A$
$$
\alpha_R\colon H^i(L_{R/A}^{\wedge_p})\to H^i(\overline{\prism}_{R/A}\{1\}[1])
$$
induces the canonical morphism, and thus, that $\alpha_R$ is an isomorphism in general.
\end{proof}

Recall the following proposition, which is a general consequence of the theory of the cotangent complex.

\begin{proposition}
  \label{sec:prism-cohom-degr-general-lifting-via-cotangent-complex}
  Let $S$ be a ring, $I\subseteq S$ an invertible ideal and $X$ a flat $\overline{S}:=S/I$-scheme. Then the class $\gamma\in \mathrm{Ext}^2_{\mathcal{O}_X}(L_{X/\Spec(\overline{S})},I/I^2\otimes_{\overline{S}}\mathcal{O}_X)$ defined by $L_{X/\Spec(S)}$ is $\pm$ the obstruction class for lifting $X$ to a flat $S/I^2$-scheme. 
\end{proposition}
\begin{proof}
See \cite[III.2.1.2.3]{illusie_complexe_cotangent_et_deformations_I} resp.\ \cite[III.2.1.3.3]{illusie_complexe_cotangent_et_deformations_I}.
\end{proof}

As before, let $(A,I)$ be a bounded prism. 
\begin{corollary}
\label{sec:obstruction-split-same}
Let $X$ be a $p$-completely flat $p$-adic formal scheme over $A/I$. The complex $\mathrm{Fil}_{1}^{\mathrm{conj}}\overline{\prism}_{X/A}$ splits in $D(X)$ (i.e., is isomorphic in $D(X)$ to a complex with zero differentials) if and only if $X$ admits a lifting to a $p$-completely flat formal scheme over $A/I^2$.
\end{corollary}
\begin{proof}
Indeed, $\mathrm{Fil}_{1}^{\mathrm{conj}}\overline{\prism}_{X/A}$ splits if and only if the class in 
\[ \mathrm{Ext}^1_{\mathcal{O}_X}(\mathrm{gr}_{1}^{\mathrm{conj}}\overline{\prism}_{X/A},\mathrm{gr}_{0}^{\mathrm{conj}}\overline{\prism}_{X/A}) = \mathrm{Ext}^2_{\mathcal{O}_X}(L_{X/\Spf(A/I)}^{\wedge_p}\{-1\},\mathcal{O}_X) \]
defined by $\mathrm{Fil}_1^{\rm conj}(\overline{\prism}_{X/A})$ vanishes. \Cref{sec:isom-trunc-prism-cohom-cotangent-complex} shows that this class is the same as the class defined by the $p$-completed cotangent complex $L_{X/\mathrm{Spf}(A)}^{\wedge_p}\{-1\}$. Lifting $X$ to a $p$-completely flat formal scheme over $A/I^2$ is the same as lifting $X\otimes_{A/I}A/(I,p^n)$ to a flat scheme over $A/(I^2,p^n)$ for all $n\geq 1$ (here we use that $(A,I)$ is bounded in order to know that $A/I$ is classically $p$-complete). One concludes by applying \Cref{sec:prism-cohom-degr-general-lifting-via-cotangent-complex}, together with the fact that the $p$-completion of the cotangent complex does not affect the (derived) reduction modulo $p^n$.
\end{proof}


\subsection{Quasi-syntomic rings}
\label{sec:quasi-syn-rings}

We shortly recall some key definitions from \cite[Chapter 4]{bhatt_morrow_scholze_topological_hochschild_homology}.

\begin{definition}
\label{sec:quasi-syn-rings-definition-quasi-syn-rings}

A ring $R$ is \textit{quasi-syntomic} if $R$ is $p$-complete with bounded
$p^{\infty}$-torsion and if the cotangent complex $L_{R/\zp}$ has $p$-complete
Tor-amplitude in $[-1,0]$\footnote{This means that the complex $M=L_{R/\zp}
\otimes_R^{\mathbb{L}} R/p \in D(R/p)$ is such that $M \otimes_R^{\mathbb{L}} N \in
D^{[-1,0]}(R/p)$ for any $R/p$-module $N$.}. The category of all quasi-syntomic
rings is denoted by $\mathrm{QSyn}$. 

Similarly, a map $R \to R'$ of $p$-complete rings with bounded $p^{\infty}$-torsion
is a \textit{quasi-syntomic morphism} (resp. a \textit{quasi-syntomic cover}) if
$R'$ is $p$-completely flat (resp. $p$-completely faithfully flat) over $R$ and
$L_{R'/R} \in D(R')$ has $p$-complete Tor-amplitude in $[-1,0]$. 
\end{definition}

  For a quasi-syntomic ring $R$ the $p$-completed cotangent complex $(L_{R/\Z_p})^{\wedge}_p$ will thus be in $D^{[-1,0]}$ (cf.\ \cite[Lemma 4.6]{bhatt_morrow_scholze_topological_hochschild_homology}).

\begin{remark}
  \label{sec:quasi-syntomic-rings-remark-quasi-syntomic-generalizes-lci}
This definition extends (in the $p$-complete world) the usual notion of locally complete intersection
ring and syntomic morphism (flat and local complete intersection) to the
non-Noetherian, non finite-type setting, as shown by the next example. 
\end{remark}

\begin{example}
  \label{sec:quasi-syntomic-rings-example-quasi-syntomic-rings}
\begin{enumerate}
\item Any $p$-complete l.c.i. Noetherian ring is in $\mathrm{QSyn}$ (cf.\ \cite[Theorem 1.2]{avramov_locally_complete_intersection_homomorphisms_and_a_conjecture_of_quillen_on_the_vanishing_of_cotangent_homology}). 

\item There are also big rings in $\mathrm{QSyn}$. For example, any (integral)
perfectoid ring (i.e., a ring $R$ which is $p$-complete, such that $\pi^p=pu$
for some $\pi\in R$ and $u \in R^{\times}$, Frobenius is surjective on $R/p$ and
$\ker(\theta)$ is principal.) is in $\mathrm{QSyn}$ (cf.\ \cite[Proposition 4.18]{bhatt_morrow_scholze_topological_hochschild_homology}). We give a short explanation : if $R$ is such a ring,
the transitivity triangle for
\[  \zp \to A_{\rm inf}(R) \to R \]
and the fact that $A_{\rm inf}(R)$ is relatively perfect over $\zp$ modulo $p$
imply that after applying $- \otimes_R^{\mathbb{L}} R/p$, $L_{R/\zp}$ and $L_{R/A_{\rm inf}(R)}$ identify. But
\[ L_{R/A_{\rm inf}(R)} = \ker(\theta)/\ker(\theta)^2 [1]= R[1], \]
as $\ker(\theta)$ is generated by a non-zero divisor\footnote{One also proves that
$R[p^{\infty}]=R[p]$, which shows that $R$ has bounded $p^{\infty}$-torsion.}.

\item As a consequence of (ii), the $p$-completion of a smooth algebra over a
perfectoid ring is also quasi-syntomic, as well as any $p$-complete bounded $p^{\infty}$-torsion ring which can be presented as the quotient of an integral
perfectoid ring by a finite regular sequence. 
\end{enumerate}
\end{example} 

The (opposite of the) category $\mathrm{QSyn}$ is endowed with the structure of a site.

\begin{definition}
\label{sec:quasi-syn-rings-definition-quasi-syn-site}
Let $\mathrm{QSyn}_{\rm qsyn}^{\mathrm{op}}$ be the site whose underlying category is the opposite category of the category $\mathrm{QSyn}$ and endowed with the Grothendieck topology generated by quasi-syntomic covers.

If $R \in \mathrm{QSyn}$ we will denote by $(R)_{\mathrm{QSYN}}$ (resp.\ $(R)_{\rm qsyn}$) the big (resp. the small) quasi-syntomic site of $R$, given by all $p$-complete with bounded $p^{\infty}$-torsion \red{$R$-algebras (resp. by all quasi-syntomic $R$-algebras, i.e. all $p$-complete with bounded $p^{\infty}$-torsion $R$-algebras $S$ such that the structure map $R \to S$ is quasi-syntomic)} endowed with the quasi-syntomic topology. 
\end{definition}

The authors of \cite{bhatt_morrow_scholze_topological_hochschild_homology} isolated an interesting class of quasi-syntomic rings.

\begin{definition}
\label{sec:quasi-syn-rings-definition-quasi-regular semiperfectoid}

A ring $R$ is \textit{quasi-regular semiperfectoid} if $R \in \mathrm{QSyn}$ and
there exists a perfectoid ring $S$ mapping surjectively to $R$.
\end{definition}

\begin{example}
  \label{sec:quasi-syntomic-rings-example-quasi-regular-semiperfectoids}
Any perfectoid ring, or any $p$-complete bounded $p^{\infty}$-torsion quotient of a perfectoid ring by a finite regular
sequence, is quasi-regular semiperfectoid.
\end{example}

The interest of quasi-regular semiperfectoid rings comes from the fact that they form a basis of the site $\mathrm{QSyn}_{\rm qsyn}^{\rm op}$. 

\begin{proposition}
\label{sec:quasi-syn-rings-proposition-quasi-regular semiperfectoid-basis}

Let $R$ be quasi-syntomic ring. There exists a quasi-syntomic cover $R \to R'$, with $R'$ quasi-regular semiperfectoid. Moreover, all terms of the \u{C}ech nerve $R^{'\bullet}$ are quasi-regular semiperfectoid. 
\end{proposition}
\begin{proof}
See \cite[Lemma 4.27]{bhatt_morrow_scholze_topological_hochschild_homology} and \cite[Lemma 4.29]{bhatt_morrow_scholze_topological_hochschild_homology}.
\end{proof}

Finally, recall the following result, which is \cite[Prop 7.11]{bhatt_scholze_prisms_and_prismatic_cohomology}.

\begin{proposition}
 \label{sec:quasi-syntomic-rings-proposition-quasi-syntomic-cover-of-prism-admits-refinement}
Let $(A,I)$ be a bounded prism, and $R$ be a quasi-syntomic $A/I$-algebra. There exists a prism $(B,IB) \in (R/A)_{\prism}$ such that the map $R \to B/IB$ is $p$-completely faithfully flat. In particular, if $A/I \to R$ is a quasi-syntomic cover, then $(A,I) \to (B,IB)$ is a faithfully flat map of prisms. 
\end{proposition}
\begin{proof}
Since the proof is short, we recall it. Choose a surjection
$$A/I\langle x_j, j\in J \rangle \to R, $$for some index set $J$. Set
\[ S = A/I\langle x_j^{1/p^{\infty}} \rangle \hat{\otimes}_{A/I\langle x_j, j\in J \rangle}^{\mathbb{L}} R. \]
Then $R \to S$ is a quasi-syntomic cover and by assumption $A/I \to R$ is quasi-syntomic : hence, the map $A/I \to S$ is quasi-syntomic. Moreover the $p$-completion of $\Omega_{S/(A/I)}^1$ is zero. We deduce that the map $A/I \to S$ is such that $(L_{S/(A/I)})^{\wedge_p}$ has $p$-complete Tor-amplitude in degree $[-1,-1]$. Therefore, by the Hodge-Tate comparison, the derived prismatic cohomology $\prism_{S/A}$ is concentrated in degree $0$ and the map $S \to \overline{\prism_{S/A}}$ is $p$-completely faithfully flat. One can thus just take $B=\prism_{S/A}$. 
\end{proof}

As observed in \cite{bhatt_scholze_prisms_and_prismatic_cohomology}, a corollary of \Cref{sec:quasi-syntomic-rings-proposition-quasi-syntomic-cover-of-prism-admits-refinement} is Andr\'e's lemma.

\begin{theorem}[Andr\'e's lemma]
  \label{sec:generalities-prisms-andres-lemma}
  Let $R$ be perfectoid ring. Then there exists a $p$-completely faithfully flat map $R\to S$ of perfectoid rings such that $S$ is absolutely integrally closed, i.e., every monic polynomial with coefficients in $S$ has a solution. 
\end{theorem}
\begin{proof}
 This is \cite[Theorem 7.12]{bhatt_scholze_prisms_and_prismatic_cohomology}. Since the proof is also short, we recall it. Write $R=A/I$, for a perfect prism $(A,I)$ (\Cref{sec:generalities-prisms-theorem-perfect-prisms-equivalent-to-perfectoid-rings}). The $p$-complete $R$-algebra $\tilde{R}$ obtained by adding roots of all possible monic polynomials over $R$ is a quasi-syntomic cover, so by \Cref{sec:quasi-syntomic-rings-proposition-quasi-syntomic-cover-of-prism-admits-refinement}, we can find a prism $(B,J)$ over $(A,I)$ with a $p$-completely faithfully flat morphism $\tilde{R} \to R_1:=B/J$. Moreover, we can (and do) assume that $(B,J)$ is a perfect prism. \red{Indeed, as $(A,I)$ is perfect, the underlying $A$-algebra of the perfection\footnote{The \textit{perfection} of a prism is the $(p,I)$-derived completion (or classical) of its colimit along $\varphi$. See \cite{bhatt_scholze_prisms_and_prismatic_cohomology}.} of $(B,J)$ is the $(p,I)$-completion of a filtered colimit of $(p,I)$-completely faithfully flat $A$-algebras, hence is $(p,I)$-completely faithfully flat.} Transfinitely iterating the construction $R\mapsto R_1$ produces the desired ring $S$. 
\end{proof}

 Let us recall that a functor $u\colon \mathcal{C}\to \mathcal{D}$ between sites is cocontinuous (cf.\ \cite[Tag 00XI]{stacks_project}) if \blue{for every object $C\in \mathcal{C}$ and any covering $\{V_j\to u(C)\}_{j\in J}$ of $u(C)$ in $\mathcal{D}$ there exists a covering $\{ C_i\to C\}_{i\in I}$ of $C$ in $\mathcal{C}$ such that the family $\{ u(C_i)\to u(C)\}_{i\in I}$ refines the covering $\{V_j\to u(C)\}_{j\in J}$}. For a cocontinuous functor
  $u\colon \mathcal{C}\to \mathcal{D}$ the functor
  $$
  u^{-1} \colon \mathrm{Shv}(\mathcal{D})\to \mathrm{Shv}(\mathcal{C}),\ \mathcal{F}\to (\mathcal{F}\circ u)^\sharp
  $$
  (here $()^\sharp$ denotes sheafification)
  is left-exact (even exact) with right adjoint
  $$
  \mathcal{G}\in \mathrm{Shv}(\mathcal{C})\mapsto (D\mapsto \varprojlim\limits_{\{ u(C)\to D\}^\mathrm{op}} \mathcal{G}(C)).
  $$
  Thus, a cocontinuous functor $u\colon \mathcal{C}\to \mathcal{D}$ induces a morphism of topoi
  $$
  u\colon \mathrm{Shv}(\mathcal{C})\to \mathrm{Shv}(\mathcal{D}).
  $$
  Note that in the definition of a cocontinuous functor the morphisms $u(C_j)\to u(C)$ are not required to form a covering of $C$.

\begin{corollary}
  \label{sec:quasi-syntomic-rings-morphism-of-topoi-from-prismatic-topos}
Let $R$ be a $p$-complete ring. The functor $u\colon (R)_{\prism} \to (R)_{\mathrm{QSYN}}$, sending $(A,I)$
to $$R \to A/I$$ is cocontinuous. Consequently, it defines a morphism of
topoi, still denoted by $u$ :
        \[ u\colon  \mathrm{Shv}((R)_{\prism}) \to \mathrm{Shv}((R)_{\mathrm{QSYN}}).
 \]
 \end{corollary}
\begin{proof}
Immediate from the definition (cf. \cite[Tag 00XJ]{stacks_project}) and the previous proposition.
\end{proof}

This yields the following important corollary. 

\begin{corollary}
        \label{sec:prismatic-to-syntomic-corollary-exactness-if-kernel-is-affine-p-completely-syntomic}
        Let $R$ be a $p$-complete ring. Let $$0\to G_1\to G_2\to G_3\to 0$$ be a short
exact sequence of abelian sheaves on $(R)_{\mathrm{QSYN}}$. Then
the sequence
        $$
        0\to u^{-1}(G_1)\to u^{-1}(G_2)\to u^{-1}(G_3)\to 0
        $$
        is an exact sequence on $(R)_{\prism}$. This applies for example when $G_1, G_2,
G_3$ are finite locally free group schemes over $R$.
\end{corollary}
\begin{proof}
The first assertion is just saying that $u^{-1}$ is exact, as $u$ is a
cocontinuous functor (\cite[Tag 00XL]{stacks_project}). The second assertion follows, as
any finite locally free group scheme is syntomic (cf.\ \cite[Proposition 2.2.2]{breuil_groupes_p_divisibles_groupes_finis_et_modules_filtres}).   
\end{proof}

\subsection{Prismatic cohomology of quasi-regular semiperfectoid rings}
\label{sec:prism-cohom-quasi-regular semiperfectoid}

\blue{In this short subsection, we collect a few facts about prismatic cohomology of quasi-regular semiperfectoid rings for later reference.}

For the moment, fix a bounded base prism $(A,I)$ and let $R$ be $p$-complete $A/I$-algebra. There are several cohomologies attached to $R$ :
\begin{enumerate}
\item The derived prismatic cohomology
$$
\prism_{R/A}
$$
of $R$ over $(A,I)$ defined in \Cref{sec:def-prism-cohom-general} via left Kan extension of prismatic cohomology.
\item The cohomology
$$
\prism_{R/A}^{\mathrm{init}}:=R\Gamma((R/A)_{\prism},\mathcal{O}_\prism)
$$
of the prismatic site of $(R/A)_\prism$ (with its $p$-completely faithfully flat topology). 
\item Finally (and only for technical purposes),
$$
\prism_{R/A}^{\mathrm{init, unbdd}}:=R\Gamma((R/A)_{\prism,\mathrm{unbdd}},\mathcal{O}_\prism),
$$
the prismatic cohomology of $R$ with respect to the site $(R/A)_{\prism,\mathrm{unbdd}}$ of not necessarily bounded prisms $(B,J)$ over $(A,I)$ together with a morphism $R\to B/J$ of $A/I$-algebras. We equip $(R/A)_{\prism,\mathrm{unbdd}}$ with the chaotic topology.
\end{enumerate}
%

Assume from now on that $(A,I)$ is a perfect prism and that $A/I\to R$ is a surjection with $R$ quasi-regular semiperfectoid.
The prism $\prism_{R/A}^{\mathrm{init, unbdd}}$ admits then a more concrete (but in general rather untractable) description. Let $K$ be the kernel of $A\to R$. Then
$$
\prism_{R/A}^{\mathrm{init, unbdd}}\cong A\{\frac{K}{I}\}^{\wedge_{(p,I)}}
$$
is the prismatic envelope of the $\delta$-pair $(A,K)$ from \cite[Lemma V.5.1]{bhatt_lectures_on_prismatic_cohomology} as follows from the universal property of the latter. In particular, the site $(R/A)_{\prism,\mathrm{unbdd}}$ has a final object\footnote{Up to now this discussion did not use that $R$ is quasi-regular, it was sufficient that $A/I\to R$ is surjective.}.

\begin{proposition}
  \label{sec:prism-cohom-quasi-proposition-all-prisms-agree-for-quasi-regular semiperfectoid}
  Let as above $(A,I)$ be a perfect prism and $R$ quasi-regular semiperfectoid with a surjection $A/I\twoheadrightarrow R$. Then the canonical morphisms induce isomorphisms
  $$
  \prism_{R/A}\cong \prism_{R/A}^{\mathrm{init}}\cong \prism_{R/A}^{\mathrm{init, unbdd}}
  $$
  as $\delta$-rings.
\end{proposition}
\begin{proof}
\blue{This is \cite[Proposition 7.10]{bhatt_scholze_prisms_and_prismatic_cohomology} (the second isomorphism, i.e. the fact that $\prism_{R/A}^{\mathrm{init, unbb}}$ is bounded, follows from the last assertion of loc. cit.).}
\end{proof}

If $pR=0$, i.e., $R$ is quasi-regular semiperfect, there is moreover the universal $p$-complete PD-thickening
$$
A_\crys(R)
$$
of $R$ (cf.\ \cite[Proposition 4.1.3]{scholze_weinstein_moduli_of_p_divisible_groups}).
The ring $A_\crys(R)$ is $p$-torsion free by \cite[Theorem 8.14]{bhatt_morrow_scholze_topological_hochschild_homology}.

\begin{lemma}
  \label{sec:prism-cohom-quasi-lemma-for-qr-semiperfect-prism-isomorphic-to-acrys} Let $(A,I)$, $R$ be as above and assume that $pR=0$. Then there is a canonical $\varphi$-equivariant isomorphism
  $$
  \prism_{R/A}\cong A_\crys(R).
  $$
\end{lemma}
\begin{proof}
  As $A_\crys(R)$ is $p$-torsion free (cf.\ \cite[Theorem 8.14]{bhatt_morrow_scholze_topological_hochschild_homology}) and carries a canonical Frobenius lift there we get a natural morphism
  $$
  \prism_{R/A}\to A_\crys(R).
  $$
  Conversely, the kernel of the natural morphism (cf.\ \Cref{sec:prism-cohom-quasi-theorem-identification-of-the-graded-pieces-of-nygaard-filtration}, which does not depend on this lemma)
  $$
  \theta\colon \prism_{R/A}\to R
  $$
  has divided powers (as one checks similarly to \cite[Proposition 8.12]{bhatt_morrow_scholze_topological_hochschild_homology}, using that the proof of \Cref{sec:crystalline-comp-smooth} goes through in the syntomic case, cf.\ \Cref{sec:prism-site-prism-remarks-after-crystalline-comparison}). This provides a canonical morphism
  $$
  A_\crys(R)\to \prism_R
  $$
  in the other direction. Similarly, to \cite[Theorem 8.14]{bhatt_morrow_scholze_topological_hochschild_homology} one checks that both are inverse to each other.
\end{proof}

\begin{remark}
Both rings $\prism_{R/A}$ and $A_\crys(R)$ are naturally $W(R^\flat)$-algebras, but the isomorphism of \Cref{sec:prism-cohom-quasi-lemma-for-qr-semiperfect-prism-isomorphic-to-acrys} restricts to the Frobenius on $W(R^\flat)$. Concretely, if $R=R^\flat/x$ for some non-zero divisor $x\in R^\flat$, then
$$
\prism_{R/W(R^\flat)}\cong W(R^\flat)\{\frac{x}{p}\}^\wedge
$$
and (cf.\ \cite[Corollary 2.37]{bhatt_scholze_prisms_and_prismatic_cohomology})
$$
A_\crys(R)\cong W(R^\flat)\{\frac{x^p}{p}\}^\wedge\cong \prism_{R/W(R^\flat)}\otimes_{W(R^\flat),\varphi} W(R^\flat).
$$
\end{remark}

\orange{The prismatic cohomology $\prism_R$ of a quasi-regular semiperfectoid ring $R$ comes equipped with its \textit{Nygaard filtration}, \cite[\S 12]{bhatt_scholze_prisms_and_prismatic_cohomology}, an $\mathbb{N}$-indexed decreasing multiplicative filtration defined for $i\geq 0$ by 
$$
\mathcal{N}^{\geq i}(\prism_R) = \{ x \in \prism_R, \varphi(x) \in d^i \prism_R \},
$$
$d$ denoting a generator of the  ideal $I$. The graded pieces of the Nygaard filtration can be described as follows.}

\begin{theorem}
  \label{sec:prism-cohom-quasi-theorem-identification-of-the-graded-pieces-of-nygaard-filtration} Let $R$ be a quasi-regular semiperfectoid ring. Then
  $$
  \mathcal{N}^{\geq i}(\prism_R)/\mathcal{N}^{\geq i+1}(\prism_R)\cong \mathrm{Fil}^{\mathrm{conj}}_i(\overline{\prism}_R)\{i\}
  $$
  for $i\geq 0$. In particular, $\prism_R/\mathcal{N}^{\geq 1}\prism_R\cong R$.
\end{theorem}

Here $\mathrm{Fil}^{\mathrm{conj}}_\bullet(\overline{\prism}_R)$ denotes the conjugate filtration on $\overline{\prism}_R$ with graded pieces given by $\mathrm{gr}_i^{\mathrm{conj}}(\overline{\prism}_R)\cong \Lambda^i{L}_{R/S}^{\wedge_p}[-i]$, for any choice of perfectoid ring $S$ mapping to $R$, \orange{cf. \ref{sec:hodge-tate-comp-general})}. 

\begin{proof}
  See \cite[Theorem 12.2]{bhatt_scholze_prisms_and_prismatic_cohomology}.
\end{proof}

\subsection{The K\"unneth formula in prismatic cohomology}
\label{sec:kunn-form-prism}

The Hodge-Tate comparison implies a K\"unneth formula. Here is the precise statement.
Note that for a bounded prism $(A,I)$ the functor $R\mapsto \prism_{R/A}$ is naturally defined on all derived $p$-complete simplicial $A/I$-algebras.

\begin{proposition}
  \label{sec:kunn-form-prism-proposition-kunneth-formula}
  Let $(A,I)$ be a bounded prism. Then the functor
  $$
  R\mapsto \prism_{R/A}
  $$
  from derived $p$-complete simplicial rings over $A/I$ to derived $(p,I)$-complete $E_\infty$-algebras over $A$ preserves tensor products, i.e., for all morphism $R_1\leftarrow R_3 \to R_2$ the canonical morphism
  $$
  \prism_{R_1/A}\hat{\otimes}^{\mathbb{L}}_{\prism_{R_3/A}} \prism_{R_2/A}\to \prism_{R_1\hat{\otimes}^{\mathbb{L}}_{R_3}R_2/A}
  $$
  is an equivalence. 
\end{proposition}
\begin{proof}
  Using \cite[Construction 2.1]{bhatt_morrow_scholze_topological_hochschild_homology} (resp.\ \cite[Proposition 5.5.8.15]{lurie_higher_topos_theory}) the functor $R\mapsto \prism_{R/A}$, which is the left Kan extension from $p$-completely smooth algebras to all derived $p$-complete simplicial $A/I$-algebras, commutes with colimits if it preserves finite coproducts. Clearly, $\prism_{(A/I)/A}\cong A$, i.e., $\prism_{-/A}$ preserves the final object. Moreover, for $R,S$ $p$-completely smooth over $A/I$ the canonical morphism
  $$
  \prism_{R/A}\hat{\otimes}^{\mathbb{L}}_{A} \prism_{S/A}\to \prism_{S\hat{\otimes} R/A} 
  $$
  is an isomorphism because this may by $I$-completeness be checked for $\overline{\prism}_{-/A}$ where it follows from the Hodge-Tate comparison.
\end{proof}

Gluing the isomorphism in \Cref{sec:kunn-form-prism-proposition-kunneth-formula} we can derive, using as well the projection formula and flat base change for quasi-coherent cohomology, the following statement.
\begin{corollary}
  \label{sec:kunn-form-prism-smooth-proper-case}
  If $X$ and $Y$ are \red{quasi-compact quasi-separated} $p$-completely smooth $p$-adic formal schemes over $\Spf(A/I)$), then
$$
R\Gamma(X\times_{\Spf(A/I)} Y,\prism_{X\times_{\Spf(A/I)} Y/A})\cong R\Gamma(X,\prism_{X/A})\hat{\otimes}_A^{\mathbb{L}}R\Gamma(Y,\prism_{Y/A}).
$$
\end{corollary}

\newpage
\section{Prismatic Dieudonn\'e theory for $p$-divisible groups}
\label{sec:prism-dieu-theory-for-p-divisible-groups}

This chapter is the heart of this paper. \purple{We construct the \textit{prismatic Dieudonn\'e functor} over any quasi-syntomic ring and prove that it gives an antiequivalence between $p$-divisible groups over $R$ and \textit{\red{admissible} prismatic Dieudonn\'e crystals over $R$}.} The strategy to do this is to use \textit{quasi-syntomic descent} to reduce to the case where $R$ is quasi-regular semiperfectoid, in which case the \red{(admissible)} prismatic Dieudonn\'e crystals over $R$ can be replaced by simpler objects, the \textit{\red{(admissible)} prismatic Dieudonn\'e modules}.

\subsection{Abstract prismatic Dieudonn\'e crystals and modules}
\label{sec:abstr-divid-dieud}
Let $R$ be a $p$-complete ring. We defined in \Cref{sec:quasi-syntomic-rings-morphism-of-topoi-from-prismatic-topos} a morphism of topoi :
\[ u : \mathrm{Shv}((R)_{\prism}) \to \mathrm{Shv}((R)_{\rm QSYN}).   \]
\blue{We let 
$$
\epsilon_\ast: \mathrm{Shv}((R)_{\rm QSYN}) \to \mathrm{Shv}((R)_{\rm qsyn})
$$
be the functor defined by $\epsilon_\ast \mathcal{F}(R') = \mathcal{F}(R')$ for $\mathcal{F} \in \mathrm{Shv}((R)_{\rm QSYN})$ and $R' \in (R)_{\rm qsyn}$. It has a left adjoint $\epsilon^\natural: \mathrm{Shv}((R)_{\rm qsyn}) \to \mathrm{Shv}((R)_{\rm QSYN})$. We warn the reader that the restriction functor from the big to the small quasi-syntomic site does not induce a morphism of sites\footnote{We thank Kazuhiro Ito for drawing our attention to this point.}, i.e., this left adjoint need not preserve finite limits (which explains why we denoted it $\epsilon^\natural$ instead of $\epsilon^{-1}$). 

We let 
$$ v_\ast = \epsilon_\ast \circ u_\ast :  \mathrm{Shv}((R)_{\prism}) \to \mathrm{Shv}((R)_{\rm qsyn})
$$
and
$$
v^\natural = u^{-1} \circ \epsilon^\natural : \mathrm{Shv}((R)_{\rm qsyn}) \to  \mathrm{Shv}((R)_{\prism}).
$$
We still have the formula $Rv_\ast\cong R\varepsilon_\ast\circ Ru_\ast$ as $\varepsilon_\ast$ is exact.
}

\begin{definition}
 \label{sec:abstr-divid-prism-definition-of-the-sheaves}

Let $R$ be a $p$-complete ring. We define :
\[ \mathcal{O}^{\rm pris} := v_* \mathcal{O}_{\prism} ~~ ; ~~ \mathcal{N}^{\geq 1} \mathcal{O}^{\rm pris} := v_* \mathcal{N}^{\geq 1} \mathcal{O}_{\prism} ~~ ; ~~  \mathcal{I}^{\rm pris} := v_* \mathcal{I}_{\prism} , \]
where $\mathcal{I}_\prism\subseteq \mathcal{O}_{\prism}$ denotes the canonical invertible ideal sheaf sending a prism $(B,J)\in (R)_{\prism}$ to $J$.
The sheaf $\mathcal{O}^{\rm pris}$ is endowed with a Frobenius lift $\varphi$. 
\end{definition}

Although these sheaves are defined in general, we will only use them over quasi-syntomic rings. 

\begin{proposition}
 \label{sec:abstr-divid-prism-proposition-quotient-sheaf-is-structure-sheaf}

Let $R$ be quasi-syntomic ring. The quotient sheaf $$\mathcal{O}^{\rm pris} / \mathcal{N}^{\geq 1} \mathcal{O}^{\rm pris}$$ is isomorphic to the structure sheaf $\mathcal{O}$ of $(R)_{\mathrm{qsyn}}$. 
\end{proposition}
\begin{proof}
It is enough to produce such an isomorphism functorially on a basis of $(R)_{\mathrm{qsyn}}$. By \Cref{sec:quasi-syn-rings-proposition-quasi-regular semiperfectoid-basis}, we can thus assume that $R$ is quasi-regular semiperfectoid. In this case, we conclude by \Cref{sec:prism-cohom-quasi-theorem-identification-of-the-graded-pieces-of-nygaard-filtration}.
\end{proof} 

\begin{definition}
  \label{sec:prism-dieud-cryst-definition-prismatic-crystals}
  Let $R$ be a $p$-complete ring. A \textit{prismatic crystal} over $R$ is an $\mathcal{O}_\prism$-module $\mathcal{M}$ on the prismatic site $(R)_\prism$ of $R$ such that for all morphisms $(B,J)\to (B^\prime,J^\prime)$ in $(R)_\prism$ the canonical morphism
  $$
  \mathcal{M}(B,J)\otimes_{B}B^\prime\red{\to} \mathcal{M}(B^\prime,J^\prime)
  $$
  is an isomorphism.
\end{definition}

Note that a prismatic crystal in finitely generated projective $\mathcal{O}_\prism$-modules (resp. in finitely generated projective $\overline{\mathcal{O}}_{\prism}$-modules) is the same thing as a finite locally free $\mathcal{O}_{\prism}$-module (resp. a finite locally free $\overline{\mathcal{O}}_{\prism}$-module). In what follows, we will essentially consider only this kind of prismatic crystal.

\begin{proposition}
 \label{sec:abstr-divid-prism-proposition-finite-locally-free}
Let $R$ be a quasi-syntomic ring. The functors $v_*$ and \blue{$v^{\ast}(-):=\mathcal{O}_{\prism}\otimes_{v^{\natural}\mathcal{O}^\pris}v^{\natural}\green{(-)}$} induce equivalences between the category of finite locally free $\mathcal{O}_{\prism}$-modules and the category of finite locally free $\mathcal{O}^{\rm pris}$-modules. 
\end{proposition}
\begin{proof}
  Because $v_\ast(\mathcal{O}_\prism)=\mathcal{O}^\pris$ it is clear that for all finite locally free $\mathcal{O}^\pris$-modules $\mathcal{M}$ the canonical morphism
    $$
    \mathcal{M}\to v_{\ast}(v^\ast(\mathcal{M}))
    $$
    is an isomorphism as this can be checked locally on $(R)_{\mathrm{qsyn}}$. Conversely, let $\mathcal{N}$ be a finite locally free $\mathcal{O}_\prism$-module. We have to show that the counit
    $$
    v^{\ast}v_\ast(\mathcal{N})\to \mathcal{N}
    $$
    is an isomorphism.
    For any morphism $R\to R^\prime$ with $R^\prime$ quasi-syntomic there are equivalences
    $$
    (R)_\prism/h_{R^\prime}\cong (R^\prime)_{\prism}\ ,\ (R)_{\mathrm{qsyn}}/{R^\prime}\cong (R^\prime)_{\mathrm{qsyn}}
    $$
    of slice topoi where $h_{R^\prime}(B,J):=\mathrm{Hom}_R(R^\prime,B/J)$. By passing to a quasi-syntomic cover $R\to R^\prime$ we can therefore assume that $R$ is quasi-regular semiperfectoid, in particular that the site $(R)_{\prism}$ has a final object given by $\prism_R$. By $(p,I)$-completely faithfully flat descent of finitely generated projective modules over $(p,I)$-complete rings of bounded $(p,I)$-torsion (cf. \Cref{sec:descent-p-completely-1-proposition-descent-of-finite-projective-modules-for-faithfully-morphisms-of-prisms}), the category of finite locally free $\mathcal{O}_\prism$-modules on $(R)_{\prism}$ is equivalent to finitely generated projective $\prism_R$-modules\footnote{The non-trivial point is that the global sections of a finite locally free $\mathcal{O}_{\prism}$-module are locally free over $\prism_R$.}.
    As the morphism $\prism_R\to R$ (the ``$\theta$''-map) is henselian along its kernel, cf.\ \Cref{sec:essent-surj-lemma-prism_r-henselian-along-ker-theta}, finite locally free $\prism_R$-modules split on the pullback of an open cover of $\Spf(R)$. Thus, after passing to a quasi-syntomic cover of $\mathrm{Spf}(R)$, we may assume that $\mathcal{N}$ is finite free. Then the isomorphism
    $$
    v^\ast v_\ast(\mathcal{N})\cong \mathcal{N}
    $$
    is clear.
\end{proof}

\begin{definition}
  \label{sec:abstr-divid-prism-definition-prismatic-dieudonne-crystals}
  Let $R$ be a quasi-syntomic ring. A \textit{prismatic Dieudonn\'e crystal over $R$} is a finite locally free $\mathcal{O}^{\rm pris}$-module $\mathcal{M}$ together with $\varphi$-linear morphism
  $$
  \varphi_{\mathcal{M}} \colon  \mathcal{M}\to \mathcal{M}
  $$
  whose linearization \red{$\varphi^\ast \mathcal{M}\to \mathcal{M}$} has its cokernel killed by $\mathcal{I}^{\rm pris}$. \red{We call a prismatic Dieudonn\'e crystal $(\mathcal{M},\varphi_{\mathcal{M}})$ \textit{admissible} if the image of the composition
    \[
      \mathcal{M}\xrightarrow{\varphi_{\mathcal{M}}} \mathcal{M}\to \mathcal{M}/\mathcal{I}^\pris\cdot \mathcal{M}
    \]
    is a finite, locally free $\mathcal{O}$-module $\mathcal{F}_{\mathcal{M}}$ such that the map $(\mathcal{O}^{\rm pris}/\mathcal{I}^{\rm pris}) \otimes_{\mathcal{O}} \mathcal{F}_{\mathcal{M}} \to \mathcal{M}/\mathcal{I}^{\rm pris}\mathcal{M}$ induced by $\varphi_{\mathcal{M}}$ is a monomorphism}.
\end{definition}

Here, $\mathcal{M}/\mathcal{I}^\pris\cdot \mathcal{M}$ is an $\mathcal{O}\cong \mathcal{O}^\pris/\mathcal{N}^{\geq 1}\mathcal{O}^\pris$-module, cf.\ \ref{sec:abstr-divid-prism-proposition-quotient-sheaf-is-structure-sheaf}, via the composition $\mathcal{O}^\pris\xrightarrow{\varphi}\mathcal{O}^\pris\to \mathcal{O}^\pris/\mathcal{I}^\pris\mathcal{O}$.

%
%

\begin{remark}
\label{remarks-on-the-definition-of-a-filtered-prismatic-dieudonne-crystal}
 For a prismatic Dieudonn\'e crystal $(\mathcal{M},\varphi_{\mathcal{M}})$ the linearization $\varphi^\ast \mathcal{M} \to \mathcal{M}$ of the morphism $\varphi_{\mathcal{M}}\colon \mathcal{M} \to \mathcal{M}$ is an isomorphism after inverting a local generator $\tilde{\xi}$ of $\mathcal{I}^{\rm pris}$ and in particular is injective, since $\varphi^{\ast} \mathcal{M}$ is $\tilxi$-torsion free.
%
\end{remark}

\begin{remark}
\label{remark-on-admissibility}
Let $(\mathcal{M},\varphi_{\mathcal{M}})$ be a prismatic Dieudonn\'e crystal. Write $\mathrm{Fil} \mathcal{M}=\varphi_{\mathcal{M}}^{-1}(\mathcal{I}^{\rm pris}.\mathcal{M})$. Consider the diagram (defining $Q,K$)
  \[
    \xymatrix{
      0\ar[r]& \varphi^\ast \mathrm{Fil} \mathcal{M} \ar[r]^{\varphi_{\mathcal{M}}}\ar[d] & \mathcal{I}^{\rm pris}.\mathcal{M} \ar[r]\ar[d] & Q\ar[r]\ar[d]^\alpha & 0 \\
      0\ar[r]& \varphi^\ast \mathcal{M} \ar[r]^{\varphi_{\mathcal{M}}} &  \mathcal{M} \ar[r] & K\ar[r] & 0.   
    }
  \]
  As $\mathcal{I}^{\rm pris}.K=0$ (by definition of a prismatic Dieudonn\'e crystal) the map $\alpha$ is zero. The snake lemma implies therefore that there exists a short exact sequence
  \[
    0\to Q \to \varphi^\ast \mathcal{M}/{\varphi^\ast \Fil \mathcal{M}}\cong \mathcal{O}^{\rm pris}/\mathcal{I}^{\rm pris}\otimes_{\mathcal{O}} \mathcal{F}_{\mathcal{M}} \xrightarrow{\beta} \mathcal{M}/\mathcal{I}^{\rm pris}\mathcal{M} \to K \to 0
  \]
 (where as in \Cref{sec:abstr-divid-prism-definition-prismatic-dieudonne-crystals} we wrote $\mathcal{F}_{\mathcal{M}}= \mathcal{M}/\mathrm{Fil} \mathcal{M}$). Hence we see that the injectivity of $\beta$ (condition required in the definition of admissibility) is equivalent to the condition that $Q=0$.  
\end{remark}


\begin{definition}
  \label{sec:abstr-divid-prism-definition-category-of-divided-dieudonne-crystals}
  Let $R$ be a quasi-syntomic ring. We denote by $\mathrm{DM}(R)$ the category of prismatic Dieudonn\'e crystals over $R$ (with $\mathcal{O}^{\rm pris}$-linear morphisms commuting with Frobenius) and by $\mathrm{DM}^{\mathrm{adm}}(R)$ the full subcategory of admissible objects.
\end{definition}

\begin{proposition}
\label{sec:astr-divid-prism-proposition-descent}
The fibered category of \red{(usual or admissible)} prismatic Dieudonn\'e crystals over the category $\mathrm{QSyn}$ of quasi-syntomic rings endowed with the quasi-syntomic topology is a stack.
\end{proposition}
\begin{proof}
This follows from the definition, because by general properties of topoi modules under $\mathcal{O}^{\rm pris}$ and $\mathcal{O}$ form a stack for the quasi-syntomic topology on $(R)_{\mathrm{qsyn}}$.
\end{proof}

For quasi-regular semiperfectoid rings, these abstract objects have a more concrete incarnation, which we explain now. Let $R$ be a quasi-regular semiperfectoid ring and let $(\prism_R,I)$ be the prism associated with $R$. Note that $I$ is necessarily principal as there exists a perfectoid ring mapping to $R$.
Recall (\Cref{sec:prism-cohom-quasi-theorem-identification-of-the-graded-pieces-of-nygaard-filtration}) that
$$
\theta\colon \prism_R/\mathcal{N}^{\geq 1}\prism_R\cong R
$$
is an isomorphism.

\begin{definition}
  \label{sec:abstr-divid-prism-definition-prismatic-dieudonne-modules-for-quasi-regular semiperfectoid}
  A \textit{prismatic Dieudonn\'e module over $R$} is a finite locally free $\prism_R$-module $M$ together with a $\varphi$-linear morphism
  $$
  \varphi_M \colon M\to M
  $$
  whose linearization \red{$\varphi^\ast M \to M$} has its cokernel killed by $I$. As in \ref{sec:abstr-divid-prism-definition-prismatic-dieudonne-crystals}, we call a prismatic Dieudonn\'e module $(M,\varphi_M)$ over $R$ \textit{admissible} if the image of the composition
  \[
    M\xrightarrow{\varphi_M}M\to M/I\cdot M
  \]
  is a finite, locally free $R\cong \prism_R/\mathcal{N}^{\geq 1}\prism_R$-module $F_M$ such that the map $\green{\prism_R/I\prism_R} \otimes_{R} F_M \to M/IM$ induced by $\varphi_{M}$ is a monomorphism.
\end{definition}

%

\begin{remark}
\label{remarks-on-the-definition-ofo-prismatic-dieudonne-module}
 For a prismatic Dieudonn\'e module $(M,\varphi_M)$ the linearization \red{$\varphi^\ast M \to M$} of the morphism $\varphi_M\colon M\to M$ is an isomorphism after inverting a generator $\tilde{\xi}$ of $I$ and in particular is injective, since $\varphi^{\ast}M$ is $\tilxi$-torsion free. \green{In \ref{remark-projectivity-referee} we will prove that these properties imply that the cokernel of $\varphi^\ast M\to M$ is a finite projective $\prism_R/I$-module.}
%
%
\end{remark}

If $R$ is perfectoid, one has $$(\prism_R,I)=(A_{\rm inf}(R),(\tilxi)). $$
A prismatic Dieudonn\'e module is the same thing as a \textit{minuscule Breuil-Kisin-Fargues module} (\cite{bhatt_morrow_scholze_integral_p_adic_hodge_theory}) over $A_{\rm inf}(R)$ with respect to $\tilxi$. In fact, the situation for perfectoid rings is simple, as shown by the following proposition.

\begin{proposition}
  \label{sec:abstr-divid-prism-divided-prismatic-dieudonne-modules-vs-bkf-modules} Let $R$ be a perfectoid ring. Any prismatic Dieudonn\'e module over $R$ is admissible.
  \end{proposition}
 We postpone the proof, it will be given below after \Cref{sec:abstr-divid-prism-proposition-equivalence-divided-prismatic-dieudonne-modules-windows}.

%
%

\red{
\begin{proposition}
 \label{sec:abstr-divid-prism-proposition-equivalence-crystals-modules-for-quasi-regular semiperfectoid} 
 Let $R$ be a quasi-regular semiperfectoid ring. The functor
 $$
 (\mathcal{M}, \varphi_{\mathcal{M}}) \mapsto (v^{\ast}\mathcal{M}(\prism_R,I), v^{\ast} \varphi_{\mathcal{M}}(\prism,I)  )
 $$
 of evaluation on the initial prism $(\prism_R,I)$ induces an equivalence between the category of \red{(usual or admissible)} prismatic Dieudonn\'e crystals over $R$ and the category of \red{(usual or admissible)} prismatic Dieudonn\'e modules over $R$, with quasi-inverse
 $$
 (M,\varphi_M) \mapsto (M \otimes_{\prism_R} \mathcal{O}^{\rm pris}, \varphi_M \otimes \varphi_{\mathcal{O}^{\rm pris}})
 $$
\end{proposition}
\begin{proof}
Let us call \blue{$G_R$, resp. $F_R$,} the first, resp. the second, functor displayed in the statement of the proposition. Using \Cref{sec:abstr-divid-prism-proposition-finite-locally-free} and the equivalence between finite locally free $\mathcal{O}_{\prism}$-modules and finite locally free $\prism_R$-modules, one immediately gets that $F_R$ is an equivalence between the category of prismatic Dieudonn\'e crystals over $R$ and the category of prismatic Dieudonn\'e modules over $R$, with quasi-inverse given by $G_R$. Hence, we only need to check that the admissibility conditions on both sides agree. 
  
Let $(M,\varphi_M)$ be an admissible Dieudonn\'e module over $R$. 

\begin{lemma}
Let $R \to R^\prime$ be a quasi-syntomic morphism, with $R^\prime$ being also quasi-regular semiperfectoid. Let $(M^\prime, \varphi_{M^\prime}):=(M\otimes_{\prism_R} \prism_{R^\prime}, \varphi_M \otimes \varphi_{\prism_{R^\prime}})$ be the base change of $(M,\varphi_M)$. Then 
$$
\varphi_{M^\prime}^{-1}(I\prism_{R^\prime}. M^\prime) = \mathcal{N}^{\geq 1} \prism_{R^\prime}. M^\prime + \mathrm{Im}(\varphi_M^{-1}(I.M) \otimes_{\prism_R} \prism_{R^\prime} \to M^\prime). 
$$
\end{lemma}
The lemma follows from \Cref{sec:abstr-divid-prism-proposition-equivalence-divided-prismatic-dieudonne-modules-windows} (and \Cref{alpha-morphisms-and-base-change}), which will be proved below; let us take it for granted and finish the proof. For any quasi-regular semiperfectoid ring $R^\prime$ quasi-syntomic over $R$, note that, using the notations from the lemma,
$$
\Gamma(R^\prime, F_R(M)) = M^\prime, \quad \Gamma(R^\prime, \varphi_{F_R(M)}^{-1}(\mathcal{I}^{\rm pris}.F_R(M))) =  \varphi_{M^\prime}^{-1}(I\prism_{R^\prime}. M^\prime). 
$$
The lemma tells us that in particular
$$
M^\prime/ \varphi_{M^\prime}^{-1}(I\prism_{R^\prime}. M^\prime) \cong R^\prime \otimes_R M/\varphi_M^{-1}(I.M).
$$
This being true for any quasi-regular semiperfectoid ring $R^\prime$ quasi-syntomic over $R$, we deduce that we have a short exact sequence of sheaves on $(R)_{\rm qsyn}$
$$ 
0 \to \varphi_{F_R(M)}^{-1}(\mathcal{I}^{\rm pris}.F_R(M)) \to F_R(M) \to \mathcal{O} \otimes_R M/\varphi_M^{-1}(I.M) \to 0
$$ 
By admissibility of $(M,\varphi_M)$ the right most term is a finite locally free $\mathcal{O}$-module, and thus \blue{$F_R(M)$} is admissible. 

Conversely, let $(\mathcal{M},\varphi_{\mathcal{M}})$ be an admissible Dieudonn\'e crystal. Consider the exact sequence of sheaves
$$
0 \to \varphi_{\mathcal{M}}^{-1}(\mathcal{I}^{\rm pris}.\mathcal{M}) \to \mathcal{M} \to  \mathcal{M}/\varphi_{\mathcal{M}}^{-1}(\mathcal{I}^{\rm pris}.\mathcal{M}) \to 0,
$$
and apply to it the functor $\Gamma(R,-)$. We get an exact sequence
$$
0 \to \Gamma(R, \varphi_{\mathcal{M}}^{-1}(\mathcal{I}^{\rm pris}.\mathcal{M})) = \varphi_{G_R(\mathcal{M})}^{-1}(I.G_R(\mathcal{M})) \to G_R(\mathcal{M}) \to  \Gamma(R,\mathcal{M}/\varphi_{\mathcal{M}}^{-1}(\mathcal{I}^{\rm pris}.\mathcal{M})).
$$
Since $\mathcal{M}/\varphi_{\mathcal{M}}^{-1}(\mathcal{I}^{\rm pris}.\mathcal{M})$ is a finite locally free $\mathcal{O}$-module by admissibility of $(\mathcal{M},\varphi_{\mathcal{M}})$, the right most term is a finite projective $R$-module, and it therefore suffices to show that the above sequence is also right-exact. The map 
$$
G_R(\mathcal{M}) \to  \Gamma(R,\mathcal{M}/\varphi_{\mathcal{M}}^{-1}(\mathcal{I}^{\rm pris}.\mathcal{M}))
$$ 
factors through 
$$
G_R(\mathcal{M})/\mathcal{N}^{\geq 1} \prism_R. G_R(\mathcal{M}) \to  \Gamma(R,\mathcal{M}/\varphi_{\mathcal{M}}^{-1}(\mathcal{I}^{\rm pris}.\mathcal{M}))
$$
which is a map of $R$-modules, and it suffices to show that this map is surjective. Since the target is a finitely generated $R$-module and $R$ is $p$-complete, it suffices by Nakayama's lemma to prove surjectivity after base-change along any surjection $R \to k$, with $k$ a perfect field of characteristic $p$. After base-change along such a morphism $R \to k$, the above map factors through
$$
G_R(\mathcal{M}) \otimes_{\prism_R} \prism_k \to  \Gamma(R,\mathcal{M}/\varphi_{\mathcal{M}}^{-1}(\mathcal{I}^{\rm pris}.\mathcal{M})) \otimes_R k.
$$
Since $G_R(\mathcal{M})$, resp. $\mathcal{M}/\varphi_{\mathcal{M}}^{-1}(\mathcal{I}^{\rm pris}.\mathcal{M})$, is a finite locally free $\mathcal{O}^{\rm pris}$-module, resp. a finite locally free $\mathcal{O}$-module, this identifies with the map
$$
G_k(\mathcal{M}_k) \to \Gamma(k,\mathcal{M}_k/\varphi_{\mathcal{M}_k}^{-1}(\mathcal{I}^{\rm pris}.\mathcal{M}_k)),
$$  
i.e. the same map as the one we originally wanted to prove is surjective, but now with $R$ replaced by $k$ (we denoted with an index $k$ the restrictions of the various objects involved to the quasi-syntomic site of $k$). But since $k$ is perfect, $(G_k(\mathcal{M}_k),\varphi_{G_k(\mathcal{M}_k)})$ is automatically admissible, \blue{by definition of admissibility using that every $k$-module is free}. Hence, as proved above, we have an exact sequence (using that $F_k \circ G_k \cong \mathrm{Id}$)
$$
0  \to \varphi_{\mathcal{M}_k}^{-1}(\mathcal{I}^{\rm pris}.\mathcal{M}_k) \to \mathcal{M}_k \to \mathcal{O} \otimes_R G_k(\mathcal{M}_k)/\varphi_{G_k(\mathcal{M}_k)}^{-1}(I.G_k(\mathcal{M}_k)) \to 0,
$$
i.e. 
$$
\mathcal{M}_k/\varphi_{\mathcal{M}_k}^{-1}(\mathcal{I}^{\rm pris}.\mathcal{M}_k) \cong \mathcal{O} \otimes_R G_k(\mathcal{M}_k)/\varphi_{G_k(\mathcal{M}_k)}^{-1}(I.G_k(\mathcal{M}_k)),
$$
hence 
$$
\Gamma(k,\mathcal{M}_k/\varphi_{\mathcal{M}_k}^{-1}(\mathcal{I}^{\rm pris}.\mathcal{M}_k)) \cong G_k(\mathcal{M}_k)/\varphi_{G_k(\mathcal{M}_k)}^{-1}(I.G_k(\mathcal{M}_k)).
$$
This shows that the map 
$$
G_k(\mathcal{M}_k) \to \Gamma(k,\mathcal{M}_k/\varphi_{\mathcal{M}_k}^{-1}(\mathcal{I}^{\rm pris}.\mathcal{M}_k))
$$ 
is surjective, as desired.
 \end{proof}
}

\begin{definition}
  \label{sec:abstr-divid-prism-definition-categories-of-divided-dieudonne-modules} We denote by $\mathrm{DM}(R)$ the category of prismatic Dieudonn\'e modules over $R$ (with morphisms commuting with the Frobenius) and by $\mathrm{DM}^{\rm adm}(R)$ the full subcategory formed by admissible objects. 
\end{definition}
\Cref{sec:abstr-divid-prism-proposition-equivalence-crystals-modules-for-quasi-regular semiperfectoid} shows that the possible conflict of notation is not an issue : for $R$ quasi-regular semiperfectoid, the two categories denoted by $\mathrm{DM}(R)$ are naturally equivalent, and similarly for $\DF(R)$.
\\



In the rest of this subsection, we will shortly recall the general notions of \textit{frame} and \textit{window}, and then discuss the connection with the definitions above.   

\begin{definition} \label{frame}
A \textit{frame} $\underline{A}=(A,\Fil ~A, \varphi, \varphi_1)$ consists of (classically) $(p,d)$-adically complete rings $A$ and $R=A/\Fil~A$, for some $d \in A$ and some ideal $\Fil ~A$, a lift of Frobenius $\varphi$, a $\varphi$-linear map $\varphi_1 : \Fil~ A \to A$ \red{(called the \textit{divided Frobenius} on $A$)} such that $\varphi=\varpi \varphi_1$ on $\Fil~ A$, with $\varpi=\varphi(d)$.

\red{Let $\underline{A}, \underline{A}^\prime$ be two frames, and let $u \in A^\prime$ be a unit. A \textit{$u$-morphism of frames} $\alpha: \underline{A} \to \underline{A}^\prime$ is a morphism of rings $\alpha : A \to A^\prime$ intertwinning $\varphi$ and $\varphi^\prime$, carrying $\Fil~ A$ into $\Fil~ A^\prime$, and satisfying $\varphi_1^\prime \circ \alpha = u \alpha \circ \varphi_1$ and $\alpha(\varpi)=u \varpi^\prime$.}
\end{definition}

\begin{remark} \label{firstremark}
	In many situations (such as those considered in this paper), the image of $\varphi_1$ will always generate the unit ideal of $A$.
\end{remark}

Here is an important source of examples. 

\begin{example} \label{exprism}
	Let $(A,I=(d))$ be an oriented prism. There are usually two natural ways of attaching a frame to $(A,(d))$. One possibility is to consider the frame
	\[ \underline{A}_d = (A,(d),\varphi,\varphi_1), \]
	where $\varphi_1$ is defined by $\varphi_1(dx)=\varphi(x)$ (recall that $A$ is $d$-torsion free). Here, $\varphi=\varphi(d)\varphi_1$ on $\mathrm{Fil} A=(d)$. The other possibility \red{works when $d$ is of the form $d=\varphi(d')$ for some $d'\in A$}: one can then consider the frame
	\[ \underline{A}_{\rm Nyg}=(A, \mathcal{N}^{\geq 1} A, \varphi, \varphi_1) \]
	where $\varphi_1:=\varphi/d$ on $\mathcal{N}^{\geq 1}A$ (using again that $A$ is $d$-torsion free). Here, $\varphi=d\varphi_1$ on $\mathrm{Fil} A$. Note that in the first case, the divided Frobenius is with respect to $\varphi(d)$, whereas in the second case the divided Frobenius is with respect to $d$. 
   \end{example} 

\begin{definition} \label{window}
	A \textit{window} $\underline{M}=(M,\mathrm{Fil}~M,\varphi_M,\varphi_{M,1})$ over a frame $\underline{A}$ consists of a finite locally free $A$-module $M$, an $A$-submodule $\Fil~M \subset M$, and $\varphi$-linear maps $\varphi_M : M \to M$ and $\varphi_{M,1} : \mathrm{Fil}~M \to M$, such that :
	\begin{itemize}
		\item $\Fil~A \cdot M \subset \mathrm{Fil}~M$ and $M/\mathrm{Fil}~M$ is a finite locally free $R$-module.
		\item If $a \in \mathrm{Fil}~A$, $m \in M$, $\varphi_{M,1}(am)= \varphi_1(a)\varphi_M(m)$.
		\item If $m \in \mathrm{Fil}~M$, $\varphi_M(m)=\varpi \varphi_{M,1}(m)$.
		\item $\varphi_{M,1}(\mathrm{Fil}~M) + \varphi_M(M)$ generates $M$ as an $A$-module. 
	\end{itemize} 
A morphism of windows is an $A$-linear map preserving the filtrations and commuting with $\varphi_M$ and $\varphi_{M,1}$. The category of windows over $\underline{A}$ is denoted by $\Win(\underline{A})$.  
\end{definition}

\begin{remark} \label{simplifications}
	If the surjectivity condition on the image of $\varphi_1$ of \Cref{firstremark} is satisfied, then the third point of the previous definition follows from the second and the last one simply says that $\varphi_{M,1}(\mathrm{Fil}~M)$ generates $M$ \red{(indeed, by assumption one can write $1=\sum_{i=1}^r a_i \varphi_1(b_1)$ for some $a_i \in A, b_i \in \mathrm{Fil}~A$, whence $\varpi=\sum_{i=1}^r a_i \varphi(b_i)$)}.
	\end{remark}
	
\red{	
\begin{remark}
\label{alpha-morphisms-and-base-change}
If $\alpha : \underline{A} \to \underline{A}^\prime$ is a $u$-morphism of frames as \Cref{frame} (for some unit $u\in A^\prime$), and $\underline{M}$, resp. $\underline{M}^\prime$, is a window over $\underline{A}$, resp. $\underline{A}^\prime$, an \textit{$\alpha$-morphism of windows} $f: \underline{M} \to \underline{M}^\prime$ is a morphism $f: M\to M^\prime$ of $A$-modules, intertwinning $\varphi_M$ and $\varphi_{M^\prime}$, sending $\Fil~ M$ into $\Fil ~ M^\prime$ and satisfying $\varphi_{M^\prime,1} \circ f = u f \circ \varphi_{M,1}$ (hence if $\underline{A}=\underline{A}^\prime$, $\alpha=\mathrm{Id}_{\underline{A}}$, an $\alpha$-morphism of windows is just a morphism of windows over $\underline{A}$). There is a base change functor 
$$
\alpha^\ast : \Win(\underline{A}) \to \Win(\underline{A}^\prime)
$$ characterized by the universal property that if $\underline{M}\in \Win(\underline{A})$, $\underline{M}^\prime \in \Win(\underline{A}^\prime)$, homomorphisms in $\Win(\underline{A}^\prime)$ from $\alpha^\ast \underline{M}$ to $\underline{M}^\prime$ identify with $\alpha$-morphisms of windows from $\underline{M}$ to $\underline{M}^\prime$. Concretely, if  $\underline{M}\in \Win(\underline{A})$, then $\alpha^\ast \underline{M}=(M^\prime, \Fil ~ M^\prime, \varphi_{M^\prime}, \varphi_{M^\prime, 1})$ is given by $M^\prime = A^\prime \otimes_A M$, $\Fil ~ M^\prime$ is the submodule generated by $(\Fil~ A^\prime).M^\prime$ and the image of $\Fil~ M$, and $\varphi_{M^\prime}, \varphi_{M^\prime,1}$ are uniquely determined by the requirement that $M \to M^\prime$, $m \mapsto 1 \otimes m$, is an $\alpha$-morphism of windows.
\end{remark}	
}

\red{	
\begin{proposition}
  \label{sec:abstr-divid-prism-proposition-existence-of-normal-decompositions}
  Let $\underline{A}=(A,\Fil A, \varphi,\varphi_1)$ be a frame, such that any finite projective $A/\Fil~A$-module lifts to a finite projective $A$-module. Let $(M,\mathrm{Fil}~M,\varphi_M,\varphi_{M,1})$ be a window over $\underline{A}$. Then there exist finite projective $A$-modules $L,T$ such that $M=L\oplus T$ and $\mathrm{Fil}~M=L\oplus \Fil~A. T$. Moreover, given $L,T$ there exists a bijection between $\varphi$-semilinear isomorphisms (i.e. $\varphi$-semilinear maps which become isomorphisms after linearization) $\Psi\colon \blue{L}\oplus T \to L\oplus T$ and $\underline{A}$-window structures on the pair $(L\oplus T, L \oplus \Fil~A. T)$. 
\end{proposition}
\begin{proof}
  This is a combination of \cite[Remark 2.4]{lau_frames_and_finite_group_schemes_over_complete_regular_local_rings} and \cite[Lemma 2.5]{lau_frames_and_finite_group_schemes_over_complete_regular_local_rings}. Let us give some details, and set $S:=A/\Fil~A$. The module $S\otimes_{A}M$ decomposes, as $M/\mathrm{Fil}~M$ is finite projective, into a direct sum $S\otimes_{A}M\cong M/\mathrm{Fil}~M\oplus Q$ for some finite projective $S$-module $Q$. Let $L,T$ be finite projective $A$-modules such that $L$ is a lift of $Q$ and $T$ a lift of $M/\mathrm{Fil}~M$. We can then lift the decomposition $S\otimes_{A}M$ to a decomposition $M=L\oplus T$ by projectivity. The property $\mathrm{Fil}~M=L\oplus \Fil~A T$ follows. Given $\varphi_{M}$ we define $\Psi(l+t):=\varphi_{M,1}(l)+\varphi_M(t)$ for $l\in L, t\in T$ on $M=L\oplus T$, and conversely given $\Psi$ we set $\varphi_M(l+t):=\varpi \Psi(l)+\Psi(t)$ and $\varphi_{M,1}(l+at):=\Psi(l)+\varphi_1(a)\Psi(t)$ for $l\in L,t\in T,a\in \Fil~A$.
\end{proof}

\begin{lemma}
\label{filtration-must-be-the-pullback-by-frobenius}
Let $\underline{A}=(A,\Fil A, \varphi,\varphi_1)$ as in \Cref{sec:abstr-divid-prism-proposition-existence-of-normal-decompositions} such that $\varpi$ is a non-zero divisor and $\Fil A= \varphi^{-1}(\varpi A)$. Then if $(M, \Fil~ M, \varphi_M, \varphi_{M,1})$ is a window over $\underline{A}$, we have
$$
\Fil ~ M = \varphi_M^{-1}(\varpi M)
$$
(note that one always has an inclusion $\Fil ~ M \subset \varphi_M^{-1}(\varpi M)$). \green{Moreover, $\varphi_M\colon M\to M$ induces an injection $M/\Fil M\to M/\varpi M$, and the latter extends to an injection $A/\varpi\otimes_{A/\Fil A} M/\Fil M\to M/I$ of a locally direct summand.} 
\end{lemma} 
\begin{proof}
Let
$$
M=L \oplus T
$$
be a normal decomposition of $M$ as in \ref{sec:abstr-divid-prism-proposition-existence-of-normal-decompositions}, and
$$
\Psi= (\varphi_{M,1})_{|_L} + (\varphi_{M})_{|_T},
$$
 so that $\Fil ~ M= L \oplus \Fil~ A.T$. Let $x=l+t \in M$ such that $\varphi_M(x) \in \varpi M$. We have 
$$
\varphi_M(x) = \varpi \Psi(l) + \Psi(t)
$$ 
so the condition is equivalent to requiring that $\Psi(t) \in \varpi. M$. \green{For simplicity we assume that $L, T$ are free $A$-modules in the following. The general case follows by localization.} Fix a basis $t_1,\dots,t_r$ of $T$ and a basis $l_1,\dots,l_s$ of $L$, as $A$-modules. Since $\Psi$ is a $\varphi$-linear isomorphism, the family $(\Psi(t_1),\dots,\Psi(t_r),\Psi(l_1),\dots,\Psi(l_s))$ is a basis of $M$, and so the reduction of the family $(\Psi(t_1),\dots,\Psi(t_r))$ modulo $\varpi$ is linearly independent. Write $t=\sum_{i=1}^r a_i t_i$, with $a_i \in A$ for all $i=1,\dots, r$. By assumption, we have that
$$
\Psi(t) = \sum_{i=1}^r \varphi(a_i) \Psi(t_i) \in \varpi. M
$$
and therefore we must have $\varphi(a_i) \in \varpi A$ for all $i=1,\dots,r$, i.e. $a_i \in \Fil ~A$ for all $i=1,\dots,r$, by the condition on $\Fil~A$. Hence $t \in \Fil~A.T$ and thus $x \in \Fil~M$, as desired.
\green{For the last statements note that the map $\varphi_M\colon M/\Fil M\cong T/\Fil A.T\to M/\varpi$ identifies with the map induced by $\Psi$. As $\varphi_M(t_i)=\Psi(t_i), i=1,\ldots, r,$ are linearly independent (over $A/\varpi$) this map extends to an inclusion
  \[
    A/\varpi\otimes_{A/\Fil A} M/\Fil \blue{M} \to M/\varpi
  \]
  of a direct summand. This finishes the proof.}
  \end{proof}

Let us now see what the categories of windows look like for the frames attached to prisms discussed in \Cref{exprism}.}

\begin{definition}
  \label{sec:comp-case-mathc-definition-breuil-kisin-module}
  Let $(A,I=(d))$ be a prism. A \textit{Breuil-Kisin module $(M,\varphi_M)$ over $(A,I)$, or just $A$ if $I$ is understood}, is a finite free $A$-module $M$ together with an isomorphism
  $$
  \varphi_M\colon \varphi^{\ast}M[\frac{1}{I}]\cong M[\frac{1}{I}].
  $$
  If $\varphi_M(\varphi^{\ast}M)\subseteq M$ with cokernel killed by $I$, then $(M,\varphi_M)$ is called \textit{minuscule}.
  
  We denote by $\mathrm{BK}(A)$ the category of Breuil-Kisin modules over $A$ and by $\mathrm{BK}_{\mathrm{min}}(A)\subseteq \mathrm{BK}(A)$ its full subcategory of minuscule ones.
\end{definition}

\begin{remark}
  \label{remark-projectivity-referee}
  
If $(M,\varphi_M)$ is a minuscule Breuil-Kisin module over $(A,I)$, the cokernel $N$ of $\varphi_M(\varphi^{\ast}M)\subseteq M$ is a finite projective $A/I$-module. Indeed $N$ is pseudocoherent as an $A$-module (having a $2$-term resolution by finite projective $A$-modules), hence as an $A/I$-module. Moreover, if $k$ is the residue field of $\mathrm{Spec}(A/I)$ at any closed point, then the derived tensor
$$ \bar{k} \otimes_{A/I}^L N = W(\bar{k}) \otimes_A^L N $$
is a perfect complex of $W(\bar{k})$-modules, hence bounded. It follows that the complex $k \otimes_{A/I}^L N$ is also bounded, so that $N$ has a finite resolution by finite projective $A/I$-modules (\cite[Tag 068W]{stacks_project}). Since $N$ has projective dimension $\leq 1$ as an $A$-module, it is necessarily projective as an $A/I$-module. We thank the referee for poiting out this argument to us.
\end{remark}

\begin{proposition}
  \label{sec:abstr-divid-prism-remark-cais-lau-principal}
  Let $(A,(d))$ is an oriented prism. The functor 
$$
(M,\mathrm{Fil} M,\varphi_M, \varphi_{M,1}) \mapsto (\mathrm{Fil} M, d.\varphi_{M,1})
$$
induces an equivalence between the category of windows over the frame $\underline{A}_d$ of \Cref{exprism} and the category $\mathrm{BK}_{\rm min}(A)$.  \end{proposition}
  \begin{proof}
  See \cite[Lemma 2.1.16]{cais_lau_dieudonne_crystals_and_wach_modules_for_p_divisible_groups} (taking \Cref{remark-projectivity-referee} into account). 
    \end{proof}	


Before turning to the second example introduced in \Cref{exprism}, let us recall some facts about henselian pairs. Let $A$ be a ring and let $I\subseteq A$ be an ideal. We recall that the pair $(A,I)$ is henselian if $I$ is contained in the Jacobson radical of $A$ and if for any monic polynomial $f\in A[T]$ and each factoriztion $\overline{f}=g_0h_0$ with $g_0,h_0\in A/I[T]$ monic and generating the unit ideal, there exists a factorization $f=gh$ with $g,h$ monic and $g_0=\overline{g}$, $h_0=\overline{h}$ (cf.\ \cite[Tag 09XE]{stacks_project}).

If $I$ is locally nilpotent\footnote{That is, every element in $I$ is nilpotent.} or $A$ is $I$-adically complete, then the pair $(A,I)$ is henselian (cf.\ \cite[Tag 0ALI]{stacks_project}, \cite[Tag 0ALJ]{stacks_project}).

For us the following well-known property of henselian pairs will be important (cf.\ \cite[Lemma 4.20]{clausen_mathew_morrow_k_theory_and_topological_cyclic_homology_of_henselian_pairs}).

\begin{lemma}
  \label{sec:essent-surj-lemma-iso-classes-of-finite-projective-modules-for-henselian-pairs} Let $(A,I)$ be an henselian pair. The base change $M\mapsto M\otimes_{A} A/I$ induces a bijection on isomorphism classes of finite projective modules over $A$, resp.\ $A/I$.
\end{lemma} 
\begin{proof}
  If $M,N$ are finite projective $A$-modules, then any isomorphism $M/IM\cong N/IN$ can be lifted to a morphism $M\to N$ by projectivity of $M$. As $I\subseteq A$ lies in the Jacobson radical of $A$ this lifted homomorphism is then automatically an isomorphism. Moreover, any finite projective $A/I$-module can be lifted to a finite projective $A$-module by \cite[Tag 0D4A]{stacks_project}.
\end{proof}

Now, we provide the proof that $\prism_R$ is henselian along $\mathcal{N}^{\geq 1}\prism_R=\ker(\theta\colon \prism_R\to R)$. We learned the argument from \cite[Remark 5.2]{lau_divided_dieudonne_crystals}.

\begin{lemma}
  \label{sec:essent-surj-lemma-prism_r-henselian-along-ker-theta}
  The pair $(\prism_R,\mathrm{ker}(\theta))$ is henselian.
\end{lemma}
\begin{proof}
  Because $\prism_R$ is $(p,\xi)$-adically complete it suffices to prove that the pair
  $$
  (\prism_R/(p,\xi),(p,\mathrm{ker}(\theta))/(p,\xi))
  $$
  is henselian (cf.\ \cite[Tag 0DYD]{stacks_project}). We know $\ker(\theta)=\mathcal{N}^{\geq 1}\prism_R$. Hence,
  for every element $x\in \ker(\theta)$, $x^p\in (p,\tilxi)$. As
  locally nilpotent ideals are henselian the claim follows.
\end{proof}

%
%
\red{
\begin{proposition}
  \label{sec:abstr-divid-prism-proposition-equivalence-divided-prismatic-dieudonne-modules-windows}
Let $R$ be a quasi-regular semiperfectoid ring. Fix a generator $\tilde{\xi}=\varphi(\xi)$ of the ideal $I$ of the prism $(\prism_R,I)$, giving rise to a frame $\underline{\prism}_{R, \rm Nyg}$ of \Cref{exprism} (with $d=\tilxi$). The forgetful functor
$$
\mathrm{Win}(\underline{\prism}_{R, \rm Nyg}) \to \DM(R), \quad (M, \Fil~ M, \varphi_M,\varphi_{M,1}) \mapsto (M,\varphi_M)
$$
is fully faithful, with essential image the subcategory $\DM^{\rm adm}(R)$. 
\end{proposition}
\begin{proof}
Thanks to \Cref{sec:essent-surj-lemma-prism_r-henselian-along-ker-theta}, we can apply \Cref{filtration-must-be-the-pullback-by-frobenius} to the frame $\underline{\prism}_{R, \rm Nyg}$. This yields fully faithfulness, and that for a window $(M,\Fil~M,\varphi_M,\varphi_{M,1})$ the image of
    \[
      M\xrightarrow{\varphi_M}M\to M/I\cdot M
    \]
    identifies with $M/\Fil~M$. \green{By \Cref{filtration-must-be-the-pullback-by-frobenius} we can deduce admissibility}. Assume conversely that $(M,\varphi_M)$ is an admissible prismatic Dieudonn\'e module. Then the datum
    $(M,\varphi_M^{-1}(I\cdot M), \varphi_M, \frac{1}{\tilxi}\varphi_M)$ is a window over $\underline{\prism}_{R,\rm Nyg}$. Indeed, the condition that $\varphi_{M,1}(\Fil M)$ generates $M$ follows from the definition of admissibility and \Cref{remark-on-admissibility}. This finishes the proof.
\end{proof}
}

\begin{remark}
Assume that $R$ is quasi-regular semiperfect, i.e. $R$ is quasi-regular semiperfectoid and $pR=0$. Let $(M,\varphi_M)$ be a prismatic Dieudonn\'e module over $R$. Let $N \subset M/\mathcal{N}^{\geq 1}\prism_R M$ be a \red{locally free $R$-module which is a direct summand}, and define $\mathrm{Fil}~M$ to be the inverse image of $N$ in $M$. Then the collection \red{$(M,\mathrm{Fil}~M,\varphi_M,1/p \varphi_M)$} is a \red{window over $\underline{\prism}_{R,\rm Nyg}=\underline{A_{\rm crys}(R)}_{\rm Nyg}$} if and only if $N$ is an ``admissible'' filtration in the sense of Grothendieck on the Dieudonn\'e module $(M,\varphi_M, V_M)$, where $V_M=\varphi_M^{-1}.p$ (which makes sense by the assumption that $(M,\varphi_M)$ is a prismatic Dieudonn\'e module). For a proof of this, see \cite[Lemma 2.5.1]{cais_lau_dieudonne_crystals_and_wach_modules_for_p_divisible_groups}).  
\end{remark}

We can now prove \Cref{sec:abstr-divid-prism-divided-prismatic-dieudonne-modules-vs-bkf-modules}.

 \begin{proof}[Proof of \Cref{sec:abstr-divid-prism-divided-prismatic-dieudonne-modules-vs-bkf-modules}]
 We know by \Cref{sec:abstr-divid-prism-proposition-equivalence-divided-prismatic-dieudonne-modules-windows} that the functor
 $$
 (M, \varphi_M) \mapsto (M,\varphi_M^{-1}(\tilxi. M), \varphi_M, \frac{1}{\tilxi} \varphi_M) 
 $$
 is an equivalence between $\DM^{\rm adm}(R)$ and $\mathrm{Win}(\underline{\prism}_{R,\rm Nyg})$. Since $R$ is perfectoid, $\mathcal{N}^{\geq 1}\prism_R =(\xi)$, and so
 $$
 \underline{\prism}_{R,\rm Nyg} = \underline{\prism}_{R,\xi}
 $$
 By \Cref{sec:abstr-divid-prism-remark-cais-lau-principal}, the functor 
  \[ (N, \mathrm{Fil}~N, \varphi_N) \mapsto (\mathrm{Fil}~N, \frac{\xi}{\tilxi} \varphi_{N}) \]
induces an equivalence between $\mathrm{Win}( \underline{\prism}_{R,\xi})$ and \blue{$\mathrm{BK}_{\rm min}(A_{\rm inf}(R))$ (the category of minuscule Breuil-Kisin modules over $A_{\rm inf}(R)$). The latter category is, however, obviously equivalent to $\DM(R')$, with $R'=A_{\rm inf}(R) / \xi$.} As $\varphi$ is bijective on $\prism_R$, base change along $\varphi$ is also an equivalence between $\DM(R')$ and $\DM(R)$. Composing these equivalences, we obtain an equivalence
$$
\DM^{\rm adm}(R) \to \DM(R)
$$
But this composite functor is nothing but the identity functor. 
  \end{proof}

Finally, we record some statements which are later used to prove essential surjectivity for the prismatic Dieudonn\'e functor.

For a ring $A$ with an endomorphism $\varphi\colon A\to A$ we denote by $\varphi-\mathrm{Mod}_A^{\mathrm{unit}}$ the category of ``unit'' $\varphi$-modules over $A$, i.e., the category of pairs $(M,\varphi_M)$ with $M$ a finite projective $A$-module and $\varphi_M\colon \varphi^\ast M\cong M$ an isomorphism.

\begin{lemma}
  \label{sec:essent-surj-lemma-equivalence-of-varphi-modules-over-delta-rings}
Let $A\to B$ be a surjection of \blue{bounded} prisms with kernel $J\subseteq A$. Assume that the Frobenius $\varphi$ of $A$ is topologically nilpotent \blue{(for the $(p,I)$-adic topology)} on $J$ and that $(A,J)$ is henselian. Then the functor
$$
\varphi-\mathrm{Mod}_A^{\mathrm{unit}}\to \varphi-\mathrm{Mod}_B^{\mathrm{unit}},\ (M,\varphi_M)\mapsto (M\otimes_A B,\varphi_M\otimes_A B)
$$ 
is an equivalence.
\end{lemma}
\begin{proof}
  To prove fully faithfulness it suffices to show (by passing to internal hom's) that for every $\varphi$-module $(M,\varphi_M)$ over $A$ the map
$$
M^{\varphi_M=1}\to (M/JM)^{\varphi_M=1}
$$
is bijective. Let $m\in M^{\varphi_M=1}\cap JM$ and write $m=\sum\limits_{i=1}^na_im_i$ with $a_i\in J$ and $m_i\in M$. Then
$$
m=\varphi^j_M(m)=\sum\limits_{i=1}^n\varphi^j(a_i)\varphi_M^j(m_i)
$$
where the $\varphi^j(a_i)$ converge to $0$ if $j\to \infty$ by our assumption on $\varphi$. Thus $m=\varphi^j_M(m)\to 0$ if $j\to \infty$ and therefore $m=0$, which proves injectivity.
Conversely, let $m\in M$ and assume that $\varphi_M(m)\equiv m$ modulo $JM$. Write
$$
z:=\varphi_M(m)-m\in JM.
$$
As above the sequence $\varphi_M^j(z)$ converges to $0$ if $j\to \infty$. Set
$$
\tilde{m}:=m+\sum\limits_{j=0}^\infty \varphi_M^j(z).
$$
Then $\tilde{m}\equiv m$ modulo $JM$ and $\varphi_M(\tilde{m})=\tilde{m}$.
Thus we showed that
$$
M^{\varphi_M=1}\cong (M/JM)^{\varphi_M=1}
$$
and the functor $\varphi-\mathrm{Mod}_A^{\mathrm{unit}}\to \varphi-\mathrm{Mod}_B^{\mathrm{unit}}$ is fully faithful and we are left with essential surjectivity.
For this let $(N,\varphi_N)\in \varphi-\mathrm{Mod}_B^{\mathrm{unit}}$. By assumption $A$ is henselian along $J$ and thus we can write $N\cong M\otimes_A B$ for some finite projective $A$-module $M$. Using projectiviy of $\varphi^\ast M$ over $A$ we can lift $\varphi_N\colon \varphi^\ast N\to N$ to some homomorphism $\varphi_M\colon \varphi^\ast M\to M$. As $J$ lies in the radical of $A$ the homomorphism $\varphi_M$ will automatically be an isomorphism as $\varphi_N$ is. Thus, we have lifted $(N,\varphi_N)$ to $(M,\varphi_M)$, which finishes the proof.
\end{proof}

The following statement is similar to \cite[Lemma 2.12]{lau_frames_and_finite_group_schemes_over_complete_regular_local_rings} or \cite[Appendix A.4]{kisin_crystalline_representations_and_f_crystals}.

  It will use the ``Nygaard frame'' associated to an oriented \blue{prisms}, which was discussed in \ref{exprism}. 
  
\begin{lemma}
  \label{sec:essent-surj-lemma-equivalence-on-window-categories}
  Let $(A,(\tilxi)) \to (B,(\tilxi))$ be a surjection of oriented \blue{bounded} prisms with kernel $J$ contained in $\mathcal{N}^{\geq 1}A$, and assume that $\tilxi=\varphi(\xi)$ for some $\xi \in A$ and that $(A,\tilxi)$ bounded. Assume that $\varphi_1$ is (pointwise) topologically nilpotent on $J$ and that $(A,J)$ is henselian.
  Then the base change functor induces an equivalence :
  $$\mathrm{Win}(\underline{A}_{\mathrm{Nyg}}) \simeq \mathrm{Win}((B,\mathcal{N}^{\geq 1} A/J,A/\mathcal{N}^{\geq 1}A, \varphi,\varphi_1)).$$
\end{lemma}

We note that $\varphi_1(J)\subseteq J$ as $B$ is $\tilxi$-torsion free and $\varphi(j)=\tilxi\varphi_1(j)$ in $A$. Thus the condition that $\varphi_1$ is topologically nilpotent on $J$ makes sense. Moreover, $\varphi_1(J)\subseteq J$ implies that $(B,\mathcal{N}^{\geq 1} A/J, A/\mathcal{N}^{\geq 1}A, \varphi,\varphi_1)$ is indeed a well-defined frame.

\begin{proof}
 In this proof, we will use the following convenient notation: if $\sigma: S\to S$ is a ring endomorphism and $f: M \to N$ is a $\sigma$-linear map between two $S$-modules, we will denote by $f^\sharp: \sigma^\ast M \to N$ its linearization. \blue{We will also abbreviate $\underline{A}_{\mathrm{Nyg}}$ as $\underline{A}$ and $(B,\mathcal{N}^{\geq 1} A/J,A/\mathcal{N}^{\geq 1}A, \varphi,\varphi_1)$ as $\underline{B}$.}
  
  By the existence of normal decompositions \blue{(cf. \Cref{sec:abstr-divid-prism-proposition-existence-of-normal-decompositions}: we can apply it since the proof of \Cref{sec:essent-surj-lemma-prism_r-henselian-along-ker-theta} shows that $A$ is henselian along $\mathcal{N}^{\geq 1} A$ and this implies that finite projective $B/\mathrm{im}(\mathcal{N}^{\geq 1} A)$-modules can be lifted to finite locally free $B$-modules -- even to finite projective $A$-modules)} and the fact that $A$ is henselian along $J$ the base change functor
  $$
  \mathrm{Win}(\underline{A})\to \mathrm{Win}(\underline{B})
  $$
  is essentially surjective. Let $\underline{M}, \underline{N}$ be two windows over $\underline{A}$. We want to prove that
  $$
  \mathrm{Hom}_{\underline{A}}(\underline{M},\underline{N})\cong \mathrm{Hom}_{\underline{B}}(\underline{M}/J,\underline{N}/J)
  $$
  where $\underline{M}/J, \underline{N}/J$ denote the base change of $\underline{M},\underline{N}$ to $\underline{B}$. The idea of proof is similar to \Cref{sec:essent-surj-lemma-equivalence-of-varphi-modules-over-delta-rings} (and \cite[Theorem 3.2]{lau_frames_and_finite_group_schemes_over_complete_regular_local_rings}).
  Let
  $$
  \beta\colon  M\to JN
  $$
  be an arbitrary homomorphism of $A$-modules. Then the $A$-module homomorphism
  $$
  U(\beta)\colon M\to JN,\ m\mapsto \varphi_{N,1}^\sharp (\mathrm{Id}\otimes \beta)(\varphi_{M,1}^\sharp)^{-1}(m) 
  $$
  is well-defined. Indeed, $\varphi_M^\sharp \colon \varphi^\ast M\to M$ is injective with cokernel killed by $\tilxi$ (which follows from the fact that $\varphi_{M,1}(\mathrm{Fil} M)$ generates $M$ and that $M,\varphi^\ast(M)$ are $\tilxi$-torsion free) and thus on $\tilxi M$ there exists a partial inverse $\blue{(\varphi_M^\sharp)}^{-1}\colon \tilxi M\to \varphi^{\ast}M$ of $\blue{\varphi_M^\sharp}$. Moreover, as $\beta$ has image in $JN$ the composition $\varphi_N^\sharp (\mathrm{Id}\otimes \beta)$ has image in $\tilxi N$. The module $M$ is finitely generated: choose generators $x_1, \dots, x_r$. For each $n\geq 1$, and each $x \in M$, we can write
  $$
    \blue{ ((\varphi^{n-1})^\ast (\varphi_{M,1}^\sharp)^{-1} \circ \dots \circ \varphi^\ast (\varphi_{M,1}^\sharp)^{-1} \circ (\varphi_{M,1}^\sharp)^{-1} (x) } = \sum_{i=1}^r b_{i,n}(x) \otimes x_i \in (\varphi^n)^\ast M, 
   $$
   with $b_{i,n}(x) \in A$.
   Hence, we get 
   $$
   U^n(\beta)(x)= \blue{(((\varphi_{N,1})^\sharp \circ \varphi^\ast (\varphi_{N,1})^\sharp  \circ \dots \circ (\varphi^{n-1})^\ast (\varphi_{N,1})^\sharp) }\circ (\varphi^n)^\ast \beta) \left(  \sum_{i=1}^r b_{i,n}(x) \otimes x_i \right)
   $$ 
    whence
   $$
   U^n(\beta)(x)= \sum_{i=1}^r \varphi^n(b_{i,n}(x)) \varphi_{N,1}^n(\beta(x_i)).
   $$
    Write for each $i=1,\dots, r$,
  $$
  \beta(x_i)=\sum_{k=1}^{s_r} j_{i,k} y_{i,k},
  $$
   with $j_{i,k} \in J$, $y_{i,k} \in N$.
   We have, for each $i=1,\dots, r$, 
  $$
  \varphi_{N,1}^n(\beta(x_i)) = \varphi_{N,1}^n \left(\sum_{k=1}^{s_r} j_{i,k} y_{i,k}\right) = \sum_{k=1}^{s_r} \varphi_1^n(j_{i,k}) \varphi_{N}^n(y_{i,k}).
  $$
  By our assumption, $\varphi_1$ on $J$ is pointwise topologically nilpotent, and so in particular for each $m_0 \geq 0$, we can find $m\geq 0$ such that $\varphi_1^m(j_{i,k}) \in (p,\tilxi)^{m_0}$, for all $i=1,\dots, r$, $j=1,\dots, s_r$. The above equalities show that for all $n\geq m$ and for all $x\in M$, 
  $$
  U^n(\beta)(x) \in (p,\tilxi)^{m_0} N.
  $$
 Hence, we deduce from the above that for every $\beta\colon M\to JN$ the sequence
  $$
  \beta, U(\beta),U(U(\beta)),\ldots, U^n(\beta),\ldots
  $$
  converges to $0$ (because $A$ is $(p,\tilxi)$-adically complete as $(A,\tilxi)$ is bounded).
  Now let $\alpha\colon M\to N$ be a homomorphism of windows such that $\alpha\equiv 0$ modulo $J$. Then $U^n(\alpha)=\alpha$ for all $n$ because $\alpha\circ \varphi_M=\varphi_N\circ \alpha$, which implies $\alpha=0$ as the sequence $U^n(\alpha)$ converges to $0$ as we saw above.
  Conversely, assume that $\alpha\colon M\to N$ is an $A$-module homomorphism, such that $\alpha$ modulo $J$ is an homomorphism of windows over $\underline{B}$. Then $\alpha$ maps $\mathrm{Fil} M$ to $\mathrm{Fil} N$ because this can be checked modulo $J$. Note that $(\varphi_M^\sharp)^{-1}(\tilxi.M)= \varphi^\ast(\mathrm{Fil} M)$ as follows from \Cref{filtration-must-be-the-pullback-by-frobenius}. Hence $U(\alpha)$ sends $M$ to $N$. Set
  $$
  \beta:= U(\alpha)-\alpha\colon M\to N.
  $$
  Then $\beta(M)\subseteq JN$ by the assumption on $\alpha$. Therefore the homomorphism
  $$
\tilde{\alpha}\colon M\to N,\ m\mapsto \alpha(m)+\sum\limits_{n=0}^\infty U^n(\beta)(m)
$$
is well-defined. Moreover, $\alpha\equiv \tilde{\alpha}$ modulo $J$ and $\tilde{\alpha}$ is a homomorphism of windows over $\underline{A}$.
 \end{proof}

From the proof of the last lemma, one can also extract the following statement.
\begin{lemma}
  \label{sec:essent-surj-lemma-faithful-dieudonne-modules-divid-frob-top-nilpotent}
Let $R \to R'$ be a morphism of quasi-regular semiperfectoid rings such that $J=\ker(\prism_R \to \prism_{R'})$ is contained in $\mathcal{N}^{\geq 1} \prism_R$, stable by $\varphi_1$ and such that $\varphi_1$ is topologically nilpotent on $J$ (for some, or equivalently any, choice of a generator of the ideal $I$ defining the prism structure of $\prism_R$). Then the base change functors
\[ \mathrm{DM}(R) \to \mathrm{DM}(R') \quad ; \quad \red{\DM^{\rm adm}(R) \to \DM^{\rm adm}(R')} \]
are faithful.
\end{lemma}
\begin{proof}
It is enough to prove that the first functor is faithful. For this, one uses the exact same argument used in the proof of \Cref{sec:essent-surj-lemma-equivalence-on-window-categories}.
\end{proof}

\begin{remark}
  \label{sec:abstr-filt-prism-remark-faithfulness-on-windows-if-divided-frob-is-topologically-nilpotent}
More generally, if one has a $1$-morphism of frames $\underline{A} \to \underline{A}'$, whose kernel $J$ is contained in $\mathrm{Fil}~ A$, stable by $\varphi_1$, and such that $\varphi_1$ is topologically nilpotent on $J$, the same proof shows that the base change functor
\[ \mathrm{Win}(\underline{A}) \to \mathrm{Win}(\underline{A}') \]
is faithful. 
\end{remark}

\subsection{Definition of the prismatic Dieudonn\'e functor}
\label{sec:divid-prism-dieud-definition-divided-prismatic-dieudonne-crystals}

In this subsection we define the prismatic Dieudonn\'e crystals of $p$-divisible groups over quasi-syntomic rings and prove some formal properties of them. More difficult properties, like the crystal property or local freeness, will be proved later (cf.\ \Cref{sec:divid-prism-dieud-definition-for-p-div-groups}) after discussing the case of abelian schemes first (cf.\ \Cref{sec:prism-dieud-modul-dieud-mod-abel-schemes}).

Let $R\in \mathrm{QSyn}$ be a quasi-syntomic ring and let $(R)_\prism$ be its absolute prismatic site. We recall from \Cref{sec:abstr-divid-prism-proposition-finite-locally-free} that the category of finite locally free crystals on $(R)_{\prism}$ is equivalent to the category of finite locally free $\mathcal{O}^{\pris}$-modules on the small quasi-syntomic site $(R)_{\mathrm{qsyn}}$ of $R$ endowed with the quasi-syntomic topology.

Recall as well that there is an exact sequence
$$
0\to \mathcal{N}^{\geq 1}\mathcal{O}^{\mathrm{pris}}\to \mathcal{O}^\pris\to \mathcal{O}\to 0
$$
where $\mathcal{O}$ is the structure sheaf $S\in (R)_{\mathrm{qsyn}}\mapsto S$  on $(R)_{\mathrm{qsyn}}$ (cf.\ \Cref{sec:abstr-divid-prism-proposition-quotient-sheaf-is-structure-sheaf}).

\begin{definition}
  \label{sec:divid-prism-dieud-definition-divided-prismatic-dieudonne-crystal-of-p-divisible-groups} Let $G$ be a $p$-divisible group over $R$. We define\footnote{For an alternative perspective on this definition, using classifying stacks, see the work of Mondal \cite{mondal2021dieudonne}.}
  $$
    \mathcal{M}_{\prism}(G):=\mathcal{E}xt^1_{(R)_{\mathrm{qsyn}}}(G,\mathcal{O}^\pris) \\
  $$
  and $\varphi_{\mathcal{M}_\prism(G)}$ as the endomorphism of $\mathcal{M}_\prism(G)$ induced from the endomorphism $\varphi$ on $\mathcal{O}^\pris$.
  We call $(\mathcal{M}_\prism(G),\varphi_{\mathcal{M}_\prism(G)})$ the \textit{prismatic Dieudonn\'e crystal} of $G$.
\end{definition}

We will check later that $(\mathcal{M}_\prism(G),\varphi_{\mathcal{M}_\prism(G)})$ is indeed a(n admissible) prismatic Dieudonn\'e crystal.

\begin{remark}
\label{no-hom-btw-o-pris-and-G}
Let us note that
$$
\mathcal{H}om(G,Q)=0
$$
for any derived $p$-adically complete quasi-syntomic sheaf $Q$. Indeed, the finite locally free group schemes $G[p^n]$ are syntomic over $R$ for $n\geq 0$ (as follows e.g.\ from \cite[II.(3.2.6)]{messing_the_crystals_associated_to_barsotti_tate_groups}) \blue{(hence multiplication by $p$ on $G$ is surjective in the syntomic topology)}. This implies that the derived $p$-completion of $G$ on the big quasi-syntomic site over $R$ is given by $T_pG$ placed in degree $-1$. As there are no morphisms from $D^{\leq -1}$ to $D^{\geq 0}$, and $Q$ is assumed to be derived $p$-adically complete, the statement follows. 

In particular, we can apply this to $Q=\mathcal{O}^\pris$ and deduce that 
$$
\mathcal{H}om(G, \mathcal{O}^\pris)=0
$$
and thus also
$$
\mathcal{H}om(G,\mathcal{N}^{\geq 1} \mathcal{O}^\pris)=0.
$$
\end{remark}

\begin{remark}
Beware that the prismatic Dieudonn\'e crystal of a $p$-divisible group is a sheaf on the quasi-syntomic site, not on the prismatic site. In particular, it is not a crystal on the prismatic site of $R$, but rather the push-forward along $v$ of a crystal on the prismatic site (as will be proved later). We hope that this choice of terminology does not create too much confusion ; from the mathematical point of view, it is justified by \Cref{sec:abstr-divid-prism-proposition-finite-locally-free}.
\end{remark}


Fix a $p$-divisible group $G$ over $R$.
We check some easy properties of $\red{\mathcal{M}}_{\prism}(G)$.

%
%

In \cite{berthelot_breen_messing_theorie_de_dieudonne_cristalline_II}, the crystalline Dieudonn\'e crystal of a $p$-divisible group is defined via the sheaf of local extensions on the crystalline site. There is a similar description of the prismatic Dieudonn\'e crystal.
%

\begin{lemma}
  \label{sec:divid-prism-dieud-lemma-divided-dieudonne-module-via-local-ext-on-prismatic-site} There is a canonical isomorphism
  $$
      \mathcal{M}_{\prism}(G)\cong v_\ast(\mathcal{E}xt^1_{(R)_\prism}(u^{-1}(G),\mathcal{O}_\prism)).
  $$
\end{lemma}
\begin{proof}
First, we claim that there is a canonical isomorphism
  $$
    \mathcal{E}xt^1_{(R)_{\mathrm{QSYN}}}(G,u_\ast \mathcal{O}_\prism)\cong u_\ast(\mathcal{E}xt^1_{(R)_\prism}(u^{-1}(G),\mathcal{O}_\prism)).
    $$
    By adjunction, there is a canonical isomorphism
    \[      
      R\mathcal{H}om_{(R)_{\mathrm{QSYN}}}(G,Ru_\ast(\mathcal{O}_{\prism}))\cong Ru_\ast(R\mathcal{H}om_{(R)_\prism}(u^{-1}G,\mathcal{O}_\prism)).
    \]

   It thus suffices to see that $\mathcal{E}xt^1_{(R)_{\mathrm{QSYN}}}(G,u_\ast \mathcal{O}_\prism)$, resp.\ $u_\ast(\mathcal{E}xt^1_{(R)_\prism}(u^{-1}(G),\mathcal{O}_\prism))$, are the first cohomology sheaves on both sides.
  The sheaves
  $$
  \mathcal{H}om(G,R^1u_\ast(\mathcal{O}_\prism)), \mathcal{H}om(u^{-1}(G),\mathcal{O}_\prism)
  $$
  are $0$: \blue{for the first this follows as $G$ is $p$-divisible and the target \red{derived $p$-complete, cf. \Cref{no-hom-btw-o-pris-and-G}} and for the second the same argument as in  \Cref{no-hom-btw-o-pris-and-G} can be applied since the multiplication by $p$ map on $u^{-1}(G)$ is surjective and the prismatic topos is replete}. This implies the claim. 
  
  To finish, the proof of the proposition, it therefore remains to show that we have
  $$
  \epsilon_\ast \mathcal{E}xt^1_{(R)_{\mathrm{QSYN}}}(G,u_\ast \mathcal{O}_\prism) \cong  \mathcal{M}_{\prism}(G).
  $$  
 We will in fact give an argument, inspired by \cite{bauer1992conjecture}, which works with $\mathcal{E}xt^1$ replaced by $\mathcal{E}xt^i$, for any $i\geq 0$. The Breen-Deligne resolution $C(G)$ of $G$, seen either as a sheaf on the big or on the small quasi-syntomic site (cf. \cite[Appendix to Lecture IV]{scholze_lectures_on_condensed_mathematics}, see also \Cref{sec:calc-ext-groups} below for a partial explicit resolution, sufficient for our purposes), give, for each $i\geq 0$, spectral sequences 
  $$
 \mathcal{E}xt^{i-j}_{(R)_{\mathrm{QSYN}}}(C_j(G),u_\ast \mathcal{O}_\prism) \Longrightarrow  \mathcal{E}xt^i_{(R)_{\mathrm{QSYN}}}(G,u_\ast \mathcal{O}_\prism), 
 $$
 and 
 $$
 \mathcal{E}xt^{i-j}_{(R)_{\mathrm{qsyn}}}(C_j(G),v_\ast \mathcal{O}_\prism) \Longrightarrow  \mathcal{E}xt^i_{(R)_{\mathrm{qsyn}}}(G,v_\ast \mathcal{O}_\prism), 
 $$
 Since for each $j$, $C_j(G)$ is a finite direct sum of terms of the form $\Z[G^n]$, $n\geq 1$, it suffices to show that for each $k\geq 0, j\geq 1$, 
 $$
 \epsilon_\ast  \mathcal{E}xt^k_{(R)_{\mathrm{QSYN}}}(\Z[G^j],u_\ast \mathcal{O}_\prism) \cong  \mathcal{E}xt^k_{(R)_{\mathrm{qsyn}}}(\Z[G^j],v_\ast \mathcal{O}_\prism).
 $$
   Since $f_n : G^n \to \mathrm{Spf}(R)$ is quasi-syntomic, it induces a morphism of topoi $f_{n,\rm qsyn}: G_{\rm qsyn}^n \to (R)_{\rm qsyn}$, identifying $G_{\rm qsyn}^n$ with the slice topos $(R)_{\rm qsyn}/G^n$. Hence,
    $$
    \mathcal{E}xt^k_{(R)_{\mathrm{qsyn}}}(\Z[G^j],v_\ast \mathcal{O}_\prism) \cong R^k f_{n,\mathrm{qsyn},\ast} f_{n, \rm qsyn}^\ast \epsilon_\ast u_\ast \mathcal{O}_\prism. 
    $$
    Analogously, if $f_{n,\rm QSYN}: G_{\rm QSYN}^n \to (R)_{\rm QSYN}$ denotes the morphism of topoi induced by $f_n$, we have 
     $$
    \mathcal{E}xt^k_{(R)_{\mathrm{QSYN}}}(\Z[G^j],u_\ast \mathcal{O}_\prism) \cong R^k f_{n,\mathrm{QSYN},\ast} f_{n, \rm QSYN}^\ast  u_\ast \mathcal{O}_\prism. 
    $$
   The sheaf $R^k f_{n,\mathrm{qsyn},\ast} f_{n, \rm qsyn}^\ast \epsilon_\ast u_\ast \mathcal{O}_\prism$ is the sheaf attached to the presheaf sending $X \in (R)_{\rm qsyn}$ to 
   $$
   H^k(f_{n, \rm qsyn}^\ast X,  f_{n, \rm qsyn}^\ast \epsilon_\ast u_\ast \mathcal{O}_\prism)
   $$
   while the sheaf $\epsilon_\ast R^k f_{n,\mathrm{QSYN},\ast} f_{n, \rm QSYN}^\ast  u_\ast \mathcal{O}_\prism$ is the sheaf attached to the presheaf sending 
   $X \in (R)_{\rm qsyn}$ to 
   $$
   H^k(f_{n, \rm QSYN}^\ast X,  f_{n, \rm QSYN}^\ast  u_\ast \mathcal{O}_\prism)
   $$
Both $f_{n, \rm qsyn}^\ast X$ and $f_{n, \rm QSYN}^\ast X$ are represented by $X \times_{\mathrm{Spf}(R)} G^n \in G_{\rm qsyn}^n$, and therefore $H^k(f_{n, \rm qsyn}^\ast X,  f_{n, \rm qsyn}^\ast \epsilon_\ast u_\ast \mathcal{O}_\prism)$, resp. $H^k(f_{n, \rm QSYN}^\ast X,  f_{n, \rm QSYN}^\ast  u_\ast \mathcal{O}_\prism)$, agrees with $H_{\rm qsyn}^k(X \times_{\mathrm{Spf}(R)} G^n, \epsilon_\ast u_\ast \mathcal{O}_\prism)$, resp. with $H_{\rm QSYN}^k(X \times_{\mathrm{Spf}(R)} G^n, u_\ast \mathcal{O}_\prism)$. But these last two cohomology groups agree, since on both sites, quasi-regular semiperfectoid rings form a basis on which the cohomology in positives degrees of $u_\ast \mathcal{O}_\prism$ vanishes. Whence our claim, and the end of the proof. 
\end{proof}

Using the $p$-adic Tate module $T_pG$ of $G$, i.e., the inverse limit
$$
\varprojlim\limits_{n} G[p^n]
$$
of sheaves on $(R)_{\mathrm{qsyn}}$, one can give a more explicit description of the prismatic Dieudonn\'e crystal $\mathcal{M}_{\prism}(G)$.

\begin{lemma}
  \label{sec:divid-prism-dieud-lemma-exactness-of-universal-covering-exact-sequence} Define the universal cover $\tilde{G}:=\varprojlim\limits_{p} G$ of $G$. Then the sequences
  $$
  \begin{matrix}
    0\to T_pG\to \tilde{G}\to G\to 0\\
    0\to u^{-1} T_pG\to u^{-1} \tilde{G}\to u^{-1} G\to 0    
  \end{matrix}
  $$
  of sheaves on $(R)_{\mathrm{qsyn}}$ resp.\ $(R)_\prism$ are exact for the quasi-syntomic topology.
\end{lemma}
\begin{proof}
  Exactness of the second follows from exactness of the first and exactness of $u^{-1}$ (cf.\ \Cref{sec:prismatic-to-syntomic-corollary-exactness-if-kernel-is-affine-p-completely-syntomic}). Each $G[p^n]$ is syntomic over $R$. This implies that $\tilde{G}\to G$ is a quasi-syntomic cover, which implies exactness of the first sequence.
\end{proof}

The following lemma will be useful when describing the prismatic Dieudonn\'e crystals of $\Q_p/\Z_p$ and $\mu_{p^\infty}$ and when proving fully faithfulness of the prismatic Dieudonn\'e functor.

\begin{lemma}
  \label{sec:divid-prism-dieud-lemma-prismatic-dieudonne-module-via-hom}
  There are canonical isomorphisms
  $$
  \mathcal{M}_{\prism}(G)\cong \mathcal{H}om_{(R)_\mathrm{qsyn}}(T_pG,\mathcal{O}^\pris) \cong v_*\mathcal{H}om_{(R)_\prism}(u^{-1}(T_pG),\mathcal{O}_\prism).
  $$
\end{lemma}
\begin{proof}
  This follows from \Cref{sec:divid-prism-dieud-lemma-exactness-of-universal-covering-exact-sequence} and the fact that
  $$
  R\mathcal{H}om_{(R)_{\prism}}(u^{-1}(\tilde{G}),\mathcal{O}_\prism)=0 \quad ; \quad
  R\mathcal{H}om_{(R)_{\mathrm{qsyn}}}(\tilde{G},\mathcal{O}^\pris)=0
  $$ as $\mathcal{O}_{\prism}, \mathcal{O}^\pris$ are derived $p$-complete sheaves and $\tilde{G}$ is a $\Q_p$-vector space. 
  \end{proof}


\begin{remark}
The universal vector extension $E(G)$ of $G$ can be seen as an extension of sheaves on $(R)_{\mathrm{qsyn}}$  :
\[ 0 \to \omega_{\check{G}} \to E(G) \to G \to 0. \]
It is defined as in \cite{messing_the_crystals_associated_to_barsotti_tate_groups} (this makes sense since $R$ is $p$-complete), or equivalently as the push-out of the universal cover exact sequence 
\[  0\to T_pG\to \tilde{G}\to G\to 0 \]
along the \textit{Hodge-Tate map}
\[ HT : T_p G \to \omega_{\check{G}}, \]
 which sends $f \in T_p G = \mathrm{Hom}_R(\mathbb{Q}_p/\mathbb{Z}_p,G)$, viewed by Cartier duality as an element of \red{$\mathrm{Hom}_{R}(\check{G},\mu_{p^{\infty}})$}, to $f^* dT/T$, $dT/T$ being the canonical generator of $\omega_{\mu_{p^{\infty}}}$. Is there a way to use \Cref{sec:divid-prism-dieud-lemma-prismatic-dieudonne-module-via-hom} to relate the prismatic Dieudonn\'e module to the dual of the Lie algebra of $E(G)$?
 \end{remark}

Assume now that $R$ is quasi-regular semiperfectoid. Then, by \Cref{sec:abstr-divid-prism-proposition-finite-locally-free}, the category of finite locally free crystals on $(R)_\prism$ is equivalent to the category of finite projective $\prism_R$-modules by evaluating a crystal on the initial prism $\prism_R$. Similarly, finite locally free $\mathcal{O}^\pris$-modules on $(R)_{\mathrm{qsyn}}$ are equivalent to finite projective $\prism_R$ by evaluating a finite locally free $\mathcal{O}^\pris$-module $\mathcal{M}$ on $R$.
This allows the following simplification of the definition of the prismatic Dieudonn\'e crystal of a $p$-divisible group $G$ over $R$.

\begin{definition}
  \label{sec:divid-prism-dieud-definition-divided-prismatic-dieudonne-module}
  Let $R$ be quasi-regular semiperfectoid and let $G$ be a $p$-divisible group over $R$. Define
  $$
     M_{\prism}(G):=\mathrm{Ext}^1_{(R)_{\mathrm{qsyn}}}(G,\mathcal{O}^\pris)\cong \mathrm{Ext}^1_{(R)_\prism}(u^{-1}(G),\mathcal{O}_\prism)
       $$
  and $\varphi_{M_\prism(G)}$ as the endomorphism induced by $\varphi$ on $\mathcal{O}^\pris$. We call
  $$
  (M_{\prism}(G),\varphi_{M_\prism(G)})
  $$
  the \textit{prismatic Dieudonn\'e module} of $G$.
\end{definition}

We will see later that $M_{\prism}(G)$ is indeed \red{a(n admissible) prismatic Dieudonn\'e module} in the sense of \Cref{sec:abstr-divid-prism-definition-prismatic-dieudonne-modules-for-quasi-regular semiperfectoid}.
Moreover, $M_{\prism}(G)$ is the evaluation of the prismatic Dieudonn\'e crystal $\mathcal{M}_\prism(G)$ as follows from the local-global spectral sequence
$$
E^{ij}_2=H^i(\Spf(R),\mathcal{E}xt^j_{(R)_\mathrm{qsyn}}(G,\mathcal{O}^\pris))\Rightarrow \mathrm{Ext}^{i+j}_{(R)_\mathrm{qsyn}}(G,\mathcal{O}^\pris)
$$
by the vanishing of the sheaf $\mathcal{H}om_{(R)_\mathrm{qsyn}}(G,\mathcal{O}^\pris)$.
Thus under the equivalence from \Cref{sec:abstr-divid-prism-proposition-equivalence-crystals-modules-for-quasi-regular semiperfectoid} the prismatic Dieudonn\'e crystal $\mathcal{M}_\prism(G)$ corresponds to the prismatic Dieudonn\'e module $M_\prism(G)$.

\subsection{Comparison with former constructions}
\label{sec:comp-with-cryst}

In this section we prove a comparison of the prismatic Dieudonn\'e functor \red{$\mathcal{M}_{\prism}$} with former constructions, in two special cases :
\begin{enumerate}
\item For quasi-syntomic rings such that $pR=0$, we relate $\mathcal{M}_{\prism}$ to the crystalline Dieudonn\'e functor of Berthelot-Breen-Messing \cite{berthelot_breen_messing_theorie_de_dieudonne_cristalline_II}. 
\item For perfectoid rings, we relate the prismatic Dieudonn\'e functor to the functor introduced by Scholze-Weinstein in \cite[Appendix to Lecture XVII]{scholze2020berkeley}. 
\end{enumerate}
The intersection of these two cases is the case of perfect rings, which was historically the first to be studied. The situation for perfect fields is briefly discussed at the end of this section.
\\

We start with the case of quasi-syntomic rings $R$ with $pR=0$. We want to compare the prismatic Dieudonn\'e functor to the crystalline Dieudonn\'e functor
$$
G\mapsto \mathcal{E}xt^1_{(R/\Z_p)_{\mathrm{crys},\mathrm{pr}}}(i_*^{\rm crys}(G),\mathcal{O}_{\mathrm{crys}})
$$
of \cite{berthelot_breen_messing_theorie_de_dieudonne_cristalline_II}.
Here $(R/\Z_p)_{\crys,\mathrm{pr}}$ is the (big) crystalline site of $R$ over $\Z_p$, $\mathcal{O}_{\mathrm{crys}}$ is the crystalline structure sheaf, \red{$\mathrm{pr}$ denotes the $p$-th root topology of \cite[Definition 7.2]{lau_divided_dieudonne_crystals}} and
$$
i^{\rm crys} \colon \mathrm{Shv}(R)_{\mathrm{pr}}\to \mathrm{Shv}(R/\Z_p)_{\mathrm{crys},\mathrm{pr}}
$$
\blue{defined as in \ \cite[Lemma 8.1]{lau_divided_dieudonne_crystals}}, \red{where the left-hand side denotes the category of all schemes over $R$ endowed with the pr topology}. As in \cite[Section 8]{lau_divided_dieudonne_crystals} we define
$$
\mathcal{O}^{\mathrm{crys}}:=u_\ast^{\rm crys}(\mathcal{O}_{\mathrm{crys}})
$$
as the pushforward of the crystalline structure sheaf $\mathcal{O}_{\mathrm{crys}}$ along the morphism
$$
u^{\rm crys} \colon \mathrm{Shv}(R/\Z_p)_{\crys,\mathrm{pr}}\to \mathrm{Shv}(R)_{\mathrm{pr}}
$$ of topoi. Note that by definition $i_*^{\rm crys}=(u^{\mathrm{crys}})^{-1}$, so we can rewrite the crystalline Dieudonn\'e functor as
\[ G\mapsto \mathcal{E}xt^1_{(R/\Z_p)_{\mathrm{crys},\mathrm{pr}}}((u^{\rm crys})^{-1}(G),\mathcal{O}_{\mathrm{crys}}) \]
Let $\mathcal{J}^\crys\subseteq \mathcal{O}^{\crys}$ be the pushforward of the crystalline ideal sheaf $\mathcal{J}_{\mathrm{crys}}\subseteq \mathcal{O}_{\mathrm{crys}}$.

The following lemma is the basic input in the comparison of the prismatic and crystalline Dieudonn\'e functor.

\begin{lemma}
  \label{sec:comp-with-cryst-lemma-comparison-of-prismatic-and-crystalline-pushforward}
  \red{Let $R^\prime$ be a quasi-syntomic $\mathbb{F}_p$-algebra.} Then there is a canonical isomorphism
  $$
  \mathcal{O}^\pris(R^\prime)\to \mathcal{O}^{\crys}(R^\prime)
  $$
  identifying $\mathcal{N}^{\geq 1}\mathcal{O}^\pris(R^\prime)$ with $\mathcal{J}^{\mathrm{crys}}(R^\prime)$.
\end{lemma}

\begin{proof}
  Using the sheaf property for the $\mathrm{pr}$-topology we may assume that $R^\prime$ is semiperfect. Then $R^\prime$ is even quasi-regular semiperfect as it is quasi-syntomic. Hence,
  $$
  \mathcal{O}^\pris(R^\prime)=\prism_{R^\prime}\cong A_{\mathrm{crys}}(R^\prime)=\mathcal{O}^{\crys}(R^\prime)
  $$
  by \Cref{sec:prism-cohom-quasi-lemma-for-qr-semiperfect-prism-isomorphic-to-acrys}. Moreover, the isomorphism in \Cref{sec:prism-cohom-quasi-lemma-for-qr-semiperfect-prism-isomorphic-to-acrys} identifies $\mathcal{N}^{\geq 1}\mathcal{O}^\pris(R^\prime)$ with $\mathcal{J}^{\mathrm{crys}}$.
\end{proof}
Let $(R)_{\mathrm{qsyn},\mathrm{pr}}$ be the category of quasi-syntomic $R$-algebras equipped with the $\mathrm{pr}$-topology, and \blue{let
$$
v^{\rm crys}_\ast \colon \mathrm{Shv}(R/\Z_p)_{\crys,\mathrm{pr}}\to \mathrm{Shv}(R)_{\mathrm{qsyn}, \mathrm{pr}}
$$ obtained by composing $u^{\rm crys}_\ast$ with restriction (the same caveat as in the beginning of \Cref{sec:abstr-divid-dieud} applies here)}. \Cref{sec:comp-with-cryst-lemma-comparison-of-prismatic-and-crystalline-pushforward} implies that the sheaves $\mathcal{O}^\pris$ and $\mathcal{O}^\crys$ on $(R)_{\mathrm{qsyn},\mathrm{pr}}$ are isomorphic. We note that the categories of finite locally free $\mathcal{O}^{\mathrm{crys}}$-modules on $(R)_{\mathrm{pr}}$ and finite locally free $\mathcal{O}^{\mathrm{crys}}_{|(R)_{\mathrm{qsyn},\mathrm{pr}}}$-modules on $(R)_{\mathrm{qsyn},\mathrm{pr}}$ are equivalent because for $R$ quasi-regular semiperfect both categories identify with finite locally free $A_{\mathrm{crys}}(R)$-modules. 
These remarks give a meaning to the comparison contained in the next two results. 

\begin{theorem}
  \label{sec:prism-dieud-theory-theorem-comparison-with-crystalline-dieudonne-functor}
  Let $R$ be a quasi-syntomic ring with $pR=0$ and $G$ a $p$-divisible group over $R$. Then there is a canonical Frobenius equivariant isomorphism
  $$
  \mathcal{M}_\prism(G)\cong v_\ast^{\rm crys}(\mathcal{E}xt^{1}_{(R/\Z_p)_{\crys,\mathrm{pr}}}((u^{\rm crys})^{-1}(G),\mathcal{O}^{\crys})) 
  $$
  from the prismatic Dieudonn\'e crystal of $G$ (cf.\ \Cref{sec:divid-prism-dieud-definition-divided-prismatic-dieudonne-crystal-of-p-divisible-groups}) to the push-forward of the crystalline Dieudonn\'e crystal of $G$. In particular, if $R$ is quasi-regular semiperfect, $M_{\prism}(G)$ is isomorphic to the evaluation $M^{\rm crys}(G)$ on $A_{\rm crys}(R)$ of the crystalline Dieudonn\'e crystal, compatibly with the Frobenius.
\end{theorem}
Of course, the isomorphism is linear over the isomorphism $\mathcal{O}^{\pris}\cong \mathcal{O}^{\crys}$ from \Cref{sec:comp-with-cryst-lemma-comparison-of-prismatic-and-crystalline-pushforward}.

\begin{proof}
  By definition
  $$
  \mathcal{M}_{\prism}(G)=\mathcal{E}xt^{1}_{(R)_{\mathrm{qsyn}}}(G,\mathcal{O}^\pris).
  $$
  But
  $$
  \mathcal{E}xt^{1}_{(R)_{\mathrm{qsyn}}}(G,\mathcal{O}^\pris)\cong \mathcal{E}xt^{1}_{(R)_{\mathrm{qsyn,pr}}}(G,\mathcal{O}^\pris).
  $$
  Indeed, by \blue{the spectral sequence constructed in \Cref{sec:calc-ext-groups} below} it suffices to see that the $\mathcal{O}^{\pris}$-cohomology for the quasi-syntomic and pr-topologies agree. But quasi-regular semiperfectoid rings form a basis for both topologies and on such the higher cohomology of $\mathcal{O}^\pris$ vanishes in both topologies.
  Thus by \Cref{sec:comp-with-cryst-lemma-comparison-of-prismatic-and-crystalline-pushforward} it suffices to see
  $$
  v_\ast^{\rm crys}(\mathcal{E}xt^{1}_{(R/\Z_p)_{\crys,\mathrm{pr}}}((u^{\rm crys})^{-1}(G),\red{\mathcal{O}_{\crys}}))\cong \mathcal{E}xt^1_{(R)_{\mathrm{qsyn,pr}}}(G,\mathcal{O}^\crys).
  $$
  \red{ 
  As $u^\crys$ is a morphism of topoi , we get
  \[
    R\mathcal{H}om_{(R)_{\mathrm{QSYN,pr}}}(G,Ru_\ast^{\crys}(\mathcal{O}^\crys))\cong Ru_\ast^\crys(R\mathcal{H}om_{(R/\Z_p)_{\rm crys,pr}}((u^{\crys})^{-1}(G),\mathcal{O}_\crys)).
  \]
  Here we use that we are dealing with the pr-topology: we don't know if this statement is true for the quasi-syntomic topology, but it holds the syntomic topology as the arguments of \cite[Proposition 1.1.5]{berthelot_breen_messing_theorie_de_dieudonne_cristalline_II} apply because syntomic morphisms can be lifted locally along PD-thickenings, cf.\ \cite[Tag 0070]{stacks_project}. 
  As in \ref{sec:divid-prism-dieud-lemma-divided-dieudonne-module-via-local-ext-on-prismatic-site} it suffices to see that
    \[
      \epsilon_\ast(\mathcal{H}om(G,R^1u^{\mathrm{crys}}_{\ast}(\mathcal{O}_{\mathrm{crys}}))),\quad \mathcal{H}om((u^{\rm crys})^{-1}(G),\mathcal{O}_{\rm crys})
    \]
    vanish (here $\epsilon_\ast$ is the (exact) pushforward to the small quasi-syntomic site). The sheaf $R^1u^\crys_\ast(\mathcal{O}_\crys)$ for the pr-topology on $(R)_{\mathrm{ QSYN,pr}}$ vanishes on every $R$-algebra $S$, which is quasi-syntomic, because it vanishes on quasi-regular semiperfects (cf.\ \cite[Section 8]{bhatt_morrow_scholze_topological_hochschild_homology}) and each quasi-syntomic $\mathbb{F}_p$-algebra admits a pr-cover by some quasi-regular semiperfect ring. Write
    \[
      \mathcal{H}om(G,R^1u^{\mathrm{crys}}_{\ast}(\mathcal{O}_{\mathrm{crys}}))=\varprojlim\limits_{n}\mathcal{H}om(G[p^n],R^1u^{\mathrm{crys}}_\ast(\mathcal{O}_\crys)).
    \]
    The set $\mathrm{Hom}_{(R)_{\mathrm{QSYN,pr}}}(G[p^n],R^1u^{\mathrm{crys}}_\ast(\mathcal{O}_\crys))$ embeds into the sections of $R^1u_{\ast}^\crys(\mathcal{O}_{\crys})$ over $G[p^n]$, but these sections vanish because $G[p^n]$ is syntomic over $R$. Applying the same reasoning to all quasi-syntomic $R$-algebras proves the desired vanishing of the first $\mathcal{H}om$. For the second $\mathcal{H}om$ note that $\mathcal{O}_\crys, (u^{\crys})^{-1}(G)$ are actually sheaves for the syntomic topology on the site $(R/\Z_p)_{\crys}$ and the local $\mathcal{H}om$ does not depend on the topology. Multiplication by $p^n$ on $(u^{\crys})^{-1}(G)$ is surjective for the syntomic topology for every $n\geq 0$ (\cite[Proposition 1.1.7]{berthelot_breen_messing_theorie_de_dieudonne_cristalline_II}) . This implies that
    \[
      \mathcal{H}om((u^{\rm crys})^{-1}(G),\mathcal{O}_{\rm crys})\cong \varprojlim\limits_{n}\mathcal{H}om((u^{\crys})^{-1}(G),\mathcal{O}_{\crys}/p^n)=0
    \]
    using that $\mathcal{O}_\crys$ is $p$-adically complete (being $p$-adically separated would be sufficient for this argument).
  }
  \Cref{sec:comp-with-cryst-lemma-comparison-of-prismatic-and-crystalline-pushforward} implies then moreover compatibility with Frobenius.
\end{proof}

In general, i.e., when $p$ is not necessarily zero in $R$, one can still relate the prismatic Dieudonn\'e crystal of a $p$-divisible group to the crystalline Dieudonn\'e crystal, as follows. Let $R$ be a $p$-complete ring and let $D$ be a $p$-complete $p$-torsion free $\delta$-ring with a surjection $D\to R$ whose kernel has divided powers.\footnote{We don't require $p^nR=0$ for some $n\geq 0$.}
As the kernel of $D\to R$ has divided powers, the Frobenius on $D$ induces a morphism $R\to D/p$. With this morphism the prism $(D,(p))$ defines an object of the absolute prismatic site $(R)_\prism$ of $R$. Via \Cref{sec:divid-prism-dieud-lemma-divided-dieudonne-module-via-local-ext-on-prismatic-site} it thus makes sense to evaluate the prismatic Dieudonn\'e module of a $p$-divisible group over $R$, more precisely $v^{\ast}$ of it, on $(D,(p))$.

\begin{lemma}
  \label{sec:comp-over-mathc-comparison-over-pd-thickening-with-crystalline-dieudonne-module} For every $p$-divisible group $G$ over $R$ there is a natural Frobenius equivariant isomorphism 
  $$
  v^\ast(\mathcal{M}_\prism(G))(D,(p))\cong \mathbb{D}(G)(D).
  $$
  Here $\mathbb{D}(G)(D)$ denotes the evaluation of the (contravariant, crystalline) Dieudonn\'e crystal of $G$ on the PD-thickening $D\to R$. 
\end{lemma}
\begin{proof}
  \red{Let $\mathcal{C}$ be the category of schemes over $R$, which are $p$-completely syntomic over $R$. For each scheme $H\in \mathcal{C}$ there is a canonical isomorphism in the $\infty$-category $\mathcal{D}(\Z)$
  $$
  \red{\eta_H\colon R\Gamma}((H^{(1)}/D)_{\prism},\mathcal{O}_\prism)\cong \red{R\Gamma}((H/D)_{\mathrm{crys}},\mathcal{O}_{\rm crys})
  $$
  by the crystalline comparison for syntomic morphisms (cf.\ \Cref{sec:prism-site-prism-remarks-after-crystalline-comparison}), where $H^{(1)}:=H\times_{\Spec(R)}\Spec(D/p)$\footnote{Note that $R\Gamma((H/D)_{\mathrm{crys}},\mathcal{O}_{\rm crys})=R\Gamma(((H/p)/D)_{\mathrm{crys}},\mathcal{O}_{\rm crys})$. This follows from the computation of crystalline cohomology by a \u{C}ech-Alexander complex and the following fact : if $A$ is a $\Z/p^n$-algebra (for some $n>0$), $P$ a free $\Z_p$-algebra surjecting onto $A$, the divided power envelopes of $P/p^m \to A$ and $P/p^m \to A/p$ agree for any $m \geq n$ : see \cite[Theorem I.2.8.2]{berthelot_cohomologie_cristalline_des_schemas_de_caracteristique_p}.}.
   We can write both sides as
  \[
    H\mapsto R\Gamma((H^{(1)}/D)_\prism,\mathcal{O}_\prism)\cong R\mathrm{Hom}_{(R)_\prism}(\Z[u^{-1}(H)],\mathcal{O}_\prism)
  \]
  resp.\
  \[
        H\mapsto R\Gamma((H/D)_\crys,\mathcal{O}_\crys)\cong R\mathrm{Hom}_{(R)_\crys}(\Z[({u^{\crys}})^{-1}(H)],\mathcal{O}_\crys),
      \]
      making it clear that both functors are actually restrictions to $\mathcal{C}$ of $R\mathrm{Hom}$-functors
      \[
        F(-):=R\mathrm{Hom}_{(R)_\prism}(-,\mathcal{O}_\prism),\ G(-):=R\mathrm{Hom}_{(R)_\crys}(-,\mathcal{O}_\crys)
      \]
      on the category of sheaves of abelian groups $\mathrm{Shv}_{\Z}((R)_\prism)$ resp.\ $\mathrm{Shv}_\Z((R)_\crys)$ on $(R)_\prism$, resp.\ $(R)_\crys$ along the functors
      \[        
        H\mapsto \iota_\pris(H):=\Z[u^{-1}(H)]\in \mathrm{Shv}_{\Z}((R)_\prism)
      \]
      resp.\
      \[
        H\mapsto \iota_\crys(H):=\Z[{u^{\crys}}^{-1}(H)]\in \mathrm{Shv}_\Z((R)_\crys).
      \]
      Assume now that $H$ is a finite locally free group scheme over $R$, in particular $H$ is syntomic over $R$. Applying $F(-), G(-)$ to the Breen-Deligne resolution, cf. \Cref{sec:calc-ext-sheav-canonical-resolution-of-abelian-group}, of $u^{-1}(H)$, ${u^{\crys}}^{-1}(H)$ (seen via Dold-Kan as simplicial objects in $\mathrm{Shv}((R)_\prism)$, resp.\ $\mathrm{Shv}((R)_\crys)$) yield two cosimplicial objects $K_1^\bullet, K_2^\bullet\colon \Delta\to \mathcal{D}(\green{\Z})$ (here $\Delta$ is the simplex category) with limits 
      \[
        R\mathrm{Hom}_{(R)_\prism}(u^{-1}(H),\mathcal{O}_\prism)
        \]
      and
      \[
        R\mathrm{Hom}_{(R)_\prism}((u^{\crys})^{-1}(H),\mathcal{O}_\crys).
      \]
    
    We claim that the natural isomorphism $\eta$ extends to a natural isomorphism $K_1^\bullet\cong K_2^\bullet$.
  }
  \red{   
    Intuitively, this is clear as the morphisms in the Breen-Deligne resolution are \textit{sums} of maps induced by morphisms between schemes. We thank Yonatan Harpaz and Fabian Hebestreit for their help with the following rigorous $\infty$-categorical argument. It suffices to argue for the left Kan extensions of $F\circ \iota_\pris, G\circ \iota_\crys\colon \mathcal{C}\to \mathcal{D}(\Z)^{\mathrm{op}}$ from $\mathcal{C}$ to the category of all schemes over $R$ (this ensures the existence of fiber products in $\mathcal{C}$ commuting with coproducts). Hence, we abuse notation and denote by $\mathcal{C}$ the category of all schemes over $R$. Let $\mathcal{D}$ be any category with action by the symmetric monodial (via tensor product) category $\mathcal{F}ree_\Z$ of finite, free $\Z$-modules, such that the action commutes with finite coproducts in each variable, e.g., $\mathcal{D}=\mathcal{D}(\Z)^{\mathrm{op}}$. In other words, $\mathcal{D}$ is required to be a module under $\mathcal{F}ree_\Z$ in the symmetric monoidal $\infty$-category $\mathcal{C}at_{\infty}(\mathcal{K}_{\mathrm{fin}})$ from \cite[Corollary 4.8.1.4]{lurie_higher_algebra} with $\mathcal{K}_{\mathrm{fin}}$ the class of finite sets. Now each functor $\varphi\colon \mathcal{C}\to \mathcal{D}$ preserving finite coproducts resp.\ each natural transformation between such functors extends to a functor $\varphi^{\mathrm{ab}}\colon \mathcal{C}^{\mathrm{ab}}:=\mathcal{F}ree_{\Z}\otimes_{\mathcal{K}_{\mathrm{fin}}} \mathcal{C}\to \mathcal{D}$ resp.\ a natural transformation between such functors with $-\otimes_{\mathcal{K}_{\mathrm{fin}}}-$ the tensor product in $\mathcal{C}at_{\infty}(\mathcal{K}_{\mathrm{fin}})$ by $\mathcal{F}ree_\Z$-linearity of $\mathcal{D}$. The category $\mathcal{C}^{\mathrm{ab}}$ can now be calculated as follows: Consider the category $\mathrm{Fun}(\mathcal{C}^{\mathrm{op}},\mathcal{D}_{\geq 0}(\Z))$ of functors, and its full subcategory $\mathrm{Fun}^{\times}(\mathcal{C}^{\mathrm{op}},\mathcal{D}_{\geq 0}(\Z))$ of product-preserving functors. The inclusion $\mathrm{Fun}^{\times}(\mathcal{C}^{\mathrm{op}},\mathcal{D}_{\geq 0}(\Z))\to \mathrm{Fun}(\mathcal{C}^{\mathrm{op}},\mathcal{D}_{\geq 0}(\Z))$ admits a left adjoint $L$, given by sheafification on $\mathcal{C}$ with respect to the Grothendieck topology in which coverings are finite collections $\{X_i\to X\}_{i\in I}$, such that $\coprod\limits_{i\in I} X_i\to X$ is an isomorphism. Now, $\mathcal{C}^{\mathrm{ab}}$ is the smallest full subcategory of $\mathrm{Fun}^{\times}(\mathcal{C}^{\mathrm{op}},\mathcal{D}_{\geq 0}(\Z))$ containing all objects $L(\Z[\mathrm{Hom}_{\mathcal{C}}(-,X)])$ with $X\in \mathcal{C}$. Note that $\mathcal{C}^{\mathrm{ab}}$ is a $1$-category because sheafification for \green{this} Grothendieck topology preserves set-valued presheaves. In fact, we only need that functors $\varphi\colon \mathcal{C}\to \mathcal{D}$ preserving coproducts extend to $\mathcal{C}^{\mathrm{ab}}$ when the latter is defined by the above concrete description. To see this note if $\mathcal{S}$ denotes the $\infty$-category of spaces, i.e., Kan complexes, that
      \[
        \mathrm{Fun}^{\times}(\mathcal{F}ree_{\Z}^{\mathrm{op}},\mathcal{S})\cong \mathcal{D}_{\geq 0}(\Z)
      \]
      by \cite[Example 1.2.9]{lurie2017elliptic}. Each functor $\varphi\colon \mathcal{C}\to \mathcal{D}$ yields a functor
      \[
        \mathrm{Fun}(\mathcal{C}^{\mathrm{op}},\mathcal{D}_{\geq 0}(\Z))\to \mathrm{Fun}(\mathcal{D}^{\mathrm{op}},\mathcal{D}_{\geq 0}(\Z)).
      \]
      Now $\mathrm{Fun}(\mathcal{D}^{\mathrm{op}},\mathcal{D}_{\geq 0}(\Z))$ embeds into
      \[
        \mathrm{Fun}(\mathcal{D}^{\mathrm{op}}\times \mathcal{F}ree_\Z^{\mathrm{op}},\mathcal{S})
        \]
        (with essential image those functors commuting with products in the second factor), and the $\mathcal{F}ree_\Z$-action $\mathcal{F}ree_\Z\times \mathcal{D}\to \mathcal{D}$ furnishes a functor from this to the category $\mathrm{Fun}(\mathcal{D}^{\mathrm{op}},\mathcal{S}),$
        which contains $\mathcal{D}$ by the Yoneda lemma. Restricting further along the inclusion $\mathcal{C}^{\mathrm{ab}}\to \mathrm{Fun}(\mathcal{C}^{\mathrm{op}},\mathcal{D}_{\geq 0}(\Z))$ then yields a functor
        \[
          \mathcal{C}^{\mathrm{ab}}\to \mathrm{Fun}(\mathcal{D}^{\mathrm{op}},\mathcal{S})
        \]
        with image in $\mathcal{D}$ as $\varphi$ preserves finite coproducts. This yields the desired extension, and similarly we see that natural transformations extend.
        Given these considerations, and in particular the description of $\mathcal{C}^{\mathrm{ab}}$, it follows by unraveling the construction that in our situation the simplicial objects given by the images under $F$ resp.\ $G$ of the Breen-Deligne resolutions of $u^{-1}(H)$ resp.\ ${u^{\crys}}^{-1}(H)$ are the images of a simplicial object in $\mathcal{C}^{\mathrm{ab}}$ under the extensions of $F\circ \iota_\pris, G\circ \iota_\crys$. This shows that $\eta$ extends as desired.
       
   Passing to the limits and taking cohomology in degree $1$} we can deduce that

  $$
  \mathcal{M}_\prism(H)(D,(p)):=\mathrm{Ext}^1(u^{-1}(H),\mathcal{O}_\prism)
  $$ resp.\ $\mathbb{D}(H)(D)$ are canonically isomorphic. Hence, we obtain the desired natural isomorphism for finite flat group schemes. The proof of \Cref{sec:divid-prism-dieud-proposition-ext-crystal} below\footnote{Which the reader can check to be independent of the present lemma.} shows that writing $$G=\varinjlim\limits_{n}G[p^n]$$ and passing to the limit yields a canonical isomorphism
  $$
\mathcal{M}_\prism(G)(D,(p))\cong \mathbb{D}(G)(D)
$$
for $G$ a $p$-divisible group over $R$.
\end{proof}

\begin{remark}
The relation between the prismatic and the crystalline Dieudonn\'e functors will mostly be used over a characteristic $p$ perfect field in the rest of this text, and it could be interesting to find a more direct proof of it in this special case, as explained at the end of this section. But it will also be used for comparison with the Scholze-Weinstein functor in the next paragraph and in \Cref{sec:comp-case-mathc}. 
\end{remark}

We turn to perfectoid rings. The following statement is a special case of a theorem of Fargues (\cite{fargues_quelques_resultats_et_conjectures_concernant_la_courbe}, \cite{scholze2020berkeley}). Let $C$ be a complete algebraically closed extension of $\mathbb{Q}_p$. We abbreviate
$$
A_{\rm inf}=A_{\rm inf}(\mathcal{O}_C)\ ,\ A_\crys:=A_{\rm crys}(\mathcal{O}_C/p).
$$
We also fix a compatible system $\varepsilon$ of $p$-th roots of unity, and let $\tilxi=[p]_q$, where $q=[\varepsilon]-1$. We identify the initial prism of $(\mathcal{O}_C)_{\prism}$ with $(A_{\rm inf},(\tilxi))$.

\begin{proposition}
A prismatic Dieudonn\'e module $(M,\varphi_M)$ over $\mathcal{O}_C$ (i.e., a minuscule Breuil-Kisin-Fargues module) is uniquely determined up to isomorphism by the triple
\[ (T_M, M_{\rm crys}, \alpha_M), \]
where $T_M$ is the finite free $\mathbb{Z}_p$-module $$ T_M=M[\frac{1}{\tilxi}]^{\varphi_M=1},$$ $$M_{\rm crys}=M\otimes_{A_{\rm inf}} A_{\rm crys}$$ is a $\varphi$-module over $A_{\rm crys}$ and $\alpha_M : T_M \otimes_{\mathbb{Z}_p} B_{\rm crys} \simeq M_{\rm crys} \otimes_{A_{\rm crys}} B_{\rm crys}$ is the $\varphi$-equivariant isomorphism coming from the natural map $M[\frac{1}{\tilxi}]^{\varphi_M=1} \to M[\frac{1}{\tilxi}]$.
\end{proposition}

\blue{Let $R$ be a perfectoid ring. In \cite[Theorem 17.5.2]{scholze2020berkeley}, Scholze-Weinstein construct a covariant functor $M^{SW}$ from $p$-divisible groups over $R$ to prismatic Dieudonn\'e modules over $R$ inducing an equivalence between the two categories. It has the following properties characterizing it uniquely, which will be used in the next proof. \begin{itemize}
\item When $R$ is perfect, $M^{SW}=M_{\rm crys}(-)$ is the (covariant) crystalline Dieudonn\'e functor dual to $M^{\rm crys}$.
\item If $R=\mathcal{O}_C$, with $C$ a complete algebraically closed extension of $\mathbb{Q}_p$,
$$M^{SW}(-) \otimes_{A_{\rm inf}} A_{\rm crys} \cong M_{\rm crys}(- \otimes_{\mathcal{O}_C} \mathcal{O}_C/p)$$ (\cite[Proposition 14.8.3]{scholze2020berkeley}). In particular, if $G=X[p^{\infty}]$, for some formal abelian scheme $X$ over $\mathcal{O}_C$, the functor $M^{SW}$ sends $G$ to the prismatic Dieudonn\'e module over $\mathcal{O}_C$ dual to $H_{A_{\rm inf}}^1(X)$: this follows from the definition of $M^{SW}(G)$ (\cite[\S 12.1]{scholze2020berkeley}), \cite[Proposition 14.8.3]{scholze2020berkeley} and the above proposition.
\item In general, for any perfectoid ring $R$, if $G$ is a $p$-divisible group over $R$,
$$M^{SW}(G) \subset M_{\rm crys}(G \otimes_R R/p)$$
is the largest submodule mapping into $M(G \otimes_R V) \subset M_{\rm crys}(G \otimes_R V)$ for all maps $R \to V$ where $V$ is an integral perfectoid valuation ring with algebraically closed fraction field.
\end{itemize}  }

\begin{proposition}
 \label{sec:prism-dieud-theory-proposition-comparison-scholze-weinstein-functor}
Let $R$ be a perfectoid ring. The functor $G \mapsto M_{\prism}(G)$ from $\BT(R)$ to $\DM(R)$ coincides with the (naive)\footnote{I.e., $(-)^\ast=\mathrm{Hom}_{A_\inf(R)}(-,A_\inf(R))$.} dual of the functor $M^{SW}$ of \cite[Appendix to Lecture XVII]{scholze2020berkeley}. 
\end{proposition}
\begin{proof}
  \red{
    If $R$ is perfect and $G$ a $p$-divisible group over $R$, then we get a natural isomorphism
    \[
      \alpha_{R,G}\colon M_\prism(G)\cong M^{SW}(G)^\ast
    \]
    because both sides identify with the (contravariant) crystalline Dieudonn\'e module, cf.\ \ref{sec:prism-dieud-theory-theorem-comparison-with-crystalline-dieudonne-functor}. Moreover, $\alpha_{R,-}$ is compatible with base change along morphisms of perfect rings.
  }
    Now assume that $R=\mathcal{O}_C$, where $C$ is a perfectoid algebraically closed field over $\mathbb{Q}_p$. In this case, assume first that $G=X[p^{\infty}]$, for some formal abelian scheme $X$ over $\mathcal{O}_C$, with rigid generic fiber $X^{\rm rig}$. \blue{As recalled above}, the functor $M^{SW}$ sends $G$ to the prismatic Dieudonn\'e module over $\mathcal{O}_C$ dual to $H_{A_{\rm inf}}^1(X)$. In particular, in this case, $M^{SW}(G)$ is isomorphic to the (naive) dual to $M_{\prism}(G)$, by \Cref{sec:prism-dieud-modul-corollary-prismatic-dieudonne-crystal-of-abelian-schemes} and the comparison theorem \cite[Theorem 17.2]{bhatt_scholze_prisms_and_prismatic_cohomology}\footnote{Note that we chose $\tilxi$ as a generator of the ideal of the prism, so the Frobenius twist in the statement of loc. cit. disappears.}. Moreover, this identification is functorial for morphisms of $p$-divisible groups of abelian schemes (and not simply for morphisms of abelian schemes) : indeed, let $X, X'$ be two abelian schemes over $\mathcal{O}_C$, and $G=X[p^{\infty}]$, $H=X'[p^{\infty}]$, with a morphism $f\colon G\to H$. We want to see that the diagram 
  	 $$
	 \xymatrix{
	 	M^{SW}(G) \ar[r]^{\cong} \ar[d]_{M^{SW}(f)} & M_{\prism}(G)^{\ast} \ar[d]^{M_{\prism}(f)^*}   \\
	 	M^{SW}(H) \ar[r]^{\cong}  & M_{\prism}(H)^{\ast}    }
	 $$
commutes. This can be checked after base change to $A_{\rm crys}$. Then, using \Cref{sec:comp-over-mathc-comparison-over-pd-thickening-with-crystalline-dieudonne-module}, the terms on the top line (resp. on the bottom line) are identified with the covariant crystalline Dieudonn\'e module of $G$ (resp. $H$), and the horizontal isomorphisms induce the identity, by construction.
 
Let now $G$ be a general $p$-divisible group over $\mathcal{O}_C$. There exists a formal abelian scheme $X$ over $\mathcal{O}_C$, such that $X[p^{\infty}]=G \times \check{G}$ (cf. \cite[Proposition 14.8.4]{scholze2020berkeley}). Let $e\colon X[p^\infty]\to X[p^\infty]$ be the idempotent with kernel $G$. Then
  $$
  M_\prism(G)^{\ast}=\mathrm{ker}(M_\prism(e)^{\ast}\colon M_\prism(X[p^\infty])^{\ast}\to M_\prism(X[p^\infty])^{\ast})
  $$
  and
  $$
  M^{\mathrm{SW}}(G)=\mathrm{ker}(M^{\mathrm{SW}}(e) \colon M^{\mathrm{SW}}(X[p^\infty]) \to M^{\mathrm{SW}}(X[p^\infty])). 
  $$
  By the functoriality explained above we can conclude \red{the proof when $R=\mathcal{O}_C$, i.e., we have constructed an isomorphism $\alpha_{R,G}\cong M_\prism(G)\cong M^{SW}(G)^\ast$ in this case, which is natural in $G$ and compatible with base change along morphisms of such $R$'s. If $k$ denotes the residue field of $\mathcal{O}_C$, then by construction the base change of $\alpha_{R,G}$ along $\prism_R\to \prism_k$ is $\alpha_{k,G\otimes_R k}$.} 

  \red{Now assume that $R$ is a general perfectoid ring. By \cite[Remark 8.8]{lau_dieudonne_theory_over_semiperfect_rings_and_perfectoid_rings} we can write
    \[
      R\cong R_1\times_{S_2}S_1
    \]
    with $R_1$ $p$-torsion free perfectoid, $S_1,S_2$ perfect and $R_1\to S_2, S_1\to S_2$ surjective. As in \cite[Lemma 9.2]{lau_dieudonne_theory_over_semiperfect_rings_and_perfectoid_rings} the category $\DM(R)$ of prismatic Dieudonn\'e modules for $R$ is naturally equivalent to the $2$-limit
    \[
      \DM(R_1)\times_{\DM(S_2)}\DM(S_1).
    \]
    Thus it suffices to construct a natural isomorphism $M_\prism(G)\cong M^{SW}(G)^\ast$ for any $p$-divisible group over a perfectoid ring $R$, which is either perfect or $p$-torsion free, and show that it is compatible with base change in $R$. If $R$ is perfect, then we are already done. Let us assume that $R$ is $p$-torsion free. Then the ring $R/p$ is quasi-regular semiperfect, and $\prism_{R/p}\cong A_\crys(R/p)$ by \ref{sec:prism-cohom-quasi-lemma-for-qr-semiperfect-prism-isomorphic-to-acrys}. By \ref{sec:prism-dieud-theory-theorem-comparison-with-crystalline-dieudonne-functor}, \ref{sec:divid-prism-dieud-proposition-ext-crystal}\footnote{The proof of \ref{sec:divid-prism-dieud-proposition-ext-crystal} does not use the comparison with \cite{scholze2020berkeley}.} and the construction of $M^{SW}(G)^\ast$ we have a natural isomorphism
    \[
      \alpha_{G\otimes_{R} R/p}\colon M_\prism(G)\otimes_{\prism_R}A_\crys(R/p) \cong M^{SW}(G)^\ast\otimes_{\prism_R} A_\crys(R/p)
    \]
    because both sides identify with the (contravariant) crystalline Dieudonn\'e module of $G\otimes_R R/p$. By \ref{sec:divid-prism-dieud-proposition-ext-crystal} $M_\prism(G)$ is a finite, locally free $\prism_R$-module. Thus $M_\prism(G)$ identifies with a $\prism_R$-submodule of $M_\prism(G)\otimes_{\prism_R}A_\crys(R/p)$ because the morphism $\prism_R\to A_\crys(R/p)$ is injective. We claim that $\alpha_{G\otimes_R R/p}$ maps (injectively) $M_\prism(G)$ into $M^{SW}(G)^\ast$.
    By the very construction of $M^{SW}(G)^\ast$ we have to check that for any perfectoid valuation ring $V$ with algebraically closed fraction field and morphism $R\to V$ the module $M_\prism(G)$ maps to $M^{SW}(G_V)^\ast\subseteq M_\crys(G_V)$ with $G_V:=G\otimes_R V$. If $V$ is perfect, this follows by \ref{sec:prism-dieud-theory-theorem-comparison-with-crystalline-dieudonne-functor}. If $V$ is of mixed characteristic we can write $V$ as the fiber product $V^\prime\times_{\kappa} S$ of a perfect valuation ring $S$ with a mixed-characteristic valuation ring $V^\prime$ of rank $1$ over the residue field $\kappa$ of $V^\prime$, and write
    \[
      M^{SW}(G_V)\cong M^{SW}(G_{V^\prime})\times_{M^{SW}(G_\kappa)}M^{SW}(G_S).
    \]
    We already checked the statement for $V^\prime,\kappa, S,$ and thus we have finished the construction of a natural injective morphism
    \[
      \alpha_{R,G}\colon M_\prism(G)\to M^{SW}(G)^\ast
    \]
    for a general perfectoid ring $R$. Assume $R\to R^\prime$ is a morphism of perfectoid rings, then we know that $\alpha_{R,G}\otimes_{\prism_R}\prism_{R^\prime}=\alpha_{R^\prime,G_{R^\prime}}$ if $R,R^\prime$ are perfect. If $R$ is $p$-torsion free and $R^\prime$ perfect, we can draw the same conclusion as then $\prism_{R}\to \prism_{R^\prime}$ factors over $A_\crys(R/p)$ and $\alpha_{R,G}\otimes_{\prism_R} A_\crys(R/p)$ is the identification coming from Dieudonn\'e theory. As $M^{SW}(G)^\ast$ is a finite free $\prism_R$-module (by \cite[Theorem 17.5.2]{scholze2020berkeley}), we can check that it is an isomorphism after base change along all morphisms $\prism_R\to \prism_k$ for $R\to k$ a morphism from $R$ to a perfect field $k$. But this case was already handled. This finishes the proof. 
  }
\end{proof} 

We obtain the following corollary, which we will need in \Cref{sec:essent-surj-1}.

\begin{corollary}
\label{sec:prism-dieud-theory-comparison-with-bkf-modules-perfectoid-case}

Let $R$ be a perfectoid ring. The prismatic Dieudonn\'e functor $M_{\prism}$ \red{takes values in $\DM^{\rm adm}(R) \cong \DM(R)$} and induces an antiequivalence between $\BT(R)$ and $\DM^{\rm adm}(R)  \cong \DM(R)$.
\end{corollary}
\begin{proof}
This follows \red{immediately} from the last proposition and \cite[Theorem 17.5.2]{scholze2020berkeley}. Note that the argument of loc. cit. shows that one only needs to prove the equivalence when $R$ is the ring of integers of a perfectoid algebraically closed field, where it is due to Berthelot \cite[Theorem 3.4.1]{berthelot_theorie_de_dieudonne_sur_un_anneau_de_valuation_parfait} and Scholze-Weinstein \cite[Theorem 5.2.1]{scholze_weinstein_moduli_of_p_divisible_groups} (in this case, one can even assume that the fraction field of $R$ is spherically complete, and the result is then an easy consequence of results of Fargues : see \cite[\S 5.2]{scholze_weinstein_moduli_of_p_divisible_groups}). 
\end{proof}

\begin{remark}
  \label{sec:comp-with-form-remark-exactness-of-equivalences}
Let $R$ be a perfectoid ring. The functor $M_{\prism}$ is exact (see below \Cref{sec:divid-prism-dieud-exactness}) and has an exact quasi-inverse (we will provide an argument for this later in \Cref{sec:complements-finite-locally-free} in the case of finite locally free group schemes, which applies verbatim for $p$-divisible groups). 
\end{remark}

Let us conclude this section by discussing the case of perfect fields. For a perfect field $k$, Fontaine \cite{fontaine_groupes_p_divisibles_sur_les_corps_locaux} was the first to give a uniform definition of a functor from $p$-divisible groups to (prismatic) Dieudonn\'e modules over $k$. Let us recall it first, as formulated in \cite[\S 4.1]{berthelot_messing_theorie_de_dieudonne_cristalline_III}. If $A$ is a commutative ring, the set $\mathrm{CW}(A)$ of \textit{Witt covectors with values in $A$} is the set of all family $(a_{-i})_{i\in \mathbb{N}}$ of elements of $A$ such that there exist integers $r, s\geq 0$ such that the ideal $J_r$ generated by the $a_{-i}$, $i\geq r$, satisfies $J_r^s=0$. One still denotes by $\mathrm{CW}$ the sheaf on the big fpqc site\footnote{We could as well use any other topology finer than the Zariski topology.} of $k$ associated to the presheaf of Witt covectors. This is an abelian sheaf of $W(k)$-modules, endowed with a Frobenius operator which is semi-linear with respect to the Frobenius on $W(k)$. Fontaines defines :
\[ M^{\rm cl}(G):= \mathrm{Hom}_{(k)_{\rm fpqc}}(G, \mathrm{CW}). \]

As a corollary of \Cref{sec:prism-dieud-theory-theorem-comparison-with-crystalline-dieudonne-functor} and results of Berthelot-Breen-Messing, one gets 
\begin{proposition}
Let $k$ be a perfect field, and let $G$ be a $p$-divisible group over $R$. One has a canonical $W(k)$-linear Frobenius-equivariant isomorphism
\[ M_{\prism}(G) \cong M^{\rm cl}(G). \]
\end{proposition}
\begin{proof}
By construction, the isomorphism of \Cref{sec:prism-dieud-theory-theorem-comparison-with-crystalline-dieudonne-functor} is linear over the isomorphism $\prism_k \simeq A_{\rm crys}(k)$, which is given by the Frobenius $\sigma$ of $W(k)$, i.e., it can be seen as a Frobenius-equivariant $W(k)$-linear isomorphism :
\[ M_{\prism}(G) \cong (\sigma^{-1})^* M^{\rm crys}(G). \]
Composing it with $\sigma^{-1}$-pullback of the inverse of the $W(k)$-linear Frobenius-equivariant isomorphism of \cite[Theorem 4.2.14]{berthelot_messing_theorie_de_dieudonne_cristalline_III}, we get the desired isomorphism.
\end{proof}

It would be interesting to get a more direct proof of this corollary. In characteristic $p$, the prismatic Dieudonn\'e crystal of a $p$-divisible group admits a description which looks similar to Fontaine's definition.

\begin{definition}
Let $R$ be a a quasi-syntomic ring with $pR=0$. We define the sheaf $\mathcal{Q}$ on $(R)_{\prism}$ as the quotient :
\[ 0\to \mathcal{O}_\prism\to \mathcal{O}_\prism[1/p]\to \mathcal{Q}\to 0. \]
\end{definition}

The morphism $\mathcal{O}_\prism\to \mathcal{O}_\prism[1/p]$ is injective since any prism in $(R)_{\prism}$ is $p$-torsion free.

\begin{proposition}
Let $R$ be a quasi-syntomic ring with $pR=0$, and let $G$ be a $p$-divisible group over $R$. The \red{connecting map of the} canonical exact sequence $$0\to \mathcal{O}_\prism\to \mathcal{O}_\prism[1/p]\to \mathcal{Q}\to 0$$ induces an isomorphism :
$$
\mathcal{H}om_{(R)_{\mathrm{qsyn}}}(G,v_*\mathcal{Q})=v_*\mathcal{H}om_{(R)_{\prism}}(u^{-1}G, \red{\mathcal{Q}}) \cong \mathcal{M}_{\prism}(G).
$$
\end{proposition}
\begin{proof}
First assume that $G$ is a finite locally free group scheme. Then the statement is clear, as
$$
R\mathcal{H}om_{(R)_{\prism}}(u^{-1}(G),\mathcal{O}_\prism[1/p])=0,
$$
because $u^{-1}(G)$ is killed by some power of $p$, whereas on $\mathcal{O}_\prism[1/p]$ multiplication by $p$ is invertible. The result for $p$-divisible groups is deduced by a limit argument.
\end{proof}

This naturally leads to the following question.
\begin{question} \label{question_fontaine}
When $R=k$ is a perfect field, what is the relation between the sheaf $v_* \mathcal{Q}$ and the sheaf $\mathrm{CW}$ of Witt covectors?
\end{question} 

\subsection{Calculating $\mathrm{Ext}$-groups in topoi}
\label{sec:calc-ext-groups}

In this section we recall the method of calculating Ext-groups in a topos as presented by Berthelot, Breen, Messing (cf.\ \cite[2.1.5]{berthelot_breen_messing_theorie_de_dieudonne_cristalline_II}\footnote{For simplicity we omit the case of the local $\mathrm{Ext}$-sheaves, which is entirely similar.}.
Let $\mathfrak{X}$ be a topos and let $G,H\in \mathfrak{X}$ be two abelian groups, i.e., two abelian group objects.

The following theorem is attributed to Deligne in \cite{berthelot_breen_messing_theorie_de_dieudonne_cristalline_II}. A proof can be found in \cite[Appendix to Lecture IV, Theorem 4.10]{scholze_lectures_on_condensed_mathematics}.

\begin{theorem}
\label{sec:calc-ext-sheav-canonical-resolution-of-abelian-group}
Let $G\in \mathrm{\mathfrak{X}}$ be an abelian group. Then there exists a natural functorial (in $G$) resolution
$$
C(G)_\bullet:=(\ldots \to \Z[X_2]\to \Z[X_1]\to \Z[X_0])\simeq G
$$
where each $X_i\in \mathfrak{X}$ is a finite disjoint unions of products of copies $G$.
\end{theorem}
\begin{proof}
See \cite[2.1.5]{berthelot_breen_messing_theorie_de_dieudonne_cristalline_II} or \cite[Appendix to Lecture IV, Theorem 4.10]{scholze_lectures_on_condensed_mathematics}  
\end{proof}

\begin{lemma}
\label{sec:calc-ext-sheav-derived-yoneda-lemma}
Let $X\in \mathfrak{X}$ be any object and let $\mathcal{F}\in \mathrm{Ab}(\mathfrak{X})$ be an abelian group. Then
$$
R\Gamma(X,\mathcal{F})\cong R\mathrm{Hom}_{\mathrm{Ab}(\mathfrak{X})}(\Z[X],\mathcal{F}),
$$
where $\Z[X]$ denotes the free abelian group on $X$. 
\end{lemma}
\begin{proof}
This follows by deriving the isomorphism $\mathcal{F}(X)\cong \mathrm{Hom}_{\mathrm{Ab}(\mathfrak{X})}(\Z[X],\mathcal{F}).$
\end{proof}

These two results show that the $\mathrm{Ext}$-groups
$$
\mathrm{Ext}^i_{\mathrm{Ab}(\mathfrak{X})}(G,H)
$$
can, \textit{in principle}, be calculated in terms of the cohomology groups
$$
H^i(G\times\ldots \times G,H)
$$
for various products $G\times\ldots \times G$. Unfortunately, the construction of the resolution in \Cref{sec:calc-ext-sheav-canonical-resolution-of-abelian-group} is rather involved. However, the first terms, which are sufficient for our applications, can be made explicit\footnote{By this, we mean that one can construct a functorial (in $G$) resolution having these terms in the beginning.}. 
For example, the first terms can be chosen to be 
$$
\begin{matrix}
  C(G)_0:=\Z[G] \\
  C(G)_1:=\Z[G^2]\\
  C(G)_2:=\Z[G^3]\oplus\Z[G^2]
\end{matrix}
$$ 
with explicit differentials (cf.\ \cite[(2.1.5.2.)]{berthelot_breen_messing_theorie_de_dieudonne_cristalline_II}).
The stupid filtration of the complex $C(G)_\bullet$ yields a spectral sequence
$$
E_1^{i,j}=\mathrm{Ext}^{j}_{\mathrm{Ab}(\mathfrak{X})}(C(G)_i,\mathcal{F})\Rightarrow \mathrm{Ext}^{i+j}_{\mathrm{Ab}(\mathfrak{X})}(C(G)_\bullet, \mathcal{F})\cong \mathrm{Ext}^{i+j}_{\mathrm{Ab}(\mathfrak{X})}(G,\mathcal{F})
$$
and the terms
$$
\mathrm{Ext}^i_{\mathrm{Ab}(\mathfrak{X})}(C(G)_j,\mathcal{F})
$$
can be calculated using the cohomology.
For later use let us make the first terms of the first page of this spectral sequence explicit:
$$
\xymatrix{
\ldots \ar@{-->}[rrd] &\ldots\ar@{-->}[rrd]  & \ldots  &\ldots  \\
H^2(G,\mathcal{F})\ar[r]^-{d_1}\ar@{-->}[rrd] & H^2(G\times G,\mathcal{F})\ar[r]^-{d_2}\ar@{-->}[rrd] & H^2(G\times G,\mathcal{F})\oplus H^2(G\times G\times G,\mathcal{F})\ar[r] &\ldots \\
H^1(G,\mathcal{F})\ar[r]^-{d_1}\ar@{-->}[rrd] & H^1(G\times G,\mathcal{F})\ar[r]^-{d_2}\ar@{-->}[rrd] & H^1(G\times G,\mathcal{F})\oplus H^1(G\times G\times G,\mathcal{F})\ar[r]&\ldots \\
H^0(G,\mathcal{F})\ar[r]^-{d_1} & H^0(G\times G,\mathcal{F})\ar[r]^-{d_2} & H^0(G\times G,\mathcal{F})\oplus H^0(G\times G\times G,\mathcal{F})\ar[r] &\ldots 
}
$$
For an element $(x_1,\ldots, x_n)\in G^n$ let us denote by $[x_1,\ldots, x_n]\in \Z[G^n]$ the corresponding element in the group ring $\Z[G^n]$.
The morphisms $d_1$ and $d_2$ are then induced by
$$
\Z[G^2]\to \Z[G],\ [x,y]\mapsto -[x]+[x+y]-[y]
$$
for $d_1$
and
$$
\begin{matrix}
  \Z[G^2]\to \Z[G^2],\ [x,y]\mapsto [x,y]-[y,x] \\
  \Z[G^3]\to \Z[G^2],\ [x,y,z]\mapsto -[y,z]+ [x+y,z]-[x,y+z]+[x,y]
\end{matrix}
$$
for $d_2$ (cf.\ \cite[(2.1.5.2.)]{berthelot_breen_messing_theorie_de_dieudonne_cristalline_II}).

\subsection{Prismatic Dieudonn\'e crystals of abelian schemes}
\label{sec:prism-dieud-modul-dieud-mod-abel-schemes}

In this section we describe the prismatic cohomology of the $p$-adic completion of abelian schemes and deduce from this the construction of the prismatic Dieudonn\'e crystal \red{
$$
\mathcal{M}_{\prism}(X[p^\infty]) = (\mathcal{M}_{\prism}(X[p^{\infty}]), \varphi_{\mathcal{M}_{\prism}(X[p^{\infty}])}). 
$$}
of the $p$-divisible group $X[p^\infty]$ of the $p$-adic completion of an abelian scheme $X$ \purple{over a quasi-syntomic ring $R$}. \red{Admissibility of this prismatic Dieudonn\'e crystal will be proved in the next section, in fact for any $p$-divisible group.}
\\

Let $(A,I)$ be a bounded prism. \red{Write $\bar{A}=A/I$.} Let $X\to \Spf(\bar{A})$ be the $p$-adic completion of an abelian scheme over $\Spec(\bar{A})$.

We first prove degeneracy of the conjugate spectral sequence (cf.\ \Cref{sec:hodge-tate-comp-general}) for $X$. The proof is an adaptation of the argument in \cite[Proposition 2.5.2]{berthelot_breen_messing_theorie_de_dieudonne_cristalline_II}, which proves degeneration of the Hodge-de Rham spectral sequence.

Recall the following statement. 
\begin{proposition}\label{sec:prism-cryst-abel-1-proposition-bbm-252}
For all $k\geq 0$ (resp. for all $i,j\geq 0$), the $\bar{A}$-module $H^k(X,\Omega_{X/\bar{A}}^{\bullet})$ (resp. $H^i(X,\Omega_{X/\bar{A}}^j)$) is finite locally free, and its formation commutes with base change.

Moreover, the algebra $H^*(X,\Omega_{X/\bar{A}}^{\bullet})$ is alternating and the canonical algebra morphism
\[ \wedge^* H^1(X,\Omega_{X/\bar{A}}^{\bullet}) \to H^*(X,\Omega_{X/\bar{A}}^{\bullet}) \]
defined by the multiplicative structure of $H^*(X,\Omega_{X/\bar{A}}^{\bullet})$, is an isomorphism.
\end{proposition}
\begin{proof}
This is \cite[Proposition 2.5.2. (i)-(ii)]{berthelot_breen_messing_theorie_de_dieudonne_cristalline_II}.
\end{proof}

\begin{proposition}
  \label{sec:prism-cryst-abel-1-proposition-conjugate-spectral-sequence-degenerates-for-abelian-scheme-affine-case}
  The conjugate spectral sequence
  $$
  E^{ij}_2=H^i(X,\Omega^{j}_{X/{\bar{A}}})\{-j\}\Rightarrow H^{i+j}(X,\overline{\prism}_{X/A})
  $$
  degenerates and each term as well as the abutment commutes with base change in the bounded prism $(A,I)$. Moreover,
  $$
  H^\ast(X,\overline{\prism}_{X/A})\cong \Lambda^{\ast}H^1(X,\overline{\prism}_{X/A})
  $$
  is an exterior $\bar{A}$-algebra on $H^1(X,\overline{\prism}_{X/A})$.
\end{proposition}
\begin{proof}
  \green{If $p\neq 2$, we can use a simple argument using the multiplication by $n\in \Z$ on $X$.}
  \red{If $n\in \Z$, then the multiplication by $n$ on $X$ induces on $H^i(X,\Omega^j_{X/{\bar{A}}})\{-j\}$ multiplication by $n^{i+j}$. As the differentials of the spectral sequence are natural in $X$ this implies that they vanish on each $E_r$-page, $r\geq 0$ \green{(this uses $p\neq 2$)}. This proves that $H^i(X,\overline{\prism}_{X/A})$ is a finite locally free $\bar{A}$-module for each $i\geq 0$. By the Hodge-Tate comparison the complex
  $$
  \overline{\prism}_{X/A}
  $$
  satisfies base change in $(A,I)$, i.e., for a morphism $(A,I)\to (A^\prime,I^\prime)$ of prisms with induced morphism $g\colon X^\prime:=X\times_{\Spf(\bar{A})} \Spf(A^\prime/I^\prime)\to X$ the canonical morphism
  $$
  Lg^{\ast}\overline{\prism}_{X/A}\to \overline{\prism}_{X^\prime/A^\prime}
  $$
  is an isomorphism. From this we can deduce that each $H^i(X,\overline{\prism}_{X/A}), i\geq 0,$ satisfies base change in $(A,I)$. To show that $H^\ast(X,\overline{\prism}_{X/A})$ is an exterior algebra on $H^1(X,\overline{\prism}_{X/A})$, we need \red{first} to see that each element in $H^1(X,\overline{\prism}_{X/A})$ squares to zero. For this we can argue as in the proof \cite[Proposition 2.5.2.(ii)]{berthelot_breen_messing_theorie_de_dieudonne_cristalline_II}. Then we obtain a canonical morphism
  \[
    \beta\colon \wedge^\ast H^1(X,\overline{\prism}_{X/A})\to H^\ast(X,\overline{\prism}_{X/A}).
  \]
  \red{We} can use \Cref{sec:prism-cryst-abel-2-morphism-from-finitely-presented-module-to-finite-free-over-p-complete-ring-is-isomorphism} and compatibility with base change to reduce to the case that $\bar{A}$ is an algebraically closed field of characteristic $p$. In particular, the Frobenius on $A$ is bijective in this case, $I=(p)$ and the twists $(-)\{j\}$ are isomorphic to the identity. We may check that $\beta$ is an isomorphism after pullback along $\varphi_{\bar{A}}$. Then
  $$
  \varphi_{\bar{A}}^\ast H^k(X,\overline{\prism}_{X/A})\cong  H^{k}(X^{(1)},(\varphi_{X/\bar{A}})_{\ast}(\Omega^\bullet_{X/\bar{A}}))\cong H^k(X,\Omega_{X/\bar{A}}^\bullet)
  $$
  where we used in the second isomorphism that the relative Frobenius
  $$
  \varphi_{X/\bar{A}}\colon X\to X^{(1)}:=X\times_{\Spec(\bar{A}),\varphi_{\bar{A}}} \Spec(\bar{A})
  $$
  is finite. This reduces the assertion to de Rham cohomology, \red{which is the content of \Cref{sec:prism-cryst-abel-1-proposition-bbm-252}}. This finishes the proof.}

\green{Alternatively (\blue{including the case $p=2$)}, we could have argued like in \cite[Theorem 2.5.2.(i)]{berthelot_breen_messing_theorie_de_dieudonne_cristalline_II} to reduce, by descending induction, to the claim that $H^1(X,\overline{\prism}_{X/A})$ is locally free of rank $2n$, where $n$ is the relative dimension of $X$ over $\Spf(\bar{A})$, and commutes with base change in $(A,I)$.
  From \Cref{sec:isom-trunc-prism-cohom-cotangent-complex} it follows that
  $$
 H^1(X,\overline{\prism}_{X/A})\cong H^1(X,\tau_{\leq 1}\overline{\prism}_{X/A})\cong H^0(X,L_{X/A}[-1]).
  $$
   As $L_{X/A}$ is a perfect complex with amplitude in $[-1,0]$ this implies compatibility of $H^1(X,\overline{\prism}_{X/A})$ with base change in $(A,I)$ if all the higher cohomology groups $H^j(X,L_{X/A}[-1])$ are locally free.
   As $X$ admits a lift to $A$ (see e.g. \cite[Theorem 2.2.1]{oort_moduli_abelian_varieties}), \Cref{sec:obstruction-split-same} shows that $L_{X/A}\cong \mathcal{O}_X[1]\oplus \Omega_{X/\bar{A}}^1$. Another application of \Cref{sec:prism-cryst-abel-1-proposition-bbm-252} implies therefore that $H^1(X,\overline{\prism}_{X/A})$ is locally free of dimension $2n$ and commutes with base change in $(A,I)$ as all the $\bar{A}$-modules $H^j(X,\mathcal{O}_X)$ and $H^j(X,\Omega_{X/\bar{A}}^1)$ are locally free for $j\geq 0$.}
\end{proof}

\begin{lemma}
  \label{sec:prism-cryst-abel-2-morphism-from-finitely-presented-module-to-finite-free-over-p-complete-ring-is-isomorphism}
  Let $S$ be a ring and let $g\colon M\to N$ be a morphism of $S$-modules with $M$ finitely generated and $N$ finite projective. If
  $$
  g\otimes_{S} k(x)\colon M\otimes_S k(x)\to N\otimes_S k(x)
  $$
  is an isomorphism for all closed points $x\in \Spec(S)$, then $g$ is an isomorphism.
\end{lemma}
\begin{proof}
  Let $Q$ be the cokernel of $g$. Then $Q$ is finitely generated and $Q\otimes_S k(x)=0$ for all closed points $x\in \Spec(S)$. By Nakayama's lemma, this implies that $Q=0$, i.e., $g$ is surjective. As $N$ is projective, this implies $M\cong N \oplus K$ for $K$ the kernel of $g$. As $M$ is finitely generated, $K$ is finitely generated. Moreover for all closed points $x\in \Spec(S)$
  $$
  K\otimes_{S} k(x)=0
  $$
  and thus another application of Nakayama's lemma implies that $K=0$.
\end{proof}

We recall that for a $p$-complete ring $R$ there is the natural morphism of topoi
$$
u\colon \mathrm{Shv}(R)_\prism\to \mathrm{Shv}(R)_{\mathrm{QSYN}}.
$$

Using the previous computations, we can first describe extension groups modulo $I$.

\begin{theorem}
  \label{sec:prism-dieud-cryst-reduced-ext-for-abelian-schemes}
  Let $R$ be a $p$-complete ring and let $f\colon X\to \Spf(R)$ be the $p$-adic completion of an abelian scheme over $\Spec(R)$. Then
  \begin{enumerate}
   \item $\mathcal{E}xt_{(R)_{\prism}}^i(u^{-1}(X),\overline{\mathcal{O}}_\prism)=0$ for $i=0,2$.
   \item $\mathcal{E}xt_{(R)_{\prism}}^1(u^{-1}(X),\overline{\mathcal{O}}_\prism)$ is a prismatic crystal over $R$. Moreover,
     $$
     \mathcal{E}xt_{(R)_{\prism}}^1(u^{-1}(X),\overline{\mathcal{O}}_\prism)\cong R^1f_{\prism,\ast}(\overline{\mathcal{O}}_\prism)
     $$
     for $f_\prism\colon \mathrm{Shv}{(X)_\prism}\to \mathrm{Shv}{(R)_\prism}$ the morphism induced by $f$ on topoi and $\mathcal{E}xt_{(R)_{\prism}}^1(u^{-1}(X),\overline{\mathcal{O}}_\prism)$ is locally free of rank $2\mathrm{dim}(X)$ over $\overline{\mathcal{O}}_{\prism}$.
   \end{enumerate}
 \end{theorem}

The proof is entirely similar to the one of \cite[Th\'eor\`eme 2.5.6]{berthelot_breen_messing_theorie_de_dieudonne_cristalline_II}.
 
 \begin{proof}
   Let $(B,J)\in (R)_\prism$. We use the spectral sequence from \Cref{sec:calc-ext-groups} to calculate for $i\in \{0,1,2\}$ the groups
   $$
   \mathrm{Ext}^i(u^{-1}(X)_{|(B,J)},\overline{\mathcal{O}}_{\prism})
   $$
   on the localised site $(R)_{\prism}/(B,J)$\footnote{Which will be implicitly the subscript of all $\mathrm{Ext}$-groups appearing in this proof.}. Set $Y:=X\times_{\Spf(\bar{A})}\Spf(B/J)$.
   As by Hodge-Tate comparison
   $$
   \purple{H^0((Y/B)_{\prism},\overline{\mathcal{O}}_\prism)\cong }H^0(Y,\overline{\prism}_{Y/B})\cong B/J
   $$
   for any $n$ the first row \red{$E_1^{\ast,0}$} of the spectral sequence \red{is
seen to be independent of $X$ and exact in the case that $X=0$ is trivial (the spectral sequence for $X=0$ is concentrated in the first row and converges to $0$), hence always exact.} In general we see that $\mathrm{Hom}(u^{-1}(X)_{|(B,J)},\overline{\mathcal{O}}_\prism)=0$ and $\mathrm{Ext}^1(u^{-1}(X)_{|(B,J)},\overline{\mathcal{O}}_\prism)$ is isomorphic to the kernel of
   $$
   H^1(Y,\overline{\prism}_{Y/B}) \xrightarrow{d_1} H^1(\purple{Y\times Y},\overline{\prism}_{\purple{Y\times Y}/B})
   $$
   and $d_1=\mathrm{pr}_1^\ast+\mathrm{pr}_2^\ast-\mu^\ast$ for $\mathrm{pr}_i$ the two projections and $\mu$ the multiplication. From the K\"unneth formula (cf.\ \Cref{sec:kunn-form-prism-smooth-proper-case}) and \blue{\Cref{sec:prism-cryst-abel-1-proposition-conjugate-spectral-sequence-degenerates-for-abelian-scheme-affine-case}} it follows that
   $$
   H^1(\purple{Y\times Y},\overline{\prism}_{\purple{Y\times Y}/B})\cong H^1(Y,\overline{\prism}_{Y/B})\oplus H^1(Y,\overline{\prism}_{Y/B}).
   $$
   This implies $\mu^\ast=\mathrm{pr}_1^\ast+\mathrm{pr}_2^\ast$, i.e., $d_1=0$ and
   $$
\mathrm{Ext}^1(u^{-1}(X)_{|(B,J)},\overline{\mathcal{O}}_\prism)\cong H^1(Y,\overline{\prism}_{Y/B}).
$$
In particular, this group is compatible with base change in $(B,J)$ and locally free of rank $2\mathrm{dim}(X)$ (by \Cref{sec:prism-cryst-abel-1-proposition-conjugate-spectral-sequence-degenerates-for-abelian-scheme-affine-case}).
Moreover, the morphism $d_2$ is injective on \red{$H^1(Y \times Y,\overline{\prism}_{Y/B})$} as follows from the K\"unneth theorem and the concrete formula for $d_2$.
Finally, from \Cref{sec:prism-dieud-modul-corollary-prismatic-cohomology-of-abelian-schemes} and \Cref{sec:prism-dieud-modul-lemma-on-primitive-elements-in-exterior-algebras} one can deduce that
$$
H^i(Y,\overline{\prism}_{Y/B})\xrightarrow{d_1} H^i(Y,\overline{\prism}_{Y/B})
$$ 
is injective for all $i\geq 2$.  
These statements \purple{(together with the mentioned exactness of the first row) imply
$$
\mathrm{Ext}^2(u^{-1}(X)_{|(B,J)},\overline{\mathcal{O}}_\prism)=0.
$$}
This finishes the proof by passing to the local Ext-groups, i.e., by letting $(B,J)$ vary.
 \end{proof}

In the proof we used the following lemma on primitive elements in exterior algebras.
 
 \begin{lemma}
   \label{sec:prism-dieud-modul-lemma-on-primitive-elements-in-exterior-algebras}Let $S$ be a ring and let $M$ be a projective $S$-module. Then
   $$
   \{x\in \Lambda(M)\ |\ \mu^\ast(x)=1\otimes x+x\otimes 1\}=\Lambda^1 M,
   $$
   where $\mu^\ast\colon \Lambda(M)\to \Lambda(M+M)\cong \Lambda(M)\otimes_S \Lambda(M)$ is the natural comultiplication on $\Lambda(M)$ coming from the diagonal $M\to M\oplus M$.
 \end{lemma}
 \begin{proof}
   This follows easily by decomposing $\Lambda(M)\otimes_S \Lambda(M)$ into its bigraded pieces $\Lambda^{i}(M)\otimes_S \Lambda^j(M)$.
 \end{proof}

Now we calculate the full extension groups, up to degree $2$.

\begin{theorem}
  \label{sec:prism-dieud-cryst-ext-for-abelian-schemes}
  Let $R$ be a $p$-complete ring and let $f\colon X\to \Spf(R)$ be the $p$-adic completion of an abelian scheme over $\Spec(R)$. Then
  \begin{enumerate}
   \item $\mathcal{E}xt_{(R)_{\prism}}^i(u^{-1}(X),\mathcal{O}_\prism)=0$ for $i=0,2$.
   \item $\mathcal{E}xt_{(R)_{\prism}}^1(u^{-1}(X),\mathcal{O}_\prism)$ is a prismatic crystal over $R$. Moreover, $$\mathcal{E}xt_{(R)_{\prism}}^1(u^{-1}(X),\mathcal{O}_\prism)\cong R^1f_{\prism,\ast}(\mathcal{O}_\prism),$$ for $f_\prism\colon \mathrm{Shv}(X)_\prism\to \mathrm{Shv}(R)_\prism$ the induced morphism on topoi and the prismatic crystal $\mathcal{E}xt_{(R)_{\prism}}^1(u^{-1}(X),\mathcal{O}_\prism)$ is locally free of rank $2\mathrm{dim}(X)$ over \red{$\mathcal{O}_{\prism}$}.
   \end{enumerate}
 \end{theorem}
 \begin{proof}
   Let $(B,J)\in (R)_{\prism}$. As the statements are local for the faithfully flat topology we may assume that $J=(\tilxi)$ is principal.
   From the exact sequence
   $$
   0\to \mathcal{O}_{\prism}/\tilxi^n\xrightarrow{\tilxi} \mathcal{O}_{\prism}/\tilxi^{n+1}\to \mathcal{O}_\prism/\tilxi=\overline{\mathcal{O}}_\prism\to 0 
   $$
   of sheaves on $(R)_\prism/(B,J)$ and \Cref{sec:prism-dieud-cryst-reduced-ext-for-abelian-schemes} we can inductively conclude that
   $$
   \mathrm{Ext}^{i}(u^{-1}(X)_{|(B,J)},\mathcal{O}_\prism/(\tilxi^n))=0
   $$
   for $i\in \{0,2\}$ and any $n\geq 0$.
   This implies that
   $$
   0\to \mathrm{Ext}^{1}(u^{-1}(X)_{|(B,J)},\mathcal{O}_\prism/(\tilxi^n))\xrightarrow{\tilxi} \mathrm{Ext}^{1}(u^{-1}(X)_{|(B,J)},\mathcal{O}_\prism/(\tilxi^{n+1}))
   $$
   $$
   \to \mathrm{Ext}^{1}(u^{-1}(X)_{|(B,J)},\overline{\mathcal{O}}_\prism)\to 0
   $$
   is exact and that \red{for $0 \leq i \leq 2$,}
   $$
   \mathrm{Ext}^{i}(u^{-1}(X)_{|(B,J)},\mathcal{O}_\prism)\cong \varprojlim\limits_{n}\mathrm{Ext}^{i}(u^{-1}(X)_{|(B,J)},\mathcal{O}_\prism/(\tilxi^n)),
   $$
   and that it is zero for $i\in \{0,2\}$ or a locally free $B$-module of rank $2\mathrm{dim}(X)$ if $i=1$.
   \red{Using the spectral sequence from \Cref{sec:calc-ext-groups}, we get as in the proof of \Cref{sec:prism-dieud-cryst-reduced-ext-for-abelian-schemes} for each $n \geq 1$ a map 
   $$
   \mathrm{Ext}^{1}(u^{-1}(X)_{|(B,J)},\mathcal{O}_\prism/(\tilxi^n))\to H^1(X\times_{\Spf(R)}\Spf(B/J),\prism_{X/A}/(\tilxi^n)).
   $$
  By induction on $n$, we deduce from \Cref{sec:prism-dieud-cryst-reduced-ext-for-abelian-schemes} that this map is an isomorphism for all $n$. Passing to the inverse limit over all $n\geq 1$ and using the above identification, we deduce an isomorphism
    $$
   \mathrm{Ext}^{1}(u^{-1}(X)_{|(B,J)},\mathcal{O}_\prism)\cong H^1(X\times_{\Spf(R)}\Spf(B/J),\prism_{X/A}).
   $$ }
   This finishes the proof by passing to local Ext-groups.   
 \end{proof}
 
 \begin{corollary}
  \label{sec:prism-dieud-modul-corollary-prismatic-dieudonne-crystal-of-abelian-schemes}
Let $R$ be a $p$-complete ring. Let $X$ be the $p$-completion of an abelian scheme over $R$. The $\mathcal{O}^{\rm pris}$-module
\[ \mathcal{M}_{\prism}(X[p^{\infty}]) = \mathcal{E}xt_{(R)_{\mathrm{qsyn}}}^1(\purple{X[p^\infty]},  \mathcal{O}^{\rm pris}) \]
is a finite locally free $\mathcal{O}^{\rm pris}$-module of rank $2 \dim(X)$, given by $R^1 f_{\prism,\ast}\mathcal{O}_\prism$.
\end{corollary}
\begin{proof}
By \Cref{sec:divid-prism-dieud-lemma-divided-dieudonne-module-via-local-ext-on-prismatic-site}, 
\[ \mathcal{M}_{\prism}(X[p^{\infty}]) = v_*(\mathcal{E}xt_{(R)_{\prism}}^1(u^{-1} G, \mathcal{O}_{\prism})). \]
Hence the corollary results from \Cref{sec:prism-dieud-cryst-ext-for-abelian-schemes} and \Cref{sec:abstr-divid-prism-proposition-finite-locally-free}.
\end{proof}

Although we will not use it, let us record the full description of the prismatic cohomology of $X$.

\begin{corollary}
  \label{sec:prism-dieud-modul-corollary-prismatic-cohomology-of-abelian-schemes} With the notation from \Cref{sec:prism-dieud-modul-corollary-prismatic-dieudonne-crystal-of-abelian-schemes},
  the prismatic cohomology
  $$
  R^\ast f_{\prism,\ast}\mathcal{O}_\prism
  $$
  is a finite locally free crystal on $(R)_\prism$ and an exterior algebra on the locally free crystal
  $$
R^1f_{\prism,\ast}(\mathcal{O}_\prism)
$$
of dimension $2\mathrm{dim}(X)$.
  
\end{corollary}
\begin{proof}
  Let $(B,J)\in (R)_\prism$ and let $Y:=X\times_{\Spf(R)}\Spf(B/J)$. It suffices to prove the analog statements for $H^\ast(Y,\prism_{Y/B})$.
  From (the proof of) \Cref{sec:prism-dieud-cryst-ext-for-abelian-schemes} we see that
  $$
  H^1(Y,\prism_{Y/B})\to H^1(Y,\overline{\prism}_{Y/B})
  $$
  is surjective and that $H^\ast(Y,\overline{\prism}_{Y/B})$ is an exterior algebra on $H^1(Y,\overline{\prism}_{Y/B})$. \blue{Since $H^1(Y,\prism_{Y/B})$ is projective, we can lift the identity $H^1(Y,\prism_{Y/B}) \to H^1(Y,\prism_{Y/B})$ to a map}
  $$
  \blue{H^1(Y,\prism_{Y/B})}[-1]\to \prism_{Y/B}.
  $$
  \red{Using multiplication in $H^\ast(Y,\prism_{Y/B})$ and that $H^\ast(Y,\overline{\prism}_{Y/B})$ is an exterior algebra, we see that for each $i\geq 0$ the morphism
    \[
      H^i(Y,\prism_{Y/B})\to H^i(Y,\overline{\prism}_{Y/B})
    \]
  is surjective. This implies that each $B$-module $H^i(Y,\prism_{Y/B})$ is $J$-torsion free, and then that it is a finite, locally free $B$-module as modulo $J$ it identifies with $H^i(Y,\overline{\prism}_{Y/B})$. The same argument \blue{as in \cite[Proposition 2.5.2.(ii)]{berthelot_breen_messing_theorie_de_dieudonne_cristalline_II})} implies then that each element in $H^1(Y,\prism_{Y/B})$ squares to zero.} \blue{We obtain a morphism
  $$
  \Lambda^i{H^1(Y,\prism_{Y/B})}[-i]\to R\Gamma(Y,\prism_{Y/B})
  $$}
  inducing an isomorphism on $H^i$ after passing to $\otimes^{\mathbb{L}}_{B}B/J$.
  Altogether, we obtain a morphism
  $$
  \blue{\Lambda^\ast(H^1(Y,\prism_{Y/B}))[-\ast]}\to R\Gamma(Y,\prism_{Y/B})
  $$
  of complexes which is an isomorphism after applying $\otimes^{\mathbb{L}}_BB/J$. By derived $J$-adic completeness it is therefore an isomorphism, which implies the statements.
 \end{proof}

\subsection{The prismatic Dieudonn\'e crystal of a $p$-divisible group}
\label{sec:divid-prism-dieud-definition-for-p-div-groups}

In this section, we establish the basic properties of the prismatic Dieudonn\'e functor for $p$-divisible groups. The idea, due to Berthelot-Breen-Messing, is to make systematic use of the following theorem of Raynaud, to reduce to statements about ($p$-divisible groups of) abelian schemes proved in the last section.

\begin{theorem}
\label{sec:divid-prism-dieud-definition-for-p-div-groups-theorem-raynaud}
Let $S$ be a scheme, and let $G$ be a finite locally free group scheme over $S$. There exists Zariski-locally on $S$, a (projective) abelian scheme $A$ and a closed immersion $G \hookrightarrow A$ \purple{of group schemes} over $S$.
\end{theorem}
\begin{proof}
See \cite[Theorem 3.1.1]{berthelot_breen_messing_theorie_de_dieudonne_cristalline_II}.
\end{proof} 

\begin{proposition}
\label{sec:divid-prism-dieud-proposition-crystals-loc-finitely-presented}
Let $R$ be a $p$-complete ring, and let $G$ be a finite locally free group scheme over $R$.  The sheaf $\Ext_{(R)_{\prism}}^1(u^{-1}G,\mathcal{O}_ {\prism})$ is a prismatic crystal of locally finitely presented $\mathcal{O}_{\prism}$-modules. 
\end{proposition}
\begin{proof}
By  \Cref{sec:divid-prism-dieud-definition-for-p-div-groups-theorem-raynaud}, one can choose locally on $R$ an exact sequence of group schemes
\[ 0 \to G \to X \to X' \to 0, \]
where $X$ and $X'$ are abelian schemes over $R$. Whence, by \Cref{sec:prism-dieud-cryst-ext-for-abelian-schemes} (1), an exact sequence
\[ \Ext_{(R)_{\prism}}^1(u^{-1}X',\mathcal{O}_ {\prism}) \to \Ext_{(R)_{\prism}}^1(u^{-1}X,\mathcal{O}_ {\prism}) \to \Ext_{(R)_{\prism}}^1(u^{-1}G,\mathcal{O}_ {\prism}) \to 0. \]
This proves the proposition, by \Cref{sec:prism-dieud-cryst-ext-for-abelian-schemes} (2). 
\end{proof}

\red{Let $n\geq 1$. Recall (\cite[Definition 1.1]{illusie_deformations_de_groupes_de_barsotti_tate}) that an a finite locally free group scheme $G$ over a scheme $S$ is called a \textit{truncated Barsotti-Tate group of level $n$} if it is killed by $p^n$ and flat over $\Z/p^n$, and, when $n=1$, if it also satisfies that the sequence
$$
G_0 \overset{F} \to \varphi_{S_0,\ast} G_0 \overset{V} \to G_0
$$
is exact, where $G_0$ denotes the base change of $G$ to $S_0=V(p) \subset S$. The rank of $G[p]$ is of the form $p^h$, for an integer $h$ locally constant on $S$ called the \textit{height} of $G$. In the sequel, we will make use of the following basic facts on truncated Barsotti-Tate groups (cf. \cite[1.3 (e), 1.3 (f), 1.6]{illusie_deformations_de_groupes_de_barsotti_tate}):
\begin{enumerate}
\item If $G$ is a $p$-divisible group over $S$ (of height $h$), $G[p^n]$ is a truncated Barsotti-Tate group of level $n$ over $S$ (of height $h$) for all $n\geq 1$.
\item If $G$ is a truncated Barsotti-Tate group of level $n$ and height $h$, then so is the Cartier dual $G^\ast$ of $G$.
\item If $0 \to G_1 \to G_2 \to G_3 \to 0$ is an exact sequence finite locally free group schemes of order $p^n$ over $S$, and if two of them are truncated Barsotti-Tate groups of level $n$, then so is the third one.
\end{enumerate}
}


\begin{remark}
  \label{remark-statement-about-truncated-bt}
Let $G$ be a finite locally free group scheme \red{killed by $p^n$ over a scheme $S$ such that $p^n \mathcal{O}_S=0$}, and let $\ell_G$ be its coLie complex. Set :
\[ \omega_G=H^0(\ell_G) ~ , ~ n_G=H^{-1}(\ell_G) ~ , ~ t_G=H^0(\check{\ell}_G) ~ ; ~ \nu_G=H^1(\check{\ell}_G). \]
Grothendieck's duality formula identifies $\check{\ell}_G$ with the truncation $\tau^{\leq 1}R\mathcal{H}om(G^*,\mathbb{G}_a)$, and this gives rise to a canonical morphism :
\[ \phi_G :  \nu_G \to t_G.\]
Then $G$ is a $\mathrm{BT}_n$ if and only if $t_G, t_{G^*}$ are locally free and the canonical morphisms $\phi_G$ and $\phi_{G^*}$ are isomorphisms (cf.\ \cite[Corollary 2.2.5]{illusie_deformations_de_groupes_de_barsotti_tate}). \red{In this situation, $\omega_G$ is finite locally free of rank called the \textit{dimension} $\dim(G)$ of $G$, and $\nu_{G^\ast}$ is finite locally free of rank $h-\dim(G)$, if $h$ is the height of $G$.} 
\end{remark}

\begin{proposition} \label{crystal O bar}
	Let $R$ be a \red{quasi-syntomic} ring, and let $G$ be a truncated Barsotti-Tate group over $R$ of level $n$. The sheaf $\Ext_{(R)_{\prism}}^1(u^{-1}G,\mathcal{O}_ {\prism})$ is a prismatic crystal of finite locally free $\mathcal{O}_{\prism}/p^n$-modules.
\end{proposition}
\begin{proof}
Fix once and for all an embedding of $G$ into an abelian scheme \red{$X^\prime$} of dimension $g$ over $R$. By \Cref{sec:divid-prism-dieud-definition-for-p-div-groups-theorem-raynaud}, this can be done Zariski-locally on \red{$\mathrm{Spf}(R)$}, and the reader can check that the different steps of the proof are all local statements on \red{$\mathrm{Spf}(R)$}. Let \red{$X$} be the cokernel of the embedding $G\to \red{X^\prime}$ ; this an abelian scheme, and one has an exact sequence 
\[ 0 \to G \to \red{X^\prime \to X} \to 0 \]
of group schemes over $R$. 

We first prove that for any $(B,J) \in (R)_{\prism}$, the $B$-module $$\Ext_{(R)_{\prism}}^1(u^{-1}G, \mathcal{O}_{\prism})_{(B,J)}$$ is locally generated by $h$ sections, where $h$ is the height of $G$. By the crystal property of $\Ext_{(R)_{\prism}}^1(u^{-1}G, \mathcal{O}_{\prism})$ (cf.\ \Cref{sec:divid-prism-dieud-proposition-crystals-loc-finitely-presented}), for any morphism of prisms $(B,J) \to (W(k),(p))$, where $k$ is a characteristic $p$ perfect field,
\[ \Ext_{(R)_{\prism}}^1(u^{-1}G, \mathcal{O}_{\prism})_{(B,J)} \otimes_B W(k) = \Ext_{(R)_{\prism}}^1(u^{-1}G_k, \mathcal{O}_{\prism})_{(W(k),(p))}. \]

By Nakayama's lemma, \red{$(p,J)$}-completeness of $B$ and the finite presentation proved in \Cref{sec:divid-prism-dieud-proposition-crystals-loc-finitely-presented}, it suffices to prove that for any morphism $B \to k$ \red{vanishing on $J$}, $k$ characteristic $p$ perfect field, $$\Ext_{(R)_{\prism}}^1(u^{-1}G, \mathcal{O}_{\prism})_{(B,J)} \otimes_B k$$ is generated by $h$ elements. Such a morphism $B \to k$ extends to a morphism of prisms $(B,J) \to (W(k),(p))$, so it suffices by the above to prove our claim when $R=k$ is a perfect field and $(B,J)=(W(k),(p))$. First, observe that
\[ \Ext_{(k)_{\prism}}^1(u^{-1}G,\mathcal{O}_{\prism})_{(W(k),(p))} \otimes k = \Ext_{(k)_{\prism}}^1(u^{-1}G,\overline{\mathcal{O}}_{\prism})_{(W(k),(p))}. \]  
This is easily seen, using that $\Ext_{(k)_{\prism}}^2(u^{-1}X,\mathcal{O}_{\prism})$ and $\Ext_{(k)_{\prism}}^2(u^{-1}X,\overline{\mathcal{O}}_{\prism})$ both vanish \red{(\Cref{sec:prism-dieud-cryst-reduced-ext-for-abelian-schemes} and \Cref{sec:prism-dieud-cryst-ext-for-abelian-schemes})}.

As a corollary of \Cref{sec:prism-cryst-abel-1-proposition-conjugate-spectral-sequence-degenerates-for-abelian-scheme-affine-case} \red{(together with the standard relation between $H^1(X,\mathcal{O})$ and $\mathrm{Lie}(X^\ast)$, cf \cite[\S 5.1.1]{berthelot_breen_messing_theorie_de_dieudonne_cristalline_II})} and \Cref{sec:prism-dieud-cryst-reduced-ext-for-abelian-schemes}, one has a short exact sequence 
\[ 0 \to u^{*} \mathrm{Lie}(X^*) \to \Ext_{(R)_{\prism}}^1(u^{-1}X,\overline{\mathcal{O}}_ {\prism}) \to u^{*}\omega_X\to 0,  \]
and similarly for $X'$. Also, note that we have exact sequences\footnote{Recall (\cite[\S 5.1.1]{berthelot_breen_messing_theorie_de_dieudonne_cristalline_II}) that if $X$ is an abelian scheme, $\mathrm{Lie}(X^*)\cong\mathrm{Ext}^1(X,\mathbb{G}_a)$.} :
    \[ 	u^{*} \mathrm{Lie}(X^*) \to u^{*} \mathrm{Lie}(X^{'*}) \to u^{*} \nu_{G^*} \to 0     \]
    (where $\nu_{G^*}=\Ext^1(G,\mathbb{G}_a)$) and 
     \[  u^{*} \omega_X \to u^{*} \omega_{X'} \to u^{*} \omega_{G} \to 0.  \]
    The map $\Ext_{(k)_{\prism}}^1(u^{-1}X',\overline{\mathcal{O}}_ {\prism}) \to \Ext_{(k)_{\prism}}^1(u^{-1}X,\overline{\mathcal{O}}_ {\prism})$ is compatible with the natural maps $u^{*} \mathrm{Lie}(X^*) \to u^{*} \mathrm{Lie}(X^{'*})$ and $u^{*} \omega_{X'} \to u^{*} \omega_X$, through the identifications of \Cref{sec:prism-dieud-cryst-reduced-ext-for-abelian-schemes}.  
     The long exact sequence of $\Ext$ gives a surjection :
	\[  \Ext_{(k)_{\prism}}^1(u^{-1}X',\overline{\mathcal{O}}_ {\prism}) \to \Ext_{(k)_{\prism}}^1(u^{-1}X,\overline{\mathcal{O}}_ {\prism}) \to \Ext_{(k)_{\prism}}^1(u^{-1}G,\overline{\mathcal{O}}_ {\prism}) \to 0,   \]
	since, as we have seen in \Cref{sec:prism-dieud-cryst-reduced-ext-for-abelian-schemes}, $\Ext_{(k)_{\prism}}^2(u^{-1}X',\overline{\mathcal{O}}_{\prism})=0$. By the above remark, we even have a commutative diagram :
	 $$
	 \xymatrix{
	 	0 \ar[r] \ar[d] & 0\ar[d] \ar[r] & 0  \ar[d]  \\
	 	u^{*} \mathrm{Lie}(X^*) \ar[r] \ar[d] & u^{*} \mathrm{Lie}(X^{'*})\ar[d] \ar[r] & u^{*} \nu_{G^*} \ar[r] \ar[d] & 0  \\
	 	\Ext_{(k)_{\prism}}^1(u^{-1}X,\overline{\mathcal{O}}_ {\prism}) \ar[r] \ar[d] & \Ext_{(k)_{\prism}}^1(u^{-1}X',\overline{\mathcal{O}}_ {\prism})\ar[d] \ar[r] & \Ext_{(k)_{\prism}}^1(u^{-1}G,\overline{\mathcal{O}}_ {\prism}) \ar[r] \ar[d] & 0 \\
	 	u^{*}\omega_X \ar[r] \ar[d]  & u^{*}\omega_{X'}  \ar[d] \ar[r] & u^{*} \omega_{G} \ar[d] \ar[r]  & 0 \\
 	   0 \ar[r]  & 0 \ar[r] & 0  }
	 $$
	 where all \red{rows} and the first two columns are exact. This proves that the map
	 \[ \Ext_{(k)_{\prism}}^1(u^{-1}G,\overline{\mathcal{O}}_ {\prism})  \to u^{*} \omega_{G}   \]
	 is surjective and an easy diagram chase prove that in fact the sequence 
	\[  u^{*} \nu_{G^*} \to \Ext_{(k)_{\prism}}^1(u^{-1}G,\overline{\mathcal{O}}_ {\prism}) \to u^{*}\omega_{G} \to 0  \]
	is exact. \red{As $G$ is a truncated Barsotti-Tate group, the sheaf $\omega_{G}$ is a locally free sheaf of rank $d=\dim G$ and $\nu_{G^*}$ is a locally free sheaf of rank $h-d$ (cf. \Cref{remark-statement-about-truncated-bt}, which applies whatever the level of $G$ is, since $p=0$ on $k$)}. Hence the sequence stays exact after evaluation on $(W(k),(p))$ and $\Ext_{(k)_{\prism}}^1(u^{-1}G,\overline{\mathcal{O}}_ {\prism})_{(W(k),(p))}$ is generated by $h$ sections. This proves the claim.

Back to the proof of the proposition, we know, as a direct consequence of \Cref{sec:prism-dieud-cryst-ext-for-abelian-schemes} that
\[ \Ext_{(R)_{\prism}}^1(u^{-1} \red{X^\prime}[p^n], \mathcal{O}_{\prism}) = \Ext_{(R)_{\prism}}^1(u^{-1} \red{X^\prime}, \mathcal{O}_{\prism})/p^n \]
is crystal of locally free $\mathcal{O}_{\prism}/p^n$-modules of rank \red{$2g$}. Consider the exact sequence
\[ 0 \to G \to \red{X^\prime}[p^n] \to H \to 0, \]
where $H$ is a Barsotti-Tate group of height $2g-h$, induced by the embedding of $G$ in $\red{X^\prime}$. This gives an exact sequence
\[ \Ext_{(R)_{\prism}}^1(u^{-1} H, \mathcal{O}_{\prism}) \to \Ext_{(R)_{\prism}}^1(u^{-1} \red{X^\prime}[p^n], \mathcal{O}_{\prism}) \to \Ext_{(R)_{\prism}}^1(u^{-1} G, \mathcal{O}_{\prism}) \to 0. \]
\green{Indeed, right-exactness follows from \ref{sec:prism-dieud-cryst-ext-for-abelian-schemes}, which implies that already $$\Ext_{(R)_{\prism}}^1(u^{-1} \red{X^\prime}, \mathcal{O}_{\prism}) \to \Ext_{(R)_{\prism}}^1(u^{-1} G, \mathcal{O}_{\prism})$$ is surjective.}
Locally on $(R)_{\prism}$, the middle term is free of rank $2g$ over $\mathcal{O}_{\prism}/p^n$, while the left (resp. right) term is generated by $2g-h$ (resp. $h$) sections. Therefore, $\Ext_{(R)_{\prism}}^1(u^{-1} H, \mathcal{O}_{\prism})$ and $\Ext_{(R)_{\prism}}^1(u^{-1} G, \mathcal{O}_{\prism})$ are free over $\mathcal{O}_{\prism}/p^n$ of rank $2g-h$ and $h$.
\end{proof}

\begin{proposition}  \label{sec:divid-prism-dieud-proposition-ext-crystal}
	Let $R$ be a $p$-complete ring, and let $G$ be a $p$-divisible group over $R$. The sheaf $$\Ext_{(R)_{\prism}}^1(u^{-1}G,\mathcal{O}_{\prism})$$ is a prismatic crystal of finite locally free $\mathcal{O}_{\prism}$-modules of rank the height of $G$.
	
	\purple{In particular, if $R$ is a quasi-syntomic ring and $G$ is a $p$-divisible group over $R$, the $\mathcal{O}^{\rm pris}$-module $\mathcal{M}_{\prism}(G)$ is a finite locally free $\mathcal{O}^{\rm pris}$-module of rank the height of $G$.} 
\end{proposition}
\begin{proof}
	Let $G$ be a $p$-divisible group over $R$. Since $G=\mathrm{colim}~G[p^n]$, we have a short exact sequence :
	\[  0 \to R^1 \underset{n} \lim ~ \mathcal{H}om_{(R)_{\prism}}(u^{-1} G[p^n], \mathcal{O}_{\prism}) \to \Ext_{(R)_{\prism}}^1(u^{-1} G, \mathcal{O}_{\prism}) \]
        \[\to   \underset{n} \lim  ~ \Ext_{(R)_{\prism}}^1(u^{-1} G[p^n], \mathcal{O}_{\prism}) \to \blue{R^2 \underset{n} \lim ~ \mathcal{H}om_{(R)_{\prism}}(u^{-1} G[p^n], \mathcal{O}_{\prism})}. \]
          \blue{The last term vanishes as the prismatic topos is replete.} We have to show that the first term vanishes, or even stronger, that for each $(B,J)\in (R)_{\prism}$ the morphism
          $$
          \mathrm{Ext}^1_{(R)_\prism/(B,J)}(u^{-1}(G),\mathcal{O}_\prism)\to \underset{n}\lim ~ \mathrm{Ext}^1_{(R)_\prism/(B,J)}(u^{-1}(G[p^n]),\mathcal{O}_\prism)
          $$
          is bijective.
          Set
          $$
          M:=\mathrm{Ext}^1_{(R)_\prism/(B,J)}(u^{-1}(G),\mathcal{O}_\prism)
          $$
          and
          $$
          M_n:=\mathrm{Ext}^1_{(R)_\prism/(B,J)}(u^{-1}(G[p^n]),\mathcal{O}_\prism)
          $$
          for $n\geq 0$. For $n,m\geq 0$ the sequence
          $$
          M_{m}\xrightarrow{p^n} M_{n+m}\to M_n\to 0
          $$
          is right exact (this follows by locally embedding $G[p^{m+n}]$ and using \Cref{sec:prism-dieud-cryst-ext-for-abelian-schemes}). Thus, the canonical morphism
          $$
          M_{n+m}\otimes_{B/p^{n+p}} B/p^n\to M_n
          $$
          is an isomorphism for $n,m\geq 0$. As all $M_n$ are finite locally free over $B/p^n$ (of rank the height of $G$) the $B$-module $N:=\varprojlim\limits_{n} M_n$ is finite locally free over $B$ (of rank the height of $G$) by \cite[Tag 0D4B]{stacks_project}. By the same reference
          $$
          N/p^n\cong M_n.
          $$
          The canonical morphism $M\to N$ is surjective (by a similar $R^1\varprojlim\limits_{n}$ sequence as above). In particular, we can conclude that $M\to M_n$ is surjective for each $n\geq 0$. The long exact sequence for $0\to u^{-1}(G[p^n])\to u^{-1}G\to u^{-1}G\to 0$ and the surjectivity of $M\to M_n$ imply that $M/p^n\cong M_n$ and $\mathrm{Ext}^{2}_{{(R)_{\prism}/(B,J)}}(u^{-1}(G),\mathcal{O}_{\prism})$ has no $p^n$-torsion. This $p$-torsion freeness of $\mathrm{Ext}^2$ in turn implies that
          $$
          M/p^n\cong \mathrm{Ext}^1_{(R)_\prism/(B,J)}(u^{-1}(G),\mathcal{O}_{\prism}/p^n).
          $$
          Our aim is to prove that $M\cong N$ or equivalently that $M$ is classically $p$-complete, i.e., $M\cong \varprojlim\limits_{n} M/p^n$.
          As all prisms in $(R)_{\prism}$ are by definition bounded, and thus classically $p$-complete,
          $$
          \mathcal{O}_{\prism}\cong \varprojlim\limits_{n} \mathcal{O}_{\prism}/p^n\cong R\varprojlim\limits_{n} \mathcal{O}_\prism/p^n.
          $$
          We can therefore calculate $M=\mathrm{Ext}^1_{(R)_\prism/(B,J)}(u^{-1}G,\mathcal{O}_\prism)$ by an exact sequence
          $$
          0\to R^1\varprojlim\limits_{n} \mathrm{Hom}_{(R)_\prism/(B,J)}(u^{-1}G,\mathcal{O}_\prism/p^n)\to M\to \varprojlim\limits_nM/p^n\to 0.
          $$
          In this sequence the $R^1\varprojlim\limits_{n}$-term vanishes as each $\mathrm{Hom}_{(R)_\prism/(B,J)}(u^{-1}G,\mathcal{O}_\prism/p^n)$ is zero because $G$ is $p$-divisible.
          \red{The isomorphisms}
          $$
          M\cong \varprojlim\limits_{n} M/p^n \cong \varprojlim\limits_n M_n
          $$
          \red{imply} that $\mathcal{M}_{\prism}(G)$ is a crystal, because they show that, even stronger,
          $$
          \mathrm{Ext}^1_{(R)_\prism/(B,J)}(u^{-1}G,\mathcal{O}_\prism)
          $$
          commutes with base change in $(B,J)$. This finishes the proof of the first sentence of the proposition. 
          
          The second sentence is an immediate corollary of the first one, together with \Cref{sec:abstr-divid-prism-proposition-finite-locally-free} and \Cref{sec:divid-prism-dieud-lemma-divided-dieudonne-module-via-local-ext-on-prismatic-site}.
	      \end{proof}
      
We can now summarize our discussion and prove the main result of this section. We need a last lemma.

 \begin{lemma}
      \label{sec:fully-faithfulness-p-lifting-p-divisible-groups-along-henselian-pairs} Let $(C,J)$ be an henselian pair and let $\overline{G}$ be a $p$-divisible group over $C/J$. Then there exists a $p$-divisible group $G$ over $C$ such that
      $$
G\otimes_{C} C/J\cong \overline{G}.
      $$
    \end{lemma}
    \begin{proof}
      Set $h$ as the height of $\overline{G}$.
      Let $\mathrm{BT}_n^h$ be the Artin stack (over $\Spec(\Z)$) of $n$-truncated Barsotti-Tate groups of height $h$. Then for any $n\geq 1$ the morphism
      $$
      \mathrm{BT}_n^h\to \mathrm{BT}^h_{n-1}
      $$
      is a smooth morphism between smooth Artin stacks (cf.\ \cite[Section 2]{lau_displays_and_formal_p_divisible_groups} resp.\ \cite[Thm 4.4]{illusie_deformations_de_groupes_de_barsotti_tate}). By \cite[Theorem, page 568]{elkik_solution_d_equation_a_coefficients_dans_un_anneau_henselian} (which extends to the non-noetherian case by passing to the limit) any section $D\to C/J$ of some smooth $C$-algebra $D$ extends to a section $D\to C$.
      These statements imply that inductively, we can lift $\overline{G}[p^n]$ to a truncated $p$-divisible group $H_n$ over $C$. Then finally
      $$
      G:=\varinjlim\limits_n H_n
      $$
      yields the desired lift over $\overline{G}$.
    \end{proof}

\red{
\begin{theorem} \label{sec:divid-prism-dieud-modul-proposition-window}
	Let $R$ be a quasi-syntomic ring, and let $G$ be a $p$-divisible group over $R$. The pair $(\mathcal{M}_{\prism}(G),\varphi_{\mathcal{M}_{\prism}(G)})$ of \Cref{sec:divid-prism-dieud-definition-divided-prismatic-dieudonne-crystal-of-p-divisible-groups} is an admissible prismatic Dieudonn\'e crystal over $R$.
\end{theorem}
\begin{proof}
Let $G$ be a $p$-divisible group over $R$. By \Cref{sec:divid-prism-dieud-proposition-ext-crystal}, we already know that $\mathcal{M}_{\prism}(G)$ is a finite locally free $\mathcal{O}^{\rm pris}$-module, endowed with the semilinear endomorphism $\varphi_{\mathcal{M}_G}$. We need to see that it this gives an admissible prismatic Dieudonn\'e crystal over $R$. The construction being functorial in $R$, it suffices by \Cref{sec:astr-divid-prism-proposition-descent} to deal with the case where $R$ is quasi-regular semiperfectoid. Choose a perfectoid ring $S$ mapping surjectively onto $R$; by \Cref{sec:essent-surj-lemma-existence-of-perfectoid-covers-which-are-henselian-along-ideal}, we can assume that $S$ is henselian along $\ker(S \to R)$. \Cref{sec:fully-faithfulness-p-lifting-p-divisible-groups-along-henselian-pairs} (applied to $(C,J)=(S,\ker(S\to R))$ and $\overline{G}=G$) shows that $G$ is the base change of a $p$-divisible group $H$ over $S$. Hence, $(\mathcal{M}_{\prism}(G),\varphi_{\mathcal{M}_{\prism}(G)})$ is the base change of $(\mathcal{M}_{\prism}(H),\varphi_{\mathcal{M}_{\prism}(H)})$, which we know to be an admissible Dieudonn\'e crystal since $S$ is perfectoid, cf. \Cref{sec:prism-dieud-theory-comparison-with-bkf-modules-perfectoid-case}. 
\end{proof}
}

We now state two useful properties of the prismatic Dieudonn\'e functor : its exactness and its compatibility with Cartier duality. 

\begin{proposition} \label{sec:divid-prism-dieud-exactness}
Let $R$ be a quasi-syntomic ring. The functor
\[ \mathcal{M}_{\prism} : \BT(R) \to \mathrm{DM}(R), \quad  G \mapsto \mathcal{M}_{\prism}(G) \]
is exact.
\end{proposition}
\begin{proof}
Let 
\[ 0 \to G' \to G \to G'' \to 0 \] 
be a short exact sequence of $p$-divisible groups over $R$, which we see as an exact sequence of abelian sheaves on $(R)_{\rm qsyn}$. Applying $R\mathcal{H}om_{(R)_{\rm qsyn}}(-, \mathcal{O}^{\rm pris})$ to it, we get a long exact sequence :
\[  \mathcal{H}om_{(R)_{\rm qsyn}}(G',\mathcal{O}^{\rm pris}) \to \mathcal{M}_{\prism}(G'') \to \mathcal{M}_{\prism}(G) \to \mathcal{M}_{\prism}(G') \to  \mathcal{E}xt_{(R)_{\rm qsyn}}^2(G'',\mathcal{O}^{\rm pris}).  \]
The first term vanishes as $G^\prime$ is $p$-divisible and $\mathcal{O}^{\rm pris}$ derived $p$-complete. Let us prove surjectivity of $\mathcal{M}_\prism(G)\to \mathcal{M}_\prism(G^\prime)$. For $n\geq 1$ consider the exact sequences
  $$
  0\to G^\prime[p^n]\to G[p^n]\to H_n\to 0.
  $$
  Then $G^{\prime\prime}=\varinjlim\limits_n H_n$ with injective transition maps $H_n\to H_{n+1}$ (as $G[p^{n}]\subseteq G^\prime=G^\prime[p^n]$ for all $n\geq 1$). As in the proof of \Cref{sec:divid-prism-dieud-proposition-ext-crystal} we can conclude that
  $$
  \mathcal{M}_\prism(G[p^n])\to \mathcal{M}_\prism(G^\prime[p^n]),\ \mathcal{M}_\prism(H_{n+1})\to \mathcal{M}_\prism(H_n) 
  $$
  are surjective. Passing to the limit of the exact sequences
  $$
  \mathcal{M}_{\prism}(H_n) \to \mathcal{M}_{\prism}(G[p^n]) \to \mathcal{M}_{\prism}(G'[p^n]) \to 0
  $$
  implies therefore that
  $$
  \mathcal{M}_\prism(G)\to \mathcal{M}_\prism(G^\prime)
  $$
  is surjective, as desired.
\end{proof}

Let $R$ be a quasi-syntomic ring and let $G$ be a $p$-divisible group over $R$ with Cartier dual $\check{G}$.
Passing to the limit for the Cartier duality on finite flat group schemes yields isomorphisms
$$
T_p(\check{G})\cong \mathcal{H}om_R(T_pG,T_p\mu_{p^\infty})\cong \mathcal{H}om_R(G,\mu_{p^\infty})
$$
of sheaves on $(R)_{\mathrm{qsyn}}$.
We first construct a canonical morphism
$$
\Phi_{G}\colon \mathcal{M}_{\prism}(G)^\vee\otimes_{\mathcal{O}^{\pris}} \mathcal{M}_{\prism}(\mu_{p^\infty})\to \mathcal{M}_{\prism}(\check{G}),
$$
where $\mathcal{M}_{\prism}(G)^\vee$ denotes the $\mathcal{O}^\pris$-linear dual of $\mathcal{M}_\prism(G)$.
Recall that
$$
\mathcal{M}_{\prism}(\check{G})\cong \mathcal{H}om(T_p\check{G},\mathcal{O}^\pris)
$$
by \Cref{sec:divid-prism-dieud-lemma-prismatic-dieudonne-module-via-hom}.
Thus we can define $\Phi_G$ by setting
$$
\Phi_G(\delta\otimes l)(\alpha):=(\delta\circ \mathcal{M}_{\prism}(\alpha))(l)\in \mathcal{O}^\pris
$$
where
$$
\delta\in \mathcal{M}_\prism(G)^\vee, l\in \mathcal{M}_\prism(\mu_{p^{\infty}}), \alpha\in \mathcal{H}om(G,\mu_{p^\infty})\cong T_p\check{G}.
$$
Clearly, the morphism $\Phi_G$ is natural in $G$ and commutes with base change in $R$.

\begin{proposition}
\label{sec:divid-prism-dieud-compatibility-cartier-duality}
Let $R$ be a quasi-syntomic ring. For every $p$-divisible group $G$ over $R$, the map
$$
\Phi_G\colon \mathcal{M}_{\prism}(G)^\vee\otimes_{\mathcal{O}^\pris}\mathcal{M}_\prism(\mu_{p^\infty})\to \mathcal{M}_{\prism}(\check{G})
$$ constructed above is an isomorphism.
\end{proposition}

\red{\Cref{sec:prism-dieud-modul-proposition-prismatic-dieudonne-module-of-mu-p-infty} implies, via quasi-syntomic descent, that $\mathcal{M}_{\prism}(G)^\vee\otimes_{\mathcal{O}^\pris}\mathcal{M}_\prism(\mu_{p^\infty})$ is naturally a prismatic Dieudonn\'e crystal when equipped with the Frobenius
  \[
    1\otimes \delta\otimes l \in \mathcal{O}^\pris\otimes_{\varphi, \mathcal{O}^\pris} (\mathcal{M}_\prism(G)^\vee\otimes_{\mathcal{O}^\pris}\mathcal{M}_\prism(\mu_{p^\infty})) \mapsto \varphi^{\ast}\delta\circ \varphi_{\mathcal{M}_{\prism}(G)}^{-1}\otimes \varphi_{\mu_{p^\infty}}(1\otimes l)
  \]
(using the identification $\varphi^{\ast}\mathcal{O}^\pris\cong \mathcal{O}^\pris$ and the inverse $\varphi^{-1}_{\mathcal{M}_\prism(G)}\colon \mathcal{M}_\prism(G)\to 1/\mathcal{I}^\pris\varphi^\ast \mathcal{M}_\prism(G)$ of the linearized Frobenius on $\mathcal{M}_\prism(G)$). With this choice of Frobenius one checks that $\Phi_G$ is a morphism of prismatic Dieudonn\'e crystals, i.e., compatible with the Frobenius.}

\begin{proof}
  Both sides are locally free $\mathcal{O}^\pris$-modules of the same rank (cf.\ \Cref{sec:divid-prism-dieud-proposition-ext-crystal}). Hence it suffices to see that $\Phi_G$ is surjective, which can be checked after base change $R\to k$ to perfect fields $k$ of characteristic $p$. Thus, assume that $R=k$. By \Cref{sec:prism-dieud-theory-theorem-comparison-with-crystalline-dieudonne-functor} the prismatic Dieudonn\'e functor over $k$ is isomorphic to the crystalline one. Let
  $$
  \Phi_G^{\mathrm{cl}}\colon \mathcal{M}_{\prism}(G)^\vee\otimes_{\mathcal{O}^\pris}\mathcal{M}_\prism(\mu_{p^\infty})\to \mathcal{M}_{\prism}(\check{G})
  $$
  be the natural isomorphism coming from classical duality for the crystalline Dieudonn\'e functor over perfect fields (cf. for example\ \cite[Proposition III 5.1.iii)]{fontaine_groupes_p_divisibles_sur_les_corps_locaux}).
Let
$$
\Psi_{(-)}\colon \mathcal{M}_{\prism}(-)^\vee\otimes_{\mathcal{O}^\pris}\mathcal{M}_\prism(\mu_{p^\infty})\to \mathcal{M}_{\prism}((-)^\ast)
$$
be any natural transformation (of functors on $p$-divisible groups over quasi-syntomic rings over $k$). Then for any morphism $\gamma\colon G\to H$ of $p$-divisible groups, there is an equality
\begin{equation}
  \label{eq:1}
\Psi_G(\delta\otimes l)(\alpha\circ \gamma)=\Psi_H(\delta\circ \mathcal{M}_{\prism}(\gamma)\otimes l)(\alpha)  
\end{equation}
where $\delta\in \mathcal{M}_\prism(G), l\in \mathcal{M}_{\prism}(\mu_{p^\infty}), \alpha\in \mathcal{H}om(H,\mu_{p^\infty})$.
We want to show that $\Phi_G=u\Phi^{\mathrm{cl}}_G$ for all $p$-divisible groups $G$ and some unit $u\in \mathcal{O}^{\pris}$ (independent of $G$). Thus pick $\delta\in \mathcal{M}_\prism(G)^\vee, l\in \mathcal{M}_\prism(\mu_{p^\infty})$ and $\alpha\in \mathcal{H}om(G,\mu_{p^\infty})$. Applying (\Cref{eq:1}) to $\gamma=\alpha\colon G\to \mu_{p^\infty}$ implies
$$
\Psi_G(\delta\otimes l)(\alpha)=\Psi_{\mu_{p^\infty}}(\delta\circ \mathcal{M}_{\prism}(\alpha)\otimes l)(\mathrm{Id}_{\mu_{p^\infty}})
$$
for any natural transformation $\Psi_{(-)}$ as above. In particular, $\Psi$ (and thus $\Phi_{(-)}$ and $\Phi^{\mathrm{cl}}_{(-)}$ as examples) are determined by their behavior on $G=\mu_{p^\infty}$. For $\mu_{p^\infty}$ both induce an isomorphism
$$
\mathcal{M}_{\prism}(\mu_{p^\infty})^\vee\otimes_{\mathcal{O}^\pris} \mathcal{M}_{\prism}(\mu_{p^\infty})\cong \mathcal{H}om(T_p(\mu_{p^\infty}),\mathcal{O}^\pris)\cong \mathcal{O}^\pris.
$$
Namely, $\Phi_{\mu_{p^\infty}}$ is given by the natural evaluation, which is an isomorphism as $\mathcal{M}_{\prism}(\mu_{p^\infty})$ is free over rank $1$ (by the crystalline comparison, cf.\ \Cref{sec:prism-dieud-theory-theorem-comparison-with-crystalline-dieudonne-functor}). That $\Phi_{\mu_{p^\infty}}^{\mathrm{cl}}$ is an isomorphism follows from classical Dieudonn\'e theory (cf.\ \cite[Proposition 5.1.iii)]{fontaine_groupes_p_divisibles_sur_les_corps_locaux}). Hence, $\Phi_{\mu_{p^\infty}}$ and $\Phi^{\mathrm{cl}}_{\mu_{p^\infty}}$ differ by some unit $u\in \mathcal{O}^\pris$\footnote{Of course, one expects $u=\pm 1$, but as this finer statement is not necessary for us, we avoided the calculation verifying this.}. This implies $\Phi_G=u\Phi^{\mathrm{cl}}_G$ for all $G$ by naturality. By \cite[Proposition 5.1.iii)]{fontaine_groupes_p_divisibles_sur_les_corps_locaux} we can conclude.
\end{proof}

The main result of this text is the following theorem, whose proof will spread out over the next sections.

\begin{theorem} 
\label{sec:divid-prism-dieud-modul-theorem-main-theorem-of-the-paper}
\purple{Let $R$ be a quasi-syntomic ring.} The \red{prismatic Dieudonn\'e functor :
\[ \mathcal{M}_{\prism} : \mathrm{BT}(R) \to \DM^{\rm adm}(R) \]
is an antiequivalence between the category of $p$-divisible groups over $R$ and the category of admissible prismatic Dieudonn\'e crystals over $R$.}
\end{theorem}

\begin{proof}
By \Cref{sec:quasi-syn-rings-proposition-quasi-regular semiperfectoid-basis} and the fact that both $\mathrm{BT}$ and \red{$\DM^{\rm adm}$} are stacks on $\mathrm{QSyn}$ for the quasi-syntomic topology (see   \Cref{sec:fully-faithf-prism-proposition-bt-is-a-stack} and \Cref{sec:astr-divid-prism-proposition-descent}), we can assume that moreover $R$ is quasi-regular semiperfectoid. Then the theorem is a consequence of \Cref{sec:fully-faithf-prism-fully-faithfulness-for-p-torsion-free-quasi-regular semiperfectoid-new} and \Cref{sec:essent-surj-main-theorem-equivalence-of-divided-prismatic-dieudonne-functor}, to be proved below.
\end{proof}

\subsection{The prismatic Dieudonn\'e modules of $\Q_p/\Z_p$ and $\mu_{p^\infty}$}
\label{sec:prism-dieud-modul}

In this subsection, we calculate the prismatic Dieudonn\'e crystals of $\Q_p/\Z_p$ and $\mu_{p^\infty}$ to explicitly work out some examples for prismatic Dieudonn\'e crystals. We deduce as well a description for all \'etale and multiplicative $p$-divisible groups.
For the analogous results for the crystalline Dieudonn\'e functor see \cite[2.2]{berthelot_messing_theorie_de_dieudonne_cristalline_III}.
Let us fix a quasi-syntomic ring $R$.
Recall that for a $p$-divisible group $G$ over $R$ the prismatic Dieudonn\'e crystal $\mathcal{M}_\prism(G)$ is defined (cf.\ \Cref{sec:divid-prism-dieud-definition-divided-prismatic-dieudonne-crystal-of-p-divisible-groups}) as the sheaf

$$
\mathcal{M}_{\prism}(G):=\mathcal{E}xt_{(R)_{\mathrm{qsyn}}}^1(G,\mathcal{O}^{\rm pris}) = v_*\mathcal{E}xt_{(R)_{\prism}}^{1}(u^{-1}(G),\mathcal{O}_\prism)
$$
on the absolute prismatic site $(R)_\prism$ of $R$ and that
$$
\mathcal{M}_\prism(G)\cong \mathcal{H}om_{(R)_{\mathrm{qsyn}}}(T_pG, \mathcal{O}^{\rm pris})= v_*\mathcal{H}om_{(R)_{\prism}}(u^{-1}(T_pG),\mathcal{O}_\prism),
$$
by \Cref{sec:divid-prism-dieud-lemma-prismatic-dieudonne-module-via-hom}.

\begin{lemma}
  \label{sec:prism-dieud-modul-lemma-prismatic-dieudonne-module-for-etale-p-divisible-groups} The $\mathcal{O}^{\rm pris}$-module $\mathcal{M}_\prism(\Q_p/\Z_p)$ is \red{freely generated by the isomorphism class of the extension of $\mathcal{O}^{\rm pris}$ by $\Q_p/\Z_p$ obtained as} the push-out of the short exact sequence
  $$
  0\to \Z_p\to \Q_p\to \Q_p/\Z_p\to 0
  $$
  on $(R)_{\mathrm{qsyn}}$ along the canonical morphism $\Z_p\to \mathcal{O}^{\rm pris}$. More generally,
  $$
  \mathcal{M}_\prism(G)\cong \mathcal{H}om_{\red{(R)_{\rm qsyn}}}(T_p(G),\Z_p)\otimes_{\Z_p} \mathcal{O}^{\rm pris}.
  $$
  if $G$ is an \'etale $p$-divisible group.
\end{lemma}
\begin{proof}
  This follows directly from the isomorphism
  $$
  \mathcal{M}_\prism(G)\cong \mathcal{H}om_{(R)_{\mathrm{qsyn}}}(T_pG,\mathcal{O}^{\rm pris})
  $$
  and the fact that for an \'etale $p$-divisible group $T_pG$ is a local system of finite free $\Z_p$-modules on $(R)_{\mathrm{qsyn}}$\footnote{Here, we did some abuse of notation and denoted by $\Z_p$ the sheaf $S\mapsto \mathrm{Hom}_{\mathrm{cts}}(\pi_0(S),\Z_p)$ on $(R)_{\mathrm{qsyn}}$, which is usually called $\underline{\Z_p}$.}.
\end{proof}

\red{Let us now describe the prismatic Dieudonn\'e crystal $\mathcal{M}_{\prism}(\mu_{p^{\infty}})$ of $\mu_{p^\infty}$ on $(R)_{\rm qsyn}$. 

\begin{definition}
\label{definition-negative-breuil-kisin-tate-twist}
Let $\mathcal{O}_{\prism}\{-1 \}$ be the sheaf
$$
\mathcal{O}_{\prism}\{-1 \} = \mathcal{H}om_{(\Z_p)_{\prism}}(u^{-1}(\Z_p(1)),\mathcal{O}_\prism)
$$
on the absolute prismatic site of $\Z_p$, with $\Z_p(1):=T_p\mu_{p^\infty}$. 
\end{definition}

Note that, if $\widehat{\Gm}$ denotes, the $p$-adic completion of the multiplicative group scheme $\Gm$, we also have
$$
\mathcal{O}_{\prism} \{-1 \} \cong \mathcal{E}xt_{(\Z_p)_{\prism}}^{1}(u^{-1} \mu_{p^{\infty}},\mathcal{O}_{\prism}) \cong \mathcal{E}xt_{(\Z_p)_{\prism}}^{1}(\widehat{\Gm},\mathcal{O}_{\prism})
$$
as $\widehat{\Gm}/\mu_{p^\infty}$ is uniquely $p$-divisible and $\mathcal{O}_{\prism}$ $p$-complete. Also, as recalled above, we have a natural isomorphism
$$
\mathcal{M}_{\prism}(\mu_{p^{\infty}}) \cong v_\ast \mathcal{O}_{\prism} \{-1 \}_{|_{(R)_{\prism}}}.
$$ 
}

We can describe the sheaf $\mathcal{O}_{\prism} \{-1\}$ in restriction to prisms $(B,J)$ which live over the ``cyclotomic'' base prism
$$
(A,I):=(\Z_p[[q-1]],([p]_q))
$$
from \Cref{sec:q-logarithm}. We point out that Mondal \cite{mondal2022computation} was able to recently get rid of this restriction, using Bhatt-Lurie's syntomic Chern classes \cite{bhatt2022absolute}.

The reason is that for such prisms we can use the $q$-logarithm from \Cref{sec:q-logarithm}
$$
\log_q\colon u^{-1}(\Z_p(1))\to \mathcal{O}_\prism
$$
which defines a canonical element, which we call $\ell_q \in \mathcal{O}_{\prism}\{-1\}(A,I)$.

\begin{proposition}
  \label{sec:prism-dieud-modul-proposition-prismatic-dieudonne-module-of-mu-p-infty} 
  The $\mathcal{O}_\prism$-linear map
  $$
  \mathcal{O}_{\prism} \to \mathcal{O}_{\prism}\{-1\}, 
    $$
  sending $1$ to $\ell_q$, of sheaves on the category of all prisms living over $(A,I)=(\Z_p[[q-1]],([p]_q))$, is an isomorphism. Moreover, the Frobenius on $\mathcal{O}_{\prism}\{-1\}$ sends $\ell_q$ to $[p]_q\ell_q$. 
\end{proposition}
\begin{proof}
Let $(B,J)$ be a prism over $(A,I)$. It suffices to show that the morphism
  $$
  B \to \mathrm{Ext}^1(u^{-1}(\widehat{\Gm})_{|(B,J)},\mathcal{O}_\prism),
  $$
  (where we mean $\mathrm{Ext}^1$ in the category of abelian sheaves on the site of prisms over $(B,J)$) given by the $q$-logarithm is an isomorphism. By \Cref{sec:divid-prism-dieud-proposition-ext-crystal} the formation of this map is compatible with base change in $(B,J)$. \red{From the proof of loc.\ cit.\ we also know that $\mathrm{Ext}^1(u^{-1}(\widehat{\Gm})_{|(B,J)},\mathcal{O}_\prism)$ is a finite, locally free $B$-module of rank $1$. Therefore, it suffices to show surjectivity.} To show surjectivity one may pass to the case that $(B,J)=(W(k),(p))$ for $k$ an algebraically closed field of characteristic $p$. Then the comparison with the crystalline Dieudonn\'e crystal (cf.\ \Cref{sec:prism-dieud-theory-theorem-comparison-with-crystalline-dieudonne-functor}) reduces to an analogous statement for the usual logarithm as for $q=1$ the $q$-logarithm becomes the logarithm. Let $R$ be a general ring of characteristic $p$ and let $R^\prime\to R$ be a surjection of schemes with a PD-structure $\{\gamma_{n}\}_{n\geq 0}$ on $K:=\ker(R^\prime\to R)$ and assume $p$ nilpotent in $R^\prime$. Then there is the canonical morphism
  $$
  \mathrm{log}\colon \Z_p(1)(R)\to R^\prime,\ x\mapsto \mathrm{log}([x])
  $$
  where $[-]\colon \lim\limits_{x\mapsto x^p} R\to R^\prime$ is the
  Teichm\"uller lift and $\log$ the crystalline logarithm
  $$
  \log\colon 1+K\to R^\prime,\ y\mapsto
  \sum\limits_{n=1}^\infty(-1)^{n-1}(n-1)! \gamma_n(y-1)
  $$
  (which makes sense as $[x]\in 1+K$). But it is known that the logarithm generates the crystalline Dieudonn\'e crystal of $\mu_{p^\infty}$ (cf.\ \cite[2.2.3.Corollaire]{berthelot_messing_theorie_de_dieudonne_cristalline_III}).
  Finally the action of Frobenius on $\ell_q$ can be calculated using \Cref{sec:q-logarithm-convergence-of-q-logarithm}:
  $$
  \varphi_{\mathcal{H}om(u^{-1}(\Z_p(1)),\mathcal{O}_\prism)}(\ell_q)(x)=\frac{q^p-1}{\log(q)}\log(x^p)=\frac{q^p-1}{q-1}\ell_q(x)=[p]_q\ell_q(x)
  $$
  for $x\in \Z_p(1)$.

\end{proof}

\begin{remark}
  \label{remark-on-identifying-q-log-with-logarithm-in-characteristic-p}
Note that, when $pR=0$, the identification between the prismatic and crystalline Dieudonn\'e modules from \Cref{sec:prism-dieud-theory-theorem-comparison-with-crystalline-dieudonne-functor} is linear \textit{over the isomorphism $\prism_R\cong A_{\rm crys}(R)$} from \Cref{sec:prism-cohom-quasi-lemma-for-qr-semiperfect-prism-isomorphic-to-acrys}. This explains why the map $x\mapsto \log_q([x^{1/p}]_{\tilde{\theta}})$ is sent to $x\mapsto \log([x])$ (and not something like $x\mapsto \log([x^{1/p}])$, which would not make sense as $[x^{1/p}]-1$ need not have divided powers), cf.\ the remark after \Cref{sec:prism-cohom-quasi-lemma-for-qr-semiperfect-prism-isomorphic-to-acrys}.
\end{remark}

Assume now that $R$ is an $A/I=\Z[\zeta_p]$-algebra.

\begin{corollary}
  \label{sec:prism-dieud-modul-corollary-description-of-prismatic-dieudonne-module-for-multiplicatives}
  Let $G$ be a multiplicative $p$-divisible group over $R$. Then there is a canonical isomorphism
  $$
  u^{-1}(\mathcal{H}om(G,\mu_{p^\infty})) \otimes_{\Z_p} \mathcal{O}_{\prism} \cong \mathcal{E}xt_{(R)_{\prism}}^1(u^{-1}G,\mathcal{O}_{\prism})_{|(R/A)_\prism}
  $$
  induced by sending $f\colon G\to \mu_{p^\infty}$ to the evaluation of the morphism induced by $f$:
 \[ \mathcal{E}xt_{(R)_{\prism}}^1(u^{-1}\mu_{p^{\infty}},\mathcal{O}_{\prism})_{|(R/A)_\prism} \to \mathcal{E}xt_{(R)_{\prism}}^1(u^{-1}G,\mathcal{O}_{\prism})_{|(R/A)_\prism} \]
 on $\ell_q$.
\end{corollary}
\begin{proof}
The morphism (and the claim that it is an isomorphism) commutes with \'etale localisation on $R$. In particular, we may assume that $G\cong \mu_{p^\infty}^d$. Then the claim follows from \Cref{sec:prism-dieud-modul-proposition-prismatic-dieudonne-module-of-mu-p-infty} and additivity of the right hand side.  
\end{proof}

As a corollary of these computations, we can concretely describe the action of the prismatic Dieudonn\'e functor on morphisms $\Q_p/\Z_p\to \mu_{p^\infty}$.
Set
$$
\Z_p^{\mathrm{cycl}}:=(\varinjlim\limits_n\Z_p[\zeta_{p^n}])^{\wedge}_p.
$$
As usual we get the elements $\varepsilon=(1,\zeta_p,\ldots)$, $q:=[\varepsilon]\in A_\inf(\Z_p^{\mathrm{cycl}})$ and $\tilxi:=\frac{q^p-1}{q-1}$.

\begin{lemma}
  \label{sec:prism-dieud-modul-lemma-identifiaction-of-prismatic-dieudonne-functor-on-morphisms} Let $R$ be a quasi-regular semiperfectoid ring over $\Z_p^{\mathrm{cycl}}$. Then the morphism
  $$
  \Z_p(1)(R)\cong \mathrm{Hom}_R(\Q_p/\Z_p,\mu_{p^\infty})\xrightarrow{M_{\prism}(-)} \mathrm{Hom}_{\mathrm{DM}(R)}(M_{\prism}(\mu_{p^\infty})),M_{\prism}(\Q_p/\Z_p))\cong \prism_{R}^{\varphi=\tilxi}
  $$
  is given the map which sends $x\in \Z_p(1)(R)$ to $\log_q([x^{1/p}]_{\tilde{\theta}})\in \prism_R^{\varphi=\tilxi}$.
\end{lemma}
\begin{proof}
  First note, that
  $$
  \mathrm{Hom}_{\mathrm{DM}(R)}(M_{\prism}(\mu_{p^\infty})),M_{\prism}(\Q_p/\Z_p))\cong \prism_{R}^{\varphi=\tilxi}
  $$
  by evaluating a homomorphism $M_{\prism}(\mu_{p^\infty})\to M_{\prism}(\Q_p/\Z_p)\cong \prism_R$ on $\ell_q$. The identification of $M_{\prism}(-)$ on a homomorphism $f\colon \Q_p/\Z_p\to \mu_{p^\infty}$ follows easily from the natural isomorphism
  $$
  M_{\prism}(G)\cong \mathrm{Hom}_{(R)_{\prism}}(u^{-1}(T_p(G)),\mathcal{O}_\prism)
  $$
  for a $p$-divisible group $G$ over $R$ and \Cref{sec:prism-dieud-modul-proposition-prismatic-dieudonne-module-of-mu-p-infty}, \Cref{sec:prism-dieud-modul-lemma-prismatic-dieudonne-module-for-etale-p-divisible-groups}.
\end{proof}

\begin{remark}
\label{sec:prism-dieud-modul-theorem-fully-faithfulness-for-morphisms-from-qp-tomu-p}
This description together with \cite[Theorem 7.5.6]{bhatt2022absolute} imply fully faithfulness of the prismatic Dieudonn\'e functor in the special case of morphisms from $\Q_p/\Z_p$ to $\mu_{p^{\infty}}$. We will give in the next section a proof of fully faithfulness, still relying on the same input from \cite{bhatt2022absolute}.
\end{remark}

\subsection{Fully faithfulness}
\label{sec:fully-faithful-new}
The main result of this subsection is the following. 

\begin{theorem}
  \label{sec:fully-faithf-prism-fully-faithfulness-for-p-torsion-free-quasi-regular semiperfectoid-new}
  \purple{If $R$ is a quasi-regular semiperfectoid ring, the prismatic Dieudonn\'e functor over $R$ is fully faithful for $p$-divisible groups.}
\end{theorem}

The proof we offer was kindly suggested to us by Akhil Mathew. Recall that the prismatic Dieudonn\'e functor is given, according to \Cref{sec:divid-prism-dieud-lemma-prismatic-dieudonne-module-via-hom}, by the formula
 \begin{equation}
   \label{eq:3}
\mathcal{M}_{\prism}(G)= \mathcal{H}om_{(R)_{\rm qsyn}}(T_pG, \mathcal{O}^{\rm pris})   
 \end{equation}
for any $p$-divisible group $G$ over the quasi-regular semiperfectoid ring $R$. \blue{We also set
$$
\mathcal{N}^{\geq 1} \mathcal{M}_\prism(G) := \mathcal{H}om_{(R)_{\rm qsyn}}(T_pG, \mathcal{N}^{\geq 1}\mathcal{O}^{\rm pris}).
$$   
}

From now on, we fix a quasi-regular semiperfectoid ring $R$, and a generator $\tilxi$ of the prismatic ideal in $\prism_R$. For simplicity we assume that $R$ lives over the cyclotomic prism and that $\tilxi=[p]_q$, cf.\ \ref{sec:prism-dieud-modul-proposition-prismatic-dieudonne-module-of-mu-p-infty}. By descent this assumption is harmless.

\begin{proposition}
\label{identification_from_bms2}
If $G$ is a $p$-divisible group over $R$, there is a natural (in $R$ and $G$) identification of quasi-syntomic sheaves
$$
T_p \check{G} \cong \ker(\mathcal{N}^{\geq 1} \mathcal{M}_\prism(G)  \overset{\varphi/\tilxi-1} \longrightarrow \mathcal{M}_\prism(G)).
$$
\end{proposition}
\begin{proof}
We have, cf. \cite[Theorem 7.5.6]{bhatt2022absolute}\footnote{See also \cite[Proposition 7.17]{bhatt_morrow_scholze_topological_hochschild_homology} for a proof using algebraic $K$-theory.}, an isomorphism of quasi-syntomic sheaves
$$
T_p \mathbb{G}_m \cong \ker(\mathcal{N}^{\geq 1} \mathcal{O}^{\rm pris} \overset{\varphi/\tilxi-1} \longrightarrow \mathcal{O}^{\rm pris}).
$$
To conclude, it suffices to apply the functor $\mathcal{H}om_{(R)_{\rm qsyn}}(T_p G,-)$ to both sides and to note that $T_p \check{G} = \mathcal{H}om_{(R)_{\rm qsyn}}(T_p G,T_p \mathbb{G}_m)$.
\end{proof}

Now we start the proof of \Cref{sec:fully-faithf-prism-fully-faithfulness-for-p-torsion-free-quasi-regular semiperfectoid-new}. Let us denote by $\mathrm{Sh}_R$ the category of abelian sheaves on $(R)_{\rm qsyn}$ (so that $\mathrm{Hom}_{\mathrm{Sh}_R}(-,-)=\mathrm{Hom}_{(R)_{\rm qsyn}}(-,-)$) and by $\mathcal{D}_R$ the category of $\mathcal{O}^{\rm pris}[F]$-modules, which contains as a full subcategory the category of (admissible) prismatic Dieudonn\'e crystals. The functor
$$
\mathcal{R}: \mathrm{Sh}_R^{\rm op} \to \mathcal{D}_R, ~~ \mathcal{F} \mapsto \mathcal{H}om_{(R)_{\rm qsyn}}(\mathcal{F}, \mathcal{O}^{\rm pris})
$$
admits the left adjoint 
$$
\mathcal{L}: \mathcal{D}_R \to \mathrm{Sh}_R^{\rm op}, ~~ \mathcal{M} \mapsto \mathcal{H}om_{\mathcal{O}^{\rm pris}[F]}(\mathcal{M},\mathcal{O}^{\rm pris}).
$$
Indeed, if $\mathcal{F}\in \mathrm{Sh}_R$ is any abelian sheaf and $\mathcal{M}\in \mathcal{D}_R$, then
  $$
  \mathrm{Hom}_{\mathcal{O}^\pris[F]}(\mathcal{M},\mathcal{H}om_{(R)_{\rm qsyn}}(\mathcal{F},\mathcal{O}^\pris))\cong \mathrm{Hom}_{(R)_{\rm qsyn}}(\mathcal{F},\mathcal{H}om_{\mathcal{D}_R}(\mathcal{M},\mathcal{O}^\pris)) 
  $$
  because both sides identify with bilinear maps $\omega\colon \mathcal{M}\times \mathcal{F}\to \mathcal{O}^\pris$, which are $\mathcal{O}^\pris[F]$-linear in the first component.

Note that by the above displayed formula (\ref{eq:3}), if $G$ is a $p$-divisible group over $R$, 
$$
\mathcal{M}_\prism(G)= \mathcal{R}(T_pG).
$$
Hence, to prove the theorem, we are reduced to proving the following proposition.

\begin{proposition}
\label{right_adjoint_R_ff}
The functor $\mathcal{R}$ is fully faithful on the subcategory of $\mathrm{Sh}_R^{\rm op}$ spanned by the Tate modules of $p$-divisible groups over $R$.
\end{proposition}
\begin{proof}
Given a sheaf $\mathcal{F} \in \mathrm{Sh}_R^{\rm op}$ which is the Tate module of a $p$-divisible group, we have a natural counit map in $\mathrm{Sh}_R^{\rm op}$
$$
\mathcal{L}\mathcal{R}\mathcal{F} \to \mathcal{F}
$$
and we will show that it is an isomorphism. Switching back from $\mathrm{Sh}_R^{\rm op}$ to $\mathrm{Sh}_R$, this counit is the biduality map
$$
\mathcal{F}(-) \to \mathcal{H}om_{\mathcal{O}^{\rm pris}[F]}(\mathcal{H}om_{(R)_{\rm qsyn}}(\mathcal{F}(-),\mathcal{O}^{\rm pris}),\mathcal{O}^{\rm pris}).
$$
The formation of this map is compatible with base change in the quasi-regular semiperfectoid ring $R$. We claim that this map is an isomorphism whenever $\mathcal{F}$ is the Tate module of a $p$-divisible group $G$ over $R$. Applying \Cref{identification_from_bms2} to $\check{G}$, we get a natural (in $\mathcal{F}$ and $R$) identification 
$$
\mathcal{F}=T_p G \cong \left(\mathcal{N}^{\geq 1} \mathcal{M}_\prism(\check{G})\right)^{\varphi_{\mathcal{M}_\prism(\check{G})}=\tilxi} = \left(\mathcal{M}_\prism(\check{G})\right)^{\varphi_{\mathcal{M}_\prism(\check{G})}=\tilxi}
$$
\blue{(for the last equality, note that if $f \in \mathcal{H}om_{(R)_{\rm qsyn}}(T_pG, \mathcal{O}^{\rm pris})$ satisfies $\varphi(f)=\tilxi f$, then for any section $s$ of $T_p G$, $f(s) \in  \mathcal{N}^{\geq 1}\mathcal{O}^{\rm pris}$, i.e. $f \in \mathcal{H}om_{(R)_{\rm qsyn}}(T_pG, \mathcal{N}^{\geq 1}\mathcal{O}^{\rm pris})$).} \Cref{sec:divid-prism-dieud-compatibility-cartier-duality} and the remark following it for the identification of Frobenius allow us to rewrite this as a natural (in $\mathcal{F}$ and $R$) identification
$$
\mathcal{F} \cong \mathcal{H}om_{\mathcal{O}^{\rm pris}[F]}(\mathcal{M}_\prism(G),\mathcal{O}^{\rm pris})= \mathcal{L}(\mathcal{M}_\prism(G)) = \mathcal{L} \mathcal{R} \mathcal{F}
$$
as we can identify $(\mathcal{M}_\prism(\mu_{p^\infty}), \varphi_{\mathcal{M}_\prism(\mu_{p^\infty})})=(\mathcal{O}^\pris, \tilxi \varphi_{\mathcal{O}^\pris})$ by \Cref{sec:prism-dieud-modul-proposition-prismatic-dieudonne-module-of-mu-p-infty}. However, this natural isomorphism may not a priori coincide with above counit map. But, composing the latter with the inverse of this isomorphism, we obtain a natural endomorphism of $\mathcal{F}$, i.e., an endomorphism of any $p$-divisible group $G$ over any quasi-regular semiperfectoid ring $R$, natural in $G$ and $R$. Any such endomorphism acts on the $p$-divisible group $\Q_p/\Z_p$ by multiplication by some scalar (in $\Z_p$), at least on each connected component of $R$. It also does act by multiplication by the same scalar (depending on a connected component of $\Spec(R)$) on any $p$-divisible group $G$: indeed, this can be checked on the Tate module and since $T_p(G)=\mathcal{H}om_{(R)_{\rm qsyn}}(\Q_p/\Z_p,G)$ this follows there by naturality. 

Hence, to conclude the proof of fully faithfulness, it suffices to show that these scalars are units. This can be checked for one specific $p$-divisible group $G$ and we can take $G=\Q_p/\Z_p$, for which the claim is immediate. Indeed, $\mathcal{F}=\Z_p$ in this case and $\mathcal{F}\to \mathcal{L}\mathcal{R}(\mathcal{F})\cong \mathcal{H}om_{\mathcal{O}^\pris[F]}(\mathcal{O}^\pris, \mathcal{O}^\pris)$ sends $1\in \Z_p$ to the identity of $\mathcal{O}^\pris$, which generates $\mathcal{H}om_{\mathcal{O}^\pris[F]}(\mathcal{O}^\pris, \mathcal{O}^\pris)$ by \cite[Remark 9.3]{bhatt_scholze_prisms_and_prismatic_cohomology}.
\end{proof}

\blue{
\begin{remark}
\label{cohomology-tate-module}
In fact, as was also pointed out by Akhil Mathew and the referee, the results used in this section can be strenghtened. Indeed, \cite[Theorem 7.5.6]{bhatt2022absolute} already quoted above even gives a short exact sequence
$$
0 \to T_p \mathbb{G}_m \to \mathcal{N}^{\geq 1} \mathcal{O}^{\rm pris} \overset{\varphi/\tilxi-1} \longrightarrow \mathcal{O}^{\rm pris} \to 0.
$$
Applying $R\mathcal{H}om_{(R)_{\rm qsyn}}(T_pG, -)$ to it, we get an exact sequence of sheaves
$$
0 \to T_p \check{G} \to \mathcal{N}^{\geq 1} \mathcal{M}_\prism(\check{G}) \overset{\varphi/\tilxi -1} \longrightarrow \mathcal{M}_\prism(\check{G}) \to \mathcal{E}xt_{(R)_{\rm qsyn}}^1(T_pG , T_p \mathbb{G}_m).
$$
But $\mathcal{E}xt_{(R)_{\rm qsyn}}^1(T_pG , T_p \mathbb{G}_m)=\mathcal{E}xt_{(R)_{\rm qsyn}}^1(G , \mu_{p^\infty})=0$ (cf. \cite[Theorem 1]{waterhouse1971principal} together with the fact that the set of splittings of such an extension is a torsor under $\check{G}$ which is syntomic). Hence we get a short exact sequence and, after taking cohomology, an isomorphism
$$
R\Gamma((R)_{\rm qsyn}, T_p\check{G}) \cong \mathrm{fib}\left(\mathcal{N}^{\geq 1} M_\prism(\check{G}) \overset{\varphi/\tilxi -1} \longrightarrow M_\prism(\check{G})\right).
$$
\end{remark}
}

\subsection{Essential surjectivity}
\label{sec:essent-surj-1}

Let $R$ be quasi-regular semiperfectoid and let as before
$$
M_\prism(-)\colon \mathrm{BT}(R)\to \red{\mathrm{DM}^{\rm adm}(R)},\ G\mapsto (M_\prism(G),\varphi_{M_\prism(G)})
$$
be the prismatic Dieudonn\'e functor with values in the category of \red{admissible prismatic Dieudonn\'e modules $\mathrm{DM}^{\rm adm}(R)$} (cf.\ \Cref{sec:divid-prism-dieud-definition-divided-prismatic-dieudonne-crystals} and \Cref{sec:divid-prism-dieud-modul-proposition-window}).

Let us fix a perfect prism $(A,I)$, a generator $\tilxi\in I$ and a surjection $\blue{\bar{A}=}A/I\twoheadrightarrow R$. Let $\xi:=\varphi^{-1}(\tilxi)$. \red{In this section, we will make repeated use of \Cref{sec:abstr-divid-prism-proposition-equivalence-divided-prismatic-dieudonne-modules-windows}, which tells us that admissible prismatic Dieudonn\'e modules over $R$ (or any other quasi-regular semiperfectoid ring living over $\blue{\bar{A}}$) are the same as windows over the frame $\underline{\prism}_{R,\rm Nyg}$ (associated to $\tilxi$).}

By \Cref{sec:essent-surj-lemma-existence-of-perfectoid-covers-which-are-henselian-along-ideal} we may assume that $\blue{\bar{A}}$ is henselian along $\ker(\blue{\bar{A}}\to R)$.

Let us first assume that $\mathrm{ker}(\blue{\bar{A}}\to R)$ is generated by some elements $a_j$, $j\in J$, that admit compatible systems $(a_j,a_j^{1/p},a_j^{1/p^2},\ldots)$ of $p^n$-roots. Define
$$
\blue{S:=\left(\blue{\bar{A}}\langle X_j^{1/p^\infty} |\ j\in J\rangle/(X_j)\right)^{\wedge_p}}
$$
and $S\to R,\ X_j^{1/p^n}\mapsto \overline{a_j^{1/p^n}}.$


\begin{lemma}
  \label{sec:essent-surj-lifting-divided-prismatic-dieudonne-modules-for-special-surjection-of-quasi-regular semiperfectoid}
  The base change functor \red{$\mathrm{DM}^{\rm adm}(S)\to \mathrm{DM}^{\rm adm}(R)$ on admissible} prismatic Dieudonn\'e modules is essentially surjective.
\end{lemma}
\begin{proof}
  \blue{Using \Cref{sec:abstr-divid-prism-proposition-existence-of-normal-decompositions},} it suffices to see that $\prism_S\to \prism_R$ is surjective and henselian along its kernel (cf.\ \Cref{sec:essent-surj-lemma-equivalence-of-varphi-modules-over-delta-rings}). The surjectivity follows from the Hodge-Tate comparison as $L_{S/\blue{\bar{A}}}\to L_{R/\blue{\bar{A}}}$ is surjective by our assumption that the $a_j,j\in J$, generate $\mathrm{ker}(\blue{\bar{A}}\to R)$.
  First note that the pair $(S,\ker(S\to R))$ is henselian because the $X_j^{1/p^n}$ are nilpotent in $S$ and we assumed that $\blue{\bar{A}}$ is henselian along $\mathrm{ker}(\blue{\bar{A}}\to R)$. \blue{By \Cref{sec:essent-surj-lemma-prism_r-henselian-along-ker-theta},} to show that $\prism_S$ is henselian along $K:=\ker(\prism_S\to \prism_R)$ it suffices to see $S\cong \prism_S/\ker(\theta_S)$ is henselian along $\overline{K}:=(K+\ker(\theta))/\ker(\theta)$ (cf.\ \cite[Tag 0DYD]{stacks_project})). But $\overline{K}\subseteq S$ is contained in $\ker(S\to R)$. Another application of \cite[Tag 0DYD]{stacks_project} therefore implies that $S$ is henselian along $\overline{K}$ because $(S,\ker(S\to R))$ is henselian. This finishes the proof.
\end{proof}

Note that the ring
$$
\blue{S=\left( \blue{\bar{A}}\langle X_j^{1/p^\infty} |\ j\in J\rangle/(X_j\ |\ j\in J) \right)^{\wedge_p}}
$$
admits a surjection from the perfectoid ring \red{
$$
\tilde{S}:=\blue{\bar{A}}[[X_j^{1/p^\infty} |\ j\in J]] := \left( \varinjlim\limits_{n, J' \subset J ~ \mathrm{finite}}\blue{\bar{A}}[[X_j^{1/p^n} |\ j\in J']]  \right)^{\wedge_p}
$$
}
by sending $X_j^{1/p^n}\mapsto X_j^{1/p^n}$.

\begin{lemma}
  \label{sec:essent-surj-lemma-essential-surjectivity-from-perfectoid-to-infinite-regular-semiperfectoid-ring}
  The natural functor
  $$
  \red{\mathrm{DM}^{\rm adm}(\tilde{S})\to \mathrm{DM}^{\rm adm}(S)}
  $$
  is essentially surjective.
\end{lemma}
\begin{proof}
 The ring $\tilde{S}$ is henselian along $(X_j\ |\ j\in J)$. \red{The prism $\prism_{\tilde{S}}$ is the $(p,I)$-adic completion of 
 $$
 \varinjlim\limits_{n, J' \subset J ~ \mathrm{finite}}A[[X_j^{1/p^n} |\ j\in J']] 
 $$
\blue{Call a $\delta$-pair $(B,K)$ over $(A,I)$ a \textit{good pair} if it satisfies the following conditions:
\begin{itemize}
\item $B$ is $(p,I)$-completely flat over $A$ and $K$ is $(p,I)$-complete.
\item There exists a universal map $(B,K)\to (C,IC)$ of $\delta$-pairs to a prism $(C,IC)$ over $(A,I)$. Moreover, $(C,IC)$ is flat over $(A,I)$, and its formation commutes with $(p,I)$-completely flat base change on $B$.
\end{itemize}
}For each $n\geq 1$ and $J' \subset J$ finite, the $\delta$-pair
 $$
 \left( A[[X_j^{1/p^n} |\ j\in J']]^{\wedge_{\blue{(p,I)}}},(I, X_j, j\in J)^{\wedge_{(p,I)}} \right)
 $$
 over $(A,I)$ is a \blue{good pair, by \cite[Proposition 3.13]{bhatt_scholze_prisms_and_prismatic_cohomology}}. Since good pairs are stable under filtered colimits in the category of all \blue{$\delta$-pairs $(B,K)$ over $(A,I)$ with $B
$ and $K$ $(p,I)$-complete}, we deduce that the pair 
$$
\left(\prism_{\tilde{S}}, (I,X_j,j\in J) \right) 
$$
is a good pair, too. Therefore, by definition of a good pair and \Cref{sec:prism-cohom-quasi-proposition-all-prisms-agree-for-quasi-regular semiperfectoid}, we have 
 }
  $$
  \prism_{S}\cong \prism_{\tilde{S}}\{\frac{X_j}{\tilxi}\ |\ j\in J\}^{\wedge_{(p,\tilxi)}}.
  $$
  Define
  $$
  \blue{B:=\prism_{\tilde{S}}/(X_j\ |\ j\in J)^{\wedge_{(p,\tilxi)}}}.
  $$
  Then $B$ is $p$-torsion free and $\tilxi$-torsion free and thus
  defines a prism. Moreover, canonically $S\cong B/\tilxi$. By the
  universal property of $\prism_S$ there exists therefore a canonical
  morphism
  $$
  \alpha\colon \prism_{S}\to B.
  $$
  Concretely, the morphism $\alpha$ sends $X_j\mapsto 0$. Using a
  variant of
  \Cref{sec:essent-surj-lemma-prism_r-henselian-along-ker-theta} we
  see that $\prism_S$ is henselian along
  $\mathrm{ker}(\alpha)$. By
    \Cref{sec:essent-surj-divided-varphi-topological-nilpotent-on-the-kernel},
    $\varphi(\ker(\alpha))\subseteq \tilxi\prism_S$ and
    $\varphi/\tilxi$ is topologically nilpotent on
    $\mathrm{ker}(\alpha)$. Thus by
  \Cref{sec:essent-surj-lemma-equivalence-on-window-categories} the
  categories of windows over $\prism_S$ and $B$ are
  equivalent. Therefore it suffices to see that windows over $B$ can
  be lifted to windows over $\prism_{\tilde{S}}$. After choosing a
  normal decomposition, this follows as the functor
  $$
\varphi-\mathrm{Mod}_{\prism_{\tilde{S}}}^{\red{\mathrm{unit}}}\to \varphi-\mathrm{Mod}_B^{\red{\mathrm{unit}}} 
$$
is essentially surjective, which is true as \red{$\prism_{\tilde{S}}$} is henselian
along the kernel of $\red{\prism_{\tilde{S}}}\twoheadrightarrow B$ (cf.\ the end of
the proof of
\Cref{sec:essent-surj-lemma-equivalence-of-varphi-modules-over-delta-rings}). This
finishes the proof.
\end{proof}

To finish the proof of
  \Cref{sec:essent-surj-lifting-divided-prismatic-dieudonne-modules-for-special-surjection-of-quasi-regular
    semiperfectoid} we have to prove the following lemmas.

  \begin{lemma}
    \label{sec:essent-surj-divided-varphi-topological-nilpotent-on-the-kernel}
    With the notations from the proof of
    \Cref{sec:essent-surj-lemma-essential-surjectivity-from-perfectoid-to-infinite-regular-semiperfectoid-ring} we get
    $\varphi(\ker(\alpha))\subseteq \tilxi \prism_S$ and
    $\varphi_1:=\varphi/\tilxi$ is topologically nilpotent on
    $\ker(\alpha)$.
  \end{lemma}
  \begin{proof}
    Set $K:=\ker(\alpha)$. Then $K$ is the closure in the
    $(p,\tilxi)$-adic topology of the $\prism_S$-submodule generated
    by $\delta^n(X_j/\tilxi)$ for $j\in J$ and $n\geq 0$. By \Cref{sec:concr-pres-prism-proposition-presentation-of-prismatic-envelope-for-rank-one-element} \blue{below} the module $K$ equals the closure of the ideal generated by
    $$
    z_{j,n}:=\frac{X_j^{p^n}}{\varphi^{n}(\tilxi)\varphi^{n-1}(\tilxi)^p\cdots
    \tilxi^{p^n}}
  $$
  for $j\in J$ and $n\geq 0$. Let us show that $\varphi(K)\subseteq \tilxi \prism_S$. Clearly,

  \begin{equation}
    \label{equation-how-varphi-acts-on-z_jn}
    \varphi(z_{j,n})=\tilxi^{p^{n+1}}z_{j,n+1}.
  \end{equation}

  As $\mathcal{N}^{\geq 1}\prism_S$ is closed in $\prism_S$ (being the kernel of the continuous surjection $\prism_S\to S$), we can conclude $K\subseteq \mathcal{N}^{\geq 1}\prism_S$.
  Next, let us check that $\varphi_1$ is topologically nilpotent on $K$. Fix $l\geq 1$. We claim that for every $m\geq 1$ such that $p^m> l$ and any $k\in K$ we have
  $$
  \varphi_1^m(k)\in \tilxi^{l}K.
  $$
  This implies as desired that $\varphi_1$ is topologically nilpotent on $K$. As $\tilxi^lK$ is closed and $\varphi_1^m$ continuous (for the $(p,\tilxi)$-adic topology on $K$) it is enough to assume that $k=z_{j,n}$ for some $j\in J, n\geq 1,$ because the $z_{j,n}$ generate a dense submodule in $K$\footnote{Dense for the $(p,\tilxi)$-adic topology.}. Using (\Cref{equation-how-varphi-acts-on-z_jn}) we can calculate
  $$
  \varphi_1^m(z_{j,n})=\varphi_1^{m-1}(\tilxi^{p^{n+1}-1}z_{j,n+1})=\ldots = a\tilxi^{p^{n+m}-1}z_{j,n+m}\in \tilxi^{p^{n+m}-1}K
  $$
  for some $a\in \prism_S$. But $\tilxi^{p^{n+m}-1}K\subseteq \tilxi^lK$ because $p^{n+m}-1\geq l$. This finishes the proof.
\end{proof}

\begin{lemma}
  \label{sec:concr-pres-prism-proposition-presentation-of-prismatic-envelope-for-rank-one-element}
  Let $(A,I)$ be a \red{bounded} prism and let $d\in A$ be distinguished.  Let
  furthermore $x\in A$ be an element of rank $1$. Then for \green{$n\geq 0$} there exist natural (in $A,x$) elements
  $$
  z_n\in  A\{\frac{x}{d}\}^{\wedge_{(p,d)}}
  $$
  such that $\varphi^{n}(d)\varphi^{n-1}(d)^p\cdots
    d^{p^n}\cdot z_n=x^{p^n}$. \blue{Moreover, for all $n\geq 0$, $\delta^n(\frac{x}{d})$ lies in the subring $A[z_0,\ldots, z_{n}]$ of $A\{\frac{x}{d}\}\red{^{\wedge_{(p,d)}}}$ generated by $z_0,\ldots,z_{n}$.}
\end{lemma}
\blue{Note that the last part of the lemma implies that the resulting morphism
  $$
  A[y_1,y_2,\ldots]/(x-d y_1,
  y_1^p-\varphi(d)y_2,y_2^p-\varphi^2(d)y_3,\ldots)\to
  A\{\frac{x}{d}\}^{\wedge_{(p,d)}},\ y_n\mapsto z_n
  $$
  is surjective after $(p,d)$-completion.} \green{We expect that this surjection is actually an isomorphism.}
\begin{proof}
  We can argue in the universal case
  $A=\Z_p[x]\{d,\frac{1}{\delta(d)}\}^{\wedge_{(p,d)}}$ where
    $\delta(x)=0$, thus we may assume that $A$ is transversal, i.e.,
  that $(p,d)$ is a regular sequence in $A$, and that $(x,d)$ is a
  regular sequence. This implies that for all $r\geq 1$ the sequence
  $(\varphi^r(d),\varphi^{r-1}(d))$ is regular as well (cf.\
  \Cref{sec:concr-pres-prism-lemma-regular-sequence-for-transversal-prism}).
  We first claim that for all $n\geq 0$ the element
  $$
  z_n:=\frac{x^{p^n}}{\varphi^{n}(d)\varphi^{n-1}(d)^p\cdots d^{p^n}}
  $$
  lies in $A\{\frac{x}{d}\}$. If $n=0$, then $z_n=\frac{x}{d}\in A\{\frac{x}{d}\}$. For $n\geq 0$ we can calculate
  $$
  \varphi(z_n)=\frac{x^{p^{n+1}}}{\varphi^{n+1}(d)\cdots \varphi(d)^{p^n}}
  $$
  because $\varphi(x)=x^p$.
  The numerator $x^{p^{n+1}}$ is divisible by $d^{p^{n+1}}$ in $A\{\frac{x}{d}\}$. \red{We claim that $(d^{p^{n+1}},\varphi^{n+1}(d)\cdots \varphi(d)^{p^n})$ is a regular sequence in $A\{\frac{x}{d}\}^{\wedge_{(p,d)}}$. Granting this }we can conclude that $d^{p^{n+1}}$ divides $\frac{x^{p^{n+1}}}{\varphi^{n+1}(d)\cdots \varphi(d)^{p^n}}$, i.e., that $z_{n+1}\in A\{\frac{x}{d}\}\red{^\wedge_{(p,d)}}$.
    \blue{Write $s=\varphi^{n+1}(d)\cdots \varphi(d)^{p^n}$. To prove that $(d^{p^{n+1}},s)$ is a regular sequence in $A\{\frac{x}{d}\}^{\wedge_{(p,d)}}$, it suffices to show the same for $(d,s)$. One proves by induction on $m$ that for all $m\geq 1$, $\varphi^m(d)$ is congruent to $pu_m$ modulo $d$ 
    for a unit $u_m$. In particular, one concludes that $s$ is congruent to $up^{\green{k}}$ modulo $d$, for $k\geq 1$ and $u$ a unit. Hence to prove that $(d,s)$ is a regular sequence in $A\{\frac{x}{d}\}^{\wedge_{(p,d)}}$, it suffices to show that $(d,p)$ is a regular sequence in $A\{\frac{x}{d}\}^{\wedge_{(p,d)}}$. But this follows from transversality of $A$ and the fact that $A \to A\{\frac{x}{d}\}^{\wedge_{(p,d)}}$ is $(p,d)$-completely flat.}

\blue{Next, we show that for all $n\geq 0$, $\delta^n(\frac{x}{d})$ lies in the subring $A[z_0,\ldots, z_{n}]$ of $A\{\frac{x}{d}\}\red{^{\wedge_{(p,d)}}}$ generated by $z_0,\ldots,z_{n}$}. This claim follows from the assertion that $\delta(z_{n})\in A[\green{z_0},\ldots, z_{n+1}]$ using induction and how $\delta$ acts on sums and products. For $n=0$ we can calculate
  $$
  \delta(z_0)=\delta(\frac{x}{d})=\frac{1}{p}(\varphi(\frac{x}{d})-\frac{x^p}{d^p})=\frac{1}{p}(d^p-\varphi(d))z_{1}=\delta(d)z_1\in A\{\frac{x}{d}\}.
  $$
  Similarly, we see
  $$
  \delta(z_n)=\frac{1}{p}(d^{p^{n+1}}-\varphi^{n+1}(d))z_{n+1}
  $$
  where the term $\frac{1}{p}(d^{p^{n+1}}-\varphi^{n+1}(d))$ lies in $A$. This finishes the proof. 
\end{proof}


We can derive essential surjectivity.

\begin{theorem}
  \label{sec:essent-surj-main-theorem-equivalence-of-divided-prismatic-dieudonne-functor} \purple{Let $R$ be a quasi-regular semiperfectoid ring.} Then the prismatic Dieudonn\'e functor \red{
  $$
  M_\prism(-)\colon \mathrm{BT}(R)\to \mathrm{DM}^{\rm adm}(R)
  $$
  from the category of $p$-divisible groups over $R$ to the category of admissible} prismatic Dieudonn\'e crystals over $R$ is essentially surjective. 
\end{theorem}
\begin{proof} 
 To prove the theorem, we may pass to a quasi-syntomic cover $R'$ of $R$: indeed, let $M \in \red{\DM^{\rm adm}}(R)$ such that its base change along the map $R \to R'$ is of the form $M_{\prism}(G')$, for some $p$-divisible group $G'$ over $R$. The descent datum for $M_{\prism}(G')$ expressing that it comes from an admissible prismatic Dieudonn\'e module over $R$ (namely, $M$) gives rise to a descent datum for $G'$, since fully faithfulness over $R^\prime\hat{\otimes}_R R^\prime$ is already proved (cf.\ \Cref{sec:fully-faithf-prism-fully-faithfulness-for-p-torsion-free-quasi-regular semiperfectoid-new}). This descent datum is effective, by $p$-completely faithfully flat descent for $p$-divisible groups (cf.\ \Cref{sec:fully-faithf-prism-proposition-bt-is-a-stack}), so there exists a $p$-divisible group $G$ over $R$, with $M_{\prism}(G)=M$.
  
 Therefore, by \Cref{sec:generalities-prisms-andres-lemma}, we may and do assume that \blue{$R\cong \bar{A}/(a_j\ |\ j\in J)$ for $\bar{A}=A/I$} a perfectoid ring and $a_j\in R$ admitting compatible systems of $p^n$-roots of unity. Using \Cref{sec:essent-surj-lifting-divided-prismatic-dieudonne-modules-for-special-surjection-of-quasi-regular semiperfectoid} we may even assume that
 $$
 R\cong \blue{\bar{A}}\langle X_j^{1/p^\infty} |\ j\in J\rangle/(X_j).
 $$
 In this case we can invoke \Cref{sec:essent-surj-lemma-essential-surjectivity-from-perfectoid-to-infinite-regular-semiperfectoid-ring} and reduce to the case that $R$ is perfectoid. Then we can cite \Cref{sec:prism-dieud-theory-comparison-with-bkf-modules-perfectoid-case} to conclude that $M_\prism(-)$ is essentially surjective.
\end{proof}

This concludes the proof of the main \Cref{sec:divid-prism-dieud-modul-theorem-main-theorem-of-the-paper}.

\begin{remark}
\label{writing-a-quasi-inverse}
Let $R$ be quasi-syntomic ring. The arguments used in \Cref{sec:fully-faithful-new} show that the functor $\mathcal{G}$ from $\mathrm{DM}^{\rm adm}(R)$ to the category of abelian sheaves of $(R)_{\rm qsyn}$, sending $\mathcal{M} \in \DM^{\rm adm}(R)$ to
 \[ (\mathcal{M}^{\vee})^{\varphi=1} \otimes_{\Z_p} \mathbb{Q}_p/\Z_p, \]
where $\mathcal{M}^{\vee}$ denotes the $\mathcal{O}^{\rm pris}$-linear dual of $\mathcal{M}$, defines a quasi-inverse of the prismatic Dieudonn\'e functor.

It seems difficult to prove directly that $\mathcal{G}$ takes values in the category of (quasi-syntomic sheaves attached to) $p$-divisible groups. In the case of \'etale $p$-divisible groups \Cref{sec:divid-prism-dieud-modul-theorem-main-theorem-of-the-paper} yields an equivalence of $\Z_p$-local systems on $R$ and finite locally free $\mathcal{O}^\pris$-modules (resp.\ $\prism_R$-modules if $R$ is quasi-regular semiperfectoid) $\mathcal{M}$ together with an isomorphism $\varphi_{\mathcal{M}}\colon \varphi^\ast(\mathcal{M})\cong \mathcal{M}$. This is a generalization of Katz' correspondence between $\Z_p$-local systems on the spectrum $\Spec(k)$ of a perfect field $k$ and $\varphi$-modules over $W(k)$ (cf.\ \cite[Proposition 4.1.1]{katz_p_adic_properties_of_modular_schemes_and_modular_forms}.
We thank Beno\^it Stroh for pointing this out to us.
\end{remark}

\newpage
\section{Complements} \label{complements}

\subsection{Prismatic Dieudonn\'e theory for finite locally free group schemes} \label{sec:complements-finite-locally-free}

Let $R$ be a perfectoid ring. We fix a generator $\xi$ of $\ker(\theta)$ and let $\tilxi=\varphi(\xi)$. 


\begin{definition}
  \label{definition-torsion-bkf-modules}
A \textit{torsion prismatic Dieudonn\'e module over $R$} is a triple
$$
(M,\varphi_M,\psi_M),
$$
where $M$ is a finitely presented $A_{\rm inf}(R)$-module of projective dimension $\leq 1$ which is annihilated by a power of $p$ and where $\varphi_M : M \to M$ and $\psi_M : M\to M$ are respectively $\varphi$-linear and $\varphi^{-1}$-linear, and satisfy $$\varphi_M \circ \psi_M=\tilxi, \quad \psi_M \circ \varphi_M=\xi.$$
The category of torsion prismatic Dieudonn\'e modules over $R$ is denoted by $\mathrm{DM}_{\rm tors}(R)$. \blue{It is an exact category.}
\end{definition}

  The base change of torsion prismatic Dieudonn\'e modules behaves well.
  
  \begin{lemma}
    \label{sec:prism-dieud-theory-lemma-base-change-for-torsion-modules}
    Let $R\to R^\prime$ be a morphism of perfectoid rings and $M\in \mathrm{DM}_{\mathrm{tors}}(R)$. Then $M\otimes_{A_\inf(R)}A_\inf(R^\prime)$ is concentrated in degree $0$. In particular, the base change functor $\mathrm{DM}_{\mathrm{tors}}(R)\to \mathrm{DM}_{\mathrm{tors}}(R^\prime)$ is exact.
  \end{lemma}
  \begin{proof}
    Let
    $$
    0\to M_1\xrightarrow{f} M_2\to M\to 0
    $$
    be a resolution of $M$ by finite locally free $A_\inf(R)$-modules. As $M$ is killed by $p^n$ for some $n\geq 0$, there exists $g\colon M_2\to M_1$ such that $f\circ g=p^n$. Then $p^n=g\circ f$ (using that $f$ is injective). The base change $M_1\otimes_{A_\inf(R)}A_\inf(R^\prime)$ is $p$-torsion free as $A_\inf(R^\prime)$ is. This implies that the base change of $f$ to $A_\inf(R^\prime)$ remains injective, which finishes the proof. 
  \end{proof}

Before stating the main result, let us introduce a notation, which will be in use only in this section.
\begin{notation}
If $S$ is a $p$-complete ring, let $\mathcal{B}_S$ (resp. $\mathcal{C}_S$) denote the category whose objects are $\mathcal{O}_{\prism}$-modules on $(S)_{\prism}$ (resp. $\mathcal{O}_{\prism}$-modules on $(S)_{\prism}$ endowed with a $\varphi$-linear Frobenius), and whose morphisms are $\mathcal{O}_{\prism}$-linear morphisms (resp. $\mathcal{O}_{\prism}$-linear morphisms commuting with Frobenius).
\end{notation}

\begin{theorem}
  \label{theorem-exact-antiequivalence-for-finite-flat-group-schemes}
There is a natural exact\footnote{This includes the non-formal assertion that the inverse equivalence is exact, too.} antiequivalence
\[ H \mapsto (M_{\prism}(H), \varphi_{M_{\prism}(H)}, \psi_{M_{\prism}(H)}) \]
between the \blue{exact} category of finite locally free group schemes of $p$-power order on $R$ and the \blue{exact} category $\mathrm{DM}_{\rm tors}(R)$ of torsion prismatic Dieudonn\'e modules over $R$, such that the $A_{\rm inf}(R)$-module $M_{\prism}(H)$ is given by the formula 
\[ M_{\prism}(H)=\mathrm{Ext}_{(R)_{\prism}}^1(u^{-1} H, \mathcal{O}_{\prism}) \]
and such that $\varphi_{M_{\prism}(H)}$ is the map induced by the Frobenius of $\mathcal{O}_{\prism}$. 

\end{theorem}

\begin{remark}
  \label{remark-previous-results-on-classification-of-finite-flat-group-schemes}
A similar statement can be found in \cite[Theorem 10.12]{lau_dieudonne_theory_over_semiperfect_rings_and_perfectoid_rings}. Apart from the change of terminology, the only difference with the result in loc. cit. is that we remove the assumption that $p\geq 3$ and provide a formula for the underlying $A_{\rm inf}$-module of the torsion minuscule Breuil-Kisin-Fargues module attached to a finite locally free group scheme of $p$-power order.
\end{remark}

The proof of \Cref{theorem-exact-antiequivalence-for-finite-flat-group-schemes} will make use of the following lemma.

  \begin{lemma}
    \label{sec:prism-dieud-theory-lemma-h-0-p-torsion-free-for-syntomic-things}
    Let $(A,I)$ be a bounded prism, such that $A$ is $p$-torsion free and let $S$ be a $p$-completely syntomic $A/I$-algebra\footnote{A morphism $R\to R^\prime$ between $p$-complete rings of bounded $p^\infty$-torsion is $p$-completely syntomic if $R^\prime/p\cong R^\prime\otimes_R^{\mathbb{L}} R/p$ and $R/p\to R^\prime/p$ is syntomic in the sense of \cite[Tag 00SL]{stacks_project}.}. Then
    $$
    H^0(S,\prism_{S/A})
    $$
    is $p$-torsion free.
  \end{lemma}
  \begin{proof}
    As $S$ is a $p$-completely syntomic $A/I$-algebra the derived prismatic cohomology $\prism_{S/A}$ agrees with the cohomology $R\Gamma((S/A)_\prism,\mathcal{O}_\prism)$ of the prismatic site of $S$ over $A$ (this follows by descent from the quasi-regular semiperfectoid case and \Cref{sec:prism-cohom-quasi-proposition-all-prisms-agree-for-quasi-regular semiperfectoid}). 
    By \cite[Proposition 3.13]{bhatt_scholze_prisms_and_prismatic_cohomology} and the assumption that $S$ is a $p$-completely syntomic $A/I$-algebra, one can calculate $\prism_{S/A}$ by some \u{Cech}-Alexander complex \red{whose first term is $p$-complete and $p$-completely flat over $A$}. Therefore it suffices to see that each $p$-complete $p$-completely flat $A$-algebra $B$ has no $p$-torsion. As $A$ is $p$-torsion free, $A$, and thus $B$, is $p$-completely flat over $\Z_p$. But any $p$-completely flat $p$-complete module over $\Z_p$ is topologically free and thus $p$-torsion free.
  \end{proof}

\begin{proof}[Proof of \Cref{theorem-exact-antiequivalence-for-finite-flat-group-schemes}]
The construction of the antiequivalence is exactly similar to the one of \cite[Theorem 10.12]{lau_dieudonne_theory_over_semiperfect_rings_and_perfectoid_rings}, replacing Theorem 9.8 in loc.\ cit.\ by \Cref{sec:prism-dieud-theory-comparison-with-bkf-modules-perfectoid-case}, so we do not give it and refer the reader to \cite{lau_dieudonne_theory_over_semiperfect_rings_and_perfectoid_rings}. The simple principle is that Zariski-locally on $\mathrm{Spec}(R)$, any finite locally free group scheme of $p$-power order is the kernel of an isogeny of $p$-divisible groups (and even an isogeny of $p$-divisible groups associated to abelian schemes, cf. \Cref{sec:divid-prism-dieud-definition-for-p-div-groups-theorem-raynaud}); similarly, Zariski-locally on $\mathrm{Spec}(R)$, any torsion prismatic Dieudonn\'e module is the cokernel of an isogeny of prismatic Dieudonn\'e modules (\cite[Lemma 10.10]{lau_dieudonne_theory_over_semiperfect_rings_and_perfectoid_rings}).

  Let us now prove that
  $$
  M_{\prism}(H)=\mathrm{Ext}_{(R)_{\prism}}^1(u^{-1} H, \mathcal{O}_{\prism})
  $$
  and that the functor $M_\prism(-)$ preserves exactness for a short exact sequence
  $$
  0\to H^\prime\to H\to H^{\prime\prime}\to 0
  $$
  of finite locally free group schemes of $p$-power order over $R$. Note that this implies by Mittag-Leffler exactness of
  $$
  \red{0\to M_\prism(H^{\prime\prime})\to M_\prism(H)\to M_\prism(H^{\prime})\to 0}
  $$
  if $H^\prime,H,H^{\prime\prime}$ are finite locally free group schemes of $p$-power order or $p$-divisible groups. 

By construction of the antiequivalence, it suffices to check that if $H$ is the kernel of an isogeny $X \to X'$, with $X, X'$ are abelian schemes over $R$, the natural map
\[ M_{\prism}(X[p^{\infty}])= \mathrm{Ext}_{(R)_{\prism}}^1(u^{-1} X, \mathcal{O}_{\prism})  \to \mathrm{Ext}_{(R)_{\prism}}^1(u^{-1} H, \mathcal{O}_{\prism})  \]
is surjective. But the cokernel of this map embeds in $\mathrm{Ext}_{(R)_{\prism}}^2(u^{-1} X', \mathcal{O}_{\prism})$, which is zero by \Cref{sec:prism-dieud-cryst-ext-for-abelian-schemes}.

For exactness, start with a short exact sequence of finite locally free group schemes of $p$-power order on $R$ 
\[ 0 \to H^\prime \to H \to H'' \to 0, \]
which we see as an exact sequence of abelian sheaves on $(R)_{\rm qsyn}$. The surjectivity of the map
\[ M_{\prism}(H) \to M_{\prism}(H') \]
can be checked locally and so we can assume that $H$, and so also $H'$, embeds in an abelian scheme $X$. But we know that the map
\[ M_{\prism}(X[p^{\infty}]) \to M_{\prism}(H') \]
is already surjective, again because $\mathrm{Ext}_{(R)_{\prism}}^2(u^{-1} X/H^\prime, \mathcal{O}_{\prism})=0$. Thus, the same holds for the map
\[ M_{\prism}(H) \to M_{\prism}(H'). \]
To prove injectivity of the map 
\[ M_{\prism}(H'') \to M_{\prism}(H), \]
it suffices by the long exact sequence for $R\mathrm{Hom}_{(R)_{\prism}}(-, \mathcal{O}_{\prism})$ to prove that 
\[ \mathrm{Hom}_{(R)_{\prism}}(u^{-1} H', \mathcal{O}_{\prism}) = 0. \]
Let us prove that $\mathrm{Hom}_{(R)_{\prism}}(u^{-1} H', \mathcal{O}_{\prism})$ is $p$-torsion free. This is enough : indeed, we know it is also killed by a power of $p$, because $u^{-1}H'$ is. As
$$\mathrm{Hom}_{(R)_{\prism}}(u^{-1} H', \mathcal{O}_{\prism}) \subset H^0(u^{-1}H^\prime,\mathcal{O}_{\prism})=H^0(H^\prime, \prism_{H^{\prime}/A_\inf}), $$
it suffices to prove that the latter is $p$-torsion free. This is the content of \Cref{sec:prism-dieud-theory-lemma-h-0-p-torsion-free-for-syntomic-things} when applied to the $p$-completely syntomic $R$-scheme $H^\prime$.

Let
$$
\mathcal{G}\colon \mathrm{DM}_{\mathrm{tors}}\to \{\text{finite locally free group schemes of $p$-power order over }R \}
$$
be an inverse functor to $M_\prism(-)$.
We claim that $\mathcal{G}$ is exact.
Let
$$
0\to M_1\to M_2\to M_3\to 0
$$
be an exact sequence in $\mathrm{DM}_{\mathrm{tors}}(R)$. For any morphism $R\to R^\prime$ the base change of it along $A_\inf(R)\to A_\inf(R^\prime)$ will stay exact by \Cref{sec:prism-dieud-theory-lemma-base-change-for-torsion-modules}.
By \cite[Proposition 1.1]{dejong_finite_locally_free_group_schemes_in_characteristic_p_and_dieudonne_modules} and compatibility of $\mathcal{G}$ with base change in $R$ we can therefore assume that $R$ is a perfect field of characteristic $p$. In this case the category of finite locally free group schemes of $p$-power order and the category $\mathrm{DM}_{\mathrm{tors}}$ are abelian and thus any equivalence between them is automatically exact.
\end{proof}

\begin{remark}
  \label{remark-bk-case-finite-flat-group-schemes}
  \purple{Let $R$ be quasi-syntomic ring.} Although the same trick allows in principle to deduce from \Cref{sec:divid-prism-dieud-modul-theorem-main-theorem-of-the-paper} a classification result for finite locally free group schemes of $p$-power order over $R$, it seems more subtle to obtain a nice description of the target category, i.e. of the objects which can locally on $R$ be written as the cokernel of an isogeny of \red{admissible} prismatic Dieudonn\'e crystals on $R$. At least the arguments given above should go through whenever the forgetful functor
  \[ \DF(R) \to \mathrm{DM}(R) \]
  is an equivalence, like in the case of perfectoid rings or in the Breuil-Kisin case to be discussed in the next section (where the classification of finite flat group schemes is already known, and was proved by Kisin following the same technique, cf.\ \cite[Section 2.3]{kisin_crystalline_representations_and_f_crystals}).
  \end{remark}

\subsection{Comparison over $\mathcal{O}_K$}
\label{sec:comp-case-mathc}

In this section, we want to extract from \Cref{sec:divid-prism-dieud-modul-theorem-main-theorem-of-the-paper} a concrete classification of $p$-divisible groups over complete regular local rings with perfect residue field of characteristic $p$. This will in particular recover Breuil-Kisin's classification (\cite{breuil_schemas_en_groupes_et_corps_de_normes}, \cite{kisin_crystalline_representations_and_f_crystals}), as extended to all $p$ by Kim \cite{kim_the_classification_of_p_divisible_groups_over_2_adic_discrete_valuation_rings}, Lau \cite{lau_relations_between_dieudonne_displays_and_crystalline_dieudonne_theory} and Liu \cite{liu_the_correspondence_between_barsotti_tate_groups_and_kisin_modules_when_p_equal_2}, over $\mathcal{O}_K$, for a complete discretely valued extension of $\mathbb{Q}_p$ with perfect residue field.  

\begin{proposition}
\label{sec:comp-regular-perfectoid-cover}
Let $R$ be a \blue{complete Noetherian local ring with perfect residue field of characteristic $p$}. If $R$ is regular, there exists a quasi-syntomic perfectoid cover $R_{\infty}$ of $R$. 
\end{proposition}
\begin{proof}
  The existence of a faithfully flat cover $R \to R_{\infty}$, with $R_{\infty}$ perfectoid, is explained in \cite[Theorem 4.7]{bhatt_iyengar_ma_regular_rings_and_perfectoid_algebras}. Assume first that $pR=0$ or that $R$ is unramified\footnote{The case $R$ unramified is explained in \cite[Ex. 3.8 (4)]{bhatt_iyengar_ma_regular_rings_and_perfectoid_algebras}, too.}.  $R$ is either flat over $\Z_p$ or $pR=0$. In the first case set $\Lambda:=\Z_p$ and in the second $\Lambda:=\F_p$. By \cite[Tag 07GB]{stacks_project} the morphism $\Lambda\to R$ is a filtered colimit of smooth ring maps and thus $L_{R/\Lambda}$ has $p$-complete Tor-amplitude in degree $0$. The triangle attached to the composite $\Lambda \to R \to R_{\infty}$ shows that $L_{R_{\infty}/R}$ has $p$-complete Tor-amplitude in degree $-1$. Therefore the map $R \to R_{\infty}$ is indeed a quasi-syntomic cover. Finally, when $R$ is ramified of mixed characteristic, one sees from the explicit construction of \cite[Ex. 3.8 (5)]{bhatt_iyengar_ma_regular_rings_and_perfectoid_algebras} that $R \to R_{\infty}$ is the $p$-completion of a colimit of syntomic morphisms (obtained by extracting $p$th-roots), hence is quasi-syntomic. 
\end{proof}

\begin{remark}
In the converse direction, the main result of \cite{bhatt_iyengar_ma_regular_rings_and_perfectoid_algebras} asserts that a Noetherian ring with $p$ in its Jacobson radical which admits a faithfully flat map to a perfectoid ring has to be regular (this is a generalization of a theorem of Kunz \cite{kunz_characterizations_of_regular_local_rings_of_characteristic_p} in positive characteristic).
\end{remark}

\begin{proposition}
\label{sec:comp-case-mathc-proposition-forgetful-functor-regular}
\red{Let $R$ be a \blue{complete regular local ring with perfect residue field of characteristic $p$}. Any prismatic Dieudonn\'e crystal over $R$ is admissible.} 
\end{proposition}
\begin{proof}
\red{Let $(\mathcal{M},\varphi_{\mathcal{M}}) \in \mathrm{DM}(R)$. Let $R_{\infty}$ be a perfectoid quasi-syntomic cover of $R$, as in \Cref{sec:comp-regular-perfectoid-cover}. Let $\mathcal{M}_{\infty} \in \mathrm{DM}(R_{\infty})$ be the base change of $\mathcal{M}$, which we see as a prismatic Dieudonn\'e module $M_{\infty}$ over $R_\infty$, via the equivalence of \Cref{sec:abstr-divid-prism-proposition-equivalence-crystals-modules-for-quasi-regular semiperfectoid}. We know (\Cref{sec:abstr-divid-prism-divided-prismatic-dieudonne-modules-vs-bkf-modules}) that $M_\infty$ is admissible. Since the natural functor $\DM^{\rm adm} \to \DM$ is (tautologically) fully faithful, $M_\infty$ descends to an admissible prismatic Dieudonn\'e crystal over $R$, which must identify with $(\mathcal{M},\varphi_{\mathcal{M}})$.} 
\end{proof}

Recall the following definition, which already appeared in \Cref{sec:abstr-divid-prism-remark-cais-lau-principal} before.

\begin{definition}
  \label{sec:comp-case-mathc-definition-breuil-kisin-module}
  Let $(A,I=(d))$ be a prism. A \textit{Breuil-Kisin module $(M,\varphi_M)$ over $(A,I)$, or just $A$ if $I$ is understood}, is a finite free $A$-module $M$ together with an isomorphism
  $$
  \varphi_M\colon \varphi^{\ast}M[\frac{1}{I}]\cong M[\frac{1}{I}].
  $$
  If $\varphi_M(\varphi^{\ast}M)\subseteq M$ with cokernel \blue{killed by $I$}, then $(M,\varphi_M)$ is called \textit{minuscule}.
  
  We denote by $\mathrm{BK}(A)$ the category of Breuil-Kisin modules over $A$ and by $\mathrm{BK}_{\mathrm{min}}(A)\subseteq \mathrm{BK}(A)$ its full subcategory of minuscule ones.
\end{definition}


\blue{If $R$ is a complete regular local ring with perfect residue field $k$ of characteristic $p$, it can be written as}
\[ R = W(k)[[u_1,\dots,u_d]]/(E), \]
where $d=\dim R$ and $E$ is a power series with constant term of $p$-value one (cf.\ \cite[Theorem 29.7, Theorem 29.8 (ii)]{matsumura_commutative_ring_theory}). Let $(A,I)$ be the prism $$(A,I)=(W(k)[[u_1,\dots,u_d]],(E)),$$
where the $\delta$-ring structure on $A$ is the usual one on $W(k)$ and is such that $\delta(u_i)=0$, for $i=1,\dots,d$. \blue{For simplicity, we assume $d=1$ in the following. We hope that the general case works similarly.}

\begin{theorem}
\label{sec:comp-case-mathc-theorem-comparison-with-bk}
 Let $R$ be a complete regular local ring with perfect residue field of characteristic $p$. The functor 
\[ \BT(R) \to \mathrm{BK}_{\mathrm{min}}(A) \quad ; \quad G \mapsto v^* \mathcal{M}_{\prism}(G)((A,I)) = \Ext_{(R)_{\prism}}^1 (u^{-1} G, \mathcal{O}_{\prism})_{(A,I)} \]
is an equivalence of categories.
\end{theorem}

The case where $pR=0$ follows from \Cref{sec:prism-dieud-theory-theorem-comparison-with-crystalline-dieudonne-functor}, the classical fact that a Dieudonn\'e crystal over $R$ is the same thing as a minuscule Breuil-Kisin module over $A$ (with respect to $p$) together with an integrable topologically quasi-nilpotent connection making Frobenius horizontal and \cite[Proposition 2.7.3]{cais_lau_dieudonne_crystals_and_wach_modules_for_p_divisible_groups}, which proves that for this particular ring $A$, the connection is necessarily unique. Hence in the following, we will always assume that $R$ is $p$-torsion free. In this case, the pair $(p,E)$ is transversal. 

\begin{remark}
When $R=\mathcal{O}_K$, with $K$ a complete discretely valued extension of $\mathbb{Q}_p$ with perfect residue field, $A$ is usually denoted by $\mathfrak{S}$ (a notation which seems to originate from \cite{breuil_schemas_en_groupes_et_corps_de_normes}). We will see below that the antiequivalence of the theorem coincides in this case with the one studied by Kisin for $p$ odd and Kim, Lau and Liu when $p=2$.
\end{remark}

We will describe prismatic Dieudonn\'e crystals over $\mathcal{O}_K$ via descent using the following lemma. 
\begin{lemma}
  \label{sec:comp-over-mathc-lemma-breuil-kisin-prism-covers-final-object}
The natural map from the sheaf represented by $(A,I)$ to the final object of $\mathrm{Shv}((R)_{\prism})$ is an epimorphism for the $p$-completely faithfully flat topology.
\end{lemma}
\begin{proof}
Indeed, let $(B,J) \in (R)_{\prism}$. Let $A_{\infty}$ be the perfection of $A$; the map $R=A/I \to R_{\infty}=A_{\infty}/IA_{\infty}$ is a quasi-syntomic cover. By base change, the map
\[ B/J \to B/J \hat{\otimes}_R R_{\infty} \]
is therefore a quasi-syntomic cover as well. By \Cref{sec:quasi-syntomic-rings-proposition-quasi-syntomic-cover-of-prism-admits-refinement} there exists a prism $(C,JC)$ which is $p$-completely faithfully flat over $(B,J)$ such that there exists a morphism of $B/J$-algebras $B/J \hat{\otimes}_R R_{\infty}\to C/J$. Since $R_{\infty}$ is perfectoid, it implies that $(C,JC)$ lives over $(A_{\infty},IA_{\infty})$ (cf.\ \Cref{sec:prisms-perf-rings-lemma-initial-object-for-perfectoid-rings}), and a fortiori over $(A,I)$, as desired.
\end{proof}

\begin{proof}[Proof of \Cref{sec:comp-case-mathc-theorem-comparison-with-bk}]
By \Cref{sec:divid-prism-dieud-modul-theorem-main-theorem-of-the-paper} and \Cref{sec:comp-case-mathc-proposition-forgetful-functor-regular}, we know that the prismatic Dieudonn\'e functor 
\[ \mathcal{M}_{\prism} : \BT(R) \to \mathrm{DM}(R) \]
is an antiequivalence. Therefore, it suffices to prove that the functor
\[ \mathcal{M} \to v^* \mathcal{M}((A,I))  \]
from prismatic Dieudonn\'e crystals $\mathrm{DM}(R)$ to minuscule Breuil-Kisin modules $\mathrm{BK}_{\rm min}(A)$ is an equivalence.
Let $B$ be the absolute product of $A$ with itself in $(R)_{\prism}$. One has (cf.\ \cite[Proposition 3.13]{bhatt_scholze_prisms_and_prismatic_cohomology})
\[ B= \left(W(k)[[u]] \otimes_{W(k)} W(k)[[v]] \right)\{ \frac{u-v}{E(u)} \}_{\delta}^{\wedge_{(p,E(u))}} \]
where we wrote $E(u)$ for $E\otimes 1$\footnote{If similarly, $E(v)=1\otimes E$, then $E(u)/E(v)$ is a unit in $B$ by \cite[Lemma 2.24]{bhatt_scholze_prisms_and_prismatic_cohomology} because $E(u)$ divides $E(v)$ in $B$. Namely, $E(v)=E(u)(\frac{E(v)-E(u)}{E(u)}+1)$ in $B$ and $u-v$ divides $E(u)-E(v)$.}.
By \Cref{sec:comp-over-mathc-lemma-breuil-kisin-prism-covers-final-object} below and \Cref{sec:astr-divid-prism-proposition-descent}, a prismatic Dieudonn\'e crystal $\mathcal{M}$ over $R$ is the same thing as a minuscule Breuil-Kisin module $N$ over $A$, together with a descent datum, i.e., an isomorphism
$$
N \otimes_{A,p_1} B \cong N \otimes_{A,p_2} B
$$
(where $p_1, p_2 : A \to B$ are the two natural maps), satisfying the usual cocycle condition. 

We claim that any $N \in \mathrm{BK}_{\rm min}(A)$ is equipped with a unique descent datum. Indeed, let $f\colon B\to A$ be the map extending the multiplication map
\[
  f_0\colon B_0:=A\hat{\otimes}_{W(k)} A \to A
\]
and, for $i=1,2$, set $E_i:=p_i(E)\in B_0$, with $p_i\colon A\to B_0$ the two inclusions. Let $M_0$ be a minuscule Breuil-Kisin module over $B_0$ with respect to the element $E_1$ and $N_0$ a minuscule Breuil-Kisin module with respect to $E_2$. Let $M_A=M_0 \otimes_{B_0,f_0} A$, $N_A=N_0 \otimes_{B_0,f_0} A$ be their base changes along $f_0$.
Let $\alpha_0\colon M_0\to N_0$ be any $B_0$-linear map such that $\alpha_A:= f^\ast_0 \alpha_0 \colon M_A\to N_A$ is a morphism of Breuil-Kisin modules over $A$. Consider the composition
  \[
    U_0(\alpha_0):=\frac{1}{E_1}\varphi_{N_0}\circ \varphi^\ast \alpha_0 \circ \varphi_{M_0}^{-1}(E_1(-))\colon M_0\to \frac{1}{E_1}N_0
  \]
as in the proof of \Cref{sec:essent-surj-lemma-equivalence-on-window-categories}.
Then the morphism $U_0(\alpha_0)-\alpha_0$ maps $M_0$ to $\frac{1}{E_1} KN_0$ where $K=\ker(f_0)$ as $\alpha_A$ is a morphism of minuscule Breuil-Kisin modules over $A$.
By construction of $B$ we have $K\subseteq E_1J$, \red{if $J=\ker(f)$}. In particular, if $\alpha$ denotes the base change of $\alpha_0$ to $B$, then
\[
  U(\alpha)-\alpha
\]
maps $M_0\otimes_{B_0}B$ to $J(N_0\otimes_{B_0}B)$, where $U(\alpha)$ is the base change of $U_0(\alpha_0)$. Thanks to \Cref{sec:comp-case-mathc-lemma-top-nilpotent-J} below, we can use the same arguments in the proof of \ref{sec:essent-surj-lemma-equivalence-on-window-categories} to see that there exists an isomorphism $\alpha\colon M_0 \otimes_{B_0} B \cong N_0 \otimes_{B_0} B$ of Breuil-Kisin modules over $B$ with $f^\ast \alpha=\alpha_A$. Indeed, if $\beta_0:=U_0(\alpha_0)-\alpha_0$ with $f^\ast_0\alpha_0=\alpha_A$, then the series
\[
  \sum\limits_{n=0}^\infty U^n_0(\beta_0)
\]
converges after base change to $B$, since $\beta$ sends $M_0 \otimes_{B_0} B$ to $J.(N_0\otimes_{B_0} B)$. In other words, the map induced by $f$
$$
\delta_{M_0,N_0}: \mathrm{Hom}_{\mathrm{BK}_{\rm min}(B)} (M_0 \otimes_{B_0} B, N_0 \otimes_{B_0} B) \to \mathrm{Hom}_{\mathrm{BK}_{\rm min}(A)} (M_A, N_A)
$$
is a surjection. We claim that $\delta_{M_0,N_0}$ is also injective. Indeed, assume that $\alpha\colon M_0\otimes_{B_0}B\to N_0\otimes_{B_0}B$ is a morphism of minuscule Breuil-Kisin modules over $B$ reducing to $0$ after base change to $A$. \red{Define
$$
  U(\alpha)\colon  M_0\otimes_{B_0}B \to J. (N_0\otimes_{B_0}B),\ m\mapsto \frac{1}{E_1}\varphi_{N_0\otimes_{B_0}B}\circ \varphi^\ast \alpha \circ \varphi_{M_0\otimes_{B_0}B}^{-1}(E_1.m)
  $$
Then, since $\alpha$ is a morphism of minuscule Breuil-Kisin modules,
\[
  U^n(\alpha)=\alpha
\]
for all $n\geq 1$. But as $\varphi_1:=\frac{\varphi}{E_1}$ is topologically nilpotent on $J$, we see that $U^n(\alpha)$ converges to $0$ for $n \to \infty$ by the same exact same argument as in the proof of \Cref{sec:essent-surj-lemma-equivalence-on-window-categories}}.

Recall that we started with $N \in \mathrm{BK}_{\rm min}(A)$ and want to produce a descent datum on $N$. To apply the above discussion, we set $M_0:=N\otimes_{A,p_1}B_0, N_0:=N\otimes_{A,p_2} B_0$, and let $\varphi_{M_0}, \varphi_{N_0}$ be the respective base changes of $\varphi_N$. Since the compositions $f\circ p_1$, $f\circ p_2$ are the identity map, $M_A, N_A$ are isomorphic to $N$. Let
$$
\alpha_N \colon M_0 \to N_0
$$
corresponding via the bijection $\delta_{M_0,N_0}$ to the identity map from $M_A=N$ to $N_A=N$. If $N^\prime \in \mathrm{BK}_{\rm min}(A)$ is another minuscule Breuil-Kisin module over $A$, and $g\in \mathrm{Hom}_{\mathrm{BK}_{\rm min}(A)} (N, N^\prime)$. We claim that
$$
\alpha_{N^\prime} \circ g_1 = g_2 \circ \alpha_N 
$$
where $g_1$, resp $g_2$, is the base change of $g$ along $p_1$, resp. $p_2$. Indeed, this can be rewritten as an equality
$$
\alpha_{N^\prime} \circ g_1 \circ \alpha_N^{-1} = g_2 \in \mathrm{Hom}_{\mathrm{BK}_{\rm min}(B)} (N_0 \otimes_{B_0} B, N_0^\prime \otimes_{B_0} B)
$$
(using for $N^\prime$ notations analogous at the ones we used for $N$), which, by the considerations above, can be checked after base change along $f: B \to A$, where it becomes obvious (since $\alpha_N$, resp. $\alpha_{N^\prime}$, reduces to the identity of $N$, resp. $N^\prime$, and since $f\circ p_1=f\circ p_2$ is the identity). This shows that the formation of $\alpha_N$ is functorial in $N$. As each descent datum on $N$ reduces to the identity on $N$ after base change along $f$ the descent datum on $N$ is unique, if it exists, since $\delta_{M_0,N_0}$ is injective.

To conclude, it therefore remains to prove that $\alpha_N$ is a descent datum, i.e. that it satisfies the cocycle condition.
Let
  \[
    C
  \] be the prism representing the triple absolute product of $(A,(E))$ in $(R)_{\prism}$. We have to see that
  \begin{equation}
    \label{eq:2}
    p^\ast_{1,2}\alpha_N\circ p^\ast_{2,3}\alpha_N=p^\ast_{1,3}\alpha_N,
  \end{equation}
  where the $p_{i,j}\colon B\to C$ are induced by the respective projections. Let us note that fiber products in $(R)_\prism$ are calculated by (completed) tensor products and that
  \[
    X\times X\times X\cong (X\times X)\times_{X} (X\times X)
  \]
  for any object $X$ in a category $\mathcal{C}$ admitting fiber products.
  This implies that
   \[
     C\cong B\widehat{\otimes}_A B.
\]
Let $C_0= B\otimes_A B$ be the uncompleted tensor product. Note that $p^\ast_{i,j}\alpha_N$, for each $1\leq i<j\leq 3$, is already defined over $C_0$. The kernel $L$ of the natural morphism $C_0 \to A$ is generated by
    \[
     J\otimes_A B, ~ B\otimes_A J.
   \]
  In particular, $\varphi_1:=\red{\frac{\varphi}{E_1\otimes 1}}$ stabilizes $L$, and $\varphi_1$ is elementwise topologically nilpotent on it. Therefore, arguing as above, we see that any morphism of minuscule Breuil-Kisin modules over $C_0$ which vanishes after base change along $C_0 \to A$, must vanish after base change to $C$. After reduction to $A$, (\ref{eq:2}) becomes
   \[
     \Id_N\circ \Id_N=\Id_N
   \]
   by construction of $\alpha_N$. This finishes the proof. 
\end{proof}

The proof of \Cref{sec:comp-case-mathc-theorem-comparison-with-bk} relied on the following technical lemma.

\begin{lemma}
\label{sec:comp-case-mathc-lemma-top-nilpotent-J} With the notation from the proof of \Cref{sec:comp-case-mathc-theorem-comparison-with-bk} the ideal $J\subseteq B$ is contained in $\mathcal{N}^{\geq 1} B$, stable by $\varphi_1:=\frac{\varphi}{E(u)}$ and $\varphi_1$ is topologically nilpotent on $J$, with respect to the $(p,E)$-adic topology. 
\end{lemma}
\begin{proof}
  Write $E:=E(u)$.
  The ideal $J$ is generated (up to completion) by the $\delta$-translates of
  $$
  z:=(u-v)/E,
  $$
  so to check that $J \subset \mathcal{N}^{\geq 1} B$, it is enough to prove that $\delta^n (z) \in \mathcal{N}^{\geq 1} B$ for all $n$. We prove by induction on $n$ that for all \red{$k \geq 1$}, $\varphi^k(\delta^n(z))$ is divisible by $E$. For $n=0$, one has, for any $k\geq 1$,
\[ \varphi^k(z)= \frac{u^{pk}-v^{pk}}{\varphi^k(E)}=\frac{(u-v)(u^{pk-1} + u^{pk-2}v + \dots + uv^{pk-2} + v^{pk})}{\varphi^k(E)}. \]
Since $(E,\varphi^k(E))$ is regular (as $(p,E)$ is transversal because $B$ is $(p,E)$-completely faithfully flat over $W(k)[[u]]$ by \cite[Proposition 3.13]{bhatt_scholze_prisms_and_prismatic_cohomology}) and $u-v$ is divisible by $E$ in $B$, we deduce that $E$ divides $\varphi^k(z)$. Let now $n\geq 0$ and assume the result is known for $\delta^n(z)$. We have, for $k\geq 0$, 
\[ p\varphi^k(\delta^{n+1}(z)) =\varphi^k(p\delta^{n+1}(z))= \varphi^k(\varphi(\delta^{n}(z))-\delta^n(z)^p) =\varphi^{k+1}(\delta^n(z)) - \varphi^k(\delta^n(z))^p, \]
so the statement for $\delta^{n+1}(z)$ follows by induction hypothesis, and the fact that $p$ and $E$ are transversal. This concludes the proof that $J \subset \mathcal{N}^{\geq 1} B$. 

Let $x \in J$. We have
\[ E.f(\varphi_1(x)) = f(\varphi(x)) =\varphi(f(x))=0. \]
Since $E$ is a non-zero divisor in $A$, we must have $f(\varphi_1(x))=0$ and therefore $\varphi_1(x) \in J$, i.e., $\varphi_1$ stabilizes $J$.

It remains to prove that the divided Frobenius is topologically nilpotent on $J$, endowed with the $(p,E)$-adic topology. Let 
\[ A'= A\left\{\frac{\varphi(E)}{p}\right\}^{\wedge_p},\]
which by \cite[Lemma 2.35]{bhatt_scholze_prisms_and_prismatic_cohomology} identifies with the ($p$-completed) divided power envelope $D_A((E))^{\wedge_p}$ of $A$ in $(E)$. \red{Let $\iota\colon A\to A^\prime$ be the natural inclusion.}
The composition 
\[ \alpha : A \overset{\varphi} \longrightarrow A \overset{\iota}{\to} A' \]
defines a morphism of prisms $(A,(E)) \to (A', (p))$. Let
\[ B':=  D_{A \hat{\otimes}_{W(k)} A}(J')^{\wedge_p}, \]
where $J'$ is the kernel of the map $A \hat{\otimes}_{W(k)} A \to R$. The ideal $J'$ is generated by $E$ and $u-v$, which form a regular sequence in ${A \hat{\otimes}_{W(k)} A/p}$, and therefore
\begin{equation*}
\begin{split}
B' \cong (A \hat{\otimes}_{W(k)} A)\left\{\frac{\varphi(E),\varphi(u-v)}{p} \right\}_\delta^{\wedge_{p}} & \cong (A \hat{\otimes}_{W(k)} A)\left\{\frac{p,\varphi(u-v)}{\varphi(E)} \right\}_\delta^{\wedge_{\varphi(E)}} \\
& \cong D_{\varphi_{A\hat{\otimes}_{W(k)} A}^*B}((E))^{\wedge_p}. 
\end{split}
\end{equation*}
(In the second isomorphism we used again \cite[Lemma 2.24]{bhatt_scholze_prisms_and_prismatic_cohomology}, and in the first and last \cite[Lemma 2.37]{bhatt_scholze_prisms_and_prismatic_cohomology}.)
In particular, the map $\alpha$ induces a map:
\[ \red{\alpha_B} : B \to B'
\]
\red{because $B\cong A\hat{\otimes}_{W(k)}A\{\frac{u-v}{E}\}^{\wedge_{(p,E)}}$.}
It sends $J \subseteq B$ to the kernel $K \subset B'$ of the map $B' \to A'$ (which extends the multiplication on $\mu\colon A\hat{\otimes}_{W(k)} A\to A$), and commutes with the divided Frobenius (because $B^\prime$ is $p$- and thus $\varphi(E)$-torsion free). \red{We thus have a diagram:
  \[
  \xymatrix{
    & J \ar[d]& &     & K \ar[d]& \\
    A\hat{\otimes}_{W(k)} A\ar[r]\ar[rd]_\mu\ar@{-->}@/^1pc/[rrr]^{\varphi} &B \ar@{-->}@/^2pc/[rrr]^{\alpha_B} \ar[d] & &    A\hat{\otimes}_{W(k)} A\ar[r]\ar[rd]^{\iota\circ \mu} &B^\prime \ar[d] \\
    & A \ar@{-->}@/^1pc/[rrr]^{\alpha}&  &   & A^\prime & 
  }
  \]
}
The ideal $K\subseteq B^\prime$ is generated (up to completion) by $(u-v)$ and the $\delta$-translates of $$\frac{\varphi(u-v)}{p}=\mathrm{unit} \cdot \frac{\varphi(u-v)}{\varphi(E)}.$$ As the kernel $J$ of $B\to A$ is stable by $\varphi_1$, this implies that $K=JB^\prime$ is stable by $\varphi_1$, and thus in particular contained in $\mathcal{N}^{\geq 1}B^\prime$.

Observe also that
\[ pB' \cap B = (p,E).B \]
To see this, one needs to show that the map induced by $\alpha_B$ 
\[ B/(p,E) \to B'/p \]
is injective, i.e., by faithful flatness of $\varphi\colon A\to A$ that the natural map
\[ B/(p,\varphi(E))=B/(p,E^p) \to B'/p=D_{B}((E))/p \] 
is injective. But since $B$ is $p$-torsion free, 
\[ B'/p = B/(p,E^p)[X_0,X_1,\dots]/(X_0^p,X_1^p,\dots)^{\wedge p}\]
and the above map is simply the natural inclusion map. Hence, it suffices to prove topological nilpotence of $\varphi_1="\varphi/\varphi(E)"$ on $K$ with respect to the $p$-adic topology\footnote{Let us clarify what we mean by the various $\varphi_1$'s, whenever they are defined. On $A$ we set $\varphi_1=\varphi/E$ which is the restriction of $\varphi_1=\varphi/\varphi(E)$ along $\alpha$. In $B^\prime$ the element $\varphi(E)/p$ is a unit and thus $\varphi_1=\frac{p}{\varphi(E)}\frac{\varphi}{p}$, i.e., both possible definition of the divided Frobenius differ by a unit.}. We do it in two steps.

Note first that $\varphi$ is topologically nilpotent on $K$. More precisely, using that $K$ is stable by $\varphi_1$, one easily sees by induction that $\varphi^k(z)$ is divisible by $p^k$, for all $z \in K$ and $k\geq 1$ (with $\varphi^k(z)/p^k\in K$, because $A^\prime$ is $p$-torsion free). The equality
\[ \varphi_1(xy)=\varphi(x) \varphi_1(y) \]
for $x, y \in K$, implies by induction that for any $n \geq 1$ :  
\[ \varphi_1^n(xy)=\varphi^n(x) \varphi_1^n(y). \]
This shows that the second divided power ideal $K^{[2]}$ is stable by $\varphi_1$ (since $K$ is stable by $\varphi$,$\varphi_1$) and, by what we just said, that the left hand side is divisible by $p^n$ in $K$. In fact, one can do better. Let $m \geq 1$ and $x \in K$. In the previous equality, take $y=x^{m-1}$. Seeing it in $B'[1/p]$ (recall that $B'$ is $p$-torsion free), one can divide both sides by $m!$. It reads :
\[ \varphi_1^n(\gamma_m(x))= \frac{\varphi^n(x)}{m!} \varphi_1^n(x^{m-1}). \]
The left hand side always makes sense in $K$ since $K$ has divided powers, and for $n$ big enough, the right hand side as well since $\varphi^n(x)$ tends $p$-adically to $0$ and thus is divisible by $m!$ for $n$ big enough. Letting $n$ go to infinity, we see that the left hand side goes to $0$ in $K$. These considerations prove that $\varphi_1$ is topologically nilpotent (with respect to the $p$-adic topology) on $K^{[2]}$, as it is topologically nilpotent on $K^2$ and all divided powers $\gamma_m(x)$, $m\geq 2$, for $x\in K$. 

Let $e$ be the degree of the polynomial $E$. Since $K^{[2]}$ is stable by $\varphi_1$, $\varphi_1$ defines a semi-linear endomorphism of the quotient $K/K^{[2]}$. Let us now prove that $\varphi_1^{pe}(K/K^{[2]}) \subset p. K/K^{[2]}$. We know that the $A'$-module $K/K^{[2]}$ is isomorphic to $(\Omega_A^{1})^{\wedge_p} \otimes_A A'$ (where the map $A \to A'$ is the natural inclusion \red{$\iota$}). It is a free $A'$-module of rank generated by $du$ and via this identification, one has $\varphi_1(du)= u^{p-1} du$. But the image of $u^{pe}$ in $A'$ is divisible by $p$ since $p$ divides $E^p$ in $A'$ and $E$ is an Eisenstein polynomial. Therefore $p$ (even $p^{p-1}$) divides $\varphi^{pe}_{1}(du \otimes 1)$ in $K/K^{[2]}$.        

Finally, let us check that these two steps imply the desired topological nilpotence. Let $x \in K$, $\bar{x}$ its class in $K/K^{[2]}$. Fix an integer $n\geq 1$. By the second step, we have 
\[ \varphi_1^{pne}(\bar{x}) \in p^n K/K^{[2]}, \]
i.e., there exists $y \in K^{[2]}$ such that  
\[ \varphi_1^{pne}(x)\in y + p^n K. \]
By the first step, there exists $m \geq 1$ such that $\varphi_1^m(y) \in p^n K$, and so 
\[ \varphi_1^{pne+m}(x) \in p^n K, \]
as desired.
\end{proof}

\begin{remark}
\label{bhatt-scholze-absolute-prysmatic-crystals}
We have seen above that prismatic Dieudonn\'e crystals over $\mathcal{O}_K$ are the same as minuscule Breuil-Kisin modules. One cannot expect the same kind of result to hold for non-minuscule finite locally free $F$-crystals on the absolute prismatic site of $\mathcal{O}_K$: one really needs to remember the (unexplicit) descent datum to reconstruct the $F$-crystal. \red{In fact, Bhatt and Scholze, \cite{bhatt2021prismatic}, have recently proved} the remarkable result that finite locally free $F$-crystals on the absolute prismatic site of $\mathcal{O}_K$ are the same as $\mathrm{Gal}_K$-stable lattices in crystalline representations of $\mathrm{Gal}_K$. In the minuscule case, i.e. for prismatic Dieudonn\'e crystals, combined with the result above and the considerations below, this recovers the known equivalence between $p$-divisible groups of $\mathcal{O}_K$ and $\mathrm{Gal}_K$-stable lattices in crystalline representations of $\mathrm{Gal}_K$ with Hodge-Tate weights in $\{0,1\}$. 
\end{remark}



Finally, let $K$ be a complete, discretely valued extension of $\mathbb{Q}_p$, let $\mathcal{O}_K\subseteq K$ be its ring of integers and assume the residue field $k$ of $\mathcal{O}_K$ is perfect. We will show that the equivalence of \Cref{sec:comp-case-mathc-theorem-comparison-with-bk} coincides with the equivalence established by Kisin (cf.\ \cite[Theorem 0.4]{kisin_crystalline_representations_and_f_crystals}). Set
$$
\mathfrak{S}:=W(k)[[u]]
$$
with Frobenius lift $\varphi\colon W(k)[[u]]\to W(k)[[u]]$ sending $u\mapsto u^p$.
Fix a uniformizer $\pi\in \mathcal{O}_K$ and define the morphism
$$
\tilde{\theta}\colon \mathfrak{S}\to \mathcal{O}_K,\ u\mapsto \pi.
$$
Then the kernel $\mathrm{ker}(\tilde{\theta})=(E)$ is generated by an Eisenstein polynomial $E\in W(k)[u]$. Let $S$ be the $p$-completed divided power envelope of the ideal $(E)\subseteq \mathfrak{S}$, i.e.,
$$
S=\mathfrak{S}\{\frac{\varphi(E)}{p}\}^\wedge_p
$$
in the category of $\delta$-rings. Note that the composition
$$
\psi_K\colon \mathfrak{S}\xrightarrow{\varphi}\mathfrak{S}\to S
$$
induces to a morphism $(\mathfrak{S},(E))\to (S,(p))$ of prisms. Via the composition $\mathcal{O}_K\cong \mathfrak{S}/(E)\xrightarrow{\psi_K} S/(p)$  we consider $(S,(p))$ as an object of the (absolute) prismatic site $(\mathcal{O}_K)_\prism$.
The antiequivalence
$$
M^{\mathrm{Kis}}(-)\colon \mathrm{BT}(\mathcal{O}_K)\cong \mathrm{BK}_{\mathrm{min}}(\mathcal{O}_K)
$$
of Kisin has the characteristic property (cf.\ \cite[Theorem (2.2.7)]{kisin_crystalline_representations_and_f_crystals}) that for a $p$-divisible group $G$ over $\mathcal{O}_K$ there is a canonical Frobenius equivariant isomorphism
$$
M^{\mathrm{Kis}}(G)\otimes_{\mathfrak{S},\psi}S\cong \mathbb{D}(G)(S)
$$
where the right hand side denotes the evaluation of the crystalline Dieudonn\'e crystal of $G$ on the PD-thickening $S\to \mathcal{O}_K$ (which sends all divided powers of $E$ to zero).

Let $G$ be a $p$-divisible group over $\mathcal{O}_K$ with absolute prismatic Dieudonn\'e crystal \red{$\mathcal{M}_\prism(G)$}. We use \Cref{sec:divid-prism-dieud-lemma-divided-dieudonne-module-via-local-ext-on-prismatic-site} and \Cref{sec:abstr-divid-prism-proposition-finite-locally-free} and consider $\red{\mathcal{M}_\prism(G)}$ as a crystal on the absolute prismatic site $(\mathcal{O}_K)_\prism$.

\begin{lemma}
  \label{sec:comp-over-mathc-comparison-over-s-with-crystalline-dieudonne-module} There is a natural Frobenius-equivariant isomorphism
  $$
  \alpha_K\colon \mathcal{M}_\prism(G)(S,(p))\xrightarrow{\simeq} \mathbb{D}(G)(S).
  $$
  Here $\mathbb{D}(G)(S)$ denotes the evaluation of the Dieudonn\'e crystal of $G$ at the PD-thickening $S\to \mathcal{O}_K$. 
\end{lemma}
\begin{proof}
  This follows from \Cref{sec:comp-over-mathc-comparison-over-pd-thickening-with-crystalline-dieudonne-module}.
\end{proof}

We want to show that the natural isomorphism $\alpha_K$ restricts to an isomorphism $\mathcal{M}_\prism(G)((\mathfrak{S},(E))\cong M^{\mathrm{Kis}}(G)$. In other words, we want to prove the existence of the dotted morphisms in the diagram
$$
\xymatrix{
  \mathcal{M}_\prism(G)((\mathfrak{S},(E))\ar@{-->}@/^1pc/[r]\ar@{_{(}->}[d] & \ar@{-->}@/^1pc/[l]\ar@{_{(}->}[d] M^{\mathrm{Kis}}(G)\\ 
  \mathcal{M}_\prism(G)(S,(p))\ar@{=}^-{\sim}[r]& \mathbb{D}(G)(S).
}
$$

Let $C$ be the completion of an algebraic closure of $K$ and let $\mathcal{O}_C\subseteq C$ be its ring of integers. Set $A_\inf:=A_\inf(\mathcal{O}_C)$, $A_\crys:=A_\crys(\mathcal{O}_C)$.

We can extend the morphism $\mathcal{O}_K\to \mathcal{O}_C$ to a morphism of prisms\footnote{Note that we take $\xi$, not $\tilxi$.}
$$
f : (\mathfrak{S},(E))\to (A_\inf,(\xi))
$$
by sending $u\mapsto \red{[\pi^\flat]}=[(\pi,\pi^{1/p},\ldots)]$ (after choosing a compatible system of $p$-power roots $\pi^{1/p^n}\in \mathcal{O}_C$ of $\pi$).
Let
$$
\psi_C\colon A_\inf\xrightarrow{\varphi}A_\inf\to A_\crys.
$$
Then analogous $\psi_C$ induces a morphism $(A_\inf,(\xi))\to (A_\crys,(p))$ of prisms.

By faithful flatness of $\mathfrak{S}\to A_\inf$ (cf.\ \cite[Lemma 4.30]{bhatt_morrow_scholze_integral_p_adic_hodge_theory}\footnote{But note that our map $f$ differs from the one of \cite{bhatt_morrow_scholze_integral_p_adic_hodge_theory}, which is $\varphi \circ f$.}) it suffices to prove the existence of the dotted arrows after base change to $A_\inf$:
\begin{equation}
  \label{commutative_diag_over_a_inf_in_comparison_with_bk_stuff}
\xymatrix{
  \mathcal{M}_\prism(G)((\mathfrak{S},(E))\otimes_{\mathfrak{S},f}A_\inf\ar@{-->}@/^1pc/[r]\ar@{_{(}->}[d] & \ar@{-->}@/^1pc/[l]\ar@{_{(}->}[d] M^{\mathrm{Kis}}(G)\otimes_{\mathfrak{S},f}A_\inf\\ 
  \mathcal{M}_\prism(G)(S,(p))\otimes_{\mathfrak{S},f}A_\inf\ar@{=}^-{\sim}[r]& \mathbb{D}(G)(S)\otimes_{\mathfrak{S},f}A_\inf.
}
\end{equation}

By flat base change of PD-envelopes (cf.\ \cite[Tag
07HD]{stacks_project}) we get
  $$
  S\hat{\otimes}_{\mathfrak{S}}A_\inf\cong A_\crys
  $$
  and thus
  $\mathbb{D}(G)(S)\otimes_{\mathfrak{S}}A_\inf\cong
  \mathbb{D}(G_{\mathcal{O_C}})(A_\crys)$.

Similar to \Cref{sec:comp-over-mathc-comparison-over-s-with-crystalline-dieudonne-module} there is a canonical isomorphism
$$
\alpha_C\colon \mathcal{M}_{\prism}(G)((A_\crys,(p))\cong \mathbb{D}(G_{\mathcal{O}_C})(A_\crys)
$$
by \Cref{sec:comp-over-mathc-comparison-over-pd-thickening-with-crystalline-dieudonne-module} and thus the lower horizontal isomorphism in (\Cref{commutative_diag_over_a_inf_in_comparison_with_bk_stuff}) identifies with $\alpha_C$.
By the crystal property of $\mathcal{M}_\prism(G)$ the left vertical injection
  $$
  \mathcal{M}_{\prism}(G)((\mathfrak{S},(E)))\otimes_{\mathfrak{S},f}A_\inf \hookrightarrow \mathcal{M}_\prism(G_{\mathcal{O}_C})(S,(p))\otimes_{\mathfrak{S},f}A_\inf.
  $$
  identifies with the inclusion
  $$
  \mathcal{M}_{\prism}(G)((A_\inf,(\xi))) \hookrightarrow \mathcal{M}_\prism(G_{\mathcal{O}_C})(A_\crys,(p))
  $$
  along the morphisms of prism $\psi_C\colon (A_\inf,(\xi))\to (A_\crys,(p))$.
  By \Cref{sec:prism-dieud-theory-proposition-comparison-scholze-weinstein-functor} there is a canonical isomorphism
  $$
  \beta\colon \varphi_{A_{\rm inf}}^* \mathcal{M}_{\prism}(G)((A_\inf,(\xi)))= \mathcal{M}_{\prism}(G)((A_\inf,(\tilxi)))\cong M^{\mathrm{SW}}(G_{\mathcal{O}_C})^*
  $$
  to the dual of the functor constructed by Scholze-Weinstein. By \cite[Theorem 14.4.3]{scholze2020berkeley} $M^{\mathrm{SW}}(G)^* \otimes_{A_\inf}A_\crys\cong \mathbb{D}(G_{\mathcal{O}_C/p})(A_\crys)$ and moreover the diagram
  $$
  \xymatrix{
    \varphi_{A_{\rm inf}}^* \mathcal{M}_{\prism}(G)((A_\inf,(\xi)))\ar[r]^{\beta}\ar[d] &  M^{\mathrm{SW}}(G_{\mathcal{O}_C})^*\ar[d] \\
\mathcal{M}_{\prism}(G)((A_\crys,(p)))\ar[r]^-\simeq & \mathbb{D}(G)(A_\crys)\cong M^{\mathrm{SW}}(G_{\mathcal{O}_C})^*\otimes_{A_\inf}A_\crys
  }
  $$
  commutes \red{by construction of $\beta$, cf.\ \Cref{sec:prism-dieud-theory-proposition-comparison-scholze-weinstein-functor} and its proof}.
  Hence, it suffices to prove that there exists an isomorphism
  $$
  \gamma\colon M^{\mathrm{Kis}}(G)\otimes_{\mathfrak{S},g}A_\inf\to M^{\mathrm{SW}}(G_{\mathcal{O}_C})^*
  $$
  where $g=\varphi \circ f$ is a morphism of prisms
  \[ g : (\mathfrak{S},(E)) \to (A_{\rm inf},(\tilxi)), \]
  such that the diagram
  $$
  \xymatrix{
    M^{\mathrm{Kis}}(G)\otimes_{\mathfrak{S},g}A_\inf \ar[d]\ar[r]^\gamma & M^{\mathrm{SW}}(G_{\mathcal{O}_C})^*\ar[d]\\
    \mathbb{D}(G_{\mathcal{O}_C})(A_\crys,(p))\ar[r]^{\simeq} & M^{\mathrm{SW}}(G_{\mathcal{O}_C})^* \otimes_{A_{\rm inf}} A_{\rm crys}
  }
  $$
  commutes.
  
    Let $T$ be the dual of the $p$-adic Tate module $T_pG$ of $G$. Then $T$ is a lattice in a crystalline representation of $\mathrm{Gal}(\overline{K}/K)$ (where $\overline{K}\subseteq C$ is the algebraic closure of $K$) and $M^{\mathrm{Kis}}(G)\cong M(T)$ where $M(-)$ is Kisin's functor from lattices in crystalline representations to Breuil-Kisin modules.
  By \cite[Proposition 4.34]{bhatt_morrow_scholze_integral_p_adic_hodge_theory} $M(T)\otimes_{\mathfrak{S},g}A_\inf$ corresponds under Fargues' equivalence (cf.  \cite[Theorem 14.1.1]{scholze2020berkeley}) to the pair $(T,\Xi)$ with $\Xi\subseteq T\otimes_{\Z_p} B_{\dR}$ the $B^+_{\dR}$-lattice generated by $D_{\dR}(T_{\Q_p})=(T\otimes_{\Z_p}B_{\dR})^{\mathrm{Gal}(\overline{K}/K)}$. But this pair is exactly the one associated to $G_{\mathcal{O}_C}$ by Scholze-Weinstein.

  Thus in the end our discussion implies the following proposition.
  
\begin{proposition}
  \label{sec:comp-over-mathc-comparison-of-kisins-functor-with-prismatic-one}
  The two functors
  $$
  \begin{matrix}
    G\mapsto M^{\mathrm{Kis}}(G)\\
    G\mapsto \mathcal{M}_{\prism}(G)(\mathfrak{S},(E))
  \end{matrix}
  $$
  from $p$-divisible groups over $\mathcal{O}_K$ to minuscule Breuil-Kisin modules are naturally isomorphic.
\end{proposition}

\subsection{\red{Admissible} prismatic Dieudonn\'e crystals and displays} 
\label{sec:comparison-with-zink-displays}
The work of Zink provides a classification of \textit{connected} $p$-divisible groups over $p$-adically complete rings (cf.\ \cite{zink_the_display_of_a_formal_p_divisible_group}). In this section, we want to relate it to the classification obtained (for quasi-syntomic rings) in \Cref{sec:essent-surj-main-theorem-equivalence-of-divided-prismatic-dieudonne-functor}.

\begin{definition} \label{def-displays}
Let $R$ be a $p$-complete ring. A \textit{display} over $R$ is a window (cf. \Cref{sec:abstr-divid-dieud} and \cite[Example 5.4]{lau_divided_dieudonne_crystals}) over the frame 
$$ \underline{W}(\mathcal{O}) = (W(\mathcal{O}), I(\mathcal{O}):=\ker(W(\mathcal{O}) \to \mathcal{O}), F, F_1), $$
in the topos of sheaves on the $p$-completely faithfully flat site of $R$, where $F$ is the Witt vector Frobenius and \red{$F_1\colon I(\mathcal{O}) \to W(\mathcal{O})$} the inverse of the bijective Verschiebung morphism $V$.

The category of displays over $R$ is denoted by $\mathrm{Disp}(R)$. 
\end{definition}

\begin{remark} \label{descent-for-displays} \green{We have phrased the definition of a display in a manner parallel to the definition of a prismatic Dieudonn\'e crystal. In this form it is however unnecessarily abstract.}
The category of displays satisfies faithfully flat descent : see \cite[Theorem 37]{zink_the_display_of_a_formal_p_divisible_group}. Since displays over a $p$-complete ring $R$ (with bounded $p^{\infty}$-torsion) are equivalent to compatible systems of displays over $R/p^n$ for all $n\geq 1$, we see that displays even satisfy $p$-completely faithfully flat descent (cf. \cite[Corollary 4.8]{bhatt_morrow_scholze_topological_hochschild_homology}). Hence the category of displays over $R$ in the sense of \Cref{def-displays} is the same as the usual category of displays over $R$ (i.e., windows over the frame $(W(R),I(R),F,F_1)$).  
\end{remark}

\begin{proposition}
\label{qrsp-p-tf-frames}
Let $R$ be a quasi-regular semiperfectoid ring. Assume that $pR=0$ or that $R$ is $p$-torsion free. The natural morphism from \Cref{sec:prism-cohom-quasi-theorem-identification-of-the-graded-pieces-of-nygaard-filtration} 
\[ \prism_R \to R \]
(given by moding out $\mathcal{N}^{\geq 1} \prism_R$) lifts to a \red{$u$-morphism of frames (in the general sense of \Cref{frame})}
\[  f: \underline{\prism}_{R,\rm Nyg} \to \underline{W}(R), \]
where $\underline{\prism}_{R,\rm Nyg}$ is the frame associated to $(\prism_R,I)$ and $\tilde{\xi}$, as in \Cref{exprism}, \red{and $u \in W(R)$ is a unit such that $p=uf(\tilxi)$}.
\end{proposition}
\begin{proof}
By adjunction (cf.\ \cite[Theoreme 4]{joyal_delta_anneaux_et_vecteurs_de_witt}), the morphism $\prism_R \to R$ gives rise to a morphism of $\delta$-rings :
\[ f\colon \prism_R \to W(R), \]
lifting the morphism to $R$, i.e., sending $\mathcal{N}^{\geq 1} \prism_R$ to $I(R)$. In particular, $f(\xi) \in I(R)$, and thus
\[ f(\tilxi) = \varphi(f(\xi)) = p \varphi_1(f(\xi)) \]
and so $p$ divides $f(\tilxi)$. By \cite[Lemma 2.24]{bhatt_scholze_prisms_and_prismatic_cohomology}, \red{we deduce that $(p)=(f(\tilxi))$, and thus there exists a unit $u\in W(R)$ such that $p=uf(\tilxi)$}. It is then easy to conclude when $W(R)$ is $p$-torsion free since the commutation (up to a unit) of $f$ with the divided Frobenius can be proved after multiplying by $p$. In the case where $pR=0$ one argues as in \cite[Lemma 7.4]{lau_dieudonne_theory_over_semiperfect_rings_and_perfectoid_rings}.   
\end{proof}

It would be nice to prove that for any $R$ quasi-regular semiperfectoid, the morphism of the proposition always defines a morphism of frames. Although we did not succeed in doing so, the next proposition shows that one can circumvent this difficulty.

\begin{proposition}
  \label{sec:filt-prism-dieud-proposition-display-associated-to-p-divisible-group}
Let $R$ be a quasi-syntomic ring. \red{If $G$ is a $p$-divisible group over $R$, set
$$
Z_R(G) = \mathcal{M}_\prism(G) \otimes_{\mathcal{O}^{\rm pris}} W(\mathcal{O})
$$
with Frobenius $F_{Z_R(G)}= \varphi_{\mathcal{M}_\prism(G)} \otimes F$, and let $\mathrm{Fil} ~ Z_R(G)$ be the submodule of $Z_R(G)$ generated by $I(\mathcal{O}).Z_R(G)$ and the image of $\varphi_{\mathcal{M}_\prism(G)}^{-1}(\mathcal{I}^{\rm pris}. \mathcal{M}_\prism(G))$. There exists a unique way to extend the functor
$$
G \mapsto (Z_R(G),\mathrm{Fil} ~ Z_R(G), F_{Z_R(G)})
$$
to a functor
\[ \underline{Z}_R : \mathrm{BT}(R) \to \mathrm{Disp}(R), ~ G  \mapsto \underline{Z}_R(G)=(Z_R(G),\mathrm{Fil} ~ Z_R(G), F_{Z_R(G)}, F_{Z_R(G),1}) \]
natural in $R$ which moreover coincides (through \Cref{sec:abstr-divid-prism-proposition-equivalence-divided-prismatic-dieudonne-modules-windows}) with the composition of the prismatic Dieudonn\'e functor with the functor induced by the morphism of frames of \Cref{qrsp-p-tf-frames} when $R$ is quasi-regular semiperfectoid and $pR=0$ or $R$ is $p$-torsion free.}
\end{proposition}
\begin{proof}
The requirement of the proposition already says what 
$$(Z_R(G), \mathrm{Fil} ~ Z_R(G), F_{Z_R(G)})$$
must be. Therefore, the only issue is to define the divided Frobenius $F_{Z_R(G),1}$. 

Assume first that $R$ is quasi-regular semiperfectoid and $p$-torsion free. \red{If it exists, $F_{Z_R(G),1}$ is necessarily unique, since $W(R)$ is $p$-torsion free; thus we only need to show its existence.} For this, we define $\underline{Z}_R$ as the composition of the prismatic Dieudonn\'e functor with the functor induced by the morphism of frames of \Cref{qrsp-p-tf-frames}. By quasi-syntomic descent (\Cref{descent-for-displays}), one gets a functor $\underline{Z}_R$ for any $p$-torsion free quasi-syntomic ring $R$. For such rings $R$, the functor $\underline{Z}_R$ is necessarily unique by $p$-torsion freeness of $W(R)$. In particular, it commutes with base change in $R$. 

To obtain the functor $\underline{Z}_R$ in general, we use smoothness of the stack of $p$-divisible groups, following an idea of Lau, \cite[Proposition 2.1]{lau_smoothness_of_the_truncated_display_functor}. Let $X=\mathrm{Spec}(A) \to \mathcal{BT} \times \mathrm{Spec}(\Z_p)$ be \red{an ind-smooth} presentation of the stack of $p$-divisible groups as in loc.\ cit. Then $\mathrm{Spec}(B)=X \times_{\mathcal{BT}} X$ is affine. The $p$-adic completions $\hat{A}$ and $\hat{B}$ are both $p$-torsion free (cf.\ \cite[Lemma 1.6]{lau_smoothness_of_the_truncated_display_functor}).

Let $R$ be a quasi-syntomic ring and $G$ be a $p$-divisible group over $R$. It gives rise to a map $\alpha : \mathrm{Spec}(R) \to \mathcal{BT} \times \mathrm{Spec}(\Z_p)$. Let
\[ \mathrm{Spec}(S) = \mathrm{Spec}(R) \times_{\mathcal{BT} \times \mathrm{Spec}(\Z_p)} \mathrm{Spec}(A), \]
and
\[ \mathrm{Spec}(T) = \mathrm{Spec}(S) \times_{\mathrm{Spec}(A)} \mathrm{Spec}(B). \]
Let $\hat{S}$ and $\hat{T}$ be their $p$-adic completions. The rings $\hat{A}$ and $\hat{B}$ \red{are quasi-syntomic}. By base change the rings $\hat{S}$ and $\hat{T}$ are also quasi-syntomic. The base change
\[ (Z_{\hat{S}}(G_{\hat{S}}), \mathrm{Fil} ~ Z_{\hat{S}}(G_{\hat{S}}), F_{Z_{\hat{S}}(G_{\hat{S}})}) \]
of the triple $(Z_R(G), \mathrm{Fil} ~ Z_R(G), F_{Z_R(G)})$ along $R \to \hat{S}$ is also the base change of the triple 
\[ (Z_{\hat{A}}(H_{\hat{A}}), \mathrm{Fil} ~ Z_{\hat{A}}(H_{\hat{A}}), F_{Z_{\hat{A}}(H_{\hat{A}})}) \]
along $\alpha \otimes \hat{A}$ of the universal $p$-divisible group $H$ over $A$. The divided Frobenius $F_{Z_{\hat{A}}(H_{\hat{A}}),1}$ on $Z_{\hat{A}}(H_{\hat{A}})$ (coming from the first part of the proof) therefore induces an operator $F_{Z_{\hat{S}}(G_{\hat{S}}),1}$ on $Z_{\hat{S}}(G_{\hat{S}})$. This operator $F_{Z_{\hat{S}}(G_{\hat{S}}),1}$ is compatible with the descent datum for the base change along the two natural maps $\hat{S} \to \hat{T}$, since the functor $Z_{\hat{B}}$ exists and is unique. By descent (\Cref{descent-for-displays}), this defines a display structure $\underline{Z}_R(G)$ on the triple $(Z_R(G), \mathrm{Fil} ~ Z_R(G), F_{Z_R(G)})$.

This display structure is uniquely determined by the requirement that it is compatible with the maps $R \to \hat{S}$, $\hat{S} \to \hat{A}$. In particular, it has to coincide with the composition of the prismatic Dieudonn\'e functor with the functor induced by the morphism of frames of \Cref{qrsp-p-tf-frames} also when $R$ is quasi-regular semiperfectoid and killed by $p$.        
\end{proof}

The functor of \Cref{sec:filt-prism-dieud-proposition-display-associated-to-p-divisible-group} is not an antiequivalence when $p=2$. Nevertheless, one has the following positive result, reproving the main result of \cite{zink_the_display_of_a_formal_p_divisible_group}, \cite{lau_displays_and_formal_p_divisible_groups} in the special case of quasi-syntomic rings.

\begin{proposition}
  \label{sec:filt-prism-dieud-equivalence-to-displays}
Let $R$ be a quasi-syntomic ring, flat over $\Z/p^n$ (for some $n>0$) or $\Z_p$. The functor $\underline{Z}_R$ restricts to an antiequivalence 
\[ \mathrm{BT}_f(R) \cong \mathrm{Disp}_{\rm nilp}(R) \]
between the category of formal $p$-divisible groups over $R$ and the category of $F$-nilpotent displays over $R$. 
\end{proposition}
Recall that a display is said to be \textit{$F$-nilpotent} if its Frobenius is nilpotent modulo $p$. 
\begin{proof}
Assume first that $R$ is quasi-regular semiperfect. The functor $\underline{Z}_R$ is the composite of the prismatic Dieudonn\'e functor, which is an antiequivalence by \Cref{sec:divid-prism-dieud-modul-theorem-main-theorem-of-the-paper}, and of the functor induced by the morphism of frames
\[ (\prism_R \cong A_{\rm crys}(R), \mathcal{N}^{\geq 1} \prism_R, \varphi, \varphi_1) \to (W(R), I(R), F, F_1). \]
The morphism $\prism_R \to W(R)$ is surjective (\red{indeed, the composition $\prism_{R^\flat} \cong W(R^\flat) \to W(R)$ is surjective, since $R^\flat \to R$ is, and factors through the map $\prism_R \to W(R)$}). \green{We claim that the divided Frobenius is topologically nilpotent on its kernel. It suffices to show the same for the surjection $A_{\mathrm{crys}}(R)\to W(R)$ coming from the PD-thickening $W(R)\to R$. We recall that $A_{\mathrm{crys}}(R)$ is obtained from $W(R^\flat)$ by passing to the PD-envelope for the ideal $\mathrm{ker}\blue{( W(R^\flat)\to R)}$. This kernel is (topologically) generated by the elements $V^m([x])$ for $m\geq 0$ and $x\in I:=\mathrm{ker}(R^\flat\to R)$. If $m\geq 1$, then $V^m([x])\in W(R^\flat)$ already has divided powers. As $A_{\mathrm{crys}}(R)$ is $p$-torsion free (by quasi-regularity of $R$), we can conclude that $A_{\mathrm{crys}}(R)$ is (topologically) generated (as a module over $W(R^\flat)$) by the divided powers $[x]^{(n)}$ of $[x]$ for $x\in I$ (i.e., the divided powers of $V^m([x])$ for $m\geq 1$ are not necessary). We note that for $x\in I$ each divided power $[x]^{(n)}\in A_{\mathrm{crys}}(R)$ lies in the kernel of $A_{\crys}(R)\to W(R)$ because $[x]\in W(R^\flat)$ maps to $0\in W(R)$. Hence, we can conclude that the kernel of $A_{\crys}(R)\to W(R)$ is (topologically) generated by $V^m([x]), [x]^{(n)}$ for $x\in I$ and $n,m\geq 1$. Now $V^m([x])=p^m[x^{1/p^m}]$ and thus $\varphi_1^m(V^m([x]))=[x]$. Hence, it suffices to show that $\varphi_1$ is topologically nilpotent on the elements $[x]^{(n)}, n\geq 1, x\in I$.}
For such an element, one has
\[ \varphi_1([x]^{(n)}) = \frac{(np)!}{n!p} [x]^{(np)}. \]
Iterating, one sees that $\varphi_1$ is topologically nilpotent on the kernel (with respect to the $p$-adic topology). By \red{\Cref{sec:abstr-filt-prism-remark-faithfulness-on-windows-if-divided-frob-is-topologically-nilpotent}}, the functor
\[ \red{\DM^{\rm adm}(R)\cong \mathrm{Win}(\underline{\prism}_{R, \rm Nyg})} \to \mathrm{Disp}(R) \]
is an equivalence. It is easily seen that it restricts to an antiequivalence between formal $p$-divisible groups and $F$-nilpotent displays.

By quasi-syntomic descent, this yields the statement of the proposition when $R$ is quasi-syntomic with $pR=0$. In general, $R/p$ is quasi-syntomic (\cite[Lemma 4.16 (2)]{bhatt_morrow_scholze_topological_hochschild_homology}) and one can consider the following commutative diagram :
$$
\xymatrix{
  \mathrm{BT}(R) \ar[r]^-{\underline{Z}_R} \ar[d] & \mathrm{Disp}(R) \ar[d] \\
  \mathrm{BT}(R/p) \ar[r]^-{\underline{Z}_{R/p}} & \mathrm{Disp}(R/p).
}
$$  
Grothendieck-Messing theory for $F$-nilpotent displays (cf. \cite[Theorem 48]{zink_the_display_of_a_formal_p_divisible_group}) coupled with Grothendieck-Messing theory for $p$-divisible groups (cf.\ \cite[V (1.6)]{messing_the_crystals_associated_to_barsotti_tate_groups} and \cite[Corollary 97]{zink_the_display_of_a_formal_p_divisible_group}) show that this diagram is $2$-cartesian. Since $\underline{Z}_{R/p}$ is an antiequivalence, $\underline{Z}_R$ also is one.
\end{proof}

\subsection{\'Etale comparison for $p$-divisible groups}
\label{sec:etale-comparison-p}

Let $R$ be a quasi-syntomic ring and let $G$ be a $p$-divisible group over $R$.
In this section we show how the (dual of the) Tate module of the generic fiber of $R$, seen as a diamond (\cite[Definition 11.1]{scholze_etale_cohomology_of_diamonds}), can be recovered from the prismatic Dieudonn\'e crystal $\mathcal{M}_\prism(G)$ of $G$.

Let
$$
\mathcal{O}^{\pris}
$$
be the prismatic sheaf on $(R)_{\mathrm{qsyn}}$ and
$$
\mathcal{I}:=\mathcal{I}^{\mathrm{pris}}\subseteq \mathcal{O}^{\pris}
$$
the natural invertible $\mathcal{O}^\pris$-module (cf.\ \Cref{sec:abstr-divid-prism-definition-of-the-sheaves}).
Fix $n\geq 0$.
Note that the Frobenius
$$
\varphi\colon \mathcal{O}^{\pris}\to \mathcal{O}^\pris
$$
induces a morphism, again called Frobenius,
$$
\varphi\colon \mathcal{O}^\pris/p^n[1/\mathcal{I}]\to \mathcal{O}^\pris/p^n[1/\mathcal{I}]
$$
as $\varphi(\mathcal{I})\subseteq (p,\mathcal{I})$ although $\mathcal{I}$ is not stable under $\varphi$.

We let
$$
(R)_{v}
$$
be the $v$-site of all maps $\Spf(S)\to \Spf(R)$ with $S$ a perfectoid ring over $R$. By definition the coverings in $(R)_v$ are $v$-covers $\Spf(S^\prime)\to \Spf(S)$ (cf.\  \cite[Section 8.1]{bhatt_scholze_prisms_and_prismatic_cohomology}).
Let
$$
(R)_{\mathrm{qsyn},\mathrm{qrsp}}
$$
be the site of all maps $\Spf(S)\to \Spf(R)$ with $S$ quasi-regular semiperfectoid (covers given by quasisyntomic covers). The perfectoidization functor
$$
S\mapsto S_{\mathrm{perfd}}
$$
from \cite[Definition 8.2]{bhatt_scholze_prisms_and_prismatic_cohomology} induces a \red{morphism of sites}
$$
\alpha\colon (R)_{v}\to (R)_{\mathrm{qsyn},\mathrm{qrsp}}
$$
sending $\Spf(S)$ to $\Spf(S_{\mathrm{perfd}})$. Indeed, by \cite[Proposition 8.10]{bhatt_scholze_prisms_and_prismatic_cohomology} and the fact that quasi-syntomic covers are $v$-covers the conditions of \cite[Tag 00WV]{stacks_project} are satisfied. 
Moreover, we have the ``inclusion of the generic fiber''
$$
j\colon \Spa(R[1/p],R)_{v}^\diamond\to (R)_v
$$
induced by sending $\Spf(S)$ to $\Spa(S[1/p],S)$\footnote{We use the notation $\Spa(S[1/p],S)$ when $S$ is not necessarily integrally closed in $S[1/p]$.}. Here $\Spa(R[1/p],R)^\diamond_v$ is the $v$-site of the diamond associated with $\Spa(R[1/p],R)$ (cf.\ \cite[Section 15.1]{scholze_etale_cohomology_of_diamonds}, \cite[Definition 14.1.iii)]{scholze_etale_cohomology_of_diamonds}).

\red{The sites $(R)_v$, $(\Spa(R[1/p],R))_v$ carry tilted structure sheaves $\mathcal{O}^\flat_{(R)_v}$, $\mathcal{O}^\flat$ sending $S\in (R)_v$ to $S^\flat$ resp.\ $\Spa(S,S^+)\in (\Spa(R[1/p],R))_v$ to $S^\flat$. We let $W(\mathcal{O}^\flat_{(R)_v})$ resp.\ $W(\mathcal{O}^\flat)$ be the associated Witt vector sheaves. It is easy to see that for every $n\geq 1$ there are natural morphism $\mathcal{O}^{\mathrm{pris}}/p^n\to \alpha_\ast(W_n(\mathcal{O}^{\flat}_{(R)_v}))$, $\mathcal{O}^\pris/p^n[1/\mathcal{I}]\to (\alpha\circ j)_{\ast}(W_n(\mathcal{O}^\flat))$. }

\begin{lemma}
  \label{sec:etale-comparison-p-1-frobenius-fixed-points-of-the-prismatic-structure-sheaf} \red{The above morphisms induces} natural isomorphisms
  $$
  \alpha_\ast(\Z/p^n)\cong (\mathcal{O}^\pris/p^n)^{\varphi=1}
  $$
  and
  $$
  (\alpha\circ j)_\ast(\Z/p^n)\cong (\mathcal{O}^\pris/p^n[1/\mathcal{I}])^{\varphi=1}
  $$
  of sheaves on $(R)_{\mathrm{qsyn},\mathrm{qrsp}}$ \red{after passing to $\varphi$-fixed points}.
\end{lemma}
Here $(-)^{\varphi=1}$ denotes the (non-derived) invariants of $\varphi$ on the sheaf $\mathcal{O}^\pris/p^n[1/\mathcal{I}]$, \red{and we use that $W_n(\mathcal{O}^\flat_{(R)_v})\cong \Z/p^n$, $W_n(\mathcal{O}^\flat)\cong \Z/p^n$ as will be explained in the proof}.
\begin{proof}
  We only prove the second statement. The first is similar (but easier). Let $S$ be a quasi-regular semiperfectoid $R$-algebra. Then
  $$
(\mathcal{O}^\pris/p^n[1/\mathcal{I}])^{\varphi=1}(S) \cong (\varinjlim\limits_{\varphi}\mathcal{O}^\pris/p^n[1/\mathcal{I}])^{\varphi=1}(S)\cong (\varinjlim\limits_{\varphi}\mathcal{O}^\pris(S)/p^n\red{[1/\mathcal{I}]})^{\varphi=1}.
$$
The first isomorphism follows from commuting \green{Frobenius fixed points the filtered colimit over $\N$ along $\varphi$} and the second as \red{$\varinjlim\limits_{\varphi}\mathcal{O}^\pris$ is $p$-torsion free (cf. \cite[Proof of Lemma 2.28]{bhatt_scholze_prisms_and_prismatic_cohomology}) and $S$ is quasi-regular semiperfectoid (which implies that the sheaf $\varinjlim\limits_{\varphi}\mathcal{O}^\pris$ has no higher cohomology over $S$)}.
Then \cite[Lemma 9.2]{bhatt_scholze_prisms_and_prismatic_cohomology} implies that
$$
(\varinjlim\limits_{\varphi}\mathcal{O}^\pris(S)/p^n\red{[1/\mathcal{I}]})^{\varphi=1}\cong (A_{\inf}(S_{\mathrm{perfd}})/p^n[1/\mathcal{I}])^{\varphi=1}.
$$
By \cite[Lemma 9.3]{bhatt_scholze_prisms_and_prismatic_cohomology}, the equivalence of underlying topological spaces under tilting of perfectoid spaces, \cite[Theorem 7.1.1]{scholze2020berkeley}, and \cite[Proposition 3.2.7]{kedlaya_liu_relative_p_adic_hodge_theory_foundations} the right-hand side becomes
$$
W_n((S_{\mathrm{perfd}}[1/p])^\flat)^{\varphi=1}\cong \mathrm{Hom}_{\mathrm{cts}}(\pi_0(\Spa(S_{\mathrm{perfd}}[1/p],S_{\mathrm{perfd}})),\Z/p^n),
$$
which agrees with
$$
(\alpha\circ j)_\ast(\Z/p^n)(S).
$$
This finishes the proof.
\end{proof}

We can derive the following description of the Tate module of the generic fiber.

\begin{proposition}
  \label{sec:etale-comparison-p-2-etale-comparison-for-prismatic-dieudonne-module-of-p-divisible-group}
  Let $G$ be a $p$-divisible group over $R$ with prismatic Dieudonn\'e crystal $\mathcal{M}_\prism(G)$ and let $n\geq 0$. Then
  $$
  j^\ast\alpha^\ast(\mathcal{M}_{\prism}(G)/p^n[1/\mathcal{I}]^{\varphi=1})
  $$
  is canonically isomorphic to $\mathcal{H}om_{\Z/p^n}(G[p^n]_{\eta},\Z/p^n)$ where $G[p^n]_\eta$ denotes the sheaf $\Spa(S[1/p],S)\mapsto G[p^n](S[1/p])$ on $\Spa(R[1/p],R)^\diamond_v$.
\end{proposition}
\begin{proof}
Set $\mathcal{M}:=\mathcal{M}_\prism(G)$. By \Cref{sec:divid-prism-dieud-lemma-prismatic-dieudonne-module-via-hom}
  $$
  \mathcal{M}\cong \mathcal{H}om_{(R)_{\mathrm{qsyn},\mathrm{qrsp}}}(T_pG,\mathcal{O}^\pris).
  $$
  From the proof of \Cref{sec:divid-prism-dieud-proposition-ext-crystal} we can conclude that
\begin{equation*}
\begin{split}
  \mathcal{H}om_{(R)_{\mathrm{qsyn},\mathrm{qrsp}}}(T_pG,\mathcal{O}^\pris)/p^n & \cong  \mathcal{H}om_{(R)_{\mathrm{qsyn},\mathrm{qrsp}}}(T_pG,\mathcal{O}^\pris/p^n) \\
     &   \cong  \mathcal{H}om_{(R)_{\mathrm{qsyn},\mathrm{qrsp}}}(G[p^n],\mathcal{O}^\pris/p^n).
\end{split}
\end{equation*}
  It follows that
  $$
  \mathcal{M}/p^n[1/\mathcal{I}]\cong \mathcal{H}om_{(R)_{\mathrm{qsyn},\mathrm{qrsp}}}(G[p^n],\mathcal{O}^\pris/p^n[1/\mathcal{I}])
  $$
  as using \Cref{sec:calc-ext-groups} the functor $\mathcal{H}om_{(R)_{\mathrm{qsyn},\mathrm{qrsp}}}(G[p^n],-)$ commutes with filtered colimits.
  Finally,
  $$
\mathcal{M}/p^n[1/\mathcal{I}]^{\varphi=1}\cong \mathcal{H}om_{(R)_{\mathrm{qsyn},\mathrm{qrsp}}}(G[p^n],\mathcal{O}^\pris/p^n[1/\mathcal{I}]^{\varphi=1}).
$$
By \Cref{sec:etale-comparison-p-1-frobenius-fixed-points-of-the-prismatic-structure-sheaf}
$$
\mathcal{O}^\pris/p^n[1/\mathcal{I}]^{\varphi=1}\cong (\alpha\circ j)_{\ast}(\Z/p^n)
$$
and thus
\begin{equation*}
\begin{split}
\mathcal{M}/p^n[1/\mathcal{I}]^{\varphi=1} &\cong  \mathcal{H}om_{(R)_{\mathrm{qsyn},\mathrm{qrsp}}}(G[p^n],(\alpha\circ j)_\ast(\Z/p^n)) \\
& \cong (\alpha\circ j)_{\ast}(\mathcal{H}om_{\Z/p^n}((\alpha\circ j)^\ast G[p^n],\Z/p^n))). 
\end{split}
\end{equation*}
\red{The definitions of $\alpha$ and $j$ imply that for any sheaf $\mathcal{F}$ on $(R)_{\mathrm{qsyn, qrsp}}$ the non-sheafified pullback $(\alpha\circ j)^{-1}\mathcal{F}$ is the presheaf $\Spa(S,S^+) \mapsto \mathcal{F}(\Spf(S^+))$. In particular, we see that
\[
  (\alpha\circ j)^\ast\circ (\alpha\circ j)_\ast
\]
is naturally isomorphic to the identity. We obtain thus
$$
(\alpha\circ j)^\ast \mathcal{M}/p^n[1/\mathcal{I}]^{\varphi=1}\cong\mathcal{H}om_{\Z/p^n}((\alpha\circ j)^\ast G[p^n],\Z/p^n).
$$
and can now conclude by} \Cref{sec:etale-comparison-p-1-pullback-of-finite-flat-group-scheme}.
\end{proof}

\begin{lemma}
  \label{sec:etale-comparison-p-1-pullback-of-finite-flat-group-scheme}
  With the notations from \Cref{sec:etale-comparison-p-2-etale-comparison-for-prismatic-dieudonne-module-of-p-divisible-group},
  $$
  (\alpha\circ j)^\ast G[p^n]\cong G[p^n]_\eta.
  $$
\end{lemma}
\begin{proof}
  By right exactness of $(\alpha\circ j)^\ast$, it suffices to show
  $$
  (\alpha\circ j)^\ast T_pG\cong T_pG_{\eta}.
  $$
  Moreover, we may assume that $R$ is perfectoid by passing to slice topoi. Let $S$ be the $R$-algebra representing $T_pG$ on $p$-complete rings. Thus $S$ is the $p$-completion of $\varinjlim\limits_m S_m$ where $S_m$ represents $G[p^m]$. Then $S$ is quasi-regular semiperfectoid. By definition $(\alpha\circ j)^{\ast}T_pG$ is represented by the perfectoid space
  $$
  \Spa(S_{\mathrm{perfd}}[1/p],S_{\mathrm{perfd}}^+)
  $$
  over $\Spa(R[1/p],R)$
  where $S_{\mathrm{perfd}}^+$ is the integral closure of $S_{\mathrm{perfd}}$ in $S_{\mathrm{perfd}}[1/p]$. Let $\Spa(T,T^+)$ be an affinoid perfectoid space over $\Spa(R[1/p],R)$, in particular we assume that $T^+$ is integrally closed in $T=T^+[1/p]$. Then any morphism $S_{\mathrm{perfd}}[1/p]\to T$ sends $S_{\mathrm{perfd}}^+\to T^+$ because $S$ is a $p$-completed direct limit of finite $R$-algebras and $T^+$ is perfectoid and integrally closed in $T$. Thus
\begin{equation*}
  \begin{split}  
    \mathrm{Hom}_{(R[1/p],R)}((S_{\mathrm{perfd}}[1/p],S_{\mathrm{perfd}}^+),(T,T^+)) & \cong  \mathrm{Hom}_{R}(S_{\mathrm{perfd}}^+,T^+) \\
  &  \cong  \mathrm{Hom}_R(S,T^+) \\
   & \cong  \mathrm{Hom}_R(\varinjlim\limits_m S_m,T^+) \\
   & \cong  \mathrm{Hom}_R(\varinjlim\limits_m S_m,T)=T_pG(T)
 \end{split}
\end{equation*}
  where $S_m$ represents $G[p^m]$ (thus $S$ is the $p$-adic completion of $\varinjlim\limits_m S_m)$). In the last isomorphism we used again that all $S_m$ are finite over $R$ and thus any morphism $S_m\to T$ of $R$-algebras factors over $T^+$.
\end{proof}

\newpage
\appendix
\section{Descent for $p$-completely faithfully flat morphisms}
\label{sec:descent-p-completely}

\green{
In this appendix we want to record some descent statements that are used in the main body of this text.

\begin{lemma}
  \label{sec:descent-p-completely-1-description-of-finite-projective-modules-on-p-complete-rings}
  Let $R$ be derived $p$-complete ring with bounded $p^\infty$-torsion. Then the natural functor
  $$
    \{\text{ finite projective }R-\text{modules}\}\to 2-\varprojlim\limits_n \{\text{ finite projective }R/p^n-\text{modules}\}
    $$
    is an equivalence.
    In particular, the fibered category $R\mapsto \{\text{ finite projective }R-\text{modules}\}$ is a stack for the $p$-completely faithfully flat topology on the category of derived $p$-complete rings with bounded $p^\infty$-torsion.
\end{lemma}
\begin{proof}
  As $R$ is classically $p$-complete the first statement follows from \cite[Tag 0D4B]{stacks_project}. If $R\to R^\prime$ is a $p$-completely faithfully flat morphism between $p$-complete rings of bounded $p^\infty$-torsion, then $R/p^n\to R^\prime/p^n$ is faithfully flat for all $n\geq 0$ (flatness follows from \cite[Lemma 4.7.(2)]{bhatt_morrow_scholze_topological_hochschild_homology} and surjectivity of $\Spec(R^\prime/p^n)\to \Spec(R/p^n)$ is implied by the case $n=1$). Thus, classical descent of finite projective modules holds for this morphism. Passing to the ($(2)$-)inverse limit implies the last statement. 
\end{proof}
}
\begin{proposition}
  \label{sec:fully-faithf-prism-proposition-bt-is-a-stack}

  The fibered categories of $p$-divisible groups and finite locally free group schemes over $p$-complete rings with bounded $p^\infty$-torsion are stacks for the $p$-completely faithfully flat topology. 
\end{proposition}
\begin{proof}
  It suffices to show the statement for finite locally free group schemes as $p$-divisible groups are canonically a colimit of such.
 From \ref{sec:descent-p-completely-1-description-of-finite-projective-modules-on-p-complete-rings} we know that finite locally free modules form a stack for the $p$-completely faithfully flat topology on $p$-complete rings with bounded $p^\infty$-torsion. As base change commutes with fiber products, this implies that finite locally free group schemes form a stack, too.
\end{proof}

  Recall that a morphism
  $$
  (A,I)\to (B,J)
  $$
  of prisms is called faithfully flat if it is $(p,I)$-completely flat.

  \begin{proposition}
    \label{sec:descent-p-completely-1-proposition-descent-of-finite-projective-modules-for-faithfully-morphisms-of-prisms}
    The fibered category
    $$
    (A,I)\mapsto \{\text{ finite projective }A-\text{modules}\}
    $$
    on the category of bounded prisms is a stack for the faithfully flat topology.
  \end{proposition}
  \begin{proof}
    If $(A,I)$ is a prism, then $A$ is classically $I$-complete and thus finite projective $A$-modules are equivalent to compatible systems of finite projective $A/I^n$-modules, i.e.,
    $$
    \{\text{ finite projective }A-\text{modules}\}\cong 2-\varprojlim\limits_n\{\text{ finite projective }A/I^n-\text{modules}\}
    $$
    (cf.\ \cite[Tag 0D4B]{stacks_project}).
    As the $2$-limit of stacks is again a stack it suffices to show that for any $n\geq 0$ the fibered category
    $$
    (A,I)\mapsto \{\text{ finite projective }A/I^n-\text{modules}\}
    $$
    is a stack on bounded prisms. If $(A,I)\to (B,J)$ is a faithfully flat morphim of prisms, then
    $$
    A/I^n\to B/J^n
    $$
    is a $p$-completely faithfully flat morphism of rings with bounded $p^\infty$-torsion. Thus the proposition follows from \ref{sec:descent-p-completely-1-description-of-finite-projective-modules-on-p-complete-rings}.
  \end{proof}
  
  \begin{example}
    We give an example of a ring $R$ which is classically $(p,f)$-complete where $f\in R$ is a non-zero divisor, such that $R/f$ has bounded $p^\infty$-torsion, but $R$ has unbounded $p^\infty$-torsion.
    Set
    $$
    R:=\Z[f,x_{i,j}|\ i\geq 0, 0\leq j\leq i]^{\wedge_{(p,f)}}/J
    $$
    with $J$ generated by the elements
    $$
    px_{i,j}-fx_{i,j+1}
    $$
    (where $x_{i,i+1}:=0$).
    Then $f$ is a non-zero divisor in $R$ and all $p^\infty$-torsion in 
    $$
    R/f\cong \Z[x_{i,j}]/(px_{i,j})
    $$
    is killed by $p$. But
    $$
    p^{i}x_{i,0}=p^{i}fx_{i,1}=\ldots = f^{i}x_{i,i}\neq 0
    $$
    while $p^{i+1}x_{i,0}=f^ipx_{i,i}=0$.
    This shows that $R$ has unbounded $p^\infty$-torsion.
    As $f$ is a non-zero divisor in $R$ the $(p,f)^\infty$-torsion in $R$ is zero.
  \end{example}

\newpage

\bibliography{biblio}
\bibliographystyle{plain}

\end{document}